\documentclass[a4paper]{article}
\usepackage{amsmath,amsthm,amsfonts,amssymb,amscd,bm,makeidx,indentfirst,tocloft}

\newcommand{\e}{\ensuremath{\epsilon}}

\newcommand{\nn}{\ensuremath{\textbf{n}}}
\newcommand{\NN}{\ensuremath{\textbf{N}}}

\newcommand{\rto}{\ensuremath{\rightarrow}}
\newcommand{\lem}{\ensuremath{\lesssim}}

\newtheorem{theorem}{Theorem}[section]

\newtheorem{lemma}[theorem]{Lemma}

\newtheorem{proposition}[theorem]{Proposition}
\newtheorem{remark}[theorem]{Remark}

\numberwithin{equation}{section}

\title{\Large Regularity Structure, Vorticity Layer and Convergence Rates of Inviscid Limit of Free Surface Navier-Stokes Equations
with or without Surface Tension}
\author{\normalsize Fuzhou Wu\thanks{E-mail: michael8723@gmail.com; fuzhou.wu@yahoo.com; wufz12@mails.tsinghua.edu.cn} \\
\small\it  Yau Mathematical Sciences Center, Tsinghua University\\
\small\it  Beijing 100084, China \\[3pt]
\small\it  Center of Mathematical Sciences and Applications, Harvard University\\
\small\it  Cambridge, Massachusetts 02138, USA
}
\date{}

\begin{document}
\maketitle
\setlength\parindent{2em}
\setlength\parskip{5pt}

\begin{abstract}
\normalsize{
In this paper, we study the inviscid limit of the free surface incompressible Navier-Stokes equations with or without surface tension.
By delicate estimates, we prove the weak boundary layer of the velocity of the free surface Navier-Stokes equations and
the existence of strong or weak vorticity layer for different conditions.
When the limit of the difference between the initial Navier-Stokes vorticity and the initial Euler vorticity
is nonzero, or the tangential projection on the free surface of the Euler strain tensor multiplying by normal vector is nonzero,
there exists a strong vorticity layer. Otherwise, the vorticity layer is weak.
We estimate convergence rates of tangential derivatives and the first order standard normal derivative in energy norms,
we show that not only tangential derivatives and standard normal derivative have different convergence rates,
but also their convergence rates are different for different Euler boundary data.
Moreover, we determine regularity structure of the free surface Navier-Stokes solutions with or without surface tension,
surface tension changes regularity structure of the solutions.
}
\end{abstract}

%\newpage
\tableofcontents

%%% find 1
\section{Introduction}
In this paper, we study the inviscid limit of the free surface incompressible Navier-Stokes equations with or without surface tension
(see \cite{Masmoudi_Rousset_2012_FreeBC,Wang_Xin_2015,Elgindi_Lee_2014}):
\begin{equation}\label{Sect1_NavierStokes_Equation}
\left\{\begin{array}{ll}
u_t + u\cdot\nabla u + \nabla p = \e\triangle u,\hspace{1.5cm} x\in\Omega_t, \\[7pt]
\nabla\cdot u =0, \hspace{3.82cm} x\in\Omega_t,\\[7pt]
\partial_t h = u\cdot \NN, \hspace{3.45cm} x\in\Sigma_t,\\[7pt]
p\nn -2\e \mathcal{S}u\,\nn =gh\nn -\sigma H\nn, \hspace{1.07cm} x\in\Sigma_t,\\[7pt]
(u,h)|_{t=0} = (u_0^{\e},h_0^{\e}).
\end{array}\right.
\end{equation}
where $x=(y,z)$, $y$ is the horizontal variable, $z$ is the vertical variable,
the normalized pressure $p=p^F + g z$, $p^F$ is the hydrodynamical pressure of the fluid, $g z$ corresponds to the gravitational force.
The surface tension in the dynamical boundary condition $(\ref{Sect1_NavierStokes_Equation})_4$, namely
$H = - \nabla_x \cdot \big(\frac{(-\nabla_y h,1)}{\sqrt{1+|\nabla_y h|^2}}\big)
= \nabla_y\cdot \big(\frac{\nabla_y h}{\sqrt{1+|\nabla_y h|^2}}\big)$, is twice the mean curvature of the free surface $\Sigma_t$.
The initial data satisfies the compatibility condition $\Pi\mathcal{S}u_0^{\e} \nn|_{z=0} =0$.
Some notations are defined as follows:
\begin{equation}\label{Sect1_FreeSurface_Definition}
\begin{array}{ll}
\Omega_t =\{x\in\mathbb{R}^3|\, -\infty< z<h(t,y)\},\\[6pt]
\Sigma_t = \{x\in\mathbb{R}^3|\, z =h(t,y)\},\\[6pt]
\NN=(-\nabla h, 1)^{\top},\quad  \nn=\frac{\NN}{|\NN|}, \\[6pt]
\mathcal{S}u =\frac{1}{2}(\nabla u +(\nabla u)^{\top}),
\end{array}
\end{equation}
where the symbol $^{\top}$ means the transposition of matrices or vectors.
We suppose $h(t,y)\rto 0$ as $|y|\rto +\infty$ for any $t\geq 0$.

In this paper, we are interested in the free surface and have no interest in the fluid dynamics on the bottom of $\Omega_t$,
thus we simply assume $-\infty< z<h(t,y)$.
Also, we neglect the Coriolis effect generated by the planetary rotation, then there is no Ekman layer near the free surface even if Rossby number is small.

Let $\e\rto 0$ in $(\ref{Sect1_NavierStokes_Equation})$, we formally get the following free surface Euler equations:
\begin{equation}\label{Sect1_Euler_Equation}
\left\{\begin{array}{ll}
u_t + u\cdot\nabla u + \nabla p = 0,\hspace{2.48cm} x\in\Omega_t, \\[7pt]
\nabla\cdot u =0, \hspace{4.34cm} x\in\Omega_t,\\[7pt]
\partial_t h = u\cdot \NN, \hspace{4cm} x\in\Sigma_t,\\[7pt]
p =gh -\sigma H, \hspace{3.8cm} x\in\Sigma_t,\\[7pt]
(u,h)|_{t=0} = (u_0,h_0) := \lim\limits_{\e\rto 0}(u_0^{\e},h_0^{\e}) ,
\end{array}\right.
\end{equation}
where $(u_0,h_0) = \lim\limits_{\e\rto 0}(u_0^{\e},h_0^{\e})$ is in the pointwise sense or even the $L^2$ sense
(see \cite{Masmoudi_Rousset_2012_FreeBC,Elgindi_Lee_2014,Mei_Wang_Xin_2015} for the sufficient conditions of the inviscid limit),
$(u_0,h_0)$ are independent of $\e$. Note that except for $(u_0,h_0) =\lim\limits_{\e\rto 0}(u_0^{\e}, h_0^{\e})$, we do not restrict their derivatives,
especially normal derivatives. 
Furtherly, note that the Navier-slip boundary case requires $u_0 =\lim\limits_{\e\rto 0} u_0^{\e}$ (see \cite{Iftimie_Planas_2006}),
while the Dirichlet boundary case requires $u_0^{\e}(y_1,y_2) \sim u_0(y_1,y_2) + u_0^P(y_1,\frac{y_2}{\sqrt{\e}}) +o(\e)$,
where $u_0^P$ is the initial data of Prandtl equations with Dirichlet boundary condition (see \cite{Sammartino_Caflisch_1998}).

The following Taylor sign condition should be imposed on $(\ref{Sect1_Euler_Equation})$ if $\sigma=0$ ,
\begin{equation}\label{Sect1_TaylorSign_1}
\begin{array}{ll}
g - \partial_z p|_{z=0} \geq\delta_p>0.
\end{array}
\end{equation}

In this paper, either $\sigma=0$ or $\sigma>0$ is fixed, we do not study the zero surface tension limit.
For both $(\ref{Sect1_NavierStokes_Equation})$ and $(\ref{Sect1_Euler_Equation})$, the analysis for the fixed $\sigma>0$ case is very different from that
for the $\sigma=0$ case.

In order to describe the strength of the initial vorticity layer, we define
\begin{equation}\label{Sect1_Vorticity_Layer_Profile_1}
\begin{array}{ll}
\varpi^{bl}_0 = \nabla\times u_0^{\e} - \nabla\times u_0 = \nabla\times u_0^{\e} - \nabla\times\lim\limits_{\e\rto 0} u_0^{\e}.
\end{array}
\end{equation}
We emphasize that the initial vorticity layer means a boundary layer at the initial time rather than a time layer in the vicinity of the initial time. 

If $u_0^{\e}$ has a profile $u_0^{\e}(y,z) \sim u_0(y,z) + \sqrt{\e} u_0^{bl}(y,\frac{z}{\sqrt{\e}})$ in its asymptotic expansion, then $\partial_z u_0^{\e}$ does not converge uniformly to $\partial_z u_0$, and then
\begin{equation}\label{Sect1_Vorticity_Layer_Profile_2}
\begin{array}{ll}
\lim\limits_{\e\rto 0}\varpi^{bl}_0 = (-\partial_z u_0^{bl,2}, \partial_z u_0^{bl,1}, 0)^{\top} \neq 0,
\end{array}
\end{equation}
which means the initial vorticity layer is strong.

For strong initial vorticity layer, there is a special case: 
if the Euler boundary data satisfies $\Pi\mathcal{S}u_0 \nn|_{z=0} =0$, then $\lim\limits_{\e\rto 0}\varpi^{bl}_0|_{z=0} =0$ 
on the free surface due to the compatibility condition $\Pi\mathcal{S}u_0^{\e} \nn|_{z=0} =0$.
However, it can not prevent $(\ref{Sect1_Vorticity_Layer_Profile_2})$ from holding in the vicinity of the free surface.
For example, we choose the boundary layer profile to be
$u_0^{bl}(y,\frac{z}{\sqrt{\e}}) = \exp\{- (\frac{z}{\sqrt{\e}})^2\} (1,1,0)^{\top}$,
for which
\begin{equation}\label{Sect1_Example_1_0}
\begin{array}{ll}
\lim\limits_{\e\rto 0}\varpi^{bl}_0 \big|_{z=0} =0, \quad
\lim\limits_{\e\rto 0}\varpi^{bl}_0 \big|_{z = -\sqrt{\e}} = 2e^{-1}(-1,1,0)^{\top} \neq 0.
\end{array}
\end{equation}
On the contrary, if $\lim\limits_{\e\rto 0}\varpi^{bl}_0 =0$, then $\lim\limits_{\e\rto 0}\varpi^{bl}_0|_{z=0} =0$ due to its continuity,
and then $\Pi\mathcal{S}u_0 \nn|_{z=0} =0$ at the initial time.

If $u_0^{\e}$ has a profile $u_0^{\e}(y,z) \sim u_0(y,z) + \e^{\frac{1}{2} + \delta_{ubl}} u_0^{bl}(y,\frac{z}{\sqrt{\e}})$ in its asymptotic expansion,
where $\delta_{ubl}>0$, then $\lim\limits_{\e\rto 0}\varpi^{bl}_0=0$,
which means the initial vorticity layer is weak.

In order to describe the discrepancy between boundary value of Navier-Stokes vorticity and that of Euler vorticity,
we investigate whether the Euler boundary data satisfies $\Pi\mathcal{S} u\nn|_{\Sigma_t} =0$.
If $\Pi\mathcal{S} u\nn|_{\Sigma_t} =0$, the boundary value of Navier-Stokes vorticity converges to that of Euler vorticity;
otherwise there is a discrepancy.

It is easy to have $\Pi\mathcal{S} u\nn|_{\Sigma_t} \neq 0$ in $(0,T]$, because it satisfies the forced transport equation.
While $\Pi\mathcal{S} u\nn|_{\Sigma_t} =0$ in $[0,T]$ is nontrivial. However, we can construct the Euler velocity field
satisfying $\Pi\mathcal{S} u\nn|_{\Sigma_t} =0$ and finite energy.
The scenario of our problem is as follows: construct a Euler velocity field satisfying $\Pi\mathcal{S} u\nn|_{\Sigma_t} =0$,
let Navier-Stokes initial data is a small perturbation of the Euler initial data, then we study the inviscid limit of 
Navier-Stokes solutions.

One example of $\Pi\mathcal{S} u\nn|_{\Sigma_t} =0$ is that
\begin{equation}\label{Sect1_Example_1_1}
\begin{array}{ll}
u= (-y_2 e^{-y_1^2 -y_2^2 -z^2}, y_1 e^{-y_1^2 -y_2^2 -z^2}, 0),\quad\,
h=0,
\end{array}
\end{equation}
and the pressure $p$ is the solution of the Poisson equation:
\begin{equation}\label{Sect1_Example_1_2}
\left\{\begin{array}{ll}
- \triangle p = e^{-2 (y_1^2 + y_2^2 + z^2)(-2 + 4 y_1^2 + 4 y_2^2 - 4y_1^2y_2^2)}, \\[4pt]
p|_{z=0} =0.
\end{array}\right.
\end{equation}
Then the Euler boundary data satisfies
\begin{equation}\label{Sect1_Example_1_3}
\begin{array}{ll}
\mathcal{S}u \nn |_{z=0}
= [y_2 z e^{-y_1^2 -y_2^2 -z^2}, -y_1 z e^{-y_1^2 -y_2^2 -z^2}, 0]^{\top} \big|_{z=0} =0.
\end{array}
\end{equation}
By deforming symmetrically the velocity field $(\ref{Sect1_Example_1_1})$ where $h$ is also symmetric, one may construct infinitely many velocity fields satisfying $\Pi\mathcal{S} u\nn|_{\Sigma_t} =0$.

\subsection{Survey of Previous Results}
In this survey, we introduce the previous results on the well-posedness and inviscid limits.

As to the irrotational fluids, refer to S. Wu \cite{Wu_2D_1997,Wu_3D_1999,Wu_2D_2009,Wu_3D_2011}, Germain, Masmoudi and Shatah \cite{Germain_Masmoudi_Shatah_2012},
Ionescu and Pusateri \cite{Ionescu_Pusateri_2015}, Alazard and Delort \cite{Alazard_Delort_2013} for the water waves without surface tension,
refer to K. Beyer and M. G$\ddot{u}$nther \cite{Beyer_Gunther_1998},
Germain, Masmoudi and Shatah \cite{Germain_Masmoudi_Shatah_2015} for the water waves with surface tension.

Before introducing previous results on the boundary layer and inviscid limit problem, we survery there some well-posedness results.
The free surface Navier-Stokes equations have both local and global well-posedness results, while the free surface Euler equations only have local well-posedness results.

As to the free surface Navier-Stokes equations, refer to
Beale \cite{Beale_1981}, Hataya \cite{Hataya_2009}, Guo and Tice \cite{Guo_Tice_2013_ARMA,Guo_Tice_2013_InftyDomain,Guo_Tice_2013_Local}
for the zero surface tension, refer to
Beale \cite{Beale_1984}, Tani \cite{Tani_1996}, Tanaka and Tani \cite{Tani_Tanaka_1995}
for the surface tension case.
Especially, \cite{Beale_1984,Tani_Tanaka_1995,Hataya_2009,Guo_Tice_2013_ARMA,Guo_Tice_2013_InftyDomain} proved the global in time results for the small initial data. Note that the viscosity is capable of producing the global well-posedness, while the surface tension only provides the regularizing effect on the free surface and enhance the decay rates of the solutions (see \cite{Guo_Tice_2013_ARMA}).

The general free surface Euler equations which are much more difficult and only have local results. Refer to
Lindblad \cite{Lindblad_2005}, Coutand and Shkiller \cite{Coutand_Shkoller_2007},
Shatah and Zeng \cite{Shatah_Zeng_2008}, Zhang and Zhang \cite{Zhang_Zhang_2008} for the zero surface tension case,
refer to
Coutand and Shkiller \cite{Coutand_Shkoller_2007}, Shatah and Zeng \cite{Shatah_Zeng_2008}
for the surface tension case.

As the viscosity approaches zero, we hope that the solutions of Navier-Stokes equations converge to the solutions of Euler equations.
However, this is only proved in the whole spaces where there are no boundary conditions,
see \cite{Swann_1971,Kato_1972,DiPerna_Majda_1987_CPAM,DiPerna_Majda_1987_CMP,Constantin_1986,Masmoudi_2007}.
However, in the presence of boundaries, the inviscid limit problem will be challenging due to the formation of boundary layers.

For Navier-Stokes equations with Dirichlet boundary condition in the fixed domain, $u|_{\partial\Omega} =0$, strong boundary layer whose width is $O(\sqrt{\e})$ and amplitude is $O(1)$ forms near the boundary.
Namely, the Navier-Stokes solution is expected to behave like $u^{\e} \sim u^0 + u^{bl}(t,y, z/{\sqrt{\e}})$ where $u^0$ is the Euler solution satisfying characteristic boundary condition $u\cdot\nn|_{\partial\Omega} =0$, $u^{bl}(t,y, z/{\sqrt{\e}})$ is the boundary layer profile. The inviscid limit is not rigorously verified except for the following two cases, i. e., the analytic setting (see \cite{Asano_1988,Sammartino_Caflisch_1998}) and the case where the vorticity is located away from the boundary (see \cite{Maekawa_2013,Maekawa_2014}).

For Navier-Stokes equations with Navier-slip boundary condition in the fixed domain,
$\Pi(2\mathcal{S} u\nn + \gamma_{s}\, u)|_{\partial\Omega} =0,\  u\cdot \nn|_{\partial\Omega} =0$,
weak boundary layer whose width and amplitude are $O(\sqrt{\e})$ forms near the boundary.
Namely, the Navier-Stokes solution is expected to behave like $u^{\e} \sim u^0 + \sqrt{\e} u^{bl}(t,y, z/{\sqrt{\e}})$,
where $u^0$ is the Euler solution satisfying characteristic boundary condition $u\cdot\nn|_{\partial\Omega} =0$.
For the inviscid limit, refer to Iftimie and Planas \cite{Iftimie_Planas_2006}, Iftimie and Sueur \cite{Iftimie_Sueur_2011},
Masmoudi and Rousset \cite{Masmoudi_Rousset_2012_NavierBC}, Xiao and Xin \cite{Xiao_Xin_2013}. Note that $H^1$ convergence
is satisfied for general Navier-slip boundary condition or curved boundary, while $H^3$ convergence happens for
complete slip boundary condition $\omega\times\nn|_{\partial\Omega} =0,\, u\cdot \nn|_{\partial\Omega} =0$ and flat boundary
(see \cite{Xiao_Xin_2007,Beirao_Crispo_2011}).

For the free surface Navier-Stokes equations with kinetical and dynamical boundary conditions in the moving domain, the recent works on the inviscid limit
are studied in conormal Sobolev spaces for which the normal differential operators vanish on the free surface.
Masmoudi and Rousset \cite{Masmoudi_Rousset_2012_FreeBC} proved the uniform estimates and inviscid limit of the free surface incompressible Navier-Stokes equations without surface tension in conormal Sobolev spaces.
By extending this conormal analysis framework, Wang and Xin \cite{Wang_Xin_2015}, Elgindi and Lee \cite{Elgindi_Lee_2014} proved the inviscid limit of the free surface incompressible Navier-Stokes equations with surface tension, Mei, Wang and Xin \cite{Mei_Wang_Xin_2015} proved the inviscid limit of the free surface compressible Navier-Stokes equations with or without surface tension.
\cite{Masmoudi_Rousset_2012_FreeBC} pointed out the free surface Navier-Stokes solutions are expected to behave like $u^{\e} \sim u^0
+ \sqrt{\e} u^{bl}(t,y, z/{\sqrt{\e}})$, where $u^0$ is the free surface Euler solutions.

\subsection{Formulation of the Problem and Our Motivations}

We first study N-S (abbreviation of Navier-Stokes) equations $(\ref{Sect1_NavierStokes_Equation})$ with $\sigma=0$.
In this subsection, we formulate the free boundary problem into the fixed coordinates domain $\mathbb{R}^3_{-}$.
Similar to \cite{Masmoudi_Rousset_2012_FreeBC}, we define
the diffeomorphism between $\mathbb{R}^3_{-}$ and the moving domain $\Omega_t$:
\begin{equation}\label{Sect1_LagrCoord_Definition_1}
\begin{array}{ll}
\Phi(t,\cdot) : \mathbb{R}^3_{-} = \mathbb{R}^2\times (-\infty,0) \quad \rto \quad \Omega_t, \\[4pt]
\hspace{2.75cm} x=(y,z) \quad \rto \quad (y,\varphi(t,y,z)),
\end{array}
\end{equation}
and define $\varphi$ as
\begin{equation}\label{Sect1_LagrCoord_Definition_2}
\begin{array}{ll}
\varphi(t,y,z) = Az + \eta(t,y,z),
\end{array}
\end{equation}
where $A>0$ is constant to be determined, $\eta$ is defined as
\begin{equation}\label{Sect1_LagrCoord_Definition_3}
\begin{array}{ll}
\eta(t,y,z) = \psi \ast_y h(t,y) ,
\end{array}
\end{equation}
here the symbol $\ast_y$ is a convolution in the $y$ variable and $\psi$ decays sufficiently fast in $z$ such that
$(1-z)\psi,\ \psi,\ \partial_z\psi, \cdots, \partial_z^{m+1}\psi \in L^1(\mathrm{d}z)$.
For example, $\psi = \mathcal{F}^{-1}[\frac{1}{(1-z)^4} e^{-(1-z)^2(1+|\xi|^2)}]$ where $\mathcal{F}^{-1}$ is the inverse Fourier transformation
with respect to $\xi\in\mathbb{R}^2$.

The constant $A>0$ is suitably chosen such that $\Phi$ is a diffeomorphism, namely
\begin{equation}\label{Sect1_LagrCoord_Definition_4}
\begin{array}{ll}
\partial_z \varphi(0,y,z) \geq 1, \quad \forall x\in\mathbb{R}^3_{-}.
\end{array}
\end{equation}

By the diffeomorphism $(\ref{Sect1_LagrCoord_Definition_1})$, we have
\begin{equation}\label{Sect1_LagrCoord_Definition_5}
\begin{array}{ll}
v(t,x) = u(t,y,\varphi(t,y,z)), \hspace{1.1cm} q(t,x) = p(t,y,\varphi(t,y,z)), \hspace{1cm} \forall x\in\mathbb{R}^3_{-}, \\[6pt]

\partial^{\varphi}_i v(t,x) = \partial_i u(t,y,\varphi(t,y,z)), \quad \partial^{\varphi}_i q(t,x) = \partial_i p(t,y,\varphi(t,y,z)), \quad i =t,1,2,3,
\end{array}
\end{equation}
while $h(t,y)$ does not change.

Then the free surface Navier-Stokes equations $(\ref{Sect1_NavierStokes_Equation})$ with $\sigma=0$ are equivalent to the following system:
\begin{equation}\label{Sect1_NS_Eq}
\left\{\begin{array}{ll}
\partial_t^{\varphi} v + v\cdot\nabla^{\varphi} v + \nabla^{\varphi} q = \e\triangle^{\varphi} v, \hspace{1.04cm} x\in\mathbb{R}^3_{-}, \\[7pt]
\nabla^{\varphi}\cdot v =0, \hspace{4cm} x\in\mathbb{R}^3_{-},\\[7pt]
\partial_t h = v(t,y,0)\cdot N, \hspace{2.78cm} z=0,\\[7pt]
q\nn -2\e \mathcal{S}^{\varphi}v\,\nn =gh\nn, \hspace{2.45cm} z=0,\\[7pt]
(v,h)|_{t=0} = (v_0^{\e},h_0^{\e}),
\end{array}\right.
\end{equation}
where
\begin{equation}\label{Sect1_NS_Eq_ComplementDef}
\begin{array}{ll}
\NN=(-\nabla h(t,y), 1)^{\top},\quad  \nn=\frac{\NN}{|\NN|}, \\[7pt]
\mathcal{S}^{\varphi}v =\frac{1}{2}(\nabla^{\varphi} v +\nabla^{\varphi} v^{\top}).
\end{array}
\end{equation}

Obviously, let $\e\rto 0$ in $(\ref{Sect1_NS_Eq})$, we formally get the following free surface Euler equations:
\begin{equation}\label{Sect1_Euler_Eq}
\left\{\begin{array}{ll}
\partial_t^{\varphi} v + v\cdot\nabla^{\varphi} v + \nabla^{\varphi} q = 0, \hspace{1.4cm} x\in\mathbb{R}^3_{-}, \\[7pt]
\nabla^{\varphi}\cdot v =0, \hspace{3.7cm} x\in\mathbb{R}^3_{-},\\[7pt]
\partial_t h = v(t,y,0)\cdot N, \hspace{2.48cm} z=0,\\[7pt]
q =gh, \hspace{4.24cm} z=0,\\[7pt]
(v,h)|_{t=0} = (v_0,h_0),
\end{array}\right.
\end{equation}
where $v_0$ is the limit of $v_0^{\e}$ in the $L^2$ sense,
$h_0$ is the limit of $h_0^{\e}$ in the $L^2$ sense for $\sigma=0$
and in the $H^1$ sense for $\sigma>0$, $(v_0,h_0)$ is independent of $\e$.
The following Taylor sign condition should be imposed on $(\ref{Sect1_Euler_Eq})$ when $\sigma=0$,
\begin{equation}\label{Sect1_TaylorSign_2}
\begin{array}{ll}
g - \partial_z^{\varphi} q |_{z=0} \geq\delta_q>0.
\end{array}
\end{equation}

D. Coutand and S. Shkoller (see \cite{Coutand_Shkoller_2007}) proved the well-posedness of the free surface
incompressible Euler equations $(\ref{Sect1_Euler_Eq})$ without surface tension.
We state their results in our formulation as follows:

{\it
Suppose the Taylor sign condition $(\ref{Sect1_TaylorSign_2})$ holds at $t=0$, $h_0\in H^3(\mathbb{R}^2), v_0\in H^3(\mathbb{R}^3_{-})$, then there exists $T>0$ and a unique solution $(v,q,h)$ of $(\ref{Sect1_Euler_Eq})$
with $v\in L^{\infty}([0,T],H^3(\mathbb{R}^3_{-})),\nabla q\in L^{\infty}([0,T],H^2(\mathbb{R}^3_{-})), h\in L^{\infty}([0,T],H^{3}(\mathbb{R}^2))$. }

Though conormal derivatives of the Navier-Stokes solutions and conormal derivatives of Euler solutions vanish on the free boundary,
their differences oscillate dramatically in the vicinity of the free boundary, thus the conormal functional spaces are not suitable
for studying the convergence rates of inviscid limit. Thus, we define the following functional spaces:
\begin{equation}\label{Sect1_Define_Spapces}
\begin{array}{ll}
\|v\|_{X^{m,s}}^2 := \sum\limits_{\ell\leq m, |\alpha|\leq m+s-\ell}\|\partial_t^{\ell} \mathcal{Z}^{\alpha} v\|_{L^2(\mathbb{R}^3_{-})}^2 \, , \hspace{0.83cm}
\|v\|_{X^{m}}^2 := \|v\|_{X^{m,0}}^2 \, , \\[17pt]

\|v\|_{X_{tan}^{m,s}}^2 := \sum\limits_{\ell\leq m, |\alpha|\leq m+s-\ell}\|\partial_t^{\ell} \partial_y^{\alpha} v\|_{L^2(\mathbb{R}^3_{-})}^2 \, , \hspace{0.91cm}
\|v\|_{X_{tan}^{m}}^2 := \|v\|_{X_{tan}^{m,0}}^2 \, , \\[17pt]

|h|_{X^{m,s}}^2 := \sum\limits_{\ell\leq m, |\alpha|\leq m+s-\ell}|\partial_t^{\ell} \partial_y^{\alpha} h|_{L^2(\mathbb{R}^2)}^2 \, , \hspace{1.27cm}
|h|_{X^{m}}^2 := |h|_{X^{m,0}}^2 \, , \\[20pt]

\|v\|_{Y_{tan}^{m,s}}^2 := \sum\limits_{\ell\leq m, |\alpha|\leq m+s-\ell}\|\partial_t^{\ell} \partial_y^{\alpha} v\|_{L^{\infty}(\mathbb{R}^3_{-})}^2 \, , \hspace{0.83cm}
\|v\|_{Y_{tan}^{m}}^2 := \|v\|_{Y_{tan}^{m,0}}^2 \, , \\[17pt]

|h|_{Y^{m,s}}^2 := \sum\limits_{\ell\leq m, |\alpha|\leq m+s-\ell}|\partial_t^{\ell} \partial_y^{\alpha} h|_{L^{\infty}(\mathbb{R}^2)}^2 \, , \hspace{1.2cm}
|h|_{Y^{m}}^2 := |h|_{Y^{m,0}}^2 \, ,
\end{array}
\end{equation}
where the differential operators $\mathcal{Z}_1=\partial_{y_1}, \mathcal{Z}_2=\partial_{y_2}, \mathcal{Z}_3 =\frac{z}{1-z}\partial_z$
(see \cite{Masmoudi_Rousset_2012_FreeBC,Elgindi_Lee_2014,Wang_Xin_2015,Mei_Wang_Xin_2015}).
Also, we use $|\cdot|_m$ to denote the standard Sobolev norm defined in the horizontal space $\mathbb{R}^2$.

Assume $\omega^{\e}=\nabla^{\varphi^{\e}}\times v^{\e}$, $\omega=\nabla^{\varphi}\times v$
are Navier-Stokes vorticity, Euler vorticity respectively, $\hat{\omega} =\omega^{\e} -\omega$.
In this paper, bounded variables or quantities mean that they are bounded by $O(1)$,
small variables or quantities mean that they are bounded by $O(\e^{\beta})$ for
some $\beta>0$. Now we state our motivations of this paper.

1. As $\e\rto 0$, \cite{Masmoudi_Rousset_2012_FreeBC} showed that the velocity converges in $L^2$ and $L^{\infty}$ norms,
the height function converges in $L^2$ and $W^{1,\infty}$ norms.
We can expect that their tangential derivatives converges, but we still do not know whether the vorticity and normal derivatives of the velocity
converge in $L^{\infty}$ norm. If they do not converge in the $L^{\infty}$ norm, there are must be a strong vorticity layer in the vicinity of the free surface.
\cite{Masmoudi_Rousset_2012_FreeBC} pointed the N-S solution is expected to behave like $u^{\e} \sim u^0 + \sqrt{\e} u^{bl}(t,y, z/{\sqrt{\e}})$,
however, this is not rigorously proved. It is expected that the velocity of the free surface N-S equations has a weak boundary layer,
we have to prove the existence of strong vorticity layer for some sufficient conditions. Note that the energy norms are too weak, thus we use the $L^{\infty}$ norm to describe the existence of strong boundary layers.

2. We want to know the sufficient and necessary conditions for the existence of strong vorticity layer,
we also want to know these conditions for the weak vorticity layer.
We show that there are two sufficient conditions for the strong vorticity layer,
note that these two conditions are almost independent.
One condition is that the initial vorticity layer is strong,
then it is transported by the velocity field for any small $\e$, and then we get a strong vorticity layer when $t\in (0,T]$.
Another condition is that the Euler boundary data satisfies $\Pi\mathcal{S}^{\varphi} v \nn|_{z=0} \neq 0$ in $(0,T]$,
then there is a discrepancy between N-S vorticity and Euler vorticity, and then we have a strong vorticity layer.
When neither of two sufficient conditions is satisfied, we show that
the vorticity layer is weak.

3. \cite{Masmoudi_Rousset_2012_FreeBC,Wang_Xin_2015,Elgindi_Lee_2014} proved the uniform regularity and inviscid limit of the free surface N-S equations
with or without surface tension. In order to prove the uniform regularities, \cite{Masmoudi_Rousset_2012_FreeBC,Elgindi_Lee_2014,Wang_Xin_2015,Mei_Wang_Xin_2015} controlled the bounded quantities in conormal functional spaces and applied
the following integration by parts formula to the a priori estimates:
\begin{equation}\label{Sect1_Formulas_CanNotUse}
\begin{array}{ll}
\frac{\mathrm{d}}{\mathrm{d}t}\int\limits_{\mathbb{R}^3_{-}} f \mathrm{d}\mathcal{V}_t
=\int\limits_{\mathbb{R}^3_{-}} \partial_t^{\varphi} f \mathrm{d}\mathcal{V}_t
+ \int\limits_{\{z=0\}} f v\cdot\NN \mathrm{d}y,  \\[12pt]

\int\limits_{\mathbb{R}^3_{-}} \vec{a} \cdot\nabla^{\varphi} f \mathrm{d}\mathcal{V}_t
= \int\limits_{\{z=0\}} \vec{a}\cdot \NN\, f  \mathrm{d}y
- \int\limits_{\mathbb{R}^3_{-}} \nabla^{\varphi}\cdot \vec{a}\, f \mathrm{d}\mathcal{V}_t, \\[12pt]

\int\limits_{\mathbb{R}^3_{-}} \vec{a} \cdot (\nabla^{\varphi}\times \vec{b}) \,\mathrm{d}\mathcal{V}_t
= \int\limits_{\{z=0\}} \vec{a} \cdot (\NN \times \vec{b}) \,\mathrm{d}y
+ \int\limits_{\mathbb{R}^3_{-}} (\nabla^{\varphi}\times \vec{a}) \cdot \vec{b} \,\mathrm{d}\mathcal{V}_t,
\end{array}
\end{equation}
where $\mathrm{d}\mathcal{V}_t = \partial_z\varphi \mathrm{d}y\mathrm{d}z$ is defined on $\mathbb{R}^3_{-}$ but measures the volume element of $\Omega_t$.
Refer to \cite{Masmoudi_Rousset_2012_FreeBC} for the first and second formulae in $(\ref{Sect1_Formulas_CanNotUse})$.
As to the last formula $(\ref{Sect1_Formulas_CanNotUse})_3$ used in the fixed domain, refer to \cite{Wang_2016,Wang_Xin_Yong_2015,Xiao_Xin_2013}.

Motivated by \cite{Masmoudi_Rousset_2012_FreeBC,Wang_Xin_2015,Elgindi_Lee_2014}, we want to know convergence rates of the inviscid limit, which involves
two moving domain, we denote Navier-Stokes domain and Euler domain by $\Omega^{\e}$ and $\Omega$ respectively.
In general, $\Omega^{\e}$ and $\Omega$ do not coincide, we can not compare these two velocity fields.
Thus, we have to map $\Omega^{\e}$ and $\Omega$ to the common fixed coordinate domain $\mathbb{R}^3_{-}$,
namely $\Omega^{\e} = \Phi^{\e}(\mathbb{R}^3_{-}), \Omega = \Phi(\mathbb{R}^3_{-})$.
For any $x\in\mathbb{R}^3_{-}$, two points $\Phi^{\e}(x)$ and $\Phi(x)$ do not coincide in general, However,
$\Phi^{\e}(x)$ converges to $\Phi(x)$ pointwisely as $\e\rto 0$, thus $|v^{\e}(x) - v(x)|$
and $|\partial_t^{\ell}\mathcal{Z}^{\alpha} v^{\e}(x) - \partial_t^{\ell}\mathcal{Z}^{\alpha} v(x)|$
must be small quantities. We have to overcome many difficulties involving two different moving domains to close the estimates of
$|\partial_t^{\ell}\mathcal{Z}^{\alpha} v^{\e}(x) - \partial_t^{\ell}\mathcal{Z}^{\alpha} v(x)|$.

4. \cite{Xiao_Xin_2007,Xiao_Xin_2011,Xiao_Xin_2013,Wang_2016} studied the inviscid limit of
the incompressible or compressible N-S equations with Navier-slip boundary condition,
where the initial Navier-Stokes data and initial Euler data are exactly the same and independent of $\e$.
If the Navier-Stokes boundary condition satisfies $\omega^{\e}\times \nn|_{z=0} =0$ and the boundary is flat,
then the Euler boundary data also satisfies $\omega\times \nn|_{z=0} =0$,
$\|u^{\e} -u\|_{L^2} \lem O(\e), \|\omega^{\e} -\omega\|_{L^2} + \|u^{\e} -u\|_{H^1} \lem O(\e^{\frac{3}{4}})$,
\cite{Xiao_Xin_2007} proved the $H^3$ convergence.
While if the Euler boundary data is general or the boundary is curved, then $\|u^{\e} -u\|_{L^2} \lem O(\e^{\frac{3}{4}}),
\|\omega^{\e} -\omega\|_{L^2} + \|u^{\e} -u\|_{H^1} \lem O(\e^{\frac{1}{4}})$. \cite{Iftimie_Sueur_2011}
showed that it is impossible to prove $H^2$ convergence.

We are also interested in the convergence rates of the inviscid limit of the free boundary problem for Navier-Stokes equations.
However, in our formulation of the free boundary problem,
the diffeomorphism between the fixed coordinates $\mathbb{R}^3_{-}$
and two moving domains are twisted, the differential operators in N-S and
Euler equations are also twisted, then the estimates of tangential derivatives and
the estimates of normal derivatives can not be decoupled, we even can not develop the $L^2$ estimate of
$(v^{\e}-v,h^{\e}-h)$ themselves without involving the normal derivative $\partial_z v^{\e} - \partial_z v$.
Thus, we want to know whether tangential derivatives and normal derivatives have different convergence rates.

If the Euler boundary data satisfies $\Pi\mathcal{S}^{\varphi} v|_{z=0} \neq 0$, $\omega^{\e}|_{z=0}$ does not converge to $\omega|_{z=0}$,
we want to know how to calculate convergence rates of the vorticity in the energy norm.
If $\Pi\mathcal{S}^{\varphi} v|_{z=0} = 0$, $\omega^{\e}|_{z=0}\rto \omega|_{z=0}$, we want to know how to improve the convergence rates.

5. To estimate the convergence rate of the inviscid limit, we need to use the time derivatives.
However, time derivatives can not be expressed in terms of space derivatives by using the equations,
since we work on conormal spaces instead of standard Sobolev spaces.
Thus, we prove the uniform regularity concluding time derivatives and determine the regularity structure of N-S solutions and Euler solutions in conormal functional spaces. When time derivatives are included, uniform estimates of tangential derivatives will be different from \cite{Masmoudi_Rousset_2012_FreeBC}.
Moreover, our estimates of normal derivatives are based on the estimates of vorticity rather than
those of $\Pi\mathcal{S}^{\varphi} v\nn$ (see \cite{Masmoudi_Rousset_2012_FreeBC,Wang_Xin_2015,Elgindi_Lee_2014}).

\subsection{Main Results for N-S Equations without Surface Tension}

\cite{Masmoudi_Rousset_2012_FreeBC} proved the uniform regularity of space derivatives
of the free surface Navier-Stokes equations (\ref{Sect1_NS_Eq}), while the following proposition
concerns the uniform regularity of time derivatives.
\begin{proposition}\label{Sect1_Proposition_TimeRegularity}
For $m\geq 6$, assume the initial data $(v_0^{\e},h_0^{\e})$ satisfy the compatibility condition $\Pi\mathcal{S}^{\varphi} v_0^{\e}\nn|_{z=0} =0$ and the regularities:
\begin{equation}\label{Sect1_Proposition_TimeRegularity_1}
\begin{array}{ll}
\sup\limits_{\e\in (0,1]} \big(
|h_0^{\e}|_{X^{m-1,1}} + \e^{\frac{1}{2}}|h_0^{\e}|_{X^{m-1,\frac{3}{2}}}
+ \|v_0^{\e}\|_{X^{m-1,1}} + \|\omega_0^{\e}\|_{X^{m-1}} \\[7pt]\quad
+ \|\omega_0^{\e}\|_{1,\infty} + \e^{\frac{1}{2}}\|\partial_z \omega_0^{\e}\|_{L^{\infty}}\big) \leq C_0,
\end{array}
\end{equation}
where $C_0>0$ is suitably small such that the Taylor sign condition $g-\partial_z^{\varphi^{\e}} q^{\e} |_{z=0} \geq c_0 >0$,
then the unique Navier-Stokes solution to $(\ref{Sect1_NS_Eq})$ satisfies
\begin{equation}\label{Sect1_Proposition_TimeRegularity_2}
\begin{array}{ll}
\sup\limits_{t\in [0,T]} \big(
|h^{\e}|_{X^{m-1,1}}^2 + \e^{\frac{1}{2}}|h^{\e}|_{X^{m-1,\frac{3}{2}}}^2
 + \|v^{\e}\|_{X^{m-1,1}}^2 + \|\partial_z v^{\e}\|_{X^{m-2}}^2 + \|\omega^{\e}\|_{X^{m-2}}^2 \\[7pt]\quad
+ \|\partial_z v^{\e}\|_{1,\infty}^2
+ \e^{\frac{1}{2}}\|\partial_{zz}v^{\e}\|_{L^{\infty}}^2 \big)

+ \|\partial_t^m h\|_{L^4([0,T],L^2)}^2
+ \e\|\partial_t^m h\|_{L^4([0,T],H^{\frac{1}{2}})}^2 \\[7pt]\quad

+ \e\int\limits_0^T \|\nabla v^{\e}\|_{X^{m-1,1}}^2 + \|\nabla\partial_z v^{\e}\|_{X^{m-2}}^2 \,\mathrm{d}t \leq C.
\end{array}
\end{equation}

As $\e\rto 0$, the Euler solution to $(\ref{Sect1_Euler_Eq})$ satisfies the following regularities:
\begin{equation}\label{Sect1_Proposition_TimeRegularity_3}
\begin{array}{ll}
\sup\limits_{t\in [0,T]} \big(
|h|_{X^{m-1,1}} + \|v\|_{X^{m-1,1}} + \|\partial_z v\|_{X^{m-2}} + \|\omega\|_{X^{m-2}} \\[8pt]\quad
+ \|\partial_z v\|_{1,\infty} \big)

+ \|\partial_t^m h\|_{L^4([0,T],L^2)}^2 \leq C,
\end{array}
\end{equation}
where the Taylor sign condition $g-\partial_z^{\varphi} q |_{z=0} \geq c_0 >0$ holds.
\end{proposition}

For the initial regularities $(\ref{Sect1_Proposition_TimeRegularity_1})$, we can not prove $\|\partial_t^m v^{\e}\|_{L^4([0,T],L^2)}$.
To prove $\|\partial_t^m v^{\e}\|_{L^4([0,T],L^2)}$, it requires $(\ref{Sect1_Proposition_TimeRegularity_1})$
as well as $\partial_t^m v_0^{\e}, \partial_t^m h_0^{\e}\in L^2(\mathbb{R}^3_{-})$.

Note that when $\sigma=0$, we must use the following Alinhac's good unknown (see \cite{Alinhac_1989,Masmoudi_Rousset_2012_FreeBC}) to estimate tangential derivatives:
\begin{equation}\label{Sect1_Good_Unknown_1}
\begin{array}{ll}
V^{\ell,\alpha} = \partial_t^{\ell}\mathcal{Z}^{\alpha} v - \partial_z^{\varphi}v \partial_t^{\ell}\mathcal{Z}^{\alpha} \eta, \ 0<\ell+|\alpha|\leq m, \ell\leq m-1,\\[6pt]

Q^{\ell,\alpha} = \partial_t^{\ell}\mathcal{Z}^{\alpha} q - \partial_z^{\varphi}q \partial_t^{\ell}\mathcal{Z}^{\alpha} \eta, \ 0<\ell+|\alpha|\leq m, \ell\leq m-1.
\end{array}
\end{equation}

Our proof of Proposition $\ref{Sect1_Proposition_TimeRegularity}$ is different from \cite{Masmoudi_Rousset_2012_FreeBC}:
(i) $\|\partial_t^{\ell} q\|_{L^2}$ has no bound in general.
When $|\alpha|=0$, we estimate $V^{\ell,0}$ and $\nabla\partial_t^{\ell}q$ where $0\leq \ell \leq m-1$,
the dynamical boundary condition can not be used.

(ii) \cite{Masmoudi_Rousset_2012_FreeBC} as well as \cite{Wang_Xin_2015,Elgindi_Lee_2014,Mei_Wang_Xin_2015} estimated normal derivatives by using
$\Pi\mathcal{S}^{\varphi}\nn$ and its evolution equations.
While in this paper, we estimate normal derivatives by using the vorticity and the following equations:
\begin{equation}\label{Sect1_Vorticity_H_Eq}
\left\{\begin{array}{ll}
\partial_t^{\varphi} \omega_h + v\cdot\nabla^{\varphi}\omega_h - \e\triangle^{\varphi}\omega_h = \vec{\textsf{F}}^0[\nabla\varphi](\omega_h,\partial_j v^i),
\\[7pt]

\omega^1|_{z=0} =\textsf{F}^1 [\nabla\varphi](\partial_j v^i),
\\[6pt]

\omega^2|_{z=0} =\textsf{F}^2 [\nabla\varphi](\partial_j v^i),
\end{array}\right.
\end{equation}
where $j=1,2,\ i=1,2,3$, $\vec{\textsf{F}}^0[\nabla\varphi](\omega_h,\partial_j v^i)$ is a quadratic polynomial vector with respect to $\omega_h$ and $\partial_j v^i$, $\textsf{F}^1 [\nabla\varphi](\partial_j v^i)$, $\textsf{F}^2 [\nabla\varphi](\partial_j v^i)$ are polynomials with respect to $\partial_j v^i$, all the coefficients are fractions of $\nabla\varphi$.

(iii) In \cite{Masmoudi_Rousset_2012_FreeBC}, the Taylor sign condition is $g-\partial_z^{\varphi^{\e}}q^{\e,E}|_{z=0} \geq c_0>0$,
that is imposed on the Euler part of the pressure $q^{\e}$. \ $q^{\e}$ has a decomposition $q^{\e} = q^{\e,E}+ q^{\e,NS}$ which satisfy
\begin{equation}\label{Sect1_Pressure_EulerPart}
\begin{array}{ll}
\left\{\begin{array}{ll}
\triangle^{\varphi^{\e}} q^{\e,E} = -\partial_i^{\varphi^{\e}} v^{\e,j}\partial_j^{\varphi^{\e}} v^{\e,i},
\\[6pt]
q^{\e,E}|_{z=0} = g h^{\e}.
\end{array}\right.
\qquad

\left\{\begin{array}{ll}
\triangle^{\varphi^{\e}} q^{\e,NS} = 0,
\\[6pt]
q^{\e,NS}|_{z=0} = 2\e\mathcal{S}^{\varphi^{\e}} v\nn\cdot\nn.
\end{array}\right.
\end{array}
\end{equation}
However, the force term of $q^{\e,E}$ has boundary layer in the vicinity of the free boundary in general,
thus $\partial_z^{\varphi^{\e}} q^{\e,E}|_{z=0}$ may also have boundary layer,
it is unknown whether $\partial_z^{\varphi^{\e}} q^{\e,E}|_{z=0}$ converges pointwisely to $\partial_z^{\varphi} q|_{z=0}$ or not.
Different from \cite{Masmoudi_Rousset_2012_FreeBC}, our Taylor sign condition is $g-\partial_z^{\varphi^{\e}} q^{\e}|_{z=0} \geq c_0 >0$.
Since $\partial_z^{\varphi^{\e}} q^{\e}|_{z=0} = \e\triangle^{\varphi^{\e}} v^3 - \partial_t v^3 - v_y^{\e} \cdot\nabla_y v^{\e,3}$
and $\|\partial_{zz} v\|_{L^{\infty}},\sqrt{\e}\|\partial_{zz} v\|_{L^{\infty}}$ are bounded, thus $\partial_z^{\varphi^{\e}} q^{\e}|_{z=0}$ converges to $\partial_z^{\varphi} q|_{z=0}$ pointwisely.

\vspace{0.25cm}
For classical solutions to the free surface Navier-Stokes equations $(\ref{Sect1_NS_Eq})$ with $\sigma=0$, we will estimate the convergence rates of
the velocity later, which implies the weak boundary layer of the velocity. Before estimating the convergence rates, we show
the following theorem which states the existence of strong vorticity layer.
\begin{theorem}\label{Sect1_Thm_StrongLayer}
Assume $T>0$ is finite, fixed and independent of $\e$, $(v^{\e},h^{\e})$ is the solution in $[0,T]$ of Navier-Stokes equations $(\ref{Sect1_NS_Eq})$ with initial data $(v^{\e}_0,h^{\e}_0)$ satisfying $(\ref{Sect1_Proposition_TimeRegularity_1})$, $\omega^{\e}$ is its vorticity.
$(v,h)$ is the solution in $[0,T]$ of Euler equations $(\ref{Sect1_Euler_Eq})$ with initial data $(v_0,h_0)\in X^{m-1,1}(\mathbb{R}^3_{-}) \times X^{m-1,1}(\mathbb{R}^2)$, $\omega$ is its vorticity.

(1) If the initial Navier-Stokes velocity satisfies $\lim\limits_{\e\rto 0}(\nabla^{\varphi^{\e}}\times v_0^{\e})
- \nabla^{\varphi}\times\lim\limits_{\e\rto 0} v_0^{\e} \neq 0$ in the initial set $\mathcal{A}_0$,
the Euler boundary data satisfies $\Pi\mathcal{S}^{\varphi} v\nn|_{z=0} = 0$ in $[0,T]$,
then the Navier-Stokes solution of $(\ref{Sect1_NS_Eq})$ has a strong vorticity layer satisfying
\begin{equation}\label{Sect1_Thm_StrongLayer_1}
\begin{array}{ll}
\lim\limits_{\e\rto 0}\|\omega^{\e} -\omega\|_{L^{\infty}(\mathcal{X}(\mathcal{A}_0)\times (0,T])} \neq 0, \\[9pt]
\lim\limits_{\e\rto 0}\|\partial_z^{\varphi^{\e}} v^{\e} -\partial_z^{\varphi} v\|_{L^{\infty}(\mathcal{X}(\mathcal{A}_0)\times (0,T])} \neq 0, \\[9pt]
\lim\limits_{\e\rto 0}\|\mathcal{S}^{\varphi^{\e}}v^{\e} -\mathcal{S}^{\varphi}v\|_{L^{\infty}(\mathcal{X}(\mathcal{A}_0)\times (0,T])} \neq 0,\\[9pt]
\lim\limits_{\e\rto 0}\|\nabla^{\varphi^{\e}} q^{\e} -\nabla^{\varphi} q\|_{L^{\infty}(\mathcal{X}(\mathcal{A}_0)\times (0,T])} \neq 0.
\end{array}
\end{equation}
where $\mathcal{X}(\mathcal{A}_0) = \{\mathcal{X}(t,x)\big|\mathcal{X}(0,x)\in\mathcal{A}_0, \partial_t \mathcal{X}(t,x) = v(t,\Phi^{-1}\circ \mathcal{X})\}$.

(2) If $\lim\limits_{\e\rto 0}(\nabla^{\varphi^{\e}}\times v_0^{\e})
- \nabla^{\varphi}\times\lim\limits_{\e\rto 0} v_0^{\e} = 0$, the Euler boundary data satisfies $\Pi\mathcal{S}^{\varphi} v\nn|_{z=0} \neq 0$ in $(0,T]$,
then the Navier-Stokes solution of $(\ref{Sect1_NS_Eq})$ has a strong vorticity layer satisfying
\begin{equation}\label{Sect1_Thm_WeakLayer_1}
\begin{array}{ll}
\lim\limits_{\e\rto 0}\big|\omega^{\e}|_{z=0} -\omega|_{z=0}\big|_{L^{\infty}(\mathbb{R}^2\times (0,T])} \neq 0, \\[9pt]

\lim\limits_{\e\rto 0}\|\omega^{\e} -\omega\|_{L^{\infty}(\mathbb{R}^2\times [0, O(\e^{\frac{1}{2} -\delta_z}))\times (0,T])} \neq 0, \\[9pt]

\lim\limits_{\e\rto 0}\|\partial_z^{\varphi^{\e}} v^{\e} -\partial_z^{\varphi} v\|
_{L^{\infty}(\mathbb{R}^2\times [0, O(\e^{\frac{1}{2} -\delta_z}))\times (0,T])} \neq 0, \\[9pt]
\lim\limits_{\e\rto 0}\|\mathcal{S}^{\varphi^{\e}}v^{\e} -\mathcal{S}^{\varphi}v\|
_{L^{\infty}(\mathbb{R}^2\times [0, O(\e^{\frac{1}{2} -\delta_z}))\times (0,T])} \neq 0,\\[9pt]
\lim\limits_{\e\rto 0}\|\nabla^{\varphi^{\e}} q^{\e} -\nabla^{\varphi} q\|
_{L^{\infty}(\mathbb{R}^2\times [0, O(\e^{\frac{1}{2} -\delta_z}))\times (0,T])} \neq 0,
\end{array}
\end{equation}
for some constant $\delta_z \geq 0$.

(3) $\lim\limits_{\e\rto 0}(\nabla^{\varphi^{\e}}\times v_0^{\e})
- \nabla^{\varphi}\times\lim\limits_{\e\rto 0} v_0^{\e} = 0$ and $\Pi\mathcal{S}^{\varphi} v\nn|_{z=0} = 0$ in $[0,T]$
are necessary conditions for the Navier-Stokes solution of $(\ref{Sect1_NS_Eq})$ to have a weak vorticity layer satisfying
\begin{equation}\label{Sect1_Thm_WeakerLayer_1}
\begin{array}{ll}
\lim\limits_{\e\rto 0}\|\omega^{\e} -\omega\|_{L^{\infty}(\mathfrak{Cl}(\mathbb{R}^3_{-})\times (0,T])} = 0, \\[9pt]
\lim\limits_{\e\rto 0}\|\partial_z^{\varphi^{\e}} v^{\e} -\partial_z^{\varphi} v\|_{L^{\infty}(\mathfrak{Cl}(\mathbb{R}^3_{-})\times (0,T])} = 0, \\[9pt]
\lim\limits_{\e\rto 0}\|\mathcal{S}^{\varphi^{\e}}v^{\e} -\mathcal{S}^{\varphi}v\|_{L^{\infty}(\mathfrak{Cl}(\mathbb{R}^3_{-})\times (0,T])} = 0,\\[9pt]
\lim\limits_{\e\rto 0}\|\nabla^{\varphi^{\e}} q^{\e} -\nabla^{\varphi} q\|_{L^{\infty}(\mathfrak{Cl}(\mathbb{R}^3_{-})\times (0,T])} = 0,
\end{array}
\end{equation}
where $\mathfrak{Cl}(\mathbb{R}^3_{-}) = \mathbb{R}^3_{-}\cup \{x|z=0\}$ is the closure of $\mathbb{R}^3_{-}$.
\end{theorem}

We give some remarks on Theorem $\ref{Sect1_Thm_StrongLayer}$:
\begin{remark}\label{Sect1_Remark_StrongLayer}
(i) To represent $\partial_z^{\varphi^{\e}} v^{\e} - \partial_z^{\varphi} v$ is more natural than $\partial_z v^{\e} - \partial_z v$.
However, $\lim\limits_{\e\rto 0}\|\partial_z^{\varphi^{\e}} v^{\e} - \partial_z^{\varphi} v\|_{L^{\infty}}\neq 0$
results from $\lim\limits_{\e\rto 0}\|\partial_z v^{\e} - \partial_z v\|_{L^{\infty}}\neq 0$
and $\lim\limits_{\e\rto 0}\|\partial_z(\eta^{\e} -\eta)\|_{L^{\infty}} =0$, due to
the formula:
\begin{equation}\label{Sect1_Transport_DifferenceEq}
\begin{array}{ll}
\partial_z^{\varphi^{\e}} v^{\e} - \partial_z^{\varphi} v
= \partial_z^{\varphi^{\e}}(v^{\e} -v) - \partial_z^{\varphi} v \, \partial_z^{\varphi^{\e}} (\eta^{\e} -\eta) \\[5pt]\hspace{1.95cm}

= \frac{1}{\partial_z\varphi^{\e}} \cdot\partial_z(v^{\e} -v)
- \partial_z^{\varphi} v \, \frac{1}{\partial_z\varphi^{\e}} \cdot \partial_z (\eta^{\e} -\eta).
\end{array}
\end{equation}

(ii) The energy norm $\|\cdot\|_{L^2}$ is weaker than the $L^{\infty}$ norm, because
$\|\omega^{\e}-\omega\|_{L^2(\mathbb{R}^3_{-})} =0$, even though we have the profile
$\omega^{\e}(t,y,z) \sim \omega(t,y,z) + \omega^{bl}(t,y,\frac{z}{\sqrt{\e}})$. While
$\|\omega^{\e}-\omega\|_{L^{\infty}(\mathbb{R}^3_{-})} \neq 0$. Thus, we use the $L^{\infty}$ norm
to describe the strong vorticity layer.

(iii) $\textsf{S}_n =\Pi \mathcal{S}^{\varphi} v \nn$ satisfies the forced transport equations:
\begin{equation}\label{Sect1_Transport_Eq_Sn}
\begin{array}{ll}
\partial_t^{\varphi} \textsf{S}_n + v\cdot\nabla^{\varphi} \textsf{S}_n =
-\frac{1}{2}\Pi\big((\nabla^{\varphi} v)^2+((\nabla^{\varphi} v)^{\top})^2\big) \nn
- \Pi((\mathcal{D}^{\varphi})^2 q)\nn\\[8pt]\quad
+ (\partial_t^{\varphi}\Pi + v\cdot\nabla^{\varphi}\Pi)\mathcal{S}^{\varphi} v\nn
+ \Pi\mathcal{S}^{\varphi} v(\partial_t^{\varphi}\nn + v\cdot\nabla^{\varphi}\nn),
\end{array}
\end{equation}
where $\big((\mathcal{D}^{\varphi})^2 q\big)$ is the Hessian matrix of $q$.
The equation $(\ref{Sect1_Transport_Eq_Sn})$ implies that even if $\textsf{S}_n|_{t=0} = 0$, then $\textsf{S}_n \neq 0$ in $(0,T]$
is possible due to the force terms of $(\ref{Sect1_Transport_Eq_Sn})$.

However, $\textsf{S}_n|_{z=0} \equiv 0$ in $[0,T]$ can be constructed, see an example constructed in $(\ref{Sect1_Example_1_1}), (\ref{Sect1_Example_1_2}),(\ref{Sect1_Example_1_3})$.

(iv) $\Pi\mathcal{S}^{\varphi} v\nn|_{z=0} = 0$ at $t=0$ implies that $\lim\limits_{\e\rto 0}(\nabla^{\varphi^{\e}}\times v_0^{\e})|_{z=0}
- \nabla^{\varphi}\times\lim\limits_{\e\rto 0} v_0^{\e}|_{z=0} = 0$. But it does not contradict with
$\lim\limits_{\e\rto 0}(\nabla^{\varphi^{\e}}\times v_0^{\e})
- \nabla^{\varphi}\times\lim\limits_{\e\rto 0} v_0^{\e} \neq 0$ in the initial set $\mathcal{A}_0$,
see $(\ref{Sect1_Example_1_0})$ where $\mathcal{A}_0 = \{x|z= -\sqrt{\e}\}$ in local coordinates.
If $\Pi\mathcal{S}^{\varphi} v\nn|_{z=0} \neq 0$ in $[0,T]$ and $\lim\limits_{\e\rto 0}(\nabla^{\varphi^{\e}}\times v_0^{\e})
- \nabla^{\varphi}\times\lim\limits_{\e\rto 0} v_0^{\e} \neq 0$ in the initial set $\mathcal{A}_0$,
then it is easy to know the results are the union of $(\ref{Sect1_Thm_StrongLayer_1})$ and $(\ref{Sect1_Thm_WeakLayer_1})$.

(v) $\NN\cdot \partial_z^{\varphi} v$ and $\NN\cdot \partial_z v$ do not have boundary layer, but
$\partial_z v^3$ has boundary layer in general. Similarly, $\NN\cdot \omega$ does not have boundary layer, but
$\omega^3$ has boundary layer in general. The reason is that both $v|_{z=0}$ and $\omega|_{z=0}$ are not perpendicular to the free surface in general.
\end{remark}

The proof of Theorem $\ref{Sect1_Thm_StrongLayer}$ is based on the analysis of the limit of $\hat{\omega} = \omega^{\e} -\omega$
which satisfies the following equations:
\begin{equation}\label{Sect1_N_Derivatives_Difference_Eq}
\left\{\begin{array}{ll}
\partial_t^{\varphi^{\e}} \hat{\omega}_h
+ v^{\e} \cdot\nabla^{\varphi^{\e}} \hat{\omega}_h
- \e\triangle^{\varphi^{\e}}\hat{\omega}_h
= \vec{\textsf{F}}^0[\nabla\varphi^{\e}](\omega_h^{\e},\partial_j v^{\e,i})
- \vec{\textsf{F}}^0[\nabla\varphi](\omega_h,\partial_j v^i) \\[5pt]\quad
+ \e\triangle^{\varphi^{\e}}\omega_h
+ \partial_z^{\varphi}\omega_h \partial_t^{\varphi^{\e}} \hat{\eta}
+ \partial_z^{\varphi} \omega_h\, v^{\e}\cdot \nabla^{\varphi^{\e}}\hat{\eta}
- \hat{v}\cdot\nabla^{\varphi} \omega_h , \\[8pt]

\hat{\omega}_h|_{z=0} =\textsf{F}^{1,2} [\nabla\varphi^{\e}](\partial_j v^{\e,i}) - \omega^b_h, \\[6pt]

\hat{\omega}_h|_{t=0} = (\hat{\omega}_0^1, \hat{\omega}_0^2)^{\top},
\end{array}\right.
\end{equation}
where $\vec{\textsf{F}}^0[\nabla\varphi](\omega_h,\partial_j v^i)$ and
$\textsf{F}^{1,2} [\nabla\varphi^{\e}](\partial_j v^{\e,i})$
= $(\textsf{F}^1 [\nabla\varphi](\partial_j v^i), \textsf{F}^2 [\nabla\varphi](\partial_j v^i))^{\top}$ are defined in $(\ref{Sect1_Vorticity_H_Eq})$.
Note that in $\vec{\textsf{F}}^0[\nabla\varphi](\omega_h,\partial_j v^i)$, $\omega_h$ has degree one.

By introducing Lagrangian coordinates $(\ref{Sect3_Preliminaries_Vorticity_Eq_3})$,
the equations $(\ref{Sect1_N_Derivatives_Difference_Eq})$ can be transformed into the heat equation with damping and force terms.

By splitting $(\ref{Sect3_BoundarLayer_Initial_Eq_1})$ and estimating $(\ref{Sect3_BoundarLayer_Initial_Eq_Force})$
and $(\ref{Sect3_BoundarLayer_Initial_Eq_Initial})$, we investigate the effect of the initial vorticity layer.
If $\lim\limits_{\e\rto 0}\big\|\hat{\omega}_h|_{t=0}\big\|_{L^{\infty}(\mathcal{A}_0)} \neq 0$,
we prove that the limit of $\hat{\omega}_h$ is equal to that of the initial vorticity layer in Lagrangian coordinates,
thus the limit of the initial vorticity layer is transported in Eulerian coordinates.
Namely, $\lim\limits_{\e\rto 0}\|\hat{\omega}\|_{L^{\infty}(\mathcal{X}(\mathcal{A}_0)\times (0,T])} \neq 0$.

By splitting $(\ref{Sect4_BoundarLayer_Boundary_Eq_1})$ and estimating $(\ref{Sect4_BoundarLayer_Boundary_Force})$
and $(\ref{Sect4_BoundarLayer_Boundary_BC})$, we investigate the effect of
the discrepancy of boundary values of the vorticities for the inviscid limits.
If $\hat{\omega}_h|_{z=0} \neq 0$ and $\lim\limits_{\e\rto 0}\big\|\hat{\omega}_h|_{t=0}\big\|_{L^{\infty}} = 0$,
there is a discrepancy between N-S vorticity and Euler vorticity,
we prove that $\lim\limits_{\e\rto 0}\|\omega^{\e} -\omega\|_{L^{\infty}(\mathbb{R}^2\times [0, O(\e^{\frac{1}{2} -\delta_z}))\times (0,T])} \\
\neq 0$ by using symbolic analysis.

The following theorem concerns the convergence rates of the inviscid limits of $(\ref{Sect1_NS_Eq})$. Note that
if some functional space has negative indices, then such a estimate does not exist.

\begin{theorem}\label{Sect1_Thm_ConvergenceRates}
Assume $T>0$ is finite, fixed and independent of $\e$, $(v^{\e},h^{\e})$ is the solution in $[0,T]$ of Navier-Stokes equations $(\ref{Sect1_NS_Eq})$ with initial data $(v^{\e}_0,h^{\e}_0)$ satisfying $(\ref{Sect1_Proposition_TimeRegularity_1})$, $\omega^{\e}$ is its vorticity.
$(v,h)$ is the solution in $[0,T]$ of Euler equations $(\ref{Sect1_Euler_Eq})$ with initial data $(v_0,h_0)\in X^{m-1,1}(\mathbb{R}^3_{-}) \times X^{m-1,1}(\mathbb{R}^2)$, $\omega$ is its vorticity.
$g-\partial_z^{\varphi^{\e}} q^{\e} |_{z=0} \geq c_0 >0$, $g-\partial_z^{\varphi} q |_{z=0} \geq c_0 >0$.
Assume there exists an integer $k$ where $1\leq k\leq m-2$, such that $\|v^{\e}_0 -v_0\|_{X^{k-1,1}(\mathbb{R}^3_{-})} =O(\e^{\lambda^v})$, $|h^{\e}_0 -h_0|_{X^{k-1,1}(\mathbb{R}^2)} =O(\e^{\lambda^h})$, $\|\omega^{\e}_0 - \omega_0\|_{X^{k-1}(\mathbb{R}^3_{-})} =O(\e^{\lambda^{\omega}_1})$, where
$\lambda^v>0, \lambda^h>0, \lambda^{\omega}_1>0$.

If the Euler boundary data satisfies $\Pi\mathcal{S}^{\varphi} v\nn|_{z=0}\neq 0$ in $[0,T]$, then the convergence rates of the inviscid limit satisfy
\begin{equation}\label{Sect1_Thm4_ConvergenceRates_1}
\begin{array}{ll}
\|v^{\e} -v\|_{X_{tan}^{k-1,1}} + |h^{\e} -h|_{X^{k-1,1}} = O(\e^{\min\{\frac{1}{4},\lambda^v,\lambda^h, \lambda^{\omega}_1\}}), \\[7pt]
\|\NN^{\e}\cdot \partial_z^{\varphi^{\e}} v^{\e} - \NN\cdot \partial_z^{\varphi} v\|_{X_{tan}^{k-1}} + \|\NN^{\e}\cdot \omega^{\e} - \NN\cdot \omega\|_{X_{tan}^{k-1}}  = O(\e^{\min\{\frac{1}{4},\lambda^v,\lambda^h, \lambda^{\omega}_1\}}), \\[7pt]
\|\partial_z^{\varphi^{\e}} v^{\e} -\partial_z^{\varphi} v\|_{X_{tan}^{k-2}} + \|\omega^{\e} -\omega\|_{X_{tan}^{k-2}}
= O(\e^{\min\{\frac{1}{8},\frac{\lambda^v}{2},\frac{\lambda^h}{2}, \frac{\lambda^{\omega}_1}{2}\}}), \\[7pt]
\|\nabla^{\varphi^{\e}} q^{\e} - \nabla^{\varphi} q\|_{X_{tan}^{k-2}}
+ \|\triangle^{\varphi^{\e}} q^{\e} - \triangle^{\varphi} q\|_{X_{tan}^{k-2}}
= O(\e^{\min\{\frac{1}{8},\frac{\lambda^v}{2},\frac{\lambda^h}{2}, \frac{\lambda^{\omega}_1}{2}\}}),
\end{array}
\end{equation}

\begin{equation*}
\begin{array}{ll}
\|v^{\e} -v\|_{Y_{tan}^{k-3}} + |h^{\e} -h|_{Y^{k-3}} = O(\e^{\min\{\frac{1}{8},\frac{\lambda^v}{2},\frac{\lambda^h}{2}, \frac{\lambda^{\omega}_1}{2}\}}), \\[7pt]
\|\NN^{\e}\cdot \partial_z^{\varphi^{\e}} v^{\e} - \NN\cdot \partial_z^{\varphi} v\|_{Y_{tan}^{k-4}} + \|\NN^{\e}\cdot \omega^{\e} - \NN\cdot \omega\|_{Y_{tan}^{k-4}} = O(\e^{\min\{\frac{1}{8},\frac{\lambda^v}{2},\frac{\lambda^h}{2}, \frac{\lambda^{\omega}_1}{2}\}}).
\end{array}
\end{equation*}

If the Euler boundary data satisfies $\Pi\mathcal{S}^{\varphi} v\nn|_{z=0} = 0$ in $[0,T]$,
assume $\|\omega^{\e}_0 - \omega_0\|_{X^{k-2}(\mathbb{R}^3_{-})} =O(\e^{\lambda^{\omega}_2})$ where $\lambda^{\omega}_2>0$,
then the convergence rates of the inviscid limit satisfy
\begin{equation}\label{Sect1_Thm4_ConvergenceRates_2}
\begin{array}{ll}
\|v^{\e} -v\|_{X_{tan}^{k-2,1}} + |h^{\e} -h|_{X^{k-2,1}} = O(\e^{\min\{\frac{1}{2},\lambda^v,\lambda^h, \lambda^{\omega}_2\}}), \\[7pt]
\|\NN^{\e}\cdot \partial_z^{\varphi^{\e}} v^{\e} - \NN\cdot \partial_z^{\varphi}v\|_{X_{tan}^{k-2}} + \|\NN^{\e}\cdot \omega^{\e} - \NN\cdot \omega\|_{X_{tan}^{k-2}}  = O(\e^{\min\{\frac{1}{2},\lambda^v,\lambda^h, \lambda^{\omega}_2\}}), \\[7pt]
\|\partial_z^{\varphi^{\e}} v^{\e} -\partial_z^{\varphi} v\|_{X_{tan}^{k-3}} + \|\omega^{\e} -\omega\|_{X_{tan}^{k-3}}
= O(\e^{\min\{\frac{1}{4},\frac{\lambda^v}{2},\frac{\lambda^h}{2}, \frac{\lambda^{\omega}_2}{2}\}}), \\[7pt]
\|\nabla^{\varphi^{\e}} q^{\e} - \nabla^{\varphi} q\|_{X_{tan}^{k-3}}
+ \|\triangle^{\varphi^{\e}} q^{\e} - \triangle^{\varphi} q\|_{X_{tan}^{k-3}}
= O(\e^{\min\{\frac{1}{4},\frac{\lambda^v}{2},\frac{\lambda^h}{2}, \frac{\lambda^{\omega}_2}{2}\}}),
\\[7pt]

\|v^{\e} -v\|_{Y_{tan}^{k-4}} + |h^{\e} -h|_{Y^{k-4}} = O(\e^{\min\{\frac{1}{4},\frac{\lambda^v}{2},\frac{\lambda^h}{2}, \frac{\lambda^{\omega}_2}{2}\}}), \\[7pt]
\|\NN^{\e}\cdot \partial_z^{\varphi^{\e}} v^{\e} - \NN\cdot \partial_z^{\varphi} v\|_{Y_{tan}^{k-5}} + \|\NN^{\e}\cdot \omega^{\e} - \NN\cdot \omega\|_{Y_{tan}^{k-5}} = O(\e^{\min\{\frac{1}{4},\frac{\lambda^v}{2},\frac{\lambda^h}{2}, \frac{\lambda^{\omega}_2}{2}\}}).
\end{array}
\end{equation}
\end{theorem}

We give some remarks on Theorem $\ref{Sect1_Thm_ConvergenceRates}$:
\begin{remark}\label{Sect1_Remark_ConvergenceRates}
(i) Convergence rates are represented in functional spaces containing only tangential derivatives,
 such as $X_{tan}^{k-1,1}$ and $Y_{tan}^{k-3}$, because we are actually interested in
standard derivatives rather than conormal derivatives.
Theorem $\ref{Sect1_Thm_StrongLayer}$ has already described the behaviors of standard normal derivatives.
However, we have to use conormal Sobolev spaces to estimate convergence rates,
the initial data is also required to converge in conormal Sobolev spaces,
thus the index $k$ depends on the strength of the initial vorticity layer.
The weaker the initial vorticity layer is, the larger $k$ is.

(ii) In general, $\nabla\cdot (v^{\e} -v) \neq 0$, $q^{\e} -q$ is infinite for the infinite fluid depth.
Thus, we even can not obtain the $L^2$ estimate of
$(v^{\e}-v,h^{\e}-h)$ without involving $\partial_z v^{\e} - \partial_z v$.
The estimates of tangential derivatives and the estimates of normal derivatives
can not be decoupled, but tangential derivatives and normal derivatives have different convergence rates.

(iii) The convergence rate of the initial vorticity is related to whether the Euler boundary data satisfies
$\Pi\mathcal{S}^{\varphi}v \nn|_{z=0,t=0} = 0$ or not.
In general, $\lambda^{\omega}_1 \neq \lambda^{\omega}_2$.
The convergence rates of the inviscid limit for the free boundary problem are slower than those of the Navier-slip boundary case.

(iv)  To represent $\partial_z^{\varphi^{\e}} v^{\e} - \partial_z^{\varphi} v$ is more natural than $\partial_z v^{\e} - \partial_z v$.
However, the estimate of $\|\partial_z^{\varphi^{\e}} v^{\e} - \partial_z^{\varphi} v\|_{X_{tan}^{k-2}}$
results from the estimate of $\|\partial_z v^{\e} - \partial_z v\|_{X_{tan}^{k-2}}$
and $\|\partial_z(\eta^{\e} -\eta)\|_{X_{tan}^{k-2}}$, due to the formula $(\ref{Sect1_Transport_DifferenceEq})$.
The $L^{\infty}$ type estimates in $(\ref{Sect1_Thm4_ConvergenceRates_1})$ are based on the formula $\|f\|_{L^{\infty}}^2 \lem
\|f\|_{H_{tan}^{s_1}}\|\partial_z f\|_{H_{tan}^{s_2}}$ where $s_1+s_2>2$ (see \cite{Masmoudi_Rousset_2012_FreeBC}).
$\|v^{\e} -v\|_{X_{tan}^k}$ and $|h^{\e}-h|_{X^k}$ can not be estimated because we can not control $\|\partial_t^k(q^{\e} -q)\|_{L^2}$.

(v) For the finite fluid depth $\mathbb{R}^2\times [-L,0]$ and $L>0$ can be very small, if the initial data satisfy
$\|v^{\e}_0 -v_0\|_{X^{m-1,1}(\mathbb{R}^2\times [-L,0])} =O(\e^{\lambda^v},L)$, $|h^{\e}_0 -h_0|_{X^{m-1,1}(\mathbb{R}^2)} =O(\e^{\lambda^h})$.
$\|\omega^{\e}_0 - \omega_0\|_{X^{m-1}(\mathbb{R}^2\times [-L,0])} =O(\e^{\lambda^{\omega}_1},L)$.
If $\Pi\mathcal{S}^{\varphi} v\nn|_{z=0} \neq 0$ in $[0,T]$,
then the pressure itself has $L^2$ type estimates:
\begin{equation*}
\begin{array}{ll}
\|q^{\e}(\cdot,z) - q(\cdot,z)\|_{X_{tan}^{k-2}(\mathbb{R}^2)} \lem \big|q^e|_{z=0} - q|_{z=0}\big|_{X_{tan}^{k-2}}
+ \|\partial_z q^{\e} -\partial_z q\|_{X_{tan}^{k-2}} \\[7pt]\hspace{4.02cm}

\lem O(\e^{\min\{\frac{1}{8},\frac{\lambda^v}{2},\frac{\lambda^h}{2}, \frac{\lambda^{\omega}_1}{2}\}}), \\[8pt]

\|q^{\e} - q\|_{X_{tan}^{k-2}(\mathbb{R}^2\times [-L,0])} \lem O(\e^{\min\{\frac{1}{8},\frac{\lambda^v}{2},\frac{\lambda^h}{2}, \frac{\lambda^{\omega}_1}{2}\}},L), \\[6pt]

\|\omega^{\e} - \omega\|_{X_{tan}^{k-2}(\mathbb{R}^2\times [-L,0])} \lem O(\e^{\min\{\frac{1}{8},\frac{\lambda^v}{2},\frac{\lambda^h}{2}, \frac{\lambda^{\omega}_1}{2}\}}), \\[8pt]

\|v^{\e} - v\|_{X_{tan}^{k-2}(\mathbb{R}^2\times [-L,0])} \lem O(\e^{\min\{\frac{1}{4},\lambda^v,\lambda^h, \lambda^{\omega}_1\}},L), \\[8pt]
\|h^{\e} - h\|_{X_{tan}^{k-1}(\mathbb{R}^2)} \lem O(\e^{\min\{\frac{1}{4},\lambda^v,\lambda^h, \lambda^{\omega}_1\}}).
\end{array}
\end{equation*}

If $\Pi\mathcal{S}^{\varphi} v\nn|_{z=0} = 0$ in $[0,T]$, we only adjust the indices of $\e$ in the above convergence rates,
the results are similar.
\end{remark}

Now we show our strategies of the proofs.
Denote $\hat{v} = v^{\e} -v$, $\hat{h}= h^{\e} -h$, $\hat{q} = q^{\e} -q$,
then $\hat{v},\hat{h},\hat{q}$ satisfy the following equations
\begin{equation}\label{Sect1_T_Derivatives_Difference_Eq}
\left\{\begin{array}{ll}
\partial_t^{\varphi^{\e}}\hat{v} -\partial_z^{\varphi} v \partial_t^{\varphi^{\e}}\hat{\eta}
+ v^{\e} \cdot\nabla^{\varphi^{\e}} \hat{v} -v^{\e}\cdot \nabla^{\varphi^{\e}}\hat{\eta}\, \partial_z^{\varphi} v
+ \nabla^{\varphi^{\e}} \hat{q} -\partial_z^{\varphi} q\nabla^{\varphi^{\e}}\hat{\eta}
 \\[5pt]\quad
= 2\e \nabla^{\varphi^{\e}} \cdot\mathcal{S}^{\varphi^{\e}} \hat{v} + \e \triangle^{\varphi^{\e}} v - \hat{v}\cdot\nabla^{\varphi} v, 
\hspace{3.96cm}  x\in\mathbb{R}^3_{-}, \\[8pt]

\nabla^{\varphi^{\e}}\cdot \hat{v} = \partial_z^{\varphi}v \cdot\nabla^{\varphi^{\e}}\hat{\eta}, \hspace{6.38cm} x\in\mathbb{R}^3_{-},\\[8pt]

\partial_t \hat{h} + v_y\cdot \nabla_y \hat{h} = \hat{v}\cdot\NN^{\e},  \hspace{6cm} \{z=0\},
\\[7pt]

(\hat{q} -g \hat{h})\NN^{\e} -2\e \mathcal{S}^{\varphi^\e}\hat{v}\,\NN^\e = 2\e \mathcal{S}^{\varphi^\e}v\,\NN^\e, \hspace{3.8cm} \{z=0\},
\\[6pt]

(\hat{v},\hat{h})|_{t=0} = (v_0^\e -v_0,h_0^\e -h_0).
\end{array}\right.
\end{equation}

In order to close the estimates for $(\ref{Sect1_T_Derivatives_Difference_Eq})$,
we define the following variables which is similar to Alinhac's good unknown (see \cite{Alinhac_1989,Masmoudi_Rousset_2012_FreeBC}):
\begin{equation}\label{Sect1_Good_Unknown_2}
\begin{array}{ll}
\hat{V}^{\ell,\alpha} = \partial_t^{\ell}\mathcal{Z}^{\alpha}\hat{v} - \partial_z^{\varphi}v \partial_t^{\ell}\mathcal{Z}^{\alpha}\hat{\eta}, \quad\
\hat{Q}^{\ell,\alpha} = \partial_t^{\ell}\mathcal{Z}^{\alpha}\hat{q} - \partial_z^{\varphi}q \partial_t^{\ell}\mathcal{Z}^{\alpha}\hat{\eta}.
\end{array}
\end{equation}
When $|\alpha|>0$, we study the equations $(\ref{Sect5_TangentialEstimates_Diff_Eq})$ and estimate $\hat{V}^{\ell,\alpha}, \hat{Q}^{\ell,\alpha}$.
When $|\alpha|=0$, we study the equations $(\ref{Sect5_Tangential_Estimates_Time})$ and estimate $\hat{V}^{\ell,0}$ and $\partial_t^{\ell}\nabla \hat{q}$.

When we prove the estimates, we always rewrite the viscous terms by using the formula
$\triangle^{\varphi} v = 2\nabla^{\varphi}\cdot\mathcal{S}^{\varphi} v -\nabla^{\varphi}(\nabla^{\varphi}\cdot v)
= \nabla^{\varphi}(\nabla^{\varphi}\cdot v) -\nabla^{\varphi}\times(\nabla^{\varphi}\times v)$.
The dissipation of the velocity and the vorticity are controlled by using inequalities:
\begin{equation}\label{Sect1_Good_Useful_Inequalities}
\begin{array}{ll}
\|\nabla v\|_{L^2}^2 \lem \int\limits_{\mathbb{R}^3_{-}}|\mathcal{S}^{\varphi} v|^2 \,\mathrm{d}\mathcal{V}_t + \|v\|_{L^2}^2, \\[8pt]
\|\nabla \omega\|_{L^2}^2 \lem \int\limits_{\mathbb{R}^3_{-}}|\nabla^{\varphi}\times \omega|^2 \,\mathrm{d}\mathcal{V}_t + \|\omega\|_{L^2}^2
+ |\omega\cdot \nn|_{\frac{1}{2}},
\end{array}
\end{equation}
where $(\ref{Sect1_Good_Useful_Inequalities})_1$ is Korn's inequality (see \cite{Masmoudi_Rousset_2012_FreeBC,Wang_2003}),
$(\ref{Sect1_Good_Useful_Inequalities})_2$ can be proved by using Hodge decomposition and $\nabla^{\varphi}\cdot \omega =0$.

\subsection{Main Results for N-S Equations with Surface Tension}

We studied the free surface Navier-Stokes equations with surface tension,
the free surface Navier-Stokes equations $(\ref{Sect1_NavierStokes_Equation})$ with fixed $\sigma>0$ are equivalent to the following system:
\begin{equation}\label{Sect1_NS_Eq_ST}
\left\{\begin{array}{ll}
\partial_t^{\varphi} v + v\cdot\nabla^{\varphi} v + \nabla^{\varphi} q = \e\triangle^{\varphi} v, \hspace{1.04cm} x\in\mathbb{R}^3_{-}, \\[6pt]
\nabla^{\varphi}\cdot v =0, \hspace{4cm} x\in\mathbb{R}^3_{-},\\[6pt]
\partial_t h = v(t,y,0)\cdot N, \hspace{2.77cm} z=0,\\[6pt]
q\nn -2\e \mathcal{S}^{\varphi}v\,\nn =gh\nn  - \sigma H\nn, \hspace{1.25cm} z=0,\\[6pt]
(v,h)|_{t=0} = (v_0^{\e},h_0^{\e}),
\end{array}\right.
\end{equation}
where $H = - \nabla_x \cdot \big(\frac{(-\nabla_y h,1)}{\sqrt{1+|\nabla_y h|^2}}\big)
= \nabla_y\cdot \big(\frac{\nabla_y h}{\sqrt{1+|\nabla_y h|^2}}\big)$.

Let $\e \rto 0$, we formally get the free surface Euler equations with surface tension:
\begin{equation}\label{Sect1_Euler_Eq_ST}
\left\{\begin{array}{ll}
\partial_t^{\varphi} v + v\cdot\nabla^{\varphi} v + \nabla^{\varphi} q = 0, \hspace{1.69cm} x\in\mathbb{R}^3_{-}, \\[6pt]
\nabla^{\varphi}\cdot v =0, \hspace{4cm} x\in\mathbb{R}^3_{-},\\[6pt]
\partial_t h = v(t,y,0)\cdot N, \hspace{2.78cm} z=0,\\[6pt]
q =gh -\sigma H, \hspace{3.62cm} z=0,\\[6pt]
(v,h)|_{t=0} = (v_0,h_0),
\end{array}\right.
\end{equation}
where $(v_0,h_0) = \lim\limits_{\e\rto 0}(v_0^{\e},h_0^{\e})$,
we do not need the Taylor sign condition $(\ref{Sect1_TaylorSign_2})$ for $(\ref{Sect1_Euler_Eq_ST})$ when $\sigma>0$ is fixed.

For the free surface N-S equations $(\ref{Sect1_NS_Eq_ST})$, we show the regularity structure of $(\ref{Sect1_NS_Eq_ST})$ and $(\ref{Sect1_Euler_Eq_ST})$ with $\sigma>0$.
\begin{proposition}\label{Sect1_Proposition_Regularity_Tension}
Fix $\sigma>0$. For $m\geq 6$, assume the initial data $(v_0^{\e},h_0^{\e})$ satisfy the compatibility condition $\Pi\mathcal{S}^{\varphi} v_0^{\e}\nn|_{z=0} =0$ and the regularities:
\begin{equation}\label{Sect1_Proposition_Regularity_Tension_1}
\begin{array}{ll}
\sup\limits_{\e\in (0,1]} \big(
|h_0^{\e}|_{X^{m}} + \e^{\frac{1}{2}}|h_0^{\e}|_{X^{m,\frac{1}{2}}} + \sigma|h_0^{\e}|_{X^{m,1}}
+ \|v_0^{\e}\|_{X^m} + \|\omega_0^{\e}\|_{X^{m-1}} \\[7pt]\quad
+ \e\|\nabla v_0\|_{X^{m-1,1}}^2 + \e\|\nabla\omega_0\|_{X^{m-1}}^2
+ \|\omega_0^{\e}\|_{1,\infty} + \e^{\frac{1}{2}}\|\partial_z \omega_0^{\e}\|_{L^{\infty}}\big)
\leq C_0,
\end{array}
\end{equation}
then the unique Navier-Stokes solution to $(\ref{Sect1_NS_Eq})$ satisfies
\begin{equation}\label{Sect1_Proposition_Regularity_Tension_2}
\begin{array}{ll}
\sup\limits_{t\in [0,T]} \big(
|h^{\e}|_{X^{m-1,1}}^2 + \e^{\frac{1}{2}}|h^{\e}|_{X^{m-1,\frac{3}{2}}}^2 + \sigma |h^{\e}|_{X^{m-1,2}}^2
 + \|v^{\e}\|_{X^{m-1,1}}^2  + \|\partial_z v^{\e}\|_{X^{m-2}}^2 \\[9pt]\quad
  + \|\omega^{\e}\|_{X^{m-2}}^2
+ \|\partial_z v^{\e}\|_{1,\infty}^2 + \e^{\frac{1}{2}}\|\partial_{zz}v^{\e}\|_{L^{\infty}}^2 \big)

+ \|\partial_z v\|_{L^4([0,T],X^{m-1})}^2 \\[8pt]\quad
+ \|\partial_t^m v\|_{L^4([0,T],L^2)}^2 + \|\partial_t^m h\|_{L^4([0,T],L^2)}^2
+ \e\|\partial_t^m h\|_{L^4([0,T],X^{0,\frac{1}{2}})}^2 \\[6pt]\quad
+ \sigma\|\partial_t^m h\|_{L^4([0,T],X^{0,1})}^2
+ \e\int\limits_0^T \|\nabla v^{\e}\|_{X^{m-1,1}}^2 + \|\nabla\partial_z v^{\e}\|_{X^{m-2}}^2 \,\mathrm{d}t \leq C.
\end{array}
\end{equation}

As $\e\rto 0$, the Euler solution to $(\ref{Sect1_Euler_Eq})$ with the initial data $(v^0,h^0)$ satisfies the following regularities:
\begin{equation}\label{Sect1_Proposition_TimeRegularity_Tension_3}
\begin{array}{ll}
\sup\limits_{t\in [0,T]} \big(
|h|_{X^{m-1,1}} + \sigma |h|_{X^{m-1,2}}^2 + \|v\|_{X^{m-1,1}} + \|\partial_z v\|_{X^{m-2}} + \|\omega\|_{X^{m-2}} \big)

\\[9pt]\quad
+ \|\partial_z v\|_{L^4([0,T],X^{m-1})}^2 + \|\partial_t^m v\|_{L^4([0,T],L^2)}^2 + \|\partial_t^m h\|_{L^4([0,T],L^2)}^2 \\[7pt]\quad
+ \sigma\|\partial_t^m h\|_{L^4([0,T],X^{0,1})}^2
\leq C.
\end{array}
\end{equation}
\end{proposition}

For any $\sigma\geq 0$, the equation of vorticity and its equivalent boundary condition $\Pi\mathcal{S}^{\varphi}v \nn|_{z=0} =0$ are the same,
thus the estimates of normal derivatives are also the same. Thus, our proof of Proposition $\ref{Sect1_Proposition_Regularity_Tension}$
is the same as the $\sigma=0$ case.
Note that our proof is different from \cite{Masmoudi_Rousset_2012_FreeBC,Wang_Xin_2015}.

However, the estimates of the pressure are very different from the $\sigma=0$ case.
If we couple $\triangle^{\varphi} q = -\partial_j^{\varphi} v^i \partial_i^{\varphi} v^j$ with
its nonhomogeneous Dirichlet boundary condition $q|_{z=0} = g h - \sigma H + 2\e\mathcal{S}^{\varphi} v\nn\cdot\nn$,
the estimates can not be closed due to the less regularity of $h$.

The elliptic equation of the pressure coupled with its Neumann boundary condition is as follows (see \cite{Wang_Xin_2015,Elgindi_Lee_2014,Masmoudi_Rousset_2012_NavierBC}):
\begin{equation}\label{Sect1_Pressure_Neumann}
\left\{\begin{array}{ll}
\triangle^{\varphi} q = -\partial_j^{\varphi} v^i \partial_i^{\varphi} v^j, \\[5pt]

\nabla^{\varphi} q\cdot\NN|_{z=0} = - \partial_t^{\varphi} v\cdot\NN - v\cdot\nabla^{\varphi} v\cdot\NN
+ \e\triangle^{\varphi} v\cdot\NN,
\end{array}\right.
\end{equation}
Using $(\ref{Sect1_Pressure_Neumann})$ to estimate the pressure, we have to prove $\partial_t^m v\in L^4([0,T],L^2)$.
When we estimate $\partial_t^m v$, we have to overcome the difficulties generated by $\partial_t^m q$.
Besides integrating the energy estimates in time twice (see \cite{Wang_Xin_2015}), we apply the following Hardy's inequality (see \cite{Masmoudi_Wong_2015})
to the terms $\partial_z (\partial_t^{m-1}q \cdot f)$ where $\|f\|_{L^2(\mathbb{R}^3_{-}})$ and $\|(1-z)\partial_z f\|_{L^2(\mathbb{R}^3_{-}})$
are bounded.
\begin{equation}\label{Sect1_HardyIneq}
\begin{array}{ll}
\|\frac{1}{1-z} \partial_t^{\ell} q\|_{L^2(\mathbb{R}^3_{-})}
\lem \big|\partial_t^{\ell} q|_{z=0}\big|_{L^2(\mathbb{R}^2)} + \|\partial_z \partial_t^{\ell} q\|_{L^2(\mathbb{R}^3_{-})},
\quad 0\leq\ell\leq m-1.
\end{array}
\end{equation}

Note that \cite{Wang_Xin_2015} considered the finite fluid depth,
for which $\|\partial_t^{\ell} q\|_{L^2(\mathbb{R}^2\times [-L,0])}$ is bounded.
While we consider the infinite fluid depth in this paper, $\|\partial_t^{\ell} q\|_{L^2}$ has no bound in general.
Another difference is that \cite{Wang_Xin_2015} needs Taylor sign condition and Alinhac's good unknown $(\ref{Sect1_Good_Unknown_1})$
to estimate tangential derivatives. While we do not need them for the fixed $\sigma>0$.

D. Coutand and S. Shkoller (see \cite{Coutand_Shkoller_2007}) proved the well-posedness of
the free surface incompressible Euler equations $(\ref{Sect1_Euler_Eq_ST})$ with surface tension.
We state their results in our formulation as follows:

{\it
Suppose that $\sigma>0$ is fixed, $h_0\in H^{5.5}(\mathbb{R}^2), v_0\in H^{4.5}(\mathbb{R}^3_{-})$, then there exists $T>0$ and a solution $(v,q,h)$ of $(\ref{Sect1_Euler_Eq_ST})$
with $v\in L^{\infty}([0,T],H^{4.5}(\mathbb{R}^3_{-})), \nabla q\in L^{\infty}([0,T],H^3(\mathbb{R}^3_{-})), h\in L^{\infty}([0,T],H^{5.5}(\mathbb{R}^2))$. The solution is unique if
$h_0\in H^{6.5}(\mathbb{R}^2), v_0\in H^{5.5}(\mathbb{R}^3_{-})$. }

\vspace{0.3cm}

Since the equations of the vorticity and its equivalent boundary condition $\Pi\mathcal{S}^{\varphi}v\nn|_{z=0} =0$ are the same as the $\sigma=0$ case,
Theorem $\ref{Sect1_Thm_StrongLayer}$ is also valid for the equations $(\ref{Sect1_NS_Eq})$ with $\sigma>0$. Thus, we are mainly concerned with convergence rates of the inviscid limit. The results are stated in the following theorem:
\begin{theorem}\label{Sect1_Thm_StrongLayer_ST}
Assume $T>0$ is finite, fixed and independent of $\e$, $(v^{\e},h^{\e})$ is the solution in $[0,T]$ of Navier-Stokes equations $(\ref{Sect1_NS_Eq})$ with initial data $(v^{\e}_0,h^{\e}_0)$ satisfying $(\ref{Sect1_Proposition_TimeRegularity_1})$, $\omega^{\e}$ is its vorticity.
$(v,h)$ is the solution in $[0,T]$ of Euler equations $(\ref{Sect1_Euler_Eq})$ with initial data $(v_0,h_0)\in X^{m-1,1}(\mathbb{R}^3_{-}) \times X^{m-1,1}(\mathbb{R}^2)$, $\omega$ is its vorticity.
Assume there exists an integer $k$ where $1\leq k\leq m-2$, such that $\|v^{\e}_0 -v_0\|_{X^{k}(\mathbb{R}^3_{-})} =O(\e^{\lambda^v})$,
$|h^{\e}_0 -h_0|_{X^{k}(\mathbb{R}^2)} =O(\e^{\lambda^h})$, $\|\omega^{\e}_0 - \omega_0\|_{X^{k-1}(\mathbb{R}^3_{-})} =O(\e^{\lambda^{\omega}_1})$, where
$\lambda^v>0, \lambda^h>0, \lambda^{\omega}_1>0$.

If the Euler boundary data satisfies $\Pi\mathcal{S}^{\varphi} v\nn|_{z=0}\neq 0$ in $[0,T]$, then the convergence rates of the inviscid limit satisfy
\begin{equation}\label{Sect1_Thm7_ConvergenceRates_1}
\begin{array}{ll}
\|v^{\e} -v\|_{X_{tan}^{k-1,1}} + |h^{\e} -h|_{X^{k-1,2}} = O(\e^{\min\{\frac{1}{4},\lambda^v,\lambda^h, \lambda^{\omega}_1\}}), \\[7pt]
\|\NN^{\e}\cdot \partial_z^{\varphi^{\e}} v^{\e} - \NN\cdot \partial_z^{\varphi} v\|_{X_{tan}^{k-1}} + \|\NN^{\e}\cdot \omega^{\e} - \NN\cdot \omega\|_{X_{tan}^{k-1}}  = O(\e^{\min\{\frac{1}{4},\lambda^v,\lambda^h, \lambda^{\omega}_1\}}), \\[7pt]
\|\partial_z^{\varphi^{\e}} v^{\e} -\partial_z^{\varphi} v\|_{X_{tan}^{k-2}} + \|\omega^{\e} -\omega\|_{X_{tan}^{k-2}}
= O(\e^{\min\{\frac{1}{8},\frac{\lambda^v}{2},\frac{\lambda^h}{2}, \frac{\lambda^{\omega}_1}{2}\}}), \\[7pt]

\|\nabla^{\varphi^{\e}} q^{\e} - \nabla^{\varphi} q\|_{X_{tan}^{k-2}} + \|\triangle^{\varphi^{\e}} q^{\e} - \triangle^{\varphi} q\|_{X_{tan}^{k-2}}
= O(\e^{\min\{\frac{1}{8},\frac{\lambda^v}{2},\frac{\lambda^h}{2}, \frac{\lambda^{\omega}_1}{2}\}}),
\\[7pt]

\|v^{\e} -v\|_{Y_{tan}^{k-3}} + |h^{\e} -h|_{Y^{k-2}} = O(\e^{\min\{\frac{1}{8},\frac{\lambda^v}{2},\frac{\lambda^h}{2}, \frac{\lambda^{\omega}_1}{2}\}}), \\[7pt]
\|\NN^{\e}\cdot \partial_z^{\varphi^{\e}} v^{\e} - \NN\cdot \partial_z^{\varphi} v\|_{Y_{tan}^{k-4}} + \|\NN^{\e}\cdot \omega^{\e} - \NN\cdot \omega\|_{Y_{tan}^{k-4}} = O(\e^{\min\{\frac{1}{8},\frac{\lambda^v}{2},\frac{\lambda^h}{2}, \frac{\lambda^{\omega}_1}{2}\}}).
\end{array}
\end{equation}

If the Euler boundary data satisfies $\Pi\mathcal{S}^{\varphi} v\nn|_{z=0} = 0$ in $[0,T]$,
assume $\|\omega^{\e}_0 - \omega_0\|_{X^{k-2}(\mathbb{R}^3_{-})} =O(\e^{\lambda^{\omega}_2})$ where $\lambda^{\omega}_2>0$,
then the convergence rates of the inviscid limit satisfy
\begin{equation}\label{Sect1_Thm7_ConvergenceRates_2}
\begin{array}{ll}
\|v^{\e} -v\|_{X_{tan}^{k-2,1}} + |h^{\e} -h|_{X^{k-2,2}} = O(\e^{\min\{\frac{1}{2},\lambda^v,\lambda^h, \lambda^{\omega}_2\}}), \\[7pt]
\|\NN^{\e}\cdot \partial_z^{\varphi^{\e}} v^{\e} - \NN\cdot \partial_z^{\varphi} v\|_{X_{tan}^{k-2}} + \|\NN^{\e}\cdot \omega^{\e} - \NN\cdot \omega\|_{X_{tan}^{k-2}}  = O(\e^{\min\{\frac{1}{2},\lambda^v,\lambda^h, \lambda^{\omega}_2\}}), \\[7pt]
\|\partial_z^{\varphi^{\e}} v^{\e} -\partial_z^{\varphi} v\|_{X_{tan}^{k-3}} + \|\omega^{\e} -\omega\|_{X_{tan}^{k-3}}
= O(\e^{\min\{\frac{1}{4},\frac{\lambda^v}{2},\frac{\lambda^h}{2}, \frac{\lambda^{\omega}_2}{2}\}}), \\[7pt]

\|\nabla^{\varphi^{\e}} q^{\e} - \nabla^{\varphi} q\|_{X_{tan}^{k-3}} + \|\triangle^{\varphi^{\e}} q^{\e} - \triangle^{\varphi} q\|_{X_{tan}^{k-3}}
= O(\e^{\min\{\frac{1}{4},\frac{\lambda^v}{2},\frac{\lambda^h}{2}, \frac{\lambda^{\omega}_2}{2}\}}),
\\[7pt]

\|v^{\e} -v\|_{Y_{tan}^{k-4}} + |h^{\e} -h|_{Y^{k-3}} = O(\e^{\min\{\frac{1}{4},\frac{\lambda^v}{2},\frac{\lambda^h}{2}, \frac{\lambda^{\omega}_2}{2}\}}), \\[7pt]
\|\NN^{\e}\cdot \partial_z^{\varphi^{\e}} v^{\e} - \NN\cdot \partial_z^{\varphi} v\|_{Y_{tan}^{k-5}} + \|\NN^{\e}\cdot \omega^{\e} - \NN\cdot \omega\|_{Y_{tan}^{k-5}} = O(\e^{\min\{\frac{1}{4},\frac{\lambda^v}{2},\frac{\lambda^h}{2}, \frac{\lambda^{\omega}_2}{2}\}}).
\end{array}
\end{equation}
\end{theorem}

Beside Remark $\ref{Sect1_Remark_ConvergenceRates}$, we supplement the following remarks:
\begin{remark}\label{Sect1_Remark_ST}
(i) When $\sigma>0$, the surface tension changes the regularity structure of Navier-Stokes solutions and Euler solutions,
it does not change convergence rates of the inviscid limit. For the fixed $\sigma>0$, we need neither Taylor sign condition
nor Alinhac's good unknown. However, if $\sigma\rto 0$ is allowed, we need both Taylor sign condition and Alinhac's good unknown to close
a priori estimates.

(ii). \cite{Ambrose_Masmoudi_2005,Ambrose_Masmoudi_2009,Wang_Xin_2015} studied the zero surface tension limit of water waves or the free surface N-S equations, but convergence rates of the zero surface tension limit are unknown.
For the equations $(\ref{Sect1_NS_Eq_ST})$, $\sigma\rto 0$ is very different from $\e\rto 0$, because
$\e\rto 0$ implies some boundary layers generate, $\sigma\rto 0$ implies the height function $h$ loses some regularities.
By using the variables $(\ref{Sect1_Good_Unknown_2})$, it is not difficult for extending our estimates to provide the convergence rates.

Assume $1\leq k\leq m-2$, $\|v^{\e}_0 -v_0\|_{X^{k}(\mathbb{R}^3_{-})} =O(\sigma^{\mu^v}, \e^{\lambda^v})$, $|h^{\e}_0 -h_0|_{X^{k}(\mathbb{R}^2)}  =O(\sigma^{\mu^h}, \e^{\lambda^h})$, $\|\omega^{\e}_0 - \omega_0\|_{X^{k-1}(\mathbb{R}^3_{-})} =O(\sigma^{\mu^{\omega}}, \e^{\lambda^{\omega}_1})$,
$g- \partial_z^{\varphi^{\e}} q^{\e} \geq c_0>0$.
If $\Pi\mathcal{S}^{\varphi} v\nn|_{z=0}$ $\neq 0$ in $[0,T]$, we have the convergence rates of the inviscid limit:
\begin{equation*}
\begin{array}{ll}
\|v^{\e} -v\|_{X_{tan}^{k-1,1}} + |h^{\e} -h|_{X^{k-1,1}} = O(\sigma^{\min\{\mu^v,\mu^h,\mu^{\omega}\}},
\e^{\min\{\frac{1}{4},\lambda^v,\lambda^h, \lambda^{\omega}_1\}}), \\[6pt]
\|\partial_z^{\varphi^{\e}} v^{\e} -\partial_z^{\varphi} v\|_{X_{tan}^{k-2}} + \|\omega^{\e} -\omega\|_{X_{tan}^{k-2}}
= O(\sigma^{\min\{\mu^v,\mu^h,\mu^{\omega}\}},
\e^{\min\{\frac{1}{8},\frac{\lambda^v}{2},\frac{\lambda^h}{2}, \frac{\lambda^{\omega}_1}{2}\}}), \\[6pt]

\|\nabla^{\varphi^{\e}} q^{\e} - \nabla^{\varphi} q\|_{X_{tan}^{k-2}} + \|\triangle^{\varphi^{\e}} q^{\e} - \triangle^{\varphi} q\|_{X_{tan}^{k-2}}
\\[5pt]\hspace{3.2cm}
= O(\sigma^{\min\{\mu^v,\mu^h,\mu^{\omega}\}},
\e^{\min\{\frac{1}{8},\frac{\lambda^v}{2},\frac{\lambda^h}{2}, \frac{\lambda^{\omega}_1}{2}\}}),
\\[6pt]

\|v^{\e} -v\|_{Y_{tan}^{k-3}} + |h^{\e} -h|_{Y^{k-3}} = O(\sigma^{\min\{\mu^v,\mu^h,\mu^{\omega}\}},
\e^{\min\{\frac{1}{8},\frac{\lambda^v}{2},\frac{\lambda^h}{2}, \frac{\lambda^{\omega}_1}{2}\}}).
\end{array}
\end{equation*}

If $\Pi\mathcal{S}^{\varphi} v\nn|_{z=0}= 0$ in $[0,T]$, we only adjust the indices of $\e$ in the above convergence rates,
the results are similar.
\end{remark}

\vspace{0.2cm}
The rest of the paper is organized as follows: In Section 2, we study the boundary value of the vorticity, determine the regularity structure of
N-S solutions with $\sigma=0$. In Section 3, we study the strong vorticity layer caused by the strong initial vorticity layer.
In Section 4, we study the strong vorticity layer caused by the discrepancy between boundary values of the vorticities.
In Section 5, we estimate the convergence rates of the inviscid limit for $\sigma=0$. In Section 6, we determine the regularity structure
of N-S solutions with $\sigma>0$. In Section 7, we estimate the convergence rates of the inviscid limit for $\sigma>0$.
In the Appendices A and B, we derive the equations and their boundary conditions which are useful for a priori estimates.

%%% find 2
\section{Vorticity, Normal Derivatives and Regularity Structure of Navier-Stokes Solutions for $\sigma=0$}
In this section, we determine the relationship between the vorticity on the free boundary and normal derivatives of the velocity on the free boundary,
and derive the equations of $\omega_h =(\omega^1,\omega^2)$ and their boundary conditions.
When $\sigma=0$, we prove Proposition $\ref{Sect1_Proposition_TimeRegularity}$ on the regularities of Navier-Stokes solutions and Euler solutions.
For simplicity, we omit the superscript ${}^{\e}$ in this section, which represents Navier-Stokes solutions.

\subsection{Vorticity and Normal Derivatives on the Free Boundary}
The following lemma states that the normal derivatives $(\partial_z v^1,\partial_z v^2)$ can be estimated by
the tangential vorticity $\omega_h$.
\begin{lemma}\label{Sect2_NormalDer_Vorticity_Lemma}
Assume $v$ and $\omega$ are the velocity and vorticity of the free surface Navier-Stokes equations $(\ref{Sect1_NS_Eq})$ respectively,
$\|v\|_{X^{m-1,1}} + |h|_{X^{m-1,1}} <+\infty$, then
\begin{equation}\label{Sect2_NormalDer_Vorticity_Estimate}
\begin{array}{ll}
\|\partial_z v^1\|_{X^k} + \|\partial_z v^2\|_{X^k}
\lem \|\omega_h\|_{X^k} + \|v\|_{X^{k,1}} + |h|_{X^{k,\frac{1}{2}}}, \quad k\leq m-1.
\end{array}
\end{equation}

\end{lemma}
\begin{proof}
We calculate the vorticity:
\begin{equation}\label{Sect2_NormalDer_Vorticity_Estimate_1}
\left\{\begin{array}{ll}
\omega^1 = \partial_2^{\varphi} v^3 - \partial_z^{\varphi} v^2
= \partial_2 v^3 - \frac{\partial_2\varphi}{\partial_z\varphi}\partial_z v^3
- \frac{1}{\partial_z\varphi}\partial_z v^2, \\[8pt]

\omega^2 = \partial_z^{\varphi} v^1 -\partial_1^{\varphi} v^3
= - \partial_1 v^3 + \frac{\partial_1\varphi}{\partial_z\varphi}\partial_z v^3
+ \frac{1}{\partial_z\varphi}\partial_z v^1, \\[8pt]

\omega^3 = \partial_1^{\varphi} v^2 - \partial_2^{\varphi} v^1
= \partial_1 v^2 - \frac{\partial_1\varphi}{\partial_z\varphi}\partial_z v^2 - \partial_2 v^1 + \frac{\partial_2\varphi}{\partial_z\varphi}\partial_z v^1.
\end{array}\right.
\end{equation}

Plug the following divergence free condition
\begin{equation}\label{Sect2_DivFreeCondition}
\begin{array}{ll}
\partial_z v^3 = \partial_1\varphi\partial_z v^1 + \partial_2\varphi\partial_z v^2 - \partial_z\varphi(\partial_1 v^1 + \partial_2 v^2),
\end{array}
\end{equation}
into $(\ref{Sect2_NormalDer_Vorticity_Estimate_1})$, we get
\begin{equation}\label{Sect2_NormalDer_Vorticity_Estimate_2}
\left\{\begin{array}{ll}
\omega^1 = - \frac{\partial_1\varphi\partial_2\varphi}{\partial_z\varphi}\partial_z v^1
- \frac{1 + (\partial_2\varphi)^2}{\partial_z\varphi}\partial_z v^2
+ \partial_2 v^3 + \partial_2\varphi(\partial_1 v^1 + \partial_2 v^2),
\\[8pt]

\omega^2 = \frac{1+(\partial_1\varphi)^2}{\partial_z\varphi}\partial_z v^1
+ \frac{\partial_1\varphi\partial_2\varphi}{\partial_z\varphi}\partial_z v^2
- \partial_1 v^3 - \partial_1\varphi(\partial_1 v^1 + \partial_2 v^2).
\end{array}\right.
\end{equation}

It follows from $(\ref{Sect2_NormalDer_Vorticity_Estimate_2})$ that
\begin{equation}\label{Sect2_NormalDer_Vorticity_Estimate_3}
\left\{\begin{array}{ll}
\frac{\partial_1\varphi\partial_2\varphi}{\partial_z\varphi}\partial_z v^1
+ \frac{1 + (\partial_2\varphi)^2}{\partial_z\varphi}\partial_z v^2
= - \omega^1 + \partial_2 v^3 + \partial_2\varphi(\partial_1 v^1 + \partial_2 v^2),
\\[8pt]

\frac{1+(\partial_1\varphi)^2}{\partial_z\varphi}\partial_z v^1
+ \frac{\partial_1\varphi\partial_2\varphi}{\partial_z\varphi}\partial_z v^2
= \omega^2 + \partial_1 v^3 + \partial_1\varphi(\partial_1 v^1 + \partial_2 v^2).
\end{array}\right.
\end{equation}

For $(\ref{Sect2_NormalDer_Vorticity_Estimate_3})$, the determinant of the coefficient matrix of $(\partial_z v^1, \partial_z v^2)^{\top}$ is
\begin{equation}\label{Sect2_NormalDer_Vorticity_Estimate_4}
\begin{array}{ll}
\left|\begin{array}{cc}
\frac{\partial_1\varphi\partial_2\varphi}{\partial_z\varphi}
& \frac{1 + (\partial_2\varphi)^2}{\partial_z\varphi} \\[4pt]
\frac{1+(\partial_1\varphi)^2}{\partial_z\varphi}
& \frac{\partial_1\varphi\partial_2\varphi}{\partial_z\varphi}
\end{array}\right|

= -\frac{1+(\partial_1\varphi)^2 + (\partial_2\varphi)^2}{(\partial_z\varphi)^2} \neq 0,
\end{array}
\end{equation}
thus we can solve $\partial_z v^1$ and $\partial_z v^2$ from $(\ref{Sect2_NormalDer_Vorticity_Estimate_3})$, namely
there exist four homogeneous polynomials $f^k[\nabla\varphi](\partial_j v^i),\, k=1,2,3,4$,
which are one order with respect to $\partial_j v^i$, the coefficients are fractions of $\nabla\varphi$.
\begin{equation}\label{Sect2_NormalDer_Vorticity_Estimate_5}
\left\{\begin{array}{ll}
\partial_z v^1 = f^1[\nabla\varphi](\omega^1,\omega^2) + f^2[\nabla\varphi](\partial_j v^i),\ j=1,2,\ i=1,2,3, \\[8pt]
\partial_z v^2 = f^3[\nabla\varphi](\omega^1,\omega^2) + f^4[\nabla\varphi](\partial_j v^i),\ j=1,2,\ i=1,2,3,
\end{array}\right.
\end{equation}
then we have the estimates:
\begin{equation}\label{Sect2_NormalDer_Vorticity_Estimate_6}
\begin{array}{ll}
\|\partial_z v^1\|_{X^k} + \|\partial_z v^2\|_{X^k}
\lem \|\omega_h\|_{X^k} + \sum\limits_{i,j}\|\partial_j v^i\|_{X^k} + \|\nabla\varphi\|_{X^k}.
\end{array}
\end{equation}
Thus, Lemma $\ref{Sect2_NormalDer_Vorticity_Lemma}$ is proved.
\end{proof}

Because the Navier-Stokes boundary data satisfies $\Pi\mathcal{S}^{\varphi} v \nn|_{z=0} =0$, 
the following lemma claims that the boundary value of normal derivatives and tangential vorticity
can be expressed in terms of that of tangential derivatives,
the tangential vorticity $\omega_h$ satisfies $(\ref{Sect1_Vorticity_H_Eq})$.
\begin{lemma}\label{Sect2_Vorticity_H_Eq_BC_Lemma}
Assume $v$ and $\omega$ are the velocity and vorticity of the free surface Navier-Stokes equations $(\ref{Sect1_NS_Eq})$ respectively.
If $\e>0$, then there exist polynomials $\vec{\textsf{F}}^0[\nabla\varphi](\omega_h,\partial_j v^i)$, $\textsf{F}^1 [\nabla\varphi](\partial_j v^i)$, $\textsf{F}^2 [\nabla\varphi](\partial_j v^i)$ such that $\omega_h$ satisfies $(\ref{Sect1_Vorticity_H_Eq})$, 
where $\vec{\textsf{F}}^0[\nabla\varphi](\omega_h,\partial_j v^i)$ is a quadratic polynomial vector with respect to $\omega_h$ and $\partial_j v^i$, $\textsf{F}^1 [\nabla\varphi](\partial_j v^i)$, $\textsf{F}^2 [\nabla\varphi](\partial_j v^i)$ are polynomials with respect to $\partial_j v^i$, all the coefficients are fractions of $\nabla\varphi$.
\end{lemma}

\begin{proof}
Firstly, we investigate the following quantity on the free boundary:
\begin{equation}\label{Sect2_Vorticity_H_BC_1}
\begin{array}{ll}
\mathcal{S}^{\varphi}v \nn
= \left(\begin{array}{c}
n^1\partial_1^{\varphi} v^1 + \frac{n^2}{2}(\partial_1^{\varphi} v^2 + \partial_2^{\varphi} v^1)
+ \frac{n^3}{2}(\partial_1^{\varphi} v^3 + \partial_z^{\varphi} v^1) \\[4pt]

\frac{n^1}{2}(\partial_1^{\varphi} v^2 + \partial_2^{\varphi} v^1) + n^2\partial_2^{\varphi} v^2
+ \frac{n^3}{2}(\partial_2^{\varphi} v^3 + \partial_z^{\varphi} v^2) \\[4pt]

\frac{n^1}{2}(\partial_1^{\varphi} v^3 + \partial_z^{\varphi} v^1) + \frac{n^2}{2}(\partial_2^{\varphi} v^3 + \partial_z^{\varphi} v^2)
- n^3\partial_1^{\varphi} v^1 - n^3\partial_2^{\varphi} v^2
\end{array}\right).
\end{array}
\end{equation}

Since $\Pi \mathcal{S}^{\varphi} v \nn =0$, $\mathcal{S}^{\varphi}v\nn = (\mathcal{S}^{\varphi}v \nn\cdot\nn)\nn$,
then $\mathcal{S}^{\varphi}v \nn$ is parallel to $\nn$. By using $\mathcal{S}^{\varphi}v \nn\times\nn =0$, we have
\begin{equation}\label{Sect2_Vorticity_H_BC_2}
\left\{\begin{array}{ll}
n^3[n^1\partial_1^{\varphi} v^1 + \frac{n^2}{2}(\partial_1^{\varphi} v^2 + \partial_2^{\varphi} v^1)
+ \frac{n^3}{2}(\partial_1^{\varphi} v^3 + \partial_z^{\varphi} v^1)] \\[5pt]\quad

= n^1[\frac{n^1}{2}(\partial_1^{\varphi} v^3 + \partial_z^{\varphi} v^1) + \frac{n^2}{2}(\partial_2^{\varphi} v^3 + \partial_z^{\varphi} v^2)
- n^3\partial_1^{\varphi} v^1 - n^3\partial_2^{\varphi} v^2],

\\[10pt]

n^3[\frac{n^1}{2}(\partial_1^{\varphi} v^2 + \partial_2^{\varphi} v^1) + n^2\partial_2^{\varphi} v^2
+ \frac{n^3}{2}(\partial_2^{\varphi} v^3 + \partial_z^{\varphi} v^2)] \\[5pt]\quad

= n^2[\frac{n^1}{2}(\partial_1^{\varphi} v^3 + \partial_z^{\varphi} v^1) + \frac{n^2}{2}(\partial_2^{\varphi} v^3 + \partial_z^{\varphi} v^2)
- n^3\partial_1^{\varphi} v^1 - n^3\partial_2^{\varphi} v^2],

\\[10pt]

n^2[n^1\partial_1^{\varphi} v^1 + \frac{n^2}{2}(\partial_1^{\varphi} v^2 + \partial_2^{\varphi} v^1)
+ \frac{n^3}{2}(\partial_1^{\varphi} v^3 + \partial_z^{\varphi} v^1)] \\[5pt]\quad

= n^1[\frac{n^1}{2}(\partial_1^{\varphi} v^2 + \partial_2^{\varphi} v^1) + n^2\partial_2^{\varphi} v^2
+ \frac{n^3}{2}(\partial_2^{\varphi} v^3 + \partial_z^{\varphi} v^2)].
\end{array}\right.
\end{equation}

Firstly, we solve $\partial_z^{\varphi} v^1$ from $(\ref{Sect2_Vorticity_H_BC_2})$:
\begin{equation}\label{Sect2_Vorticity_H_BC_4}
\begin{array}{ll}
[\frac{(n^3)^2}{2} - \frac{(n^1)^2}{2} - \frac{(n^2)^2}{2}]\partial_z^{\varphi} v^1

= -[\frac{(n^3)^2}{2} - \frac{(n^1)^2}{2} - \frac{(n^2)^2}{2}]\partial_1^{\varphi} v^3 \\[6pt]\quad

+ (\frac{n^1(n^2)^2}{n^3} - 2n^1n^3)\partial_1^{\varphi} v^1
- (n^1n^3 + \frac{n^1(n^2)^2}{n^3})\partial_2^{\varphi} v^2 \\[6pt]\quad

+ [\frac{(n^2)^2}{n^3}\frac{n^2}{2}- \frac{n^1n^2}{n^3}\frac{n^1}{2} -\frac{n^2n^3}{2}](\partial_1^{\varphi} v^2 + \partial_2^{\varphi} v^1),
\end{array}
\end{equation}

Secondly, we solve $\partial_z^{\varphi} v^2$ from $(\ref{Sect2_Vorticity_H_BC_2})$:
\begin{equation}\label{Sect2_Vorticity_H_BC_6}
\begin{array}{ll}
[\frac{(n^3)^2}{2} - \frac{(n^2)^2}{2} - \frac{(n^1)^2}{2}]\partial_z^{\varphi} v^2

= -[\frac{(n^3)^2}{2} - \frac{(n^2)^2}{2} - \frac{(n^1)^2}{2}]\partial_2^{\varphi} v^3 \\[6pt]\quad

- (n^2n^3 + \frac{(n^1)^2n^2}{n^3})\partial_1^{\varphi} v^1
+ (\frac{(n^1)^2n^2}{n^3} -2n^2n^3)\partial_2^{\varphi} v^2 \\[6pt]\quad

+ (\frac{(n^1)^2}{n^3}\frac{n^1}{2} -\frac{n^1n^3}{2} - \frac{n^1n^2}{n^3}\frac{n^2}{2})
(\partial_1^{\varphi} v^2 + \partial_2^{\varphi} v^1).
\end{array}
\end{equation}

It follows from $(\ref{Sect2_Vorticity_H_BC_4})$ and $(\ref{Sect2_Vorticity_H_BC_6})$ that
\begin{equation*}
\begin{array}{ll}
\big[(n^1)^2 + \frac{(n^3)^2}{2} + \frac{1}{2}\frac{(n^1)^4}{(n^3)^2} - \frac{1}{2}\frac{(n^2)^4}{(n^3)^2}
\big]\partial_z v^1
+ \big[n^1n^2 + \frac{(n^1)^3n^2}{(n^3)^2} + \frac{n^1(n^2)^3}{(n^3)^2} \big]\partial_z v^2 \\[8pt]

= -[\frac{(n^3)^2}{2} - \frac{(n^1)^2}{2} - \frac{(n^2)^2}{2}]\big[\partial_z\varphi\partial_1 v^3 - \partial_1\varphi
[- \partial_z\varphi(\partial_1 v^1 + \partial_2 v^2)] \big]
\\[5pt]\quad

+ (\frac{n^1(n^2)^2}{n^3} - 2n^1n^3)(\partial_z\varphi\partial_1 v^1)
- (n^1n^3 + \frac{n^1(n^2)^2}{n^3})(\partial_z\varphi\partial_2 v^2) \\[5pt]\quad

+ [\frac{(n^2)^2}{n^3}\frac{n^2}{2}- \frac{n^1n^2}{n^3}\frac{n^1}{2} -\frac{n^2n^3}{2}]
(\partial_z\varphi\partial_1 v^2 + \partial_z\varphi\partial_2 v^1),
\end{array}
\end{equation*}

\begin{equation*}
\begin{array}{ll}
\big[n^1n^2 + \frac{(n^1)^3n^2}{(n^3)^2} + \frac{n^1(n^2)^3}{(n^3)^2} \big]\partial_z v^1
+ \big[(n^2)^2 + \frac{1}{2}(n^3)^2 + \frac{(n^2)^4}{2(n^3)^2} - \frac{(n^1)^4}{2(n^3)^2}\big]\partial_z v^2 \\[8pt]

= -[\frac{(n^3)^2}{2} - \frac{(n^2)^2}{2} - \frac{(n^1)^2}{2}]\big[\partial_z\varphi\partial_2 v^3 - {\partial_z\varphi}
[- \partial_z\varphi(\partial_1 v^1 + \partial_2 v^2)]\big] \\[5pt]\quad

- (n^2n^3 + \frac{(n^1)^2n^2}{n^3})(\partial_z\varphi\partial_1 v^1)
+ (\frac{(n^1)^2n^2}{n^3} -2n^2n^3)(\partial_z\varphi\partial_2 v^2) \\[5pt]\quad

+ (\frac{(n^1)^2}{n^3}\frac{n^1}{2} -\frac{n^1n^3}{2} - \frac{n^1n^2}{n^3}\frac{n^2}{2})
(\partial_z\varphi\partial_1 v^2 + \partial_z\varphi\partial_2 v^1).
\end{array}
\end{equation*}
where the coefficient matrix of $(\partial_z v^1, \partial_z v^2)^{\top}$ is
\begin{equation}\label{Sect2_Vorticity_H_BC_9}
\begin{array}{ll}
\textsf{M} = \left(\begin{array}{cc}
(n^1)^2 + \frac{(n^3)^2}{2} + \frac{1}{2}\frac{(n^1)^4}{(n^3)^2} - \frac{1}{2}\frac{(n^2)^4}{(n^3)^2}
& n^1n^2 + \frac{(n^1)^3n^2}{(n^3)^2} + \frac{n^1(n^2)^3}{(n^3)^2} \\[4pt]
n^1n^2 + \frac{(n^1)^3n^2}{(n^3)^2} + \frac{n^1(n^2)^3}{(n^3)^2}
& (n^2)^2 + \frac{1}{2}(n^3)^2 + \frac{(n^2)^4}{2(n^3)^2} - \frac{(n^1)^4}{2(n^3)^2}
\end{array}\right).
\end{array}
\end{equation}

Assume $|\nabla h|_{\infty}$ is suitably small, then $n^3$ is suitably large and $|n^1|+|n^2|$ is suitably small such that
$\textsf{M}$ is strictly diagonally dominant matrix. By Levy-Desplanques theorem, $\textsf{M}$ is nondegenerate,
thus we can solve $\partial_z v^1$ and $\partial_z v^2$, namely
there exist two homogeneous polynomials $f^5[\nabla\varphi](\partial_j v^i)$ and $f^6[\nabla\varphi](\partial_j v^i)$,
which are one order with respect to $\partial_j v^i$, the coefficients are fractions of $\nabla\varphi$.
\begin{equation}\label{Sect2_Vorticity_H_BC_10}
\left\{\begin{array}{ll}
\partial_z v^1 = f^5[\nabla\varphi](\partial_j v^i), \ j=1,2,\ i=1,2,3,\\[6pt]
\partial_z v^2 = f^6[\nabla\varphi](\partial_j v^i), \ j=1,2,\ i=1,2,3.
\end{array}\right.
\end{equation}

By $(\ref{Sect2_NormalDer_Vorticity_Estimate_2})$, we have the boundary values of $\omega_h = (\omega^1,\omega^2)$:
\begin{equation}\label{Sect2_Vorticity_H_BC_11}
\left\{\begin{array}{ll}
\omega^1 = - \frac{\partial_1\varphi\partial_2\varphi}{\partial_z\varphi}\partial_z v^1
- \frac{1 + (\partial_2\varphi)^2}{\partial_z\varphi}\partial_z v^2
+ \partial_2 v^3 + \partial_2\varphi(\partial_1 v^1 + \partial_2 v^2) \\[6pt]\hspace{0.47cm}

= - \frac{\partial_1\varphi\partial_2\varphi}{\partial_z\varphi}f^5[\nabla\varphi](\partial_j v^i)
- \frac{1 + (\partial_2\varphi)^2}{\partial_z\varphi}f^6[\nabla\varphi](\partial_j v^i)
+ \partial_2 v^3 + \partial_2\varphi(\partial_1 v^1 + \partial_2 v^2) \\[6pt]\hspace{0.39cm}

:= \textsf{F}^1 [\nabla\varphi](\partial_j v^i),
\\[8pt]

\omega^2 = \frac{1+(\partial_1\varphi)^2}{\partial_z\varphi}\partial_z v^1
+ \frac{\partial_1\varphi\partial_2\varphi}{\partial_z\varphi}\partial_z v^2
- \partial_1 v^3 - \partial_1\varphi(\partial_1 v^1 + \partial_2 v^2) \\[6pt]\hspace{0.47cm}

= \frac{1+(\partial_1\varphi)^2}{\partial_z\varphi}f^5[\nabla\varphi](\partial_j v^i)
+ \frac{\partial_1\varphi\partial_2\varphi}{\partial_z\varphi}f^6[\nabla\varphi](\partial_j v^i)
- \partial_1 v^3 - \partial_1\varphi(\partial_1 v^1 + \partial_2 v^2) \\[6pt]\hspace{0.39cm}

:= \textsf{F}^2 [\nabla\varphi](\partial_j v^i)
\end{array}\right.
\end{equation}

Since $\omega_h$ satisfies the equation:
\begin{equation}\label{Sect2_Vorticity_H_Eq_1}
\begin{array}{ll}
\partial_t^{\varphi} \omega_h + v\cdot\nabla^{\varphi}\omega_h - \e\triangle^{\varphi}\omega_h = \omega_h\cdot\nabla^{\varphi}_h v_h + \omega^3\partial_z^{\varphi} v_h,
\end{array}
\end{equation}
where the force term can be transformed as follows:
\begin{equation}\label{Sect2_Vorticity_H_Eq_2}
\begin{array}{ll}
\omega_h\cdot\nabla^{\varphi}_h v_h + \omega^3\partial_z^{\varphi} v_h \\[7pt]
= \omega_1 (\partial_1 v_h - \frac{\partial_1\varphi}{\partial_z\varphi}\partial_z v_h)
+ \omega_2 (\partial_2 v_h - \frac{\partial_2\varphi}{\partial_z\varphi}\partial_z v_h) \\[6pt]\quad
+ (\partial_1 v^2 - \frac{\partial_1\varphi}{\partial_z\varphi}\partial_z v^2 - \partial_2 v^1 + \frac{\partial_2\varphi}{\partial_z\varphi}\partial_z v^1) \frac{1}{\partial_z\varphi}\partial_z v_h \\[7pt]

= \omega_1 \partial_1 v_h + \omega_2 \partial_2 v_h

- \omega_1 \frac{\partial_1\varphi}{\partial_z\varphi}\vec{f}^{5,6}[\nabla\varphi](\partial_j v^i)
- \omega_2 \frac{\partial_2\varphi}{\partial_z\varphi}\vec{f}^{5,6}[\nabla\varphi](\partial_j v^i) \\[6pt]\quad
+ [\partial_1 v^2 - \frac{\partial_1\varphi}{\partial_z\varphi}f^6[\nabla\varphi](\partial_j v^i) - \partial_2 v^1 + \frac{\partial_2\varphi}{\partial_z\varphi}f^5[\nabla\varphi](\partial_j v^i)] \frac{1}{\partial_z\varphi}\vec{f}^{5,6}[\nabla\varphi](\partial_j v^i)

\\[7pt]
= \vec{\textsf{F}}^0[\nabla\varphi](\omega_h,\partial_j v^i), \ j=1,2,\ i=1,2,3,
\end{array}
\end{equation}
where $\vec{\textsf{F}}^0[\nabla\varphi](\omega_h,\partial_j v^i)$ is a quadratic polynomial vector with $\omega_h$ and $\partial_j v^i$, the coefficients are fractions of $\nabla\varphi$, $\omega_h$ has degree one.

Namely, $\omega_h$ satisfies the equation $(\ref{Sect1_Vorticity_H_Eq})_1$.
Thus, Lemma $\ref{Sect2_Vorticity_H_Eq_BC_Lemma}$ is proved.
\end{proof}

\subsection{Estimates of Derivatives including Time Derivatives}

For the free surface N-S equations $(\ref{Sect1_NS_Eq})$, we develop a priori estimates of tangential derivatives including time derivatives.
The estimates for normal derivatives are very different from \cite{Masmoudi_Rousset_2012_FreeBC}. \cite{Masmoudi_Rousset_2012_FreeBC} used
the variable $\textsf{S}_n =\Pi \mathcal{S}^{\varphi} v \nn$, while we investigate the vorticity in this paper.

$q$ satisfies the elliptic equation with its nonhomogeneous Dirichlet boundary condition
\begin{equation}\label{Sect2_Tangential_Estimate_Pressure_1}
\left\{\begin{array}{ll}
\triangle^{\varphi} q = -\partial_j^{\varphi} v^i \partial_i^{\varphi} v^j, \\[5pt]

q|_{z=0} = gh +2\e\mathcal{S}^{\varphi} v\nn\cdot \nn,
\end{array}\right.
\end{equation}
then it is standard to prove the gradient estimate of $q$:
\begin{equation}\label{Sect2_Tangential_Estimate_Pressure_2}
\begin{array}{ll}
\|\nabla q\|_{X^{m-1}} \lem \|\partial_i^{\varphi} v^j \partial_j^{\varphi} v^i\|_{X^{m-1}} + \big|q|_{z=0}\big|_{X^{m-1,\frac{1}{2}}} \\[8pt]
\lem \|v\|_{X^{m-1,1}} + \|\partial_z v\|_{X^{m-1}} + g|h|_{X^{m-1,\frac{1}{2}}}
+ \e|v_{z=0}|_{X^{m-1,\frac{3}{2}}} + \e|h|_{X^{m-1,\frac{3}{2}}}.
\end{array}
\end{equation}
Note that \cite{Masmoudi_Rousset_2012_FreeBC} estimated the pressure by decomposition $q^{\e} = q^{\e,E} + q^{\e,NS}$,
which satisfy two systems $(\ref{Sect1_Pressure_EulerPart})$. While our estimate is standard.

In order to close the estimates of tangential derivatives of $v$, that is to bound $\|\partial_t^{\ell} v\|_{L^2}$ and 
$\sqrt{\e}\|\nabla \partial_t^{\ell}\mathcal{Z}^{\alpha}v\|_{L^2}$, we must prove two preliminary lemmas of 
$h$ by using the kinetical boundary condition $(\ref{Sect1_NS_Eq})_3$.

The first preliminary lemma concerns $|\partial_t^{\ell} h|_{L^2}$ where $0\leq\ell\leq m-1$.
Note that the estimates of mix derivatives $\partial_t^{\ell}\mathcal{Z}^{\alpha} h$
will be obtained when we estimate mix derivatives $\partial_t^{\ell}\mathcal{Z}^{\alpha} v$, where $|\alpha|>0$.
\begin{lemma}\label{Sect2_Height_Estimates_Lemma}
Assume $0\leq\ell\leq m-1$,
$|\partial_t^{\ell}h|_{L^2}$ have the estimates:
\begin{equation}\label{Sect2_Height_Estimates_Lemma_Eq}
\begin{array}{ll}
|\partial_t^{\ell}h|_{L^2}^2

\lem |h_0|_{X^{m-1}}^2
+ \int\limits_0^t |h|_{X^{m-1,1}}^2 + \|v\|_{X^{m-1,1}}^2 \,\mathrm{d}t
+ \|\partial_z v\|_{L^4([0,T],X^{m-1})}^2.
\end{array}
\end{equation}
\end{lemma}

\begin{proof}
By the kinetical boundary condition $(\ref{Sect1_NS_Eq})_3$, we have $\partial_t h + v_y\cdot\nabla_y h = v^3$,
apply $\partial_t^{\ell}$ to the above equation, we get
\begin{equation}\label{Sect2_Height_Estimates_Lemma_Eq_1}
\begin{array}{ll}
\partial_t \partial_t^{\ell}h + v_y\cdot \nabla_y \partial_t^{\ell}h = \partial_t^{\ell}v^3 - [\partial_t^{\ell}, v_y\cdot \nabla_y]h.
\end{array}
\end{equation}

Multiply $(\ref{Sect2_Height_Estimates_Lemma_Eq_1})$ with $\partial_t^{\ell}h$, integrate in $\mathbb{R}^2$, we have
\begin{equation}\label{Sect2_Height_Estimates_Lemma_Eq_2}
\begin{array}{ll}
\frac{\mathrm{d}}{\mathrm{d}t}\int\limits_{\mathbb{R}^2} |\partial_t^{\ell}h|^2 \,\mathrm{d}y

= 2\int\limits_{\mathbb{R}^2}\big( \partial_t^{\ell}v^3 - [\partial_t^{\ell}, v_y\cdot \nabla_y]h \big) \partial_t^{\ell}h \,\mathrm{d}y
+ \int\limits_{\mathbb{R}^2} |\partial_t^{\ell}h|^2 \nabla_y\cdot v_y  \,\mathrm{d}y \\[7pt]

\lem |\partial_t^{\ell}h|_{L^2}^2 + |h|_{X^{\ell,1}}^2 + \big|v|_{z=0}\big|_{X^{\ell}}^2 \\[7pt]
\lem |\partial_t^{\ell}h|_{L^2}^2 + |h|_{X^{k-1,1}}^2 + \|v\|_{X^{k-1,1}}^2 + \|\partial_z v\|_{X^{k-1}}^2.
\end{array}
\end{equation}

Sum $\ell$, integrate $(\ref{Sect2_Height_Estimates_Lemma_Eq_2})$ in time and apply the integral form of Gronwall's inequality, we have
\begin{equation}\label{Sect2_Height_Estimates_Lemma_Eq_3}
\begin{array}{ll}
\int\limits_{\mathbb{R}^2} |\partial_t^{\ell}h|^2 \,\mathrm{d}y
\lem |h_0|_{X^{m-1}}^2
+ \int\limits_0^t |h|_{X^{m-1,1}}^2 + \|v\|_{X^{m-1,1}}^2 + \|\partial_z v\|_{X^{m-1}}^2
\,\mathrm{d}t \\[6pt]

\lem |h_0|_{X^{m-1}}^2
+ \int\limits_0^t |h|_{X^{m-1,1}}^2 + \|v\|_{X^{m-1,1}}^2 \,\mathrm{d}t
+ \|\partial_z v\|_{L^4([0,T],X^{m-1})}^2.
\end{array}
\end{equation}
Thus, Lemma $\ref{Sect2_Height_Estimates_Lemma}$ is proved.
\end{proof}

The second preliminary lemma concerns $\sqrt{\e}|\partial_t^{\ell}\mathcal{Z}^{\alpha} h|_{\frac{1}{2}}$,
by which we bound $\sqrt{\e}\|\mathcal{S}^{\varphi}\partial_t^{\ell}\mathcal{Z}^{\alpha}\eta\|_{L^2}$ and then 
we can bound $\sqrt{\e}\|\mathcal{S}^{\varphi}\partial_t^{\ell}\mathcal{Z}^{\alpha} v\|_{L^2}$.

\begin{lemma}\label{Sect2_Height_Viscous_Estimates_Lemma}
Assume $0\leq\ell\leq m-1$,
$\sqrt{\e}|\partial_t^{\ell}\mathcal{Z}^{\alpha}h|_{\frac{1}{2}}$ have the estimates:
\begin{equation}\label{Sect2_Height_Viscous_Estimates_Lemma_Eq}
\begin{array}{ll}
\e|h|_{X^{m-1,\frac{3}{2}}}^2 \leq \e|h_0|_{X^{m-1,\frac{3}{2}}}^2 + \int\limits_0^t|h|_{X^{m-1,1}}^2
+ \e\sum\limits_{\ell\leq m-1,\ell+|\alpha|\leq m}|\nabla V^{\ell,\alpha}|_{L^2}^2 \,\mathrm{d}t.
\end{array}
\end{equation}
\end{lemma}

\begin{proof}
Let $\Lambda$ be a differential operator with respect to $y$, whose Fourier multiplier is $(1+|\xi|^2)^{\frac{1}{2}}$,  
so $|\Lambda^{\frac{1}{2}} h|_{L^2} = |h|_{\frac{1}{2}}$.

By the kinetical boundary condition $(\ref{Sect1_NS_Eq})_3$, 
we have $\partial_t h + v_y\cdot\nabla_y h = v^3$,
apply $\partial_t^{\ell}\mathcal{Z}^{\alpha}\Lambda^{\frac{1}{2}}$ to the above equation, we get
\begin{equation}\label{Sect2_Height_Viscous_Estimates_Lemma_Eq_1}
\begin{array}{ll}
\partial_t \partial_t^{\ell}\mathcal{Z}^{\alpha}\Lambda^{\frac{1}{2}} h 
+ v_y\cdot \nabla_y \partial_t^{\ell}\mathcal{Z}^{\alpha}\Lambda^{\frac{1}{2}} h 
= \partial_t^{\ell}\mathcal{Z}^{\alpha}\Lambda^{\frac{1}{2}}v^3 
- [\partial_t^{\ell}\mathcal{Z}^{\alpha}\Lambda^{\frac{1}{2}}, v_y\cdot \nabla_y]h.
\end{array}
\end{equation}

Multiply $(\ref{Sect2_Height_Viscous_Estimates_Lemma_Eq_1})$ with $\e\partial_t^{\ell}\mathcal{Z}^{\alpha}\Lambda^{\frac{1}{2}} h$, 
integrate in $\mathbb{R}^2$, we have
\begin{equation}\label{Sect2_Height_Viscous_Estimates_Lemma_Eq_2}
\begin{array}{ll}
\e\frac{\mathrm{d}}{\mathrm{d}t}\int\limits_{\mathbb{R}^2} |\partial_t^{\ell}\mathcal{Z}^{\alpha}\Lambda^{\frac{1}{2}} h|^2 \,\mathrm{d}y

= 2\e\int\limits_{\mathbb{R}^2}\big(\partial_t^{\ell}\mathcal{Z}^{\alpha}\Lambda^{\frac{1}{2}} v^3 
- [\partial_t^{\ell}\mathcal{Z}^{\alpha}\Lambda^{\frac{1}{2}}, v_y\cdot \nabla_y]h \big) 
\partial_t^{\ell}\mathcal{Z}^{\alpha}\Lambda^{\frac{1}{2}} h \,\mathrm{d}y \\[7pt]\quad
+ \e\int\limits_{\mathbb{R}^2} |\partial_t^{\ell}\mathcal{Z}^{\alpha}\Lambda^{\frac{1}{2}}h|^2 \nabla_y\cdot v_y  \,\mathrm{d}y \\[7pt]

\lem \e|h|_{X^{m-1,\frac{3}{2}}}^2 + \e\big|v|_{z=0}\big|_{X^{m-1,\frac{3}{2}}}^2 
+ \e|\partial_t^{\ell}\mathcal{Z}^{\alpha}\Lambda^{\frac{1}{2}} h|_{L^2}^2\\[7pt]
\lem \e|h|_{X^{m-1,\frac{3}{2}}}^2 
+ \e\|v\|_{X_{tan}^{m-1,2}}^2 + \e\|\partial_z v\|_{X_{tan}^{m-1,1}}^2 
+ \e|\partial_t^{\ell}\mathcal{Z}^{\alpha}\Lambda^{\frac{1}{2}} h|_{L^2}^2\\[8pt]

\lem \e|h|_{X^{m-1,\frac{3}{2}}}^2 
+ \sum\limits_{\ell\leq m-1,\ell+|\alpha|\leq m}(\e\|\nabla_y V^{\ell,\alpha}\|_{L^2}^2 + \e\|\nabla_y\partial_t^{\ell}\mathcal{Z}^{\alpha}\eta\|_{L^2}^2) \\[12pt]\quad
+ \sum\limits_{\ell\leq m-1,\ell+|\alpha|\leq m}\big[\e\|\partial_z V^{\ell,\alpha}\|_{L^2}^2 
+ \|\partial_z^{\varphi} v\|_{L^{\infty}}^2 \cdot \e\|\partial_z \partial_t^{\ell}\mathcal{Z}^{\alpha}\eta\|_{L^2}^2 \\[12pt]\quad
+ (\sqrt{\e}\|\partial_{zz}^{\varphi} v\|_{L^{\infty}})^2 \|\partial_t^{\ell}\mathcal{Z}^{\alpha}\eta\|_{L^2}^2
+ \e|\partial_t^{\ell}\mathcal{Z}^{\alpha}\Lambda^{\frac{1}{2}} h|_{L^2}^2 \big] \\[7pt]

\lem \e|h|_{X^{m-1,\frac{3}{2}}}^2 +|h|_{X^{m-1,1}}^2 + \e\sum\limits_{\ell\leq m-1,\ell+|\alpha|\leq m}|\nabla V^{\ell,\alpha}|_{L^2}^2.
\end{array}
\end{equation}

Sum $\ell,\alpha$, integrate $(\ref{Sect2_Height_Viscous_Estimates_Lemma_Eq_2})$ in time and apply the integral form of Gronwall's inequality, 
we have $(\ref{Sect2_Height_Viscous_Estimates_Lemma_Eq})$.
Thus, Lemma $\ref{Sect2_Height_Viscous_Estimates_Lemma}$ is proved.
\end{proof}

The following lemma concerns the estimates of tangential derivatives. The proof is different from \cite{Masmoudi_Rousset_2012_FreeBC}
when we estimate $\partial_t^{\ell} v$ where $1\leq \ell\leq m-1$, since $\|\partial_t^{\ell} q\|$ has no bound for infinite fluid depth.
\begin{lemma}\label{Sect2_Tangential_Estimate_Lemma}
Assume the conditions are the same with those of Proposition $\ref{Sect1_Proposition_TimeRegularity}$,
then $v$ and $h$ satisfy the a priori estimate:
\begin{equation}\label{Sect2_Tangential_Estimate}
\begin{array}{ll}
\|v\|_{X^{m-1,1}}^2 + |h|_{X^{m-1,1}}^2 + \e |h|_{X^{m-1,\frac{3}{2}}}^2
+ \e\int\limits_0^t \|\nabla v\|_{X^{m-1,1}}^2 \,\mathrm{d}t \\[5pt]
\lem \|v_0\|_{X^{m-1,1}}^2 + |h_0|_{X^{m-1,1}}^2 + \e |h_0|_{X^{m-1,\frac{3}{2}}}^2
+ \|\partial_z v\|_{L^4([0,T],X^{m-1})}^2.
\end{array}
\end{equation}
\end{lemma}

\begin{proof}
Apply $\partial_t^{\ell}\mathcal{Z}^{\alpha}$ to $(\ref{Sect1_NS_Eq})$, 
we use the following commutator (see \cite{Masmoudi_Rousset_2012_FreeBC}):
\begin{equation}\label{Sect2_Tangential_Estimate_1}
\begin{array}{ll}
[\partial_t^{\ell}\mathcal{Z}^{\alpha}, \partial_i^{\varphi}] f
= - \partial_z^{\varphi} f\, \partial_i^{\varphi} (\partial_t^{\ell}\mathcal{Z}^{\alpha}\eta) + b.t.\, ,\quad i=t,1,2,3.
\end{array}
\end{equation}
where the abbreviation $b.t.$ represents bounded terms in this paper. Note that $\partial_t^m \varphi$
and $\partial_t^m h$ are bounded in $L^4([0,T],L^2)$, thus they are also represented by $b.t.$ in $(\ref{Sect2_Tangential_Estimate_1})$.

Similar to \cite{Masmoudi_Rousset_2012_FreeBC},
we choose the Alinhac's good unknown $(\ref{Sect1_Good_Unknown_1})$ as our variable, then $V^{\ell,\alpha}$ and $Q^{\ell,\alpha}$ satisfies
\begin{equation}\label{Sect2_Tangential_Estimate_5}
\left\{\begin{array}{ll}
\partial_t^{\varphi} V^{\ell,\alpha} + v\cdot\nabla^{\varphi} V^{\ell,\alpha} + \nabla^{\varphi} Q^{\ell,\alpha}
-2\e\nabla^{\varphi}\cdot\mathcal{S}^{\varphi} V^{\ell,\alpha} \\[5pt]\quad

= -\partial_t^{\varphi}\partial_z^{\varphi} v \, \partial_t^{\ell}\mathcal{Z}^{\alpha} \eta
- v\cdot\nabla^{\varphi} \partial_z^{\varphi} v \, \partial_t^{\ell}\mathcal{Z}^{\alpha} \eta
- \nabla^{\varphi} \partial_z^{\varphi} q \, \partial_t^{\ell}\mathcal{Z}^{\alpha} \eta \\[7pt]\quad

+ 2\e\nabla^{\varphi}\cdot (\mathcal{S}^{\varphi} \partial_z^{\varphi} v \, \partial_t^{\ell}\mathcal{Z}^{\alpha} \eta)
- 2\e \partial_z^{\varphi}\mathcal{S}^{\varphi} v_y\cdot \nabla_y \partial_t^{\ell}\mathcal{Z}^{\alpha} \eta + b.t. \, , \\[10pt]

\nabla^{\varphi}\cdot V^{\ell,\alpha} = - (\nabla^{\varphi}\cdot\partial_z^{\varphi} v) \, \partial_t^{\ell}\mathcal{Z}^{\alpha} \eta + b.t.\,
= b.t. \, , \\[10pt]

\partial_t \partial_t^{\ell}\mathcal{Z}^{\alpha} h
+ v_y \cdot\nabla_y \partial_t^{\ell}\mathcal{Z}^{\alpha} h
= V^{\ell,\alpha} \cdot \NN
+ b.t.\, ,  \\[10pt]

Q^{\ell,\alpha}\NN
-2\e \mathcal{S}^{\varphi} V^{\ell,\alpha}\,\NN
\\[7pt]\quad
= (g - \partial_z^{\varphi}q) \partial_t^{\ell}\mathcal{Z}^{\alpha}h\NN
+ 2\e (\mathcal{S}^{\varphi}\partial_z^{\varphi} v \,\NN)\, \partial_t^{\ell}\mathcal{Z}^{\alpha} h

- [\partial_t^{\ell}\mathcal{Z}^{\alpha},2\e \mathcal{S}^{\varphi}v\nn\cdot\nn,\NN] \\[8pt]\quad
+ (2\e \mathcal{S}^{\varphi}v - 2\e \mathcal{S}^{\varphi}v\nn\cdot\nn)\,\partial_t^{\ell}\mathcal{Z}^{\alpha}\NN
+2\e [\partial_t^{\ell}\mathcal{Z}^{\alpha},\mathcal{S}^{\varphi}v, \NN]
+ b.t. \, , \\[10pt]

(\partial_t^{\ell}\mathcal{Z}^{\alpha}v, \partial_t^{\ell}\mathcal{Z}^{\alpha}h)|_{t=0}
= (\partial_t^{\ell}\mathcal{Z}^{\alpha}v_0, \partial_t^{\ell}\mathcal{Z}^{\alpha}h_0).
\end{array}\right.
\end{equation}

When $|\alpha|\geq 1, \, 1\leq \ell+|\alpha|\leq m$, we develop the $L^2$ estimate of $V^{\ell,\alpha}$. The estimates are similar to \cite{Masmoudi_Rousset_2012_FreeBC}, but we do not use $g-\partial_z^{\varphi} q^{E}$.
\begin{equation}\label{Sect2_Tangential_Estimate_6}
\begin{array}{ll}
\frac{1}{2}\frac{\mathrm{d}}{\mathrm{d}t}\int\limits_{\mathbb{R}^3_{-}} |V^{\ell,\alpha}|^2 \,\mathrm{d}\mathcal{V}_t
- \int\limits_{\mathbb{R}^3_{-}} Q^{\ell,\alpha} \, \nabla^{\varphi}\cdot V^{\ell,\alpha} \,\mathrm{d}\mathcal{V}_t
+ 2\e \int\limits_{\mathbb{R}^3_{-}} |\mathcal{S}^{\varphi}V^{\ell,\alpha}|^2 \,\mathrm{d}\mathcal{V}_t \\[14pt]

\leq \int\limits_{\{z=0\}} (2\e \mathcal{S}^{\varphi}V^{\ell,\alpha}\NN - Q^{\ell,\alpha}\NN)\cdot V^{\ell,\alpha} \mathrm{d}y
+ \|\partial_z v\|_{X^{m-1}}^2 + \|\nabla q\|_{X^{m-1}}^2
+ \text{b.t.} \\[14pt]

\leq -\int\limits_{\{z=0\}} (g - \partial_z^{\varphi}q) \partial_t^{\ell}\mathcal{Z}^{\alpha}h\NN\cdot V^{\ell,\alpha} \mathrm{d}y
+ \|\partial_z v\|_{X^{m-1}}^2 + \|\nabla q\|_{X^{m-1}}^2
+ \text{b.t.} \\[14pt]

\leq -\int\limits_{\{z=0\}} (g - \partial_z^{\varphi}q) \partial_t^{\ell}\mathcal{Z}^{\alpha}h
(\partial_t \partial_t^{\ell}\mathcal{Z}^{\alpha} h
+ v_y \cdot\nabla_y \partial_t^{\ell}\mathcal{Z}^{\alpha} h) \mathrm{d}y
+ \|\partial_z v\|_{X^{m-1}}^2 \\[11pt]\quad
+ \|\nabla q\|_{X^{m-1}}^2
+ \text{b.t.} \\[10pt]

\leq - \frac{1}{2}\frac{\mathrm{d}}{\mathrm{d}t}
\int\limits_{\{z=0\}} (g - \partial_z^{\varphi}q) |\partial_t^{\ell}\mathcal{Z}^{\alpha}h|^2 \mathrm{d}y
+ \|\partial_z v\|_{X^{m-1}}^2 + \|\nabla q\|_{X^{m-1}}^2
+ \text{b.t.},
\end{array}
\end{equation}
then
\begin{equation}\label{Sect2_Tangential_Estimate_7}
\begin{array}{ll}
\frac{\mathrm{d}}{\mathrm{d}t}\int\limits_{\mathbb{R}^3_{-}} |V^{\ell,\alpha}|^2 \,\mathrm{d}\mathcal{V}_t
+ \frac{\mathrm{d}}{\mathrm{d}t}
\int\limits_{\{z=0\}} (g - \partial_z^{\varphi}q) |\partial_t^{\ell}\mathcal{Z}^{\alpha}h|^2 \mathrm{d}y
+ \e \int\limits_{\mathbb{R}^3_{-}} |\mathcal{S}^{\varphi}V^{\ell,\alpha}|^2 \,\mathrm{d}\mathcal{V}_t \\[14pt]

\lem \|\partial_z v\|_{X^{m-1}}^2 + \|\nabla q\|_{X^{m-1}}^2 + \text{b.t.}
\end{array}
\end{equation}

Since $(g - \partial_z^{\varphi}q)|_{z=0} \geq c_0>0$, a priori estimates can be closed. Thus,
\begin{equation}\label{Sect2_Tangential_Estimate_8}
\begin{array}{ll}
\|\partial_t^{\ell}\mathcal{Z}^{\alpha}v\|^2 + |\partial_t^{\ell}\mathcal{Z}^{\alpha} h|^2 + \e |\partial_t^{\ell}\mathcal{Z}^{\alpha} h|_{\frac{1}{2}}^2
+ \e\int\limits_0^t \|\nabla \partial_t^{\ell}\mathcal{Z}^{\alpha}v\|^2 \,\mathrm{d}t \\[7pt]

\lem \|v_0\|_{X^{m-1,1}}^2 + |h_0|_{X^{m-1,1}}^2 + \e |h_0|_{X^{m-1,\frac{3}{2}}}^2
+ \int\limits_0^T \|\partial_z v\|_{X^{m-1}}^2 + \|\nabla q\|_{X^{m-1}}^2 \,\mathrm{d}t.
\end{array}
\end{equation}
where we used the estimate of $\e |\partial_t^{\ell}\mathcal{Z}^{\alpha} h|_{\frac{1}{2}}^2$ that is proved by Lemma $\ref{Sect2_Height_Viscous_Estimates_Lemma}$.

When $|\alpha|=0$ and $0\leq\ell\leq m-1$, we have no bounds of $q$ and $\partial_t^{\ell} q$.
Without using Hardy's inequality $(\ref{Sect1_HardyIneq})$, we have a simpler method to estimate $V^{\ell,0}$,
we neither use the variable $Q^{\ell,\alpha}$ and nor apply the integration by parts to the pressure terms.
Also, the divergence free condition and the dynamical boundary condition will not be used here. Then
\begin{equation}\label{Sect2_Tangential_Estimate_10}
\begin{array}{ll}
\frac{1}{2}\frac{\mathrm{d}}{\mathrm{d}t}\int\limits_{\mathbb{R}^3_{-}} |V^{\ell,0}|^2 \,\mathrm{d}\mathcal{V}_t
+ 2\e \int\limits_{\mathbb{R}^3_{-}} |\mathcal{S}^{\varphi}V^{\ell,0}|^2 \,\mathrm{d}\mathcal{V}_t \\[14pt]

\leq - \int\limits_{\mathbb{R}^3_{-}} \partial_t^{\ell}\nabla^{\varphi} q \cdot V^{\ell,0} \,\mathrm{d}\mathcal{V}_t
+ \int\limits_{\{z=0\}} 2\e \mathcal{S}^{\varphi}V^{\ell,0}\NN \cdot V^{\ell,0} \mathrm{d}y + \|\partial_z v\|_{X^{m-1}}^2
+ \text{b.t.} \\[14pt]

\lem \|\partial_t^{\ell}\nabla q\|_{L^2}^2
+ \e\int\limits_{\{z=0\}} |V^{\ell,0}|^2 \mathrm{d}y
+ 4\e\int\limits_{\{z=0\}} |\mathcal{S}^{\varphi}V^{\ell,0}|^2 \mathrm{d}y + \|\partial_z v\|_{X^{m-1}}^2
+ \text{b.t.} \\[14pt]

\lem \|\partial_t^{\ell}\nabla q\|_{L^2}^2
+ \e \|\partial_t^{\ell} v|_{z=0}\|_{L^2}^2 + \e |\partial_t^{\ell} h|_{L^2}^2

+ \e\big|\partial_t^{\ell}\partial_y v|_{z=0}\big|_{L^2}^2
+ \e\big|\partial_t^{\ell}\partial_z v|_{z=0}\big|_{L^2}^2 \\[8pt]\quad

+ \e |\partial_t^{\ell} h|_{X^{0,\frac{1}{2}}}^2
+ \|\partial_z v\|_{X^{m-1}}^2 + \text{b.t.}.
\end{array}
\end{equation}

Since $\partial_z v|_{z=0}$ can be expressed in terms of tangential derivatives, see $(\ref{Sect2_Vorticity_H_BC_10})$.
\begin{equation}\label{Sect2_Tangential_Estimate_11}
\begin{array}{ll}
\partial_z v^1 = f^5[\nabla\varphi](\partial_j v^i), \ j=1,2,\ i=1,2,3, \\[7pt]
\partial_z v^2 = f^6[\nabla\varphi](\partial_j v^i), \ j=1,2,\ i=1,2,3, \\[7pt]
\partial_z v^3 = \partial_1\varphi\partial_z v^1 + \partial_2\varphi\partial_z v^2 - \partial_z\varphi(\partial_1 v^1 + \partial_2 v^2) ,
\end{array}
\end{equation}
then
\begin{equation}\label{Sect2_Tangential_Estimate_12}
\begin{array}{ll}
\frac{\mathrm{d}}{\mathrm{d}t}\int\limits_{\mathbb{R}^3_{-}} |V^{\ell,0}|^2 \,\mathrm{d}\mathcal{V}_t
+ \e \int\limits_{\mathbb{R}^3_{-}} |\mathcal{S}^{\varphi}V^{\ell,0}|^2 \,\mathrm{d}\mathcal{V}_t \\[9pt]

\lem \|\partial_t^{\ell}\nabla q\|_{L^2}^2
+ \e\big|\partial_t^{\ell}\partial_y v|_{z=0}\big|_{L^2}^2 + \e |\partial_t^{\ell} h|_{X^{0,\frac{1}{2}}}^2
+ \|\partial_z v\|_{X^{m-1}}^2 + \text{b.t.} \\[9pt]

\lem \|\partial_t^{\ell}\nabla q\|_{L^2}^2
+ \e\|\nabla v\|_{X^{m-1,1}}^2 + \e |\partial_t^{\ell} h|_{X^{0,\frac{1}{2}}}^2
+ \|\partial_z v\|_{X^{m-1}}^2 + \text{b.t.} \ .
\end{array}
\end{equation}

Similarly, we have
\begin{equation}\label{Sect2_Tangential_Estimate_13}
\begin{array}{ll}
\|V^{\ell,0}\|^2 + \e\int\limits_0^t \|\nabla V^{\ell,0}\|^2 \,\mathrm{d}t

\lem \|v_0\|_{X^{m-1,1}}^2 + |h_0|_{X^{m-1,1}}^2 + \e |h_0|_{X^{m-1,\frac{3}{2}}}^2 \\[6pt]\quad
+ \|\nabla q\|_{L^4([0,T],X^{m-1})}^2 + \|\partial_z v\|_{L^4([0,T],X^{m-1})}^2
+ \e\int\limits_0^T \|\nabla v\|_{X^{m-1,1}}^2 \,\mathrm{d}t.
\end{array}
\end{equation}

Combining $(\ref{Sect2_Tangential_Estimate_13})$ and Lemma $\ref{Sect2_Height_Estimates_Lemma}$,
we obtain the estimate of $\partial_t^{\ell} v$:
\begin{equation}\label{Sect2_Tangential_Estimate_14}
\begin{array}{ll}
\int\limits_{\mathbb{R}^2} |\partial_t^{\ell}v|^2 \,\mathrm{d}y

\lem \|v_0\|_{X^{m-1,1}}^2 + |h_0|_{X^{m-1,1}}^2 + \e |h_0|_{X^{m-1,\frac{3}{2}}}^2
+ \|\partial_z v\|_{L^4([0,T],X^{m-1})}^2

\\[7pt]\quad
+ \|\nabla q\|_{L^4([0,T],X^{m-1})}^2
+ \int\limits_0^t |h|_{X^{m-1,1}}^2 + \|v\|_{X^{m-1,1}}^2 \,\mathrm{d}t
+ \e\int\limits_0^T \|\nabla v\|_{X^{m-1,1}}^2 \,\mathrm{d}t.
\end{array}
\end{equation}

Sum $\ell$ and $\alpha$ in $(\ref{Sect2_Tangential_Estimate_8}),(\ref{Sect2_Tangential_Estimate_14})$
and Lemma $\ref{Sect2_Height_Estimates_Lemma}$, we get the estimate $(\ref{Sect2_Tangential_Estimate})$.
Thus, Lemma $\ref{Sect2_Tangential_Estimate_Lemma}$ is proved.
\end{proof}

\vspace{0.2cm}

In order to study $\partial_z v$, \cite{Masmoudi_Rousset_2012_NavierBC} estimated the quantity $\omega_h - 2\alpha u_h^{\bot}$ and got $\|\partial_z v\|_{L^{\infty}([0,T],H_{co}^{m-1})}$. However, for the free surface Navier-Stokes equations $(\ref{Sect1_NS_Eq})$,
it is impossible to obtain such a higher regularity of $\partial_z v$. Similar to \cite{Masmoudi_Rousset_2012_FreeBC},
we estimate $\|\partial_z v\|_{L^4([0,T],X^{m-1})}^2$ to close energy estimates.
\begin{lemma}\label{Sect2_Vorticity_Lemma}
Assume $v$ and $\omega$ are the velocity and vorticity of the free surface Navier-Stokes equations $(\ref{Sect1_NS_Eq})$ respectively.
$\omega_h$ satisfies the following estimate:
\begin{equation}\label{Sect2_Vorticity_Estimate}
\begin{array}{ll}
\|\omega_h\|_{L^4([0,T],X^{m-1})}^2 + \|\partial_z v\|_{L^4([0,T],X^{m-1})}^2\\[7pt]
\lem \big\|\omega_h|_{t=0}\big\|_{X^{m-1}}^2
+ \int\limits_0^T\|v\|_{X^{m-1,1}}^2 + |h|_{X^{m-1,1}}^2 \,\mathrm{d}t
+ \e\int\limits_0^t\|\partial_z v\|_{X^{m-1,1}}^2\,\mathrm{d}t.
\end{array}
\end{equation}
\end{lemma}

\begin{proof}
By Lemma $\ref{Sect2_Vorticity_H_Eq_BC_Lemma}$, we have the equations of $\omega_h$:
\begin{equation}\label{Sect2_Vorticity_Estimate_1}
\left\{\begin{array}{ll}
\partial_t^{\varphi} \omega_h + v\cdot\nabla^{\varphi}\omega_h - \e\triangle^{\varphi}\omega_h = \vec{\textsf{F}}^0[\nabla\varphi](\omega_h,\partial_j v^i),
\\[8pt]

\omega_h|_{z=0} = \vec{\textsf{F}}^{1,2}[\nabla\varphi](\partial_j v^i), \\[9pt]

\omega_h|_{t=0} = (\omega_0^1, \omega_0^2)^{\top}.
\end{array}\right.
\end{equation}
where $j=1,2,\ i=1,2,3$.

Similar to \cite{Masmoudi_Rousset_2012_FreeBC}, we decompose $\omega_h = \omega_h^{nhom} + \omega_h^{hom}$, such that $\omega_h^{nhom}$ satisfies the nonhomogeneous equations:
\begin{equation}\label{Sect2_Vorticity_Estimate_2}
\left\{\begin{array}{ll}
\partial_t^{\varphi} \omega_h^{nhom} + v\cdot\nabla^{\varphi}\omega_h^{nhom} - \e\triangle^{\varphi}\omega_h^{nhom} = \vec{\textsf{F}}^0[\nabla\varphi](\omega_h,\partial_j v^i),
\\[8pt]

\omega_h^{nhom}|_{z=0} = 0, \\[6pt]

\omega_h^{nhom}|_{t=0} = (\omega_0^1, \omega_0^2)^{\top},
\end{array}\right.
\end{equation}
and $\omega_h^{hom}$ satisfies the homogeneous equations:
\begin{equation}\label{Sect2_Vorticity_Estimate_3}
\left\{\begin{array}{ll}
\partial_t^{\varphi} \omega_h^{hom} + v\cdot\nabla^{\varphi}\omega_h^{hom} - \e\triangle^{\varphi}\omega_h^{hom} = 0,
\\[7pt]

\omega_h^{hom}|_{z=0} = \vec{\textsf{F}}^{1,2}[\nabla\varphi](\partial_j v^i), \\[6pt]

\omega_h^{hom}|_{t=0} = 0.
\end{array}\right.
\end{equation}

$(\ref{Sect2_Vorticity_Estimate_2})_1$ is equivalent to
\begin{equation}\label{Sect2_Vorticity_Estimate_4}
\begin{array}{ll}
\partial_t \omega_h^{nhom} + v_y\cdot\nabla_y\omega_h^{nhom} + V_z \partial_z \omega_h^{nhom}
- \e\triangle^{\varphi}\omega_h^{nhom} = \vec{\textsf{F}}^0[\nabla\varphi](\omega_h,\partial_j v^i).
\end{array}
\end{equation}
where $V_z = \frac{1}{\partial_z\varphi}(v\cdot\NN -\partial_t\varphi)
= \frac{1}{\partial_z\varphi}(v^3 -\partial_t\eta -v_y\cdot\nabla_y\eta)$, see \cite{Masmoudi_Rousset_2012_FreeBC}.

Apply $\partial_t^{\ell}\mathcal{Z}^{\alpha}$, where $\ell+|\alpha|\leq m-1$, to the equations $(\ref{Sect2_Vorticity_Estimate_4})$,
we get
\begin{equation}\label{Sect2_Vorticity_Estimate_6}
\left\{\begin{array}{ll}
\partial_t \partial_t^{\ell}\mathcal{Z}^{\alpha}\omega_h^{nhom}
+ v_y\cdot\nabla_y \partial_t^{\ell}\mathcal{Z}^{\alpha}\omega_h^{nhom}
+ V_z\partial_z \partial_t^{\ell}\mathcal{Z}^{\alpha}\omega_h^{nhom}
- \e\triangle^{\varphi}\partial_t^{\ell}\mathcal{Z}^{\alpha}\omega_h^{nhom} \\[9pt]\quad

= \partial_t^{\ell}\mathcal{Z}^{\alpha}\vec{\textsf{F}}^0[\nabla\varphi](\omega_h,\partial_j v^i)
- [\partial_t^{\ell}\mathcal{Z}^{\alpha}, v_y\cdot\nabla_y] \omega_h^{nhom}
- [\partial_t^{\ell}\mathcal{Z}^{\alpha}, V_z\partial_z] \omega_h^{nhom} \\[9pt]\quad

+ \e\nabla^{\varphi} \cdot [\partial_t^{\ell}\mathcal{Z}^{\alpha}, \nabla^{\varphi}]\omega_h^{nhom}
+ \e[\partial_t^{\ell}\mathcal{Z}^{\alpha}, \nabla^{\varphi}\cdot]\nabla^{\varphi}\omega_h^{nhom} ,
\\[12pt]

\partial_t^{\ell}\mathcal{Z}^{\alpha}\omega_h^{nhom}|_{z=0} = 0, \\[9pt]

\partial_t^{\ell}\mathcal{Z}^{\alpha}\omega_h^{nhom}|_{t=0}
= (\partial_t^{\ell}\mathcal{Z}^{\alpha}\omega_0^1, \partial_t^{\ell}\mathcal{Z}^{\alpha}\omega_0^2)^{\top}.
\end{array}\right.
\end{equation}

Develop the $L^2$ estimate of $(\ref{Sect2_Vorticity_Estimate_6})$, we get
\begin{equation}\label{Sect2_Vorticity_Estimate_7}
\begin{array}{ll}
\frac{\mathrm{d}}{\mathrm{d}t} \|\partial_t^{\ell}\mathcal{Z}^{\alpha} \omega_h^{nhom}\|_{L^2}^2
+ 2\e\|\nabla^{\varphi} \partial_t^{\ell}\mathcal{Z}^{\alpha} \omega_h^{nhom}\|_{L^2}^2 \\[10pt]

\lem \|\partial_t^{\ell}\mathcal{Z}^{\alpha}\omega_h\|_{L^2}^2
+ \|\partial_t^{\ell}\mathcal{Z}^{\alpha}\partial_j v^i\|_{L^2}^2
+ \|\partial_t^{\ell}\mathcal{Z}^{\alpha}\nabla\varphi\|_{L^2}^2
+ \big\|[\partial_t^{\ell}\mathcal{Z}^{\alpha}, V_z\partial_z] \omega_h^{nhom}\big\|_{L^2}^2\\[10pt]\quad

+ \e\int\limits_{\mathbb{R}^3_{-}} \nabla^{\varphi} \cdot [\partial_t^{\ell}\mathcal{Z}^{\alpha}, \nabla^{\varphi}]\omega_h^{nhom}
\cdot \partial_t^{\ell}\mathcal{Z}^{\alpha} \omega_h^{nhom}\,\mathrm{d}\mathcal{V}_t \\[10pt]\quad
+ \e\int\limits_{\mathbb{R}^3_{-}}[\partial_t^{\ell}\mathcal{Z}^{\alpha}, \nabla^{\varphi}\cdot]\nabla^{\varphi}\omega_h^{nhom}
\cdot \partial_t^{\ell}\mathcal{Z}^{\alpha} \omega_h^{nhom}\,\mathrm{d}\mathcal{V}_t + b.t.\, .
\end{array}
\end{equation}

\vspace{-0.1cm}
Now we estimate the last three terms on the right hand of $(\ref{Sect2_Vorticity_Estimate_7})$,
the first term is
\begin{equation}\label{Sect2_Vorticity_Estimate_7_1}
\begin{array}{ll}
\big\|[\partial_t^{\ell}\mathcal{Z}^{\alpha}, V_z\partial_z] \omega_h^{nhom}\big\|_{L^2}^2\\[6pt]

= \sum\limits_{\ell_1 + |\alpha_1|>0}\big\|\frac{1-z}{z}\partial_t^{\ell_1}\mathcal{Z}^{\alpha_1} 
[\frac{1}{\partial_z\varphi}(v^3 -\eta_t - v\cdot\nabla_y \eta)]
\cdot \frac{z}{1-z}\partial_t^{\ell_2}\mathcal{Z}^{\alpha_2}\partial_z \omega_h^{nhom}\big\|_{L^2}^2\\[11pt]

\lem \big\|\partial_z\partial_t^{\ell}\mathcal{Z}^{\alpha} [\frac{1}{A + \partial_z(\psi\ast h)}(v^3 - \psi\ast h_t - v\cdot\nabla_y (\psi\ast h))\big\|_{L^2}^2 + b.t. \\[10pt]

\lem \big\|\partial_t^{\ell}\mathcal{Z}^{\alpha} \partial_z[\frac{1}{A + \partial_z(\psi\ast h)}(v^3 - \psi\ast (v^3 - v_y\cdot\nabla_y h) - v\cdot\nabla_y (\psi\ast h))\big\|_{L^2}^2 + b.t. \\[12pt]

\lem |\partial_z^{m+1}\psi|_{L^1(\mathrm{d}z)}^2 |\partial_t^{\ell}\partial_y^{\alpha_y} h|_{L^2}^2
+ \|\partial_t^{\ell}\mathcal{Z}^{\alpha}\partial_z(\frac{1}{\partial_z\varphi} v)\|_{L^2}^2 \\[10pt]\quad
+ |\partial_z^{m}\psi\|_{L^1(\mathrm{d}z)}^2 \|\partial_t^{\ell}\partial_y^{\alpha_y}\partial_y h\|_{L^2}^2 + b.t. \\[10pt]

\lem |h|_{X^{m-1,1}}^2 + \|\partial_z v\|_{X^{m-1}}^2 + b.t. \,,
\end{array}
\end{equation}
the second term is
\begin{equation}\label{Sect2_Vorticity_Estimate_7_2}
\begin{array}{ll}
\e\int\limits_{\mathbb{R}^3_{-}} \nabla^{\varphi} \cdot [\partial_t^{\ell}\mathcal{Z}^{\alpha}, \nabla^{\varphi}]\omega_h^{nhom}
\cdot \partial_t^{\ell}\mathcal{Z}^{\alpha} \omega_h^{nhom}\,\mathrm{d}\mathcal{V}_t \\[8pt]

= \sum\limits_{i=1}^3
\e\int\limits_{\mathbb{R}^3_{-}} [\partial_t^{\ell}\mathcal{Z}^{\alpha}, \partial_i^{\varphi}]\omega_h^{nhom}
\cdot \partial_i^{\varphi}\partial_t^{\ell}\mathcal{Z}^{\alpha} \omega_h^{nhom}\,\mathrm{d}\mathcal{V}_t \\[10pt]

= \sum\limits_{i=1}^3
\e\int\limits_{\mathbb{R}^3_{-}} \partial_i^{\varphi}\partial_t^{\ell}\mathcal{Z}^{\alpha}\eta\, \partial_z^{\varphi}\omega_h^{nhom}
\cdot \partial_i^{\varphi}\partial_t^{\ell}\mathcal{Z}^{\alpha} \omega_h^{nhom}\,\mathrm{d}\mathcal{V}_t + b.t. 
\hspace{0.7cm}
\end{array}
\end{equation}

\begin{equation*}
\begin{array}{ll}
= \sum\limits_{i=1}^3
\e\int\limits_{\mathbb{R}^3_{-}} \frac{1}{z}\partial_i^{\varphi}\partial_t^{\ell}\mathcal{Z}^{\alpha}(\psi\ast h)\,
\frac{1}{\partial_z\varphi}\mathcal{Z}^3\omega_h^{nhom}
\cdot \partial_i^{\varphi}\partial_t^{\ell}\mathcal{Z}^{\alpha} \omega_h^{nhom}\,\mathrm{d}\mathcal{V}_t + b.t.\\[12pt]

\lem |\partial_z^{m+1} \psi|_{L^1(\mathrm{d}z)}^2 |h|_{X^{m-1,1}}^2 + \e\|\nabla^{\varphi}\partial_t^{\ell}\mathcal{Z}^{\alpha} \omega_h^{nhom}\|_{L^2}^2
+ b.t. \\[10pt]

\lem |h|_{X^{m-1,1}}^2 + \e\|\nabla^{\varphi}\partial_t^{\ell}\mathcal{Z}^{\alpha} \omega_h^{nhom}\|_{L^2}^2
+ b.t. \,.
\end{array}
\end{equation*}
and the third term is
\begin{equation}\label{Sect2_Vorticity_Estimate_7_3}
\begin{array}{ll}
\e\int\limits_{\mathbb{R}^3_{-}}[\partial_t^{\ell}\mathcal{Z}^{\alpha}, \nabla^{\varphi}\cdot]\nabla^{\varphi}\omega_h^{nhom}
\cdot \partial_t^{\ell}\mathcal{Z}^{\alpha} \omega_h^{nhom}\,\mathrm{d}\mathcal{V}_t \\[9pt]

\lem \e\sum\limits_{i=1}^3
\int\limits_{\mathbb{R}^3_{-}}\partial_i^{\varphi}\partial_t^{\ell}\mathcal{Z}^{\alpha}\eta\partial_z^{\varphi}\partial_i^{\varphi}\omega_h^{nhom}
\cdot \partial_t^{\ell}\mathcal{Z}^{\alpha} \omega_h^{nhom}\,\mathrm{d}\mathcal{V}_t + b.t. \\[10pt]

\lem - \e\sum\limits_{i=1}^3
\int\limits_{\mathbb{R}^3_{-}}\partial_z^{\varphi}\partial_i^{\varphi}\partial_t^{\ell}\mathcal{Z}^{\alpha}\eta\partial_i^{\varphi}\omega_h^{nhom}
\cdot \partial_t^{\ell}\mathcal{Z}^{\alpha} \omega_h^{nhom}\,\mathrm{d}\mathcal{V}_t \\[10pt]\quad
- \e\sum\limits_{i=1}^3
\int\limits_{\mathbb{R}^3_{-}}\partial_i^{\varphi}\partial_t^{\ell}\mathcal{Z}^{\alpha}\eta\partial_i^{\varphi}\omega_h^{nhom}
\cdot \partial_z^{\varphi}\partial_t^{\ell}\mathcal{Z}^{\alpha} \omega_h^{nhom}\,\mathrm{d}\mathcal{V}_t + b.t. \\[10pt]

\lem |\partial_z^{m+1} \psi|_{L^1(\mathrm{d}z)}^2 |h|_{X^{m-1,1}}^2 + \e\|\nabla^{\varphi}\partial_t^{\ell}\mathcal{Z}^{\alpha} \omega_h^{nhom}\|_{L^2}^2
+ b.t. \hspace{1.4cm}
\\[9pt]

\lem |h|_{X^{m-1,1}}^2 + \e\|\nabla^{\varphi}\partial_t^{\ell}\mathcal{Z}^{\alpha} \omega_h^{nhom}\|_{L^2}^2
+ b.t..
\end{array}
\end{equation}

Plug $(\ref{Sect2_Vorticity_Estimate_7_1}),(\ref{Sect2_Vorticity_Estimate_7_2}),(\ref{Sect2_Vorticity_Estimate_7_3})$ into $(\ref{Sect2_Vorticity_Estimate_7})$, we get
\begin{equation}\label{Sect2_Vorticity_Estimate_8}
\begin{array}{ll}
\frac{\mathrm{d}}{\mathrm{d}t} \|\partial_t^{\ell}\mathcal{Z}^{\alpha} \omega_h^{nhom}\|_{L^2}^2
+ 2\e\|\nabla^{\varphi} \partial_t^{\ell}\mathcal{Z}^{\alpha} \omega_h^{nhom}\|_{L^2}^2 \\[10pt]

\lem \|\partial_t^{\ell}\mathcal{Z}^{\alpha}\omega_h\|_{L^2}^2
+ \|\partial_z v\|_{X^{m-1}}^2 + |h|_{X^{m-1,1}}^2 + \e\|\nabla^{\varphi}\partial_t^{\ell}\mathcal{Z}^{\alpha} \omega_h^{nhom}\|_{L^2}^2 + b.t.\, .
\end{array}
\end{equation}

Sum $\ell$ and $\alpha$ in $(\ref{Sect2_Vorticity_Estimate_8})$, and integrate $(\ref{Sect2_Vorticity_Estimate_8})$ from $0$ to $t$, we get
\begin{equation}\label{Sect2_Vorticity_Estimate_9}
\begin{array}{ll}
\|\omega_h^{nhom}\|_{X^{m-1}}^2 + \e \int\limits_0^t \|\nabla \omega_h^{nhom}\|_{X^{m-1}}^2 \,\mathrm{d}t \\[7pt]

\lem \big\|\omega_h^{nhom}|_{t=0}\big\|_{X^{m-1}}^2
+ \int\limits_0^t\|\omega_h\|_{X^{m-1}}^2 \,\mathrm{d}t + \int\limits_0^t\|v\|_{X^{m-1,1}}^2  + |h|_{X^{m-1,1}}^2 \,\mathrm{d}t
\\[7pt]

\lem \big\|\omega_h|_{t=0}\big\|_{X^{m-1}}^2
+ \sqrt{t}\|\omega_h\|_{L^4([0,t],X^{m-1})}^2 + \int\limits_0^t\|v\|_{X^{m-1,1}}^2 \,\mathrm{d}t + \int\limits_0^t|h|_{X^{m-1,1}}^2 \,\mathrm{d}t.
\end{array}
\end{equation}

It follows from $(\ref{Sect2_Vorticity_Estimate_9})$ that
\begin{equation}\label{Sect2_Vorticity_Estimate_10}
\begin{array}{ll}
\|\omega_h^{nhom}\|_{X^{m-1}}^4
\lem \big\|\omega_h|_{t=0}\big\|_{X^{m-1}}^4
+ T\|\omega_h\|_{L^4([0,t],X^{m-1})}^4 \\[7pt]\hspace{2.6cm}
+ \big(\int\limits_0^T\|v\|_{X^{m-1,1}}^2 \,\mathrm{d}t\big)^2 + \big(\int\limits_0^T|h|_{X^{m-1,1}}^2 \,\mathrm{d}t\big)^2,
\\[14pt]

\int\limits_0^t\|\omega_h^{nhom}\|_{X^{m-1}}^4 \,\mathrm{d}t
\lem T\big\|\omega_h|_{t=0}\big\|_{X^{m-1}}^4 
+ T\int\limits_0^t \|\omega_h\|_{L^4([0,t],X^{m-1})}^4 \,\mathrm{d}t \\[7pt]\hspace{3.28cm}
+ T\big(\int\limits_0^T\|v\|_{X^{m-1,1}}^2 \,\mathrm{d}t\big)^2 + T\big(\int\limits_0^T|h|_{X^{m-1,1}}^2 \,\mathrm{d}t\big)^2.
\end{array}
\end{equation}

For the homogeneous equations $(\ref{Sect2_Vorticity_Estimate_3})$, the same as the $L^{4}([0,T],L^2)$ estimate in \cite{Masmoudi_Rousset_2012_FreeBC}
and paradifferential calculus (see Theorem 10.6 in \cite{Masmoudi_Rousset_2012_FreeBC}),
when $\ell +|\alpha|\leq m-1$, we have
\begin{equation}\label{Sect2_Vorticity_Estimate_11}
\begin{array}{ll}
\|\partial_t^{\ell}\mathcal{Z}^{\alpha}\omega_h^{hom}\|_{L^4([0,T],L^2(\mathbb{R}^3_{-}))}^2 \\[8pt]
\lem \|\partial_t^{\ell}\mathcal{Z}^{\alpha}\omega_h^{hom}\|_{H^{\frac{1}{4}}([0,T],L^2(\mathbb{R}^3_{-}))}^2 \\[8pt]
\lem \sqrt{\e}\int\limits_0^T\big|\partial_t^{\ell}\mathcal{Z}^{\alpha}\omega_h^{hom}|_{z=0}\big|_{L^2(\mathbb{R}^2)}^2 \,\mathrm{d}t \\[8pt]

\lem \sqrt{\e}\int\limits_0^T\big|\partial_t^{\ell}\mathcal{Z}^{\alpha}\big(\vec{\textsf{F}}^{1,2}[\nabla\varphi](\partial_j v^i)\big)|_{z=0}\big|_{L^2(\mathbb{R}^2)}^2 \,\mathrm{d}t \\[8pt]

\lem \sqrt{\e}\int\limits_0^T|h|_{X^{m-1,1}}^2\,\mathrm{d}t
+ \sqrt{\e}\int\limits_0^T\big|\partial_j v^i|_{z=0}\big|_{X^{m-1}}^2\,\mathrm{d}t \\[8pt]

\lem \sqrt{\e}\int\limits_0^T|h|_{X^{m-1,1}}^2\,\mathrm{d}t
+ \sqrt{\e}\int\limits_0^T\big|v|_{z=0}\big|_{X_{tan}^{m-1,1}}^2\,\mathrm{d}t \\[8pt]

\lem \sqrt{\e}\int\limits_0^T|h|_{X^{m-1,1}}^2\,\mathrm{d}t
+ \sqrt{\e}\int\limits_0^T\|\partial_z v\|_{X_{tan}^{m-1,1}} \|v\|_{X_{tan}^{m-1,1}}\,\mathrm{d}t \\[8pt]

\lem \sqrt{\e}\int\limits_0^T|h|_{X^{m-1,1}}^2\,\mathrm{d}t
+ \int\limits_0^T\|v\|_{X^{m-1,1}}^2\,\mathrm{d}t + \e\int\limits_0^T\|\partial_z v\|_{X^{m-1,1}}^2\,\mathrm{d}t,
\end{array}
\end{equation}
where $\big|\partial_j v^i|_{z=0}\big|_{X^{m-1}} = \big|v|_{z=0}\big|_{H^m}$ since $j=1,2$.

Sum $\alpha$ in $(\ref{Sect2_Vorticity_Estimate_11})$, we get
\begin{equation}\label{Sect2_Vorticity_Estimate_12}
\begin{array}{ll}
\|\omega_h^{hom}\|_{L^4([0,T],X^{m-1})}^2
\lem \int\limits_0^T|h|_{X^{m-1,1}}^2 + \|v\|_{X^{m-1,1}}^2\,\mathrm{d}t + \e\int\limits_0^T\|\partial_z v\|_{X^{m-1,1}}^2\,\mathrm{d}t.
\end{array}
\end{equation}

Square $(\ref{Sect2_Vorticity_Estimate_12})$, we have
\begin{equation}\label{Sect2_Vorticity_Estimate_12_Square}
\begin{array}{ll}
\|\omega_h^{hom}\|_{L^4([0,t],X^{m-1})}^4
\lem \|\omega_h^{hom}\|_{L^4([0,T],X^{m-1})}^4
\\[6pt]
\lem \big(\int\limits_0^T|h|_{X^{m-1,1}}^2\,\mathrm{d}t\big)^2 
+ \big(\int\limits_0^T\|v\|_{X^{m-1,1}}^2\,\mathrm{d}t\big)^2 
+ \big(\e\int\limits_0^T\|\partial_z v\|_{X^{m-1,1}}^2\,\mathrm{d}t\big)^2.
\end{array}
\end{equation}

By $(\ref{Sect2_Vorticity_Estimate_10})$ and $(\ref{Sect2_Vorticity_Estimate_12_Square})$, we have
\begin{equation}\label{Sect2_Vorticity_Estimate_13}
\begin{array}{ll}
\|\omega_h\|_{L^4([0,t],X^{m-1})}^4 \lem \|\omega_h^{nhom}\|_{L^4([0,t],X^{m-1})}^4 + \|\omega_h^{hom}\|_{L^4([0,t],X^{m-1})}^4 \\[7pt]

\lem \big\|\omega_h|_{t=0}\big\|_{X^{m-1}}^4
+ \int\limits_0^t \|\omega_h\|_{L^4([0,t],X^{m-1})}^4 \,\mathrm{d}t 
+ \big(\int\limits_0^T\|v\|_{X^{m-1,1}}^2 \,\mathrm{d}t\big)^2 \\[7pt]\quad
+ \big(\int\limits_0^T|h|_{X^{m-1,1}}^2 \,\mathrm{d}t\big)^2
+ \big(\e\int\limits_0^T\|\partial_z v\|_{X^{m-1,1}}^2\,\mathrm{d}t\big)^2.
\end{array}
\end{equation}

By the integral form of Gronwall's inequality, it is easy to have
\begin{equation}\label{Sect2_Vorticity_Estimate_14}
\begin{array}{ll}
\|\omega_h\|_{L^4([0,T],X^{m-1})}^2 \\[5pt]
\lem \big\|\omega_h|_{t=0}\big\|_{X^{m-1}}^2
+ \int\limits_0^T\|v\|_{X^{m-1,1}}^2 \,\mathrm{d}t
+ \int\limits_0^T|h|_{X^{m-1,1}}^2 \,\mathrm{d}t

+ \e\int\limits_0^T\|\partial_z v\|_{X^{m-1,1}}^2\,\mathrm{d}t.
\end{array}
\end{equation}

While by $(\ref{Sect2_NormalDer_Vorticity_Estimate_5})$, we have
\begin{equation}\label{Sect2_Vorticity_Estimate_15}
\begin{array}{ll}
\|\partial_z v_h\|_{L^4([0,T],X^{m-1})}^2 \\[7pt]
\lem \|\omega_h\|_{L^4([0,T],X^{m-1})}^2 + |h|_{L^4([0,T],X^{m-1,\frac{1}{2}})}^2 + \|v\|_{L^4([0,T],X^{m-1,1})}^2.
\end{array}
\end{equation}

By the divergence free condition $(\ref{Sect2_DivFreeCondition})$, we have
\begin{equation}\label{Sect2_Vorticity_Estimate_17}
\begin{array}{ll}
\|\partial_z v^3\|_{L^4([0,T],X^{m-1})}^2 \\[7pt]
\lem \|\partial_z v_h\|_{L^4([0,T],X^{m-1})}^2 + \|\nabla\varphi\|_{L^4([0,T],X^{m-1})}^2 + \|\partial_j v^i\|_{L^4([0,T],X^{m-1})}^2 \\[7pt]
\lem \|\omega_h\|_{L^4([0,T],X^{m-1})}^2 + |h|_{L^4([0,T],X^{m-1,\frac{1}{2}})}^2 + \|v\|_{L^4([0,T],X^{m-1,1})}^2.
\end{array}
\end{equation}
Thus, Lemma $\ref{Sect2_Vorticity_Lemma}$ is proved.
\end{proof}

Refer to \cite{Masmoudi_Rousset_2012_FreeBC} for the $L^{\infty}$ estimates which imply
$\partial_z v, \mathcal{Z}^3\partial_z v, \sqrt{\e}\partial_{zz}v \in L^{\infty}$.
The argument is based on analyzing 1D Fokker Planck equation which has explicit Green function.

In the following lemma, we estimate $\|\partial_z v\|_{L^{\infty}([0,T],X^{m-2})}$.
Note that we can not have $\partial_z v\in L^{\infty}([0,T],X^{m-1})$ due to
$\omega|_{z=0} =\textsf{F}^{1,2} [\nabla\varphi](\partial_j v^i)$, see $(\ref{Sect1_Vorticity_H_Eq})$.
\begin{lemma}\label{Sect2_NormalDer_Lemma}
Assume $v$ and $\omega$ are the velocity and vorticity of the free surface Navier-Stokes equations $(\ref{Sect1_NS_Eq})$ respectively.
$\omega_h$ satisfies the following estimate:
\begin{equation}\label{Sect2_NormalDer_Estimate}
\begin{array}{ll}
\|\partial_z v\|_{X^{m-2}}^2 + \e\int\limits_0^t\|\partial_{zz}v\|_{X^{m-2}}^2
\mathrm{d}\mathcal{V}_t\mathrm{d}t
+ \|\omega\|_{X^{m-2}}^2 + \e\int\limits_0^t\|\nabla\omega\|_{X^{m-2}}^2
\mathrm{d}\mathcal{V}_t\mathrm{d}t \\[6pt]

\lem \|\partial_z v_0\|_{X^{m-2}}^2 + \int\limits_0^t \|v\|_{X^{m-1,1}}^2 + |h|_{X^{m-1}}^2\,\mathrm{d}t
+ \|\partial_z v\|_{L^4([0,T],X^{m-1})}^2.
\end{array}
\end{equation}
\end{lemma}

\begin{proof}
By the divergence free condition $\nabla^{\varphi}\cdot v =0$, we have
\begin{equation}\label{Sect2_NormalDer_Estimate_Laplacian}
\begin{array}{ll}
\triangle^{\varphi} v = \nabla^{\varphi}(\nabla^{\varphi}\cdot v) - \nabla^{\varphi}\times (\nabla^{\varphi}\times v)
= - \nabla^{\varphi}\times\omega.
\end{array}
\end{equation}

Firstly, we develop the $L^2$ estimate of $\omega$. Multiple $(\ref{Sect1_NS_Eq})$ with $\nabla^{\varphi}\times \omega$,
integrate in $\mathbb{R}^3_{-}$, use the integration by parts formula $(\ref{Sect1_Formulas_CanNotUse})_3$, we get
\begin{equation}\label{Sect2_NormalDer_Estimate_L2_1}
\begin{array}{ll}
\int\limits_{\mathbb{R}^3_{-}} \big(\partial_t^{\varphi} v + v\cdot\nabla^{\varphi} v + \nabla^{\varphi} q
+ \e\nabla^{\varphi}\times \omega \big)
\cdot \nabla^{\varphi}\times \omega \mathrm{d}\mathcal{V}_t =0, \\[16pt]

\int\limits_{\mathbb{R}^3_{-}} \big(\partial_t^{\varphi} \omega + v\cdot\nabla^{\varphi} \omega + \nabla^{\varphi}\times\nabla^{\varphi} q \big)
\cdot \omega \mathrm{d}\mathcal{V}_t
+ \e\int\limits_{\mathbb{R}^3_{-}} |\nabla^{\varphi}\times \omega|^2 \mathrm{d}\mathcal{V}_t \\[2pt]\quad

= - \int\limits_{z=0} \big(\partial_t v + v_y\cdot\nabla_y v + \nabla^{\varphi} q \big)
\cdot \NN\times \omega \mathrm{d}\mathcal{V}_t

- \int\limits_{\mathbb{R}^3_{-}} (\sum\limits_{i=1}^3 \nabla^{\varphi} v^i \partial_i^{\varphi} v
\cdot \omega \mathrm{d}\mathcal{V}_t, \\[13pt]

\|\omega\|_{L^2}^2 + \e\int\limits_0^t \|\nabla \omega\|_{L^2}^2 \mathrm{d}_t \leq \big\|\omega|_{t=0}\big\|_{L^2}^2 + b.t.\, .
\end{array}
\end{equation}
Note that $\int\limits_{z=0} \nabla^{\varphi} q \cdot \NN\times \omega \mathrm{d}\mathcal{V}_t
\lem \big|\nabla^{\varphi} q|_{z=0}\big|_{-\frac{1}{2}} + \big|\NN\times \omega|_{z=0}\big|_{\frac{1}{2}} =b.t. \, .$

\vspace{0.2cm}
When $1\leq \ell+|\alpha|\leq m-2$, apply $\partial_t^{\ell}\mathcal{Z}^{\alpha}$ to $(\ref{Sect1_NS_Eq})$ and rewrite the viscous terms, we have
\begin{equation}\label{Sect2_NormalDer_Estimate_1}
\begin{array}{ll}
\partial_t^{\varphi} \partial_t^{\ell}\mathcal{Z}^{\alpha}v
+ v\cdot\nabla^{\varphi} \partial_t^{\ell}\mathcal{Z}^{\alpha} v
+ \nabla^{\varphi} \partial_t^{\ell}\mathcal{Z}^{\alpha} q
+ \e\nabla^{\varphi}\times(\nabla^{\varphi}\times \partial_t^{\ell}\mathcal{Z}^{\alpha}v) = \e\, \mathcal{I}_{1,1} + \mathcal{I}_{1,2},
\end{array}
\end{equation}
where
\begin{equation*}
\begin{array}{ll}
\mathcal{I}_{1,1}
= - [\partial_t^{\ell}\mathcal{Z}^{\alpha}, \nabla^{\varphi} \times] \omega
- \nabla^{\varphi} \times [\partial_t^{\ell}\mathcal{Z}^{\alpha}, \nabla^{\varphi}\times] v, \\[8pt]

\mathcal{I}_{1,2}
= - [\partial_t^{\ell}\mathcal{Z}^{\alpha}, v_y \cdot \nabla_y + V_z \partial_z] v
- [\partial_t^{\ell}\mathcal{Z}^{\alpha}, \NN\partial_z^{\varphi}] q.
\end{array}
\end{equation*}

Multiple $(\ref{Sect2_NormalDer_Estimate_1})$ with
$\nabla^{\varphi}\times (\nabla^{\varphi}\times \partial_t^{\ell}\mathcal{Z}^{\alpha}v)$,
integrate in $\mathbb{R}^3_{-}$, we have
\begin{equation}\label{Sect2_NormalDer_Estimate_2}
\begin{array}{ll}
\int\limits_{\mathbb{R}^3_{-}}
\partial_t^{\varphi} \partial_t^{\ell}\mathcal{Z}^{\alpha}v \cdot \nabla^{\varphi}\times (\nabla^{\varphi}\times \partial_t^{\ell}\mathcal{Z}^{\alpha}v)
\mathrm{d}\mathcal{V}_t
\\[7pt]\
+ \int\limits_{\mathbb{R}^3_{-}}v\cdot\nabla^{\varphi} \partial_t^{\ell}\mathcal{Z}^{\alpha} v
\cdot \nabla^{\varphi}\times (\nabla^{\varphi}\times \partial_t^{\ell}\mathcal{Z}^{\alpha}v)
\mathrm{d}\mathcal{V}_t \\[7pt]\

+ \int\limits_{\mathbb{R}^3_{-}}\nabla^{\varphi} \partial_t^{\ell}\mathcal{Z}^{\alpha} q
\cdot \nabla^{\varphi}\times (\nabla^{\varphi}\times \partial_t^{\ell}\mathcal{Z}^{\alpha}v)
\mathrm{d}\mathcal{V}_t

+ \e\int\limits_{\mathbb{R}^3_{-}} |\nabla^{\varphi}\times (\nabla^{\varphi}\times \partial_t^{\ell}\mathcal{Z}^{\alpha}v)|^2
\mathrm{d}\mathcal{V}_t \\[7pt]\

= \int\limits_{\mathbb{R}^3_{-}}(\e\, \mathcal{I}_{1,1} + \mathcal{I}_{1,2}) \cdot \nabla^{\varphi}\times (\nabla^{\varphi}\times \partial_t^{\ell}\mathcal{Z}^{\alpha}v) \mathrm{d}\mathcal{V}_t,
\end{array}
\end{equation}

Use the integration by parts formula $(\ref{Sect1_Formulas_CanNotUse})_3$ and note that
$[\partial_t^{\varphi}, \nabla^{\varphi}] =0$, we have
\begin{equation}\label{Sect2_NormalDer_Estimate_3}
\begin{array}{ll}
\int\limits_{\mathbb{R}^3_{-}}
\partial_t^{\varphi} |\nabla^{\varphi}\times\partial_t^{\ell}\mathcal{Z}^{\alpha}v|^2
\mathrm{d}\mathcal{V}_t
+ \int\limits_{\mathbb{R}^3_{-}}v\cdot\nabla^{\varphi} |\nabla^{\varphi}\times \partial_t^{\ell}\mathcal{Z}^{\alpha} v |^2
\mathrm{d}\mathcal{V}_t
\\[8pt]\quad

+ \int\limits_{\mathbb{R}^3_{-}} (\nabla^{\varphi}\times \nabla^{\varphi} \partial_t^{\ell}\mathcal{Z}^{\alpha} q)
\cdot  (\nabla^{\varphi}\times \partial_t^{\ell}\mathcal{Z}^{\alpha}v)
\mathrm{d}\mathcal{V}_t

+ \e\int\limits_{\mathbb{R}^3_{-}} |\nabla^{\varphi}\times (\nabla^{\varphi}\times \partial_t^{\ell}\mathcal{Z}^{\alpha}v)|^2
\mathrm{d}\mathcal{V}_t \\[8pt]

= - \int\limits_{z=0}
(\partial_t^{\varphi} \partial_t^{\ell}\mathcal{Z}^{\alpha}v + v\cdot\nabla^{\varphi} \partial_t^{\ell}\mathcal{Z}^{\alpha} v)
\cdot \NN\times (\nabla^{\varphi}\times \partial_t^{\ell}\mathcal{Z}^{\alpha}v)
\mathrm{d}y
\\[6pt]\quad

- \int\limits_{\mathbb{R}^3_{-}} [(\sum\limits_{i=1}^3 \nabla^{\varphi}v^i\cdot\partial_i^{\varphi})\times \partial_t^{\ell}\mathcal{Z}^{\alpha} v]
\cdot (\nabla^{\varphi}\times \partial_t^{\ell}\mathcal{Z}^{\alpha}v)
\mathrm{d}\mathcal{V}_t \\[6pt]\quad

- \int\limits_{z=0}\nabla^{\varphi} \partial_t^{\ell}\mathcal{Z}^{\alpha} q
\cdot \NN\times (\nabla^{\varphi}\times \partial_t^{\ell}\mathcal{Z}^{\alpha}v)
\mathrm{d}y
+ \int\limits_{\mathbb{R}^3_{-}} \e\,\mathcal{I}_{1,1} \cdot \nabla^{\varphi}\times (\nabla^{\varphi}\times \partial_t^{\ell}\mathcal{Z}^{\alpha}v) \mathrm{d}\mathcal{V}_t\\[8pt]\quad

+ \int\limits_{z=0}\mathcal{I}_{1,2} \cdot \NN\times (\nabla^{\varphi}\times \partial_t^{\ell}\mathcal{Z}^{\alpha}v) \mathrm{d}y
+ \int\limits_{\mathbb{R}^3_{-}} \nabla^{\varphi}\times\mathcal{I}_{1,2} \cdot (\nabla^{\varphi}\times \partial_t^{\ell}\mathcal{Z}^{\alpha}v) \mathrm{d}\mathcal{V}_t.
\end{array}
\end{equation}

By $\nabla^{\varphi}\times \nabla^{\varphi} \partial_t^{\ell}\mathcal{Z}^{\alpha} q =0$, we have
\begin{equation}\label{Sect2_NormalDer_Estimate_4}
\begin{array}{ll}
\frac{\mathrm{d}}{\mathrm{d}t} \int\limits_{\mathbb{R}^3_{-}}
|\nabla^{\varphi}\times\partial_t^{\ell}\mathcal{Z}^{\alpha}v|^2
\mathrm{d}\mathcal{V}_t

+ 2\e\int\limits_{\mathbb{R}^3_{-}} |\nabla^{\varphi}\times (\nabla^{\varphi}\times \partial_t^{\ell}\mathcal{Z}^{\alpha}v)|^2
\mathrm{d}\mathcal{V}_t \\[7pt]

\lem \|\nabla^{\varphi}\times\partial_t^{\ell}\mathcal{Z}^{\alpha}v\|_{L^2}^2
+ \e\|\nabla^{\varphi}\times (\nabla^{\varphi}\times \partial_t^{\ell}\mathcal{Z}^{\alpha}v)\|_{L^2}^2

+ \big|v |_{z=0}\big|_{X_{tan}^{m-1}}^2 \\[7pt]\quad
+ \big|\nabla^{\varphi}\times\partial_t^{\ell}\mathcal{Z}^{\alpha}v |_{z=0}\big|_{\frac{1}{2}}^2

+ \big|\nabla^{\varphi}\partial_t^{\ell}\mathcal{Z}^{\alpha} q|_{z=0} \big|_{-\frac{1}{2}}^2
+ \|\nabla \partial_t^{\ell}\mathcal{Z}^{\alpha}v\|_{L^2}^2

\\[7pt]\quad
+ \sum\limits_{\ell_1 +|\alpha|_1 \leq m-3}\big|\nabla^{\varphi}\partial_t^{\ell_1}\mathcal{Z}^{\alpha_1} q|_{z=0} \big|_{-\frac{1}{2}}^2
+ \e\|\mathcal{I}_{1,1}\|_{L^2}^2
+ \|\nabla^{\varphi}\times \mathcal{I}_{1,2}\|_{L^2}^2.
\end{array}
\end{equation}

It is easy to prove that $\|\mathcal{I}_{1,1}\|_{L^2} \lem \|\nabla \omega\|_{X^{m-2}}$. Then we estimate $\nabla^{\varphi}\times \mathcal{I}_{1,2}$.
\begin{equation}\label{Sect2_NormalDer_Estimate_5}
\begin{array}{ll}
\|\nabla^{\varphi}\times [\partial_t^{\ell}\mathcal{Z}^{\alpha}, V_z \partial_z] v\|_{L^2}
\lem \sum\limits_{\ell^1+|\alpha^1|>0} \big(\|\frac{1-z}{z}\partial_t^{\ell^1}\mathcal{Z}^{\alpha^1} V_z \cdot
\frac{z}{1-z}\nabla^{\varphi}\times \partial_t^{\ell^2}\mathcal{Z}^{\alpha^2} \partial_z v \|_{L^2} \\[7pt]\qquad
+ \|\nabla^{\varphi}\partial_t^{\ell^1}\mathcal{Z}^{\alpha^1} V_z
\times \partial_t^{\ell^2}\mathcal{Z}^{\alpha^2} \partial_z v \|_{L^2} \big) 
\end{array}
\end{equation}

\begin{equation*}
\begin{array}{ll}
\quad
\lem \sum\limits_{\ell^1+|\alpha^1|>0} \big(\|\nabla\partial_t^{\ell^1}\mathcal{Z}^{\alpha^1} V_z \|_{L^2}
+ \|\nabla^{\varphi}\times \partial_t^{\ell^2}\mathcal{Z}^{\alpha^2} \mathcal{Z}^3 v \|_{L^2}
+ \|\partial_t^{\ell^2}\mathcal{Z}^{\alpha^2} \partial_z v \|_{L^2} \big) \\[10pt]\quad

\lem \|\omega\|_{X^{m-2}} + \|v\|_{X^{m-1}} + \|\partial_z v\|_{X^{m-2}}
+ \|\partial_{zz} \eta\|_{X^{m-2}} + \|\partial_{zt} \eta\|_{X^{m-2}} \\[6pt]\quad

\lem \|\omega\|_{X^{m-2}} + \|v\|_{X^{m-1}} + |h|_{X^{m-1}}.

\\[16pt]
\|\nabla^{\varphi}\times [\partial_t^{\ell}\mathcal{Z}^{\alpha}, \NN\partial_z^{\varphi}] q\|_{L^2} \\[6pt]\quad

\lem \|(\nabla_y, 0)^{\top}\times [\partial_t^{\ell}\mathcal{Z}^{\alpha}, \NN\partial_z^{\varphi}] q\|_{L^2}
+ \|\NN\partial_z^{\varphi}\times [\partial_t^{\ell}\mathcal{Z}^{\alpha}, \NN\partial_z^{\varphi}] q\|_{L^2} \\[7pt]\quad

\lem \|\nabla q\|_{X^{m-1}} + \sum\limits_{\ell^1+|\alpha^1|>0} \big(\|\NN\partial_z^{\varphi}\times \partial_t^{\ell^1}\mathcal{Z}^{\alpha^1}\NN
\cdot \partial_t^{\ell^2}\mathcal{Z}^{\alpha^2} \partial_z^{\varphi} q \|_{L^2} \\[11pt]\qquad
+ \|\NN\times \partial_t^{\ell^1}\mathcal{Z}^{\alpha^1}\NN
\cdot \partial_z^{\varphi}\partial_t^{\ell^2}\mathcal{Z}^{\alpha^2} \partial_z^{\varphi} q \|_{L^2} \\[7pt]\quad

\lem \|\nabla q\|_{X^{m-1}} + \|\partial_{zz}^{\varphi} q \|_{X^{m-3}}
\lem \|\nabla q\|_{X^{m-1}} + \|\triangle^{\varphi} q \|_{X^{m-3}} \\[7pt]\quad

\lem \|\nabla q\|_{X^{m-1}} + \|\partial_i^{\varphi} v^j \partial_j^{\varphi} v^i\|_{X^{m-3}}

\lem \|\nabla q\|_{X^{m-1}} + \|v\|_{X^{m-2}} + \|\omega\|_{X^{m-2}}.
\end{array}
\end{equation*}

Plug $(\ref{Sect2_NormalDer_Estimate_5})$ into $(\ref{Sect2_NormalDer_Estimate_4})$, integrate in time and apply the integral form of Gronwall's inequality, then we have
\begin{equation}\label{Sect2_NormalDer_Estimate_6}
\begin{array}{ll}
\|\nabla^{\varphi}\times\partial_t^{\ell}\mathcal{Z}^{\alpha}v\|_{L^2}^2
+ \e\int\limits_0^t\|\nabla^{\varphi}\times (\nabla^{\varphi}\times \partial_t^{\ell}\mathcal{Z}^{\alpha}v)\|_{L^2}^2 \mathrm{d}t \\[6pt]

\lem \big\|\nabla^{\varphi}\times\partial_t^{\ell}\mathcal{Z}^{\alpha}v |_{t=0}\big\|_{L^2}^2
+ \int\limits_0^t \|\omega\|_{X^{m-2}}^2 + \|v\|_{X^{m-1,1}}^2 + |h|_{X^{m-1}}^2 \\[9pt]\quad
+ \|\partial_z v\|_{X^{m-1}}^2 + \|\nabla q\|_{X^{m-1}}^2\,\mathrm{d}t + b.t.
\end{array}
\end{equation}

$\partial_t^{\ell}\mathcal{Z}^{\alpha} \omega$ is equivalent to $\nabla^{\varphi} \times \partial_t^{\ell}\mathcal{Z}^{\alpha} v$, due to
$\ell +|\alpha|\leq m-2$ and
\begin{equation}\label{Sect2_NormalDer_Estimate_2_Formula}
\begin{array}{ll}
\partial_t^{\ell}\mathcal{Z}^{\alpha} \omega - \partial_t^{\ell}\mathcal{Z}^{\alpha} (\nabla^{\varphi} \times v)
= \sum\limits_{\ell_1 +|\alpha_1| >0} \partial_t^{\ell_1}\mathcal{Z}^{\alpha_1} (\frac{\NN}{\partial_z\varphi})\partial_z
\times \partial_t^{\ell_2}\mathcal{Z}^{\alpha_2} v, \\[14pt]

\|\partial_t^{\ell}\mathcal{Z}^{\alpha} \omega - \partial_t^{\ell}\mathcal{Z}^{\alpha} (\nabla^{\varphi} \times v)\|_{L^2}
\lem \|\partial_z v\|_{X^{m-3}} + |h|_{X^{m-2,\frac{1}{2}}}.
\end{array}
\end{equation}
Then we have the estimate of the vorticity:
\begin{equation}\label{Sect2_NormalDer_Estimate_7}
\begin{array}{ll}
\|\omega\|_{X^{m-2}}^2 + \e\int\limits_0^t\|\nabla\omega\|_{X^{m-2}}^2
\mathrm{d}\mathcal{V}_t\mathrm{d}t \\[6pt]

\lem \|\omega_0 \|_{X^{m-2}}^2 + \int\limits_0^t \|v\|_{X^{m-1,1}}^2 + |h|_{X^{m-1}}^2
+ \|\partial_z v\|_{X^{m-1}}^2 + \|\nabla q\|_{X^{m-1}}^2\,\mathrm{d}t.
\end{array}
\end{equation}
Thus, Lemma $\ref{Sect2_NormalDer_Lemma}$ is proved.
\end{proof}

\begin{remark}\label{Sect2_NormalDer_Remark}
There is another approach to estimate $\|\omega_h\|_{X^{m-2}}$ and $\|\partial_z v\|_{X^{m-2}}$, that is, we define variables:
\begin{equation}\label{Sect2_NormalDer_Remark_1}
\left\{\begin{array}{ll}
\zeta^1 = \omega^1 -\textsf{F}^1 [\nabla\varphi](\partial_j v^i), \ j=1,2,\ i=1,2,3,
\\[9pt]

\zeta^2 = \omega^2 -\textsf{F}^2 [\nabla\varphi](\partial_j v^i), \ j=1,2,\ i=1,2,3,
\end{array}\right.
\end{equation}
then $\zeta$ satisfies the following equation:
\begin{equation}\label{Sect2_NormalDer_Remark_2}
\left\{\begin{array}{ll}
\partial_t \zeta + v_y\cdot\nabla_y\zeta + V_z\partial_z\zeta - \e\triangle^{\varphi}\zeta

= \vec{\textsf{F}}^0[\nabla\varphi](\zeta + \vec{\textsf{F}}^{1,2}[\nabla\varphi](\partial_j v^i),\partial_j v^i)  \\[6pt]\quad
+ \vec{\textsf{F}}^{1,2}[\nabla\varphi](\partial_j\partial_i^{\varphi} q + [\partial_j, \partial_t^{\varphi} + v\cdot\nabla^{\varphi} - \e\triangle^{\varphi}]v^i)

- \partial_j v^i(\partial_t^{\varphi} \vec{\textsf{F}}^{1,2}[\nabla\varphi] \\[6pt]\quad
+ v\cdot\nabla^{\varphi} \vec{\textsf{F}}^{1,2}[\nabla\varphi]
- \e\triangle^{\varphi} \vec{\textsf{F}}^{1,2}[\nabla\varphi])
+ \e\nabla^{\varphi} \vec{\textsf{F}}^{1,2}[\nabla\varphi] \cdot \nabla^{\varphi}\partial_j v^i, \\[8pt]

\zeta|_{z=0} =0, \\[6pt]

\zeta|_{t=0} = \omega_{h,0} -\textsf{F}^{1,2} [\nabla\varphi](\partial_j v^i)|_{t=0}.
\end{array}\right.
\end{equation}
Then we have the estimate of $\zeta$:
\begin{equation}\label{Sect2_NormalDer_Remark_3}
\begin{array}{ll}
\|\zeta\|_{X^{m-2}}^2 + \e \int\limits_0^t\|\nabla \zeta\|_{X^{m-2}}^2 \,\mathrm{d}t  \\[7pt]
\lem \big\|\zeta|_{t=0}\big\|_{X^{m-2}}^2
+ \int\limits_0^t\|v\|_{X^{m-2,2}}^2 + \|\nabla q\|_{X^{m-2,1}}^2 + |h|_{X^{m}}^2 \,\mathrm{d}t \\[10pt]

\lem \big\|\omega|_{t=0}\big\|_{X^{m-2}}^2 + \big\|h|_{t=0}\big\|_{X^{m-2,\frac{1}{2}}}^2 + \big\|v|_{t=0}\big\|_{X^{m-2,1}}^2
+ \int\limits_0^t \cdots \,\mathrm{d}t.
\end{array}
\end{equation}

By using $(\ref{Sect2_NormalDer_Remark_1})$ and $(\ref{Sect2_NormalDer_Remark_3})$, we can estimate $\|\omega\|_{X^{m-2}}$. However, we can not use this method and variables $(\ref{Sect2_NormalDer_Remark_1})$ to estimate convergence rates of the inviscid limit, see Remark $\ref{Sect5_NormalDer_Remark}$.
\end{remark}

By the estimates proved in Lemmas $\ref{Sect2_Tangential_Estimate_Lemma}, \ref{Sect2_Vorticity_Lemma}, \ref{Sect2_NormalDer_Lemma}$,
it is standard to prove Proposition $\ref{Sect1_Proposition_TimeRegularity}$.

%%% find 3

\section{Strong Vorticity Layer Caused by Strong Initial Vorticity Layer}

In this section, we study the strong vorticity layer for the free surface N-S equations $(\ref{Sect1_NS_Eq})$,
which arises from the strong initial vorticity layer.

\subsection{The Equations Transformed in Lagrangian Coordinates}
In this preliminaries, we derive the evolution equations of $\hat{\omega}_h = \omega_h^{\e} -\omega_h$,
and construct a variable which satisfies the heat equations with damping.

$\hat{\omega}$ satisfies the equations $(\ref{Sect1_N_Derivatives_Difference_Eq})$,
plug the following equality into $(\ref{Sect1_N_Derivatives_Difference_Eq})$,
\begin{equation}\label{Sect3_Preliminaries_Vorticity_Eq_2}
\begin{array}{ll}
\vec{\textsf{F}}^0[\nabla\varphi^{\e}](\omega_h^{\e},\partial_j v^{\e,i})
- \vec{\textsf{F}}^0[\nabla\varphi](\omega_h,\partial_j v^i) \\[7pt]

= f^7[\nabla\varphi^{\e},\nabla\varphi,\partial_j v^{\e,i},\partial_j v^i]\hat{\omega}_h
+ f^8[\nabla\varphi^{\e},\nabla\varphi,\partial_j v^{\e,i},\partial_j v^i,\omega_h^{\e},\omega_h]\partial_j\hat{v}^i \\[7pt]\quad
+ f^9[\nabla\varphi^{\e},\nabla\varphi,\partial_j v^{\e,i},\partial_j v^i,\omega_h^{\e},\omega_h]\nabla\hat{\varphi},
\end{array}
\end{equation}
where these coefficients $f^7[\cdots],f^8[\cdots],f^9[\cdots]$ are uniformly bounded with respect to $\e$.
Then we obtain the following system of $\hat{\omega}_h$:
\begin{equation}\label{Sect3_Vorticity_Eq}
\left\{\begin{array}{ll}
\partial_t^{\varphi^{\e}} \hat{\omega}_h
+ v^{\e} \cdot\nabla^{\varphi^{\e}} \hat{\omega}_h
- \e\triangle^{\varphi^{\e}}\hat{\omega}_h
- f^7[\nabla\varphi^{\e},\nabla\varphi,\partial_j v^{\e,i},\partial_j v^i]\hat{\omega}_h \\[8pt]\
= f^8[\nabla\varphi^{\e},\nabla\varphi,\partial_j v^{\e,i},\partial_j v^i,\omega_h^{\e},\omega_h]\partial_j\hat{v}^i
+ f^9[\nabla\varphi^{\e},\nabla\varphi,\partial_j v^{\e,i},\partial_j v^i,\omega_h^{\e},\omega_h]\nabla\hat{\varphi} \\[7pt]\quad\
+ \e\triangle^{\varphi^{\e}}\omega_h
+ \partial_z^{\varphi}\omega_h \partial_t^{\varphi^{\e}} \hat{\varphi}
+ \partial_z^{\varphi} \omega_h\, v^{\e}\cdot \nabla^{\varphi^{\e}}\hat{\varphi}
- \hat{v}\cdot\nabla^{\varphi} \omega_h
, \\[8pt]

\hat{\omega}_h|_{z=0} =\textsf{F}^{1,2} [\nabla\varphi^{\e}](\partial_j v^{\e,i}) - \omega_h^b := \hat{\omega}_h^b, \\[6pt]

\hat{\omega}_h|_{t=0} =\omega^{\e}_{h,0} -\omega_{h,0} := \hat{\omega}_{h,0}.
\end{array}\right.
\end{equation}

Similar to \cite{Masmoudi_Rousset_2012_FreeBC}, we eliminate the convection term by using the Lagrangian parametrization of $\Omega_t$:
\begin{equation}\label{Sect3_Preliminaries_Vorticity_Eq_3}
\begin{array}{ll}
\partial_t \mathcal{X}(t,x) = u^{\e}(t,\mathcal{X}(t,x)) = v^{\e}(t,\, \Phi^{-1} \circ \mathcal{X}), \quad \mathcal{X}(0,x) =\Phi(0,x).
\end{array}
\end{equation}
Define the Jacobian of the change of variables $J(t,x) = |\det\nabla \mathcal{X}(t,x)|$, then $J(t,x) = J(0,x):=J_0(x)$ due to the divergence free condition.
Denote $a_0 = |J_0(x)|^{\frac{1}{2}}$, define the matrix $(a_{ij}) = |J_0|^{\frac{1}{2}} P^{-1}$, where
the matrix $P$ satisfies $P_{ij} =\partial_i \mathcal{X}\cdot \partial_j \mathcal{X}$.

Define $W = e^{-\gamma t} \hat{\omega}_h(t,\, \Phi^{-1}\circ \mathcal{X})$, then $W$ satisfies the equation:
\begin{equation}\label{Sect3_Preliminaries_HeatEq_Damping}
\begin{array}{ll}
a_0\partial_t W - \e\partial_i(a_{ij}\partial_j W)
+ \big(\gamma a_0 - f^7[\nabla\varphi^{\e},\nabla\varphi,\partial_j v^{\e,i},\partial_j v^i]\big) W \\[9pt]

= \e\, e^{-\gamma t}\triangle^{\varphi^{\e}}\omega_h
+ e^{-\gamma t}\partial_z^{\varphi}\omega_h \partial_t^{\varphi^{\e}} \hat{\varphi}
+ e^{-\gamma t}\partial_z^{\varphi} \omega_h\, v^{\e}\cdot \nabla^{\varphi^{\e}}\hat{\varphi} \\[7pt]\quad\
- e^{-\gamma t}\hat{v}\cdot\nabla^{\varphi} \omega_h

+ e^{-\gamma t}f^8[\nabla\varphi^{\e},\nabla\varphi,\partial_j v^{\e,i},\partial_j v^i,\omega_h^{\e},\omega_h]\partial_j\hat{v}^i \\[7pt]\quad\
+ e^{-\gamma t}f^9[\nabla\varphi^{\e},\nabla\varphi,\partial_j v^{\e,i},\partial_j v^i,\omega_h^{\e},\omega_h]\nabla\hat{\varphi}

:=\mathcal{I}_2,
\end{array}
\end{equation}
where $\|\mathcal{I}_2\|_{L^{\infty}} \rto 0$ as $\e\rto 0$.

Since $a_0>0$, we can choose suitably large $\gamma>0$ such that
\begin{equation}\label{Sect3_Preliminaries_Vorticity_Eq_5}
\begin{array}{ll}
\gamma a_0 - f^7[\nabla\varphi^{\e},\nabla\varphi,\partial_j v^{\e,i},\partial_j v^i] >0,
\end{array}
\end{equation}
then $\big(\gamma a_0 - f^7[\nabla\varphi^{\e},\nabla\varphi,\partial_j v^{\e,i},\partial_j v^i]\big) W$ is a damping term.
Since the matrix $(a_{ij})$ is definitely positive, $- \e\partial_i(a_{ij}\partial_j W)$ is the diffusion term.

\subsection{$L^{\infty}$ Estimate of Strong Vorticity Layer}
In this subsection, we prove that if the initial vorticity layer is strong, then
the vorticity layer is strong.

Before proving our results, let us investigate the simplest model by using the heat kernel, that is the following heat equation with damping,
Dirichlet boundary condition and constant coefficients.
\begin{proposition}\label{Sect3_HeatEq_Diffusion_Proposition}
Assume $\|W\|_{X_{tan}^2} <+\infty$, $w^{ini} \neq 0$, $\gamma>0$ is constant, $W$ is the solution of the following heat equation with damping:
\begin{equation}\label{Sect3_HeatEq_Diffusion_1}
\left\{\begin{array}{ll}
\partial_t W - \e\triangle W + \gamma W =0, \\[5pt]

W|_{z=0} = 0, \\[5pt]

W|_{t=0} = w^{ini} \nrightarrow 0,
\end{array}\right.
\end{equation}
then $\lim\limits_{\e\rto 0}\|W\|_{L^{\infty}(\mathbb{R}^3_{-} \times (0,T])} \neq 0$.
\end{proposition}

\begin{proof}
Define $\tilde{W} = e^{\gamma t}W$, the equations $(\ref{Sect3_HeatEq_Diffusion_1})$ are rewritten as
\begin{equation}\label{Sect3_HeatEq_Diffusion_2}
\left\{\begin{array}{ll}
\partial_t \tilde{W} - \e\triangle \tilde{W} = 0, \\[5pt]

\tilde{W}|_{z=0} = 0, \\[5pt]

\tilde{W}|_{t=0} = w^{ini} \neq 0,
\end{array}\right.
\end{equation}

Note that $\e\triangle \tilde{W} \nrightarrow 0$ as $\e\rto 0$. Otherwise, we have $\partial_t \tilde{W}=0$,
then $\tilde{W} = w^{ini}(y,\frac{z}{\sqrt{\e}})$. However, $\e\partial_{zz} \tilde{W} = \partial_{zz} w^{ini}(y,\frac{z}{\sqrt{\e}}) \neq 0$.
This is a contradiction.

$\e\rto 0$ implies $\e t\rto 0$,
then the limit of the solution $\tilde{W}$ satisfies
\begin{equation}\label{Sect3_HeatEq_Diffusion_3}
\begin{array}{ll}
\tilde{W}(t,x) = \frac{1}{\sqrt{4\pi \e t}} \int\limits_{\mathbb{R}^3_{-}} w^{ini}(y)
\big(\exp\{-\frac{|x-y|^2}{4\e t}\} -\exp\{-\frac{|x+y|^2}{4\e t}\}\big) \,\mathrm{d}y \\[8pt]\hspace{1.22cm}

\rto w^{ini}(x), \text{\ as\ } \e t\rto 0.
\end{array}
\end{equation}
The convergence $(\ref{Sect3_HeatEq_Diffusion_3})$ is strong. Because $\lim\limits_{\e\rto 0}\|w^{ini}(x)\|_{L^{\infty}(\mathbb{R}^3_{-})} \neq 0$,
then we have $\lim\limits_{\e\rto 0}\|W\|_{L^{\infty}(\mathbb{R}^3_{-} \times (0,T])} \neq 0$.
Thus, Proposition $\ref{Sect3_HeatEq_Diffusion_Proposition}$ is proved.
\end{proof}

In order to prove that the strong initial vorticity layer is one of sufficient conditions for the
existence of strong vorticity layer, we assume that the Euler boundary data satisfies
$\Pi\mathcal{S}^{\varphi}v\nn|_{z=0} =0$.
\begin{theorem}\label{Sect3_BoundarLayer_Initial_Thm}
Assume $\omega^{\e},v^{\e}$ are the vorticity, velocity of Navier-Stokes equations $(\ref{Sect1_NS_Eq})$,
$\omega,v,\nn$ are the vorticity, velocity, normal vector of Euler equations $(\ref{Sect1_Euler_Eq})$.
If the initial Navier-Stokes velocity satisfies $\lim\limits_{\e\rto 0}(\nabla^{\varphi^{\e}}\times v_0^{\e})
- \nabla^{\varphi}\times\lim\limits_{\e\rto 0} v_0^{\e} \neq 0$ in the initial set $\mathcal{A}_0$,
the Euler solution satisfies $\Pi\mathcal{S}^{\varphi} v\nn|_{z=0} = 0$ in $[0,T]$,
then $\lim\limits_{\e\rto 0}\|\omega^{\e} -\omega\|_{L^{\infty}(\mathcal{X}(\mathcal{A}_0)\times (0,T])} \neq 0$.
\end{theorem}

\begin{proof}
Since $\Pi\mathcal{S}^{\varphi} v\nn|_{z=0} = 0$ in $[0,T]$, $|\omega_0^{\e}|_{z=0} -\omega_0|_{z=0}|_{\infty} \rto 0$ as $\e\rto 0$.
then there exist a set $\mathcal{A}_0\cap \{x|z<0\} \neq \emptyset$ such that
$\lim\limits_{\e\rto 0}(\nabla^{\varphi^{\e}}\times v_0^{\e})
- \nabla^{\varphi}\times\lim\limits_{\e\rto 0} v_0^{\e} \neq 0$ in the initial set $\mathcal{A}_0$.

We study the equations $(\ref{Sect3_Vorticity_Eq})$ in the Lagrangian coordinates:
\begin{equation}\label{Sect3_BoundarLayer_Initial_Eq_1}
\left\{\begin{array}{ll}
a_0\partial_t W - \e\partial_i(a_{ij}\partial_j W)
+ \big(\gamma a_0 - f^7[\nabla\varphi^{\e},\nabla\varphi,\partial_j v^{\e,i},\partial_j v^i]\big) W =\mathcal{I}_2, \\[8pt]

W|_{z=0} = \hat{\omega}_h^b := \omega_0^{\e}|_{z=0} -\omega_0|_{z=0} \rto 0, \\[7pt]

W|_{t=0} = \hat{\omega}_{h,0} \nrightarrow 0.
\end{array}\right.
\end{equation}

We decompose $W = W^{f\!o} + W^{bdy} + W^{ini}$, such that $W^{f\!o}$ satisfies the nonhomogeneous equations:
\begin{equation}\label{Sect3_BoundarLayer_Initial_Eq_Force}
\left\{\begin{array}{ll}
a_0\partial_t W^{f\!o} - \e\partial_i(a_{ij}\partial_j W^{f\!o})
+ \big(\gamma a_0 - f^7[\nabla\varphi^{\e},\nabla\varphi,\partial_j v^{\e,i},\partial_j v^i]\big) W^{f\!o} =\mathcal{I}_2, \\[8pt]

W^{f\!o}|_{z=0} = 0, \\[7pt]

W^{f\!o}|_{t=0} = 0,
\end{array}\right.
\end{equation}
$W^{bdy}$ satisfies the following equations:
\begin{equation}\label{Sect3_BoundarLayer_Initial_Eq_Boundary}
\left\{\begin{array}{ll}
a_0\partial_t W^{bdy} - \e\partial_i(a_{ij}\partial_j W^{bdy})
+ \big(\gamma a_0 - f^7[\nabla\varphi^{\e},\nabla\varphi,\partial_j v^{\e,i},\partial_j v^i]\big) W^{bdy} =0,\\[8pt]

W^{bdy}|_{z=0} = \hat{\omega}_h^b, \\[8pt]

W^{bdy}|_{t=0} = 0,
\end{array}\right.
\end{equation}
and $W^{ini}$ satisfies the homogeneous equations:
\begin{equation}\label{Sect3_BoundarLayer_Initial_Eq_Initial}
\left\{\begin{array}{ll}
a_0\partial_t W^{ini} - \e\partial_i(a_{ij}\partial_j W^{ini})
+ \big(\gamma a_0 - f^7[\nabla\varphi^{\e},\nabla\varphi,\partial_j v^{\e,i},\partial_j v^i]\big) W^{ini} =0,\\[8pt]

W^{ini}|_{z=0} = 0, \\[8pt]

W^{ini}|_{t=0} = \hat{\omega}_{h,0},
\end{array}\right.
\end{equation}
where $\lim\limits_{\e\rto 0}\|\hat{\omega}_{h,0}\|_{L^{\infty}(\mathcal{A}_0)} \neq 0$.

Note the diffusion term and damping term in $(\ref{Sect3_BoundarLayer_Initial_Eq_Force})$,
it is easy to use the maximal principle to prove $\|W^{f\!o}\|_{L^{\infty}} \rto 0$ as $\e\rto 0$, due to the force term that vanishes when $\e\rto 0$.

For $(\ref{Sect3_BoundarLayer_Initial_Eq_Boundary})$, we define
\begin{equation}\label{Sect3_BoundarLayer_Initial_Eq_Boundary_1}
\begin{array}{ll}
\phi = W^{bdy} - \big(\textsf{F}^{1,2} [\nabla\varphi^{\e}](\partial_j v^{\e,i}) - \textsf{F}^{1,2} [\nabla\varphi](\partial_j v^i)\big),
\end{array}
\end{equation}
then $\phi$ satisfies the following equations:
\begin{equation}\label{Sect3_BoundarLayer_Initial_Eq_Boundary_2}
\left\{\begin{array}{ll}
a_0\partial_t W^{bdy} - \e\partial_i(a_{ij}\partial_j W^{bdy})
+ \big(\gamma a_0 - f^7[\nabla\varphi^{\e},\nabla\varphi,\partial_j v^{\e,i},\partial_j v^i]\big) W^{bdy} \\[7pt]\quad
= - a_0\partial_t \big(\textsf{F}^{1,2} [\nabla\varphi^{\e}](\partial_j v^{\e,i}) - \textsf{F}^{1,2} [\nabla\varphi](\partial_j v^i)\big)
\\[7pt]\qquad
+ \e\partial_i \big[a_{ij}\partial_j \big(\textsf{F}^{1,2}
[\nabla\varphi^{\e}](\partial_j v^{\e,i}) - \textsf{F}^{1,2} [\nabla\varphi](\partial_j v^i)\big)\big]
\\[7pt]\qquad
- \big(\gamma a_0 - f^7[\nabla\varphi^{\e},\nabla\varphi,\partial_j v^{\e,i},\partial_j v^i]\big)
\big(\textsf{F}^{1,2} [\nabla\varphi^{\e}](\partial_j v^{\e,i}) - \textsf{F}^{1,2} [\nabla\varphi](\partial_j v^i)\big),\\[8pt]

\phi|_{z=0} = 0, \\[8pt]

\phi|_{t=0} = - \textsf{F}^{1,2} [\nabla\varphi^{\e}](\partial_j v^{\e,i})|_{t=0} + \textsf{F}^{1,2} [\nabla\varphi](\partial_j v^i)|_{t=0}.
\end{array}\right.
\end{equation}

It is easy to prove $\|\phi\|_{L^{\infty}} \rto 0$ as $\e\rto 0$, due to the force term and the initial data $\phi|_{t=0}$ that vanish when $\e\rto 0$.
Thus, it follows from $(\ref{Sect3_BoundarLayer_Initial_Eq_Boundary_1})$ that $\|W^{bdy}\|_{L^{\infty}} \rto 0$ as $\e\rto 0$.

Next, we study the equations $(\ref{Sect3_BoundarLayer_Initial_Eq_Initial})$ which are already expressed in the Lagrangian coordinates.
In order to prove $\lim\limits_{\e\rto 0}\|\omega^{\e} -\omega\|_{L^{\infty}(\mathcal{X}(\mathcal{A}_0)\times (0,T])} \neq 0$,
we have to prove $\lim\limits_{\e\rto 0}\|W^{ini}\|_{L^{\infty}(\mathcal{A}_0\times (0,T])} \neq 0$.

By defining the variable
\begin{equation}\label{Sect3_BoundarLayer_Initial_Eq_Initial_1}
\begin{array}{ll}
\tilde{W}^{ini} = W^{ini} \exp\{\int_0^t \frac{1}{a_0}
(\gamma a_0 - f^7[\nabla\varphi^{\e},\nabla\varphi,\partial_j v^{\e,i},\partial_j v^i]) \,\mathrm{d}t\},
\end{array}
\end{equation}
$(\ref{Sect3_BoundarLayer_Initial_Eq_Initial})$ can be rewritten as
\begin{equation}\label{Sect3_BoundarLayer_Initial_Eq_Initial_2}
\left\{\begin{array}{ll}
\partial_t \tilde{W}^{ini} - \frac{\sqrt{\e}}{a_0} \exp\{-\int_0^t \frac{1}{a_0}
(\gamma a_0 - f^7[\cdots]) \,\mathrm{d}t\} \cdot a_{ij} \partial_{ij} \tilde{W}^{ini}
 = \sqrt{\e}\,\mathcal{I}_3,\\[8pt]

\tilde{W}|_{z=0} = 0, \\[8pt]

\tilde{W}|_{t=0} = \hat{\omega}_{h,0},
\end{array}\right.
\end{equation}
where
\begin{equation}\label{Sect3_BoundarLayer_Initial_Eq_Initial_3}
\begin{array}{ll}
\mathcal{I}_3 = \frac{\sqrt{\e}}{a_0}\sum\limits_{i,j =1}^3\partial_i a_{ij} \cdot \partial_j \big(\tilde{W}^{ini} \exp\{-\int_0^t \frac{1}{a_0}
(\gamma a_0 - f^7[\cdots]) \,\mathrm{d}t\} \big)
\\[11pt]\hspace{0.85cm}

+ \frac{\sqrt{\e}}{a_0}\sum\limits_{i,j=1}^3 a_{ij} \partial_i \tilde{W}^{ini} \cdot \partial_j \big(\exp\{-\int_0^t \frac{1}{a_0}
(\gamma a_0 - f^7[\cdots]) \,\mathrm{d}t\}\big).
\end{array}
\end{equation}
Note that $\|\mathcal{I}_3\|_{L^{\infty}} <+\infty$, because $\mathcal{I}_3$ contains normal differential operator $\partial_z$ of order at most one.

$(\ref{Sect3_BoundarLayer_Initial_Eq_Initial_2})$ is uniformly parabolic, which has fundament solution satisfying the parabolic scaling.
Let $\textsf{H}(\frac{x}{\sqrt{\e t}})$ to denote the fundament solution of the following homogeneous parabolic equation in $\mathbb{R}^3$:
\begin{equation}\label{Sect3_BoundarLayer_Initial_Eq_Initial_4}
\begin{array}{ll}
\partial_t f - \frac{\e}{a_0} a_{ij} \exp\{-\int_0^t \frac{1}{a_0}
(\gamma a_0 - f^7[\cdots]) \,\mathrm{d}t\} \cdot \partial_{ij} f =0,
\end{array}
\end{equation}
by using the fundament solution $\textsf{H}$, the equations $(\ref{Sect3_BoundarLayer_Initial_Eq_Initial_2})$ have the explicit formula
by using Duhamel's principle,
\begin{equation}\label{Sect3_BoundarLayer_Initial_Eq_Initial_5}
\begin{array}{ll}
\tilde{W}^{ini}(t,x) = \int\limits_{\mathbb{R}^3_{-}} \hat{\omega}_{h,0}(y) (\textsf{H}(\frac{x-y}{\sqrt{\e t}})
- \textsf{H}(\frac{x+y}{\sqrt{\e t}})) \,\mathrm{d}y \\[7pt]\hspace{2cm}
+ \int\limits_0^t
\int\limits_{\mathbb{R}^3_{-}} \sqrt{\e}\,\mathcal{I}_3(t- s,y) (\textsf{H}(\frac{x-y}{\sqrt{\e s}}) - \textsf{H}(\frac{x+y}{\sqrt{\e s}})) \,\mathrm{d}y
\mathrm{d}s, \\[-0.3cm]
\end{array}
\end{equation}
then $\tilde{W}^{ini}(t,x)\rto \hat{\omega}_{h,0} + \sqrt{\e} O(1) \rto \hat{\omega}_{h,0}$ pointwisely, as $\e\rto 0$,
$\lim\limits_{\e\rto 0} \tilde{W}^{ini}(t,x)$ and
$\lim\limits_{\e\rto 0} \hat{\omega}_{h,0}(x)$ have the same support. The limit of the solution is equal to that of the initial data
in Lagrangian coordinates, namely the limit of the vorticity is also transported by the velocity field in Eulerian coordinates.

By $(\ref{Sect3_BoundarLayer_Initial_Eq_Initial_1})$, we have
\begin{equation}\label{Sect3_BoundarLayer_Initial_Eq_Initial_6}
\begin{array}{ll}
\lim\limits_{\e\rto 0}\|W^{ini}\|_{L^{\infty}(\mathcal{A}_0)}
= \lim\limits_{\e\rto 0}\|\tilde{W}^{ini}\exp\{- \int_0^t \frac{1}{a_0}
(\gamma a_0 - f^7[\cdots]) \,\mathrm{d}t\}\|_{L^{\infty}(\mathcal{A}_0)} \neq 0.
\end{array}
\end{equation}
So $\lim\limits_{\e\rto 0}\|\hat{\omega}_h\|_{L^{\infty}(\mathcal{X}(\mathcal{A}_0)\times (0,T])} \neq 0$.
Thus, Theorem $\ref{Sect3_BoundarLayer_Initial_Thm}$ is proved.
\end{proof}

It is easy to show that the strong vorticity layer implies the strong boundary layers of the following variables:
\begin{equation}\label{Sect3_StrongVorticityLayer_Corollary}
\begin{array}{ll}
\lim\limits_{\e\rto 0}\|\partial_z v^{\e} -\partial_z v\|_{L^{\infty}(\mathcal{X}(\mathcal{A}_0)\times (0,T])} \neq 0, \\[9pt]
\lim\limits_{\e\rto 0}\|\mathcal{S}v^{\e} -\mathcal{S}v\|_{L^{\infty}(\mathcal{X}(\mathcal{A}_0)\times (0,T])} \neq 0,\\[9pt]
\lim\limits_{\e\rto 0}\|\nabla q^{\e} - \nabla q\|_{L^{\infty}(\mathcal{X}(\mathcal{A}_0)\times (0,T])} \neq 0.
\end{array}
\end{equation}

%%% find 4
\section{Strong Vorticity Layer Caused by the Discrepancy between Boundary Values of Vorticities}
In this section, we study the strong vorticity layer for the free surface N-S equations $(\ref{Sect1_NS_Eq})$,
which arises from the discrepancy between boundary value of Navier-Stokes vorticity and boundary value of Euler vorticity.

\subsection{Discrepancy of the Vorticity on the Free Boundary}
The following lemma shows that if the tangential projection on the free surface of the Euler strain tensor multiply by normal vector
does not vanishes, then there is a discrepancy between Navier-Stokes vorticity and Euler vorticity.
\begin{lemma}\label{Sect4_Vorticity_Discrepancy_Lemma}
Assume $\omega^{\e},v^{\e},\NN^{\e}$ are the vorticity, velocity, normal vector of Navier-Stokes equations $(\ref{Sect1_NS_Eq})$,
$\omega,v,\NN$ are the vorticity, velocity, normal vector of Euler equations $(\ref{Sect1_Euler_Eq})$,
$\omega^{\e,b}$ and $\omega^b$ are boundary values of $\omega^{\e},\omega$ respectively.
If $\Pi\mathcal{S}^{\varphi} v\nn|_{z=0} \neq 0$ in $(0,T]$, then $\lim\limits_{\e\rto 0}|\omega^{\e,b}-\omega^b|
_{L^{\infty}(\mathbb{R}^2\times (0,T])} \neq 0$.
\end{lemma}

\begin{proof}
We denote $\textsf{S}_n =\Pi \mathcal{S}^{\varphi} v \nn$. Since $\textsf{S}_n \neq 0$,
$\mathcal{S}^{\varphi}v\nn = (\mathcal{S}^{\varphi}v \nn\cdot\nn)\nn + \Pi\mathcal{S}^{\varphi} v\nn$,
then $\mathcal{S}^{\varphi}v \nn$ is not parallel to $\nn$, namely,
\begin{equation}\label{Sect4_Vorticity_Discrepancy_1}
\begin{array}{ll}
\mathcal{S}^{\varphi}v \nn\times\nn = (\mathcal{S}^{\varphi}v \nn\cdot\nn)\nn \times\nn + \Pi\mathcal{S}^{\varphi} v\nn \times\nn
= \Pi\mathcal{S}^{\varphi} v\nn \times\nn \neq 0.
\end{array}
\end{equation}
Denote $\Pi\mathcal{S}^{\varphi} v\nn \times\nn := (\Theta^1,\Theta^2,\Theta^3)^{\top}$, which is a nonzero vector.

By $\mathcal{S}^{\varphi}v \nn\times\nn = (\Theta^1,\Theta^2,\Theta^3)^{\top}$, similar to $(\ref{Sect2_Vorticity_H_BC_2})$, we have
\begin{equation}\label{Sect4_Vorticity_Discrepancy_2}
\left\{\begin{array}{ll}
n^3[n^1\partial_1^{\varphi} v^1 + \frac{n^2}{2}(\partial_1^{\varphi} v^2 + \partial_2^{\varphi} v^1)
+ \frac{n^3}{2}(\partial_1^{\varphi} v^3 + \partial_z^{\varphi} v^1)] \\[5pt]\quad

= n^1[\frac{n^1}{2}(\partial_1^{\varphi} v^3 + \partial_z^{\varphi} v^1) + \frac{n^2}{2}(\partial_2^{\varphi} v^3 + \partial_z^{\varphi} v^2)
- n^3\partial_1^{\varphi} v^1 - n^3\partial_2^{\varphi} v^2] -\Theta^2,

\\[10pt]

n^3[\frac{n^1}{2}(\partial_1^{\varphi} v^2 + \partial_2^{\varphi} v^1) + n^2\partial_2^{\varphi} v^2
+ \frac{n^3}{2}(\partial_2^{\varphi} v^3 + \partial_z^{\varphi} v^2)] \\[5pt]\quad

= n^2[\frac{n^1}{2}(\partial_1^{\varphi} v^3 + \partial_z^{\varphi} v^1) + \frac{n^2}{2}(\partial_2^{\varphi} v^3 + \partial_z^{\varphi} v^2)
- n^3\partial_1^{\varphi} v^1 - n^3\partial_2^{\varphi} v^2] +\Theta^1,

\\[10pt]

n^2[n^1\partial_1^{\varphi} v^1 + \frac{n^2}{2}(\partial_1^{\varphi} v^2 + \partial_2^{\varphi} v^1)
+ \frac{n^3}{2}(\partial_1^{\varphi} v^3 + \partial_z^{\varphi} v^1)] \\[5pt]\quad

= n^1[\frac{n^1}{2}(\partial_1^{\varphi} v^2 + \partial_2^{\varphi} v^1) + n^2\partial_2^{\varphi} v^2
+ \frac{n^3}{2}(\partial_2^{\varphi} v^3 + \partial_z^{\varphi} v^2)] + \Theta^3.
\end{array}\right.
\end{equation}

Then we have the following two equations involving $\partial_z v^1$ and $\partial_z v^2$:
\begin{equation}\label{Sect4_Vorticity_Discrepancy_3}
\begin{array}{ll}
\big[(n^1)^2 + \frac{(n^3)^2}{2} + \frac{1}{2}\frac{(n^1)^4}{(n^3)^2} - \frac{1}{2}\frac{(n^2)^4}{(n^3)^2}
\big]\partial_z v^1
+ \big[n^1n^2 + \frac{(n^1)^3n^2}{(n^3)^2} + \frac{n^1(n^2)^3}{(n^3)^2} \big]\partial_z v^2 \\[10pt]

= -[\frac{(n^3)^2}{2} - \frac{(n^1)^2}{2} - \frac{(n^2)^2}{2}]\big[\partial_z\varphi\partial_1 v^3 - \partial_1\varphi
[- \partial_z\varphi(\partial_1 v^1 + \partial_2 v^2)] \big]
\\[5pt]\quad

+ (\frac{n^1(n^2)^2}{n^3} - 2n^1n^3)(\partial_z\varphi\partial_1 v^1)
- (n^1n^3 + \frac{n^1(n^2)^2}{n^3})(\partial_z\varphi\partial_2 v^2) \\[5pt]\quad

+ [\frac{(n^2)^2}{n^3}\frac{n^2}{2}- \frac{n^1n^2}{n^3}\frac{n^1}{2} -\frac{n^2n^3}{2}]
(\partial_z\varphi\partial_1 v^2 + \partial_z\varphi\partial_2 v^1)-\partial_z\varphi\Theta^2 - \frac{n^2}{n^3}\partial_z\varphi\Theta^3, \\[15pt]

\big[n^1n^2 + \frac{(n^1)^3n^2}{(n^3)^2} + \frac{n^1(n^2)^3}{(n^3)^2} \big]\partial_z v^1
+ \big[(n^2)^2 + \frac{1}{2}(n^3)^2 + \frac{(n^2)^4}{2(n^3)^2} - \frac{(n^1)^4}{2(n^3)^2}\big]\partial_z v^2 \\[10pt]

= -[\frac{(n^3)^2}{2} - \frac{(n^2)^2}{2} - \frac{(n^1)^2}{2}]\big[\partial_z\varphi\partial_2 v^3 - {\partial_z\varphi}
[- \partial_z\varphi(\partial_1 v^1 + \partial_2 v^2)]\big] \\[5pt]\quad

- (n^2n^3 + \frac{(n^1)^2n^2}{n^3})(\partial_z\varphi\partial_1 v^1)
+ (\frac{(n^1)^2n^2}{n^3} -2n^2n^3)(\partial_z\varphi\partial_2 v^2) \\[5pt]\quad

+ (\frac{(n^1)^2}{n^3}\frac{n^1}{2} -\frac{n^1n^3}{2} - \frac{n^1n^2}{n^3}\frac{n^2}{2})
(\partial_z\varphi\partial_1 v^2 + \partial_z\varphi\partial_2 v^1)+\partial_z\varphi\Theta^1 + \frac{n^1}{n^3}\partial_z\varphi\Theta^3.
\end{array}
\end{equation}

When $|\nabla h|_{\infty}$ is suitably small, the coefficient matrix of $(\partial_z v^1, \partial_z v^2)^{\top}$ is nondegenerate, then we solve
\begin{equation}\label{Sect4_Vorticity_Discrepancy_4}
\left\{\begin{array}{ll}
\partial_z v^1 = f^5[\nabla\varphi](\partial_j v^i)
-\textsf{M}^{11}(\partial_z\varphi\Theta^2 + \frac{n^2}{n^3}\partial_z\varphi\Theta^3)
+\textsf{M}^{12}(\partial_z\varphi\Theta^1 + \frac{n^1}{n^3}\partial_z\varphi\Theta^3),
\\[6pt]\hspace{1.2cm} j=1,2,\ i=1,2,3,\\[8pt]

\partial_z v^2 = f^6[\nabla\varphi](\partial_j v^i)
-\textsf{M}^{21}(\partial_z\varphi\Theta^2 + \frac{n^2}{n^3}\partial_z\varphi\Theta^3)
+\textsf{M}^{22}(\partial_z\varphi\Theta^1 + \frac{n^1}{n^3}\partial_z\varphi\Theta^3),
\\[6pt]\hspace{1.2cm} j=1,2,\ i=1,2,3,
\end{array}\right.
\end{equation}
where the matrix $\textsf{M} = (\textsf{M}_{ij})$ is defined in $(\ref{Sect2_Vorticity_H_BC_9})$, $(\textsf{M}^{ij}) = (\textsf{M}_{ij})^{-1}$.

By $(\ref{Sect2_NormalDer_Vorticity_Estimate_2})$ and $(\ref{Sect4_Vorticity_Discrepancy_4})$, we have the boundary values of $\omega_h = (\omega^1,\omega^2)$:
\begin{equation}\label{Sect4_Vorticity_Discrepancy_5}
\begin{array}{ll}
\omega^1 = - \frac{\partial_1\varphi\partial_2\varphi}{\partial_z\varphi}\partial_z v^1
- \frac{1 + (\partial_2\varphi)^2}{\partial_z\varphi}\partial_z v^2
+ \partial_2 v^3 + \partial_2\varphi(\partial_1 v^1 + \partial_2 v^2) \\[8pt]\hspace{0.47cm}

:= \textsf{F}^1 [\nabla\varphi](\partial_j v^i) + \varsigma_1\Theta^1 + \varsigma_2\Theta^2 + \varsigma_3\Theta^3, \\[9pt]

\omega^2 = \frac{1+(\partial_1\varphi)^2}{\partial_z\varphi}\partial_z v^1
+ \frac{\partial_1\varphi\partial_2\varphi}{\partial_z\varphi}\partial_z v^2
- \partial_1 v^3 - \partial_1\varphi(\partial_1 v^1 + \partial_2 v^2) \\[8pt]\hspace{0.47cm}

:= \textsf{F}^2 [\nabla\varphi](\partial_j v^i) + \varsigma_4\Theta^1 + \varsigma_5\Theta^2 + \varsigma_6\Theta^3,
\end{array}
\end{equation}
where the coefficients $\varsigma_i$ are as follows:
\begin{equation}\label{Sect4_Vorticity_Discrepancy_6}
\begin{array}{ll}
\varsigma_1 =\partial_z\varphi[\partial_1\varphi\partial_2\varphi\textsf{M}^{12}
+ (1+(\partial_2\varphi)^2)\textsf{M}^{22}], \\[8pt]
\varsigma_2 =\partial_1\varphi\partial_2\varphi\textsf{M}^{11} + (1 + (\partial_2\varphi)^2)\textsf{M}^{21}, \\[8pt]

\varsigma_3 =\big[(1+(\partial_1\varphi)^2)
\big(-\textsf{M}^{11}\frac{n^2}{n^3} +\textsf{M}^{12}\frac{n^1}{n^3}\partial_z\varphi\big) \\[7pt]\hspace{0.85cm}
+ \partial_1\varphi\partial_2\varphi
\big(-\textsf{M}^{21}\frac{n^2}{n^3}
+\textsf{M}^{22}\frac{n^1}{n^3}\partial_z\varphi\big)\big], \\[9pt]

\varsigma_4 = \partial_z\varphi[(1+(\partial_1\varphi)^2)\textsf{M}^{12} + \partial_1\varphi\partial_2\varphi\textsf{M}^{22}], \\[8pt]
\varsigma_5 = -(1+(\partial_1\varphi)^2)\textsf{M}^{11} - \partial_1\varphi\partial_2\varphi \textsf{M}^{21}, \\[8pt]

\varsigma_6 = (1+(\partial_1\varphi)^2)
\big(-\textsf{M}^{11}\frac{n^2}{n^3}
+\textsf{M}^{12}\frac{n^1}{n^3}\partial_z\varphi\big) \\[7pt]\hspace{0.85cm}
+ \partial_1\varphi\partial_2\varphi
\big(-\textsf{M}^{21}\frac{n^2}{n^3}
+\textsf{M}^{22}\frac{n^1}{n^3}\partial_z\varphi\big).
\end{array}
\end{equation}

If $|\varsigma_1\Theta^1 + \varsigma_2\Theta^2 + \varsigma_3\Theta^3|_{\infty}
= |\varsigma_4\Theta^1 + \varsigma_5\Theta^2 + \varsigma_6\Theta^3|_{\infty} = 0$, then
\begin{equation}\label{Sect4_Vorticity_Discrepancy_7}
\left\{\begin{array}{ll}
\partial_z v^1 = f^5[\nabla\varphi](\partial_j v^i),\ j=1,2,\ i=1,2,3,\\[8pt]

\partial_z v^2 = f^6[\nabla\varphi](\partial_j v^i),\ j=1,2,\ i=1,2,3.
\end{array}\right.
\end{equation}
Since the proof of Lemma $\ref{Sect2_Vorticity_H_Eq_BC_Lemma}$ is revertible,
$(\ref{Sect4_Vorticity_Discrepancy_7})$ implies $\mathcal{S}^{\varphi}v \nn\times\nn =0$. This strongly contradicts with $(\ref{Sect4_Vorticity_Discrepancy_1})$.

Thus, either $|\varsigma_1\Theta^1 + \varsigma_2\Theta^2 + \varsigma_3\Theta^3|_{\infty}\neq 0$
or $|\varsigma_4\Theta^1 + \varsigma_5\Theta^2 + \varsigma_6\Theta^3|_{\infty}\neq 0$.
Without lose of generality, we assume the former holds.

As $\e\rto 0$, $|\textsf{F}^1 [\nabla\varphi^{\e}](\partial_j v^{\e,i}) - \textsf{F}^1 [\nabla\varphi](\partial_j v^i)|_{L^{\infty}}\rto 0$,
this convergence is strong due to enough uniform regularities in conormal Sobolev space of Navier-Stokes solutions and its tangential derivatives
(see \cite{Masmoudi_Rousset_2012_FreeBC}).
Thus, when $\e$ is sufficiently small,
$|\textsf{F}^1 [\nabla\varphi^{\e}](\partial_j v^{\e,i}) - \textsf{F}^1 [\nabla\varphi](\partial_j v^i)|_{\infty}
\leq \frac{1}{2}|\varsigma_1\Theta^1 + \varsigma_2\Theta^2 + \varsigma_3\Theta^3|_{\infty}$. It follows from $(\ref{Sect4_Vorticity_Discrepancy_5})$ that
\begin{equation}\label{Sect4_Vorticity_Discrepancy_8}
\begin{array}{ll}
|\omega^{\e,1} -\omega^1|_{\infty} \geq |\varsigma_1\Theta^1 + \varsigma_2\Theta^2 + \varsigma_3\Theta^3|_{\infty}
- \big|\textsf{F}^1 [\nabla\varphi^{\e}](\partial_j v^{\e,i}) - \textsf{F}^1 [\nabla\varphi](\partial_j v^i)\big|_{\infty} \\[8pt]\hspace{2cm}

\geq \frac{1}{2}|\varsigma_1\Theta^1 + \varsigma_2\Theta^2 + \varsigma_3\Theta^3|_{\infty}.
\end{array}
\end{equation}
Then
\begin{equation}\label{Sect4_Vorticity_Discrepancy_9}
\begin{array}{ll}
|\omega_h^{\e,b} -\omega_h^b|_{L^{\infty}(\mathbb{R}^2\times (0,T])}
\geq \max\{ |\omega^{\e,1} -\omega^1|_{\infty} ,\, |\omega^{\e,2} -\omega^2|_{\infty}\}
\\[8pt]

\geq \frac{1}{2}\max\{|\varsigma_1\Theta^1 + \varsigma_2\Theta^2 + \varsigma_3\Theta^3|_{\infty}, \,
|\varsigma_4\Theta^1 + \varsigma_5\Theta^2 + \varsigma_6\Theta^3|_{\infty}\} >0.
\end{array}
\end{equation}

Thus, Lemma $\ref{Sect4_Vorticity_Discrepancy_Lemma}$ is proved.
\end{proof}

\subsection{$L^{\infty}$ Estimate of Strong Vorticity Layer}
In this subsection, we prove the existence of strong vorticity layer when the Euler boundary data satisfies
$\Pi\mathcal{S}^{\varphi} v\nn|_{z=0} \neq 0$ in $(0,T]$.

Before proving our results, let us investigate the simplest model, that is the following heat equation with damping and constant coefficients.
\begin{proposition}\label{Sect4_HeatEq_Diffusion_Proposition}
Assume $w^b \in H^4(\mathbb{R}^2\times[0,T])$, $w^b \nrightarrow 0$, $\gamma>0$ is constant, $W$ is the solution of the following heat equation with damping:
\begin{equation}\label{Sect4_HeatEq_Diffusion_1}
\left\{\begin{array}{ll}
\partial_t W - \e\triangle W + \gamma W =0, \\[5pt]

W|_{z=0} = w^b \nrightarrow 0, \\[5pt]

W|_{t=0} = 0,
\end{array}\right.
\end{equation}
then $\lim\limits_{\e\rto 0}\|W\|_{L^{\infty}(\mathbb{R}^3_{-} \times (0,T])} \neq 0$.
\end{proposition}

\begin{proof}
We define the following Fourier transformation with respect to $(t,y)\in\mathbb{R}_{+}\times\mathbb{R}^2$,
\begin{equation}\label{Sect4_BoundarLayer_Boundary_Fourier}
\begin{array}{ll}
\mathcal{F}[W](\tau,\xi,z) = \int\limits_0^{+\infty}\int\limits_{\mathbb{R}^3_{-}}
e^{-i \tau t - i\xi\cdot y} W(t,y,z) \,\mathrm{d}t\mathrm{d}y,
\end{array}
\end{equation}
Note that $W|_{t=0} = 0$, there is no term involving $W|_{t=0}$ appears as a force term.

By applying Fourier transformation $(\ref{Sect4_BoundarLayer_Boundary_Fourier})$ to $(\ref{Sect4_HeatEq_Diffusion_1})$,
we get the second-order ordinary differential equation:
\begin{equation}\label{Sect4_HeatEq_Diffusion_2}
\begin{array}{ll}
i\tau \mathcal{F}[W] -\e\partial_{zz} \mathcal{F}[W] + \e|\xi|^2\mathcal{F}[W] + \gamma \mathcal{F}[W] =0, \\[9pt]

\partial_{zz} \mathcal{F}[W] - \frac{1}{\e}(i\tau + \e|\xi|^2 + \gamma)\mathcal{F}[W] =0, \\[9pt]

\mathcal{F}[W](\tau,\xi,z) = \exp\{(i\tau + \e|\xi|^2 + \gamma)^{\frac{1}{2}}\frac{z}{\sqrt{\e}}\} \mathcal{F}[w^b](\tau,\xi),
\end{array}
\end{equation}
where the complex root $(i\tau + \e|\xi|^2 + \gamma)^{\frac{1}{2}}$ has two branches, one of which always has a positive real part
due to $\e|\xi|^2 + \gamma>0$, then we choose this branch.

If $|z| =O(\e^{\frac{1}{2} -\delta_z})$ where $\delta_z > 0$, we simply assume $z = - \e^{\frac{1}{2} -\delta_z}$, then as $\e\rto 0$,
\begin{equation}\label{Sect4_HeatEq_Diffusion_3}
\begin{array}{ll}
|\exp\{(i\tau + \e|\xi|^2 + \gamma)^{\frac{1}{2}}\frac{z}{\sqrt{\e}}\}|
=\exp\{-(\e|\xi|^2 + \gamma)^{\frac{1}{2}} \e^{-\delta_z}\} \leq \exp\{-\gamma^{\frac{1}{2}} \e^{-\delta_z}\} \rto 0, \\[7pt]

\|\mathcal{F}[W](\tau,\xi,z)\|_{L^1(\mathrm{d}\tau\mathrm{d}\xi)}
= \|\exp\{-(\e|\xi|^2 + \gamma)^{\frac{1}{2}}\e^{-\delta_z}\}
\mathcal{F}[w^b](\tau,\xi)\|_{L^1 (\mathrm{d}\tau\mathrm{d}\xi)} \rto 0,
\end{array}
\end{equation}
note that $\mathcal{F}[w^b]\in L^1 (\mathrm{d}\tau\mathrm{d}\xi)$ requires $w^b \in H^4(\mathbb{R}^2\times [0,T])$,
then
\begin{equation}\label{Sect4_HeatEq_Diffusion_4}
\begin{array}{ll}
\lim\limits_{\e\rto 0} \|W\|_{L^{\infty}(t,y,z)} = \lim\limits_{\e\rto 0}\|\mathcal{F}^{-1}[\mathcal{F}[W]]\|_{L^{\infty}(t,y,z)} \\[7pt]
\lem \|\mathcal{F}[W](\tau,\xi,z)\|_{L^1(\mathrm{d}\tau\mathrm{d}\xi)} \rto 0.
\end{array}
\end{equation}

If $|z| =O(\e^{\frac{1}{2} +\delta_z})$ where $\delta_z \geq 0$, we simply assume $z = - \e^{\frac{1}{2} +\delta_z}$,
then as $\e\rto 0$,
\begin{equation}\label{Sect4_HeatEq_Diffusion_5}
\begin{array}{ll}
|\exp\{(i\tau + \e|\xi|^2 + \gamma)^{\frac{1}{2}}\frac{z}{\sqrt{\e}}\}|
=\exp\{-(\e|\xi|^2 + \gamma)^{\frac{1}{2}} \e^{\delta_z}\} \rto 1 \text{\ or\ } e^{-\sqrt{\gamma}},
\end{array}
\end{equation}
for any finite $\xi\in \mathbb{R}^2$,
then $\|\mathcal{F}[W](\tau,\xi,z)\|_{L^{\infty} (\mathrm{d}\tau\mathrm{d}\xi)} \neq 0$.

We use the proof by contradiction.
Assume that $\|W\|_{L^{\infty}(t,y,z)} = 0$, then $\|W\|_{L^1(t,y,z)} = 0$, and then $\|\mathcal{F}[W](\tau,\xi,z)\|_{L^{\infty}} = 0$, this is a contradiction.
Thus, if $z =O(\e^{\frac{1}{2} +\delta_z})$, $\lim\limits_{\e\rto 0}\|W\|_{L^{\infty}(t,y,z)} \neq 0$.
Thus, Proposition $\ref{Sect4_HeatEq_Diffusion_Proposition}$ is proved.
\end{proof}

However, the proof of the following theorem is much more complicated than Proposition $\ref{Sect4_HeatEq_Diffusion_Proposition}$,
since we have to use the symbolic analysis and paradifferential calculus for our problem.

In order to prove that the discrepancy of boundary values of vorticities is one of sufficient conditions for the
existence of strong vorticity layer, we assume that the initial vorticity layer is weak.
\begin{theorem}\label{Sect4_BoundarLayer_Boundary_Thm}
Assume the conditions are the same with Lemma $\ref{Sect4_Vorticity_Discrepancy_Lemma}$.
If the Euler solution satisfies $\Pi\mathcal{S}^{\varphi} v\nn|_{z=0} \neq 0$ in $(0,T]$,
and $\Pi\mathcal{S}^{\varphi} v\nn|_{z=0} \in H^4(\mathbb{R}^2\times[0,T])$,
the initial Navier-Stokes velocity satisfies $\lim\limits_{\e\rto 0}(\nabla^{\varphi^{\e}}\times v_0^{\e})
- \nabla^{\varphi}\times\lim\limits_{\e\rto 0} v_0^{\e} = 0$,
then $\lim\limits_{\e\rto 0}\|\omega^{\e} -\omega\|_{L^{\infty}(\mathbb{R}^2\times [0, O(\sqrt{\e}))\times (0,T])} \neq 0$.
\end{theorem}

\begin{proof}
Since the initial Navier-Stokes velocity satisfies $\lim\limits_{\e\rto 0}(\nabla\times u_0^{\e}) - \nabla\times\lim\limits_{\e\rto 0} u_0^{\e} = 0$,
$\lim\limits_{\e\rto 0}|\omega^{\e}_0 -\omega_0|_{L^{\infty}} = 0$, then $\Pi\mathcal{S}^{\varphi} v\nn|_{z=0,t=0} = 0$,
that does not contradict with  $\Pi\mathcal{S}^{\varphi} v\nn|_{z=0} \neq 0$ in $(0,T]$.

We study the equations $(\ref{Sect3_Vorticity_Eq})$ with small initial data:
\begin{equation}\label{Sect4_BoundarLayer_Boundary_Eq_1}
\left\{\begin{array}{ll}
a_0\partial_t W - \e\partial_i(a_{ij}\partial_j W)
+ \big(\gamma a_0 - f^7[\nabla\varphi^{\e},\nabla\varphi,\partial_j v^{\e,i},\partial_j v^i]\big) W =\mathcal{I}_2, \\[10pt]

W|_{z=0} = e^{-\gamma t}\hat{\omega}_h^b \nrightarrow 0, \\[9pt]

W|_{t=0} = \hat{\omega}_{h,0} \rto 0.
\end{array}\right.
\end{equation}

We decompose $W = W^{bdy} + W^{f\!o}$, such that $W^{f\!o}$ satisfies the nonhomogeneous equations:
\begin{equation}\label{Sect4_BoundarLayer_Boundary_Force}
\left\{\begin{array}{ll}
a_0\partial_t W^{f\!o} - \e\partial_i(a_{ij}\partial_j W^{f\!o})
+ \big(\gamma a_0 - f^7[\nabla\varphi^{\e},\nabla\varphi,\partial_j v^{\e,i},\partial_j v^i]\big) W^{f\!o} =\mathcal{I}_2, \\[10pt]

W^{f\!o}|_{z=0} = 0, \\[9pt]

W^{f\!o}|_{t=0} = \hat{\omega}_{h,0} \rto 0.
\end{array}\right.
\end{equation}
and $W^{bdy}$ satisfies the homogeneous equations:
\begin{equation}\label{Sect4_BoundarLayer_Boundary_BC}
\left\{\begin{array}{ll}
a_0\partial_t W^{bdy} - \e\partial_i(a_{ij}\partial_j W^{bdy})
+ \big(\gamma a_0 - f^7[\nabla\varphi^{\e},\nabla\varphi,\partial_j v^{\e,i},\partial_j v^i]\big) W^{bdy} =0, \\[10pt]

W^{bdy}|_{z=0} = e^{-\gamma t}\hat{\omega}_h^b \nrightarrow 0, \\[11pt]

W^{bdy}|_{t=0} = 0.
\end{array}\right.
\end{equation}

Note the diffusion term and damping term of $(\ref{Sect4_BoundarLayer_Boundary_Force})$, it is easy to prove that
$\|W^{f\!o}\|_{L^{\infty}} \leq \|W^{f\!o}|_{t=0}\|_{L^{\infty}}
+ \int\limits_0^t \|\mathcal{I}_2\|_{\infty}\,\mathrm{d}t \rto 0.$

Next, we study the homogeneous equations $(\ref{Sect4_BoundarLayer_Boundary_BC})$ with variable coefficients,
which differs from the equations $(\ref{Sect4_HeatEq_Diffusion_1})$ up to coefficients.
By using symbolic analysis, it is standard to prove that the limit of the solution of
$(\ref{Sect4_BoundarLayer_Boundary_BC})$ behaves similarly to that of $(\ref{Sect4_HeatEq_Diffusion_1})$.
However, we still show some keypoints.

We rewrite $(\ref{Sect4_BoundarLayer_Boundary_BC})$ in the following form:
\begin{equation}\label{Sect4_BoundarLayer_Boundary_BC_1}
\left\{\begin{array}{ll}
\e \partial_{zz} W^{bdy} + \e(\frac{\partial_z a_{33}}{a_{33}}+ \sum\limits_{j=1,2}\frac{\partial_j a_{j3}}{a_{33}})\partial_z W^{bdy}
+ 2\e \sum\limits_{j=1,2}\frac{a_{j3}}{a_{33}}\partial_{jz} W^{bdy} \\[11pt]\quad

+ \e \sum\limits_{j=1,2}\frac{\partial_z a_{j3}}{a_{33}}\partial_j W^{bdy} + \e \sum\limits_{j=1,2}\frac{\partial_i a_{ij}}{a_{33}}\partial_j W^{bdy}
+ \e \sum\limits_{j=1,2}\frac{a_{ij}}{a_{33}}\partial_{ij} W^{bdy} \\[13pt]\quad
- \frac{a_0}{a_{33}}\partial_t W^{bdy}
- \frac{1}{a_{33}}\big(\gamma a_0 - f^7[\nabla\varphi^{\e},\nabla\varphi,\partial_j v^{\e,i},\partial_j v^i]\big) W^{bdy} =0, \\[11pt]

W^{bdy}|_{z=0} = e^{-\gamma t}\hat{\omega}_h^b \nrightarrow 0, \\[8pt]

W^{bdy}|_{t=0} = 0.
\end{array}\right.
\end{equation}

Take $z$ as a parameter, then the symbolic version of $(\ref{Sect4_BoundarLayer_Boundary_BC_1})$ is
\begin{equation}\label{Sect4_BoundarLayer_Boundary_BC_2}
\left\{\begin{array}{ll}
\e \partial_{zz} \tilde{W}^{bdy} + A_1 \sqrt{\e} \partial_z \tilde{W}^{bdy} + A_0 \tilde{W}^{bdy} =0, \\[11pt]
\tilde{W}^{bdy}|_{z=0} = \mathcal{F}[e^{-\gamma t}\hat{\omega}_h^b] \nrightarrow 0, \\[8pt]
\tilde{W}^{bdy}|_{t=0} = 0.
\end{array}\right.
\end{equation}
where the Fourier multipliers are as follows:
\begin{equation}\label{Sect4_BoundarLayer_Boundary_BC_3}
\begin{array}{ll}
A_1 = \sqrt{\e}(\frac{\partial_z a_{33}}{a_{33}}+ \sum\limits_{j=1,2}\frac{\partial_j a_{j3}}{a_{33}}
+ 2i \sum\limits_{j=1,2}\frac{a_{j3}}{a_{33}} \xi^j) \\[14pt]

A_0 = i\e \sum\limits_{j=1,2}\frac{\partial_z a_{j3}}{a_{33}}\xi^j + i\e \sum\limits_{j=1,2}\frac{\partial_i a_{ij}}{a_{33}}\xi^j
- \e \sum\limits_{j=1,2}\frac{a_{ij}}{a_{33}}\xi^i\xi^j
- i\tau \frac{a_0}{a_{33}} \\[12pt]\hspace{0.87cm}
- \frac{1}{a_{33}}\big(\gamma a_0 - f^7[\nabla\varphi^{\e},\nabla\varphi,\partial_j v^{\e,i},\partial_j v^i]\big),
\end{array}
\end{equation}

Due to $|a_0| + |a_{ij}| + \sqrt{\e}|\partial_z a_{ij}| \leq C$ for some $C>0$ (see \cite{Masmoudi_Rousset_2012_FreeBC}), when $\e\rto 0$,
\begin{equation*}
\begin{array}{ll}
A_1 \rto \sqrt{\e}\frac{\partial_z a_{33}}{a_{33}}, \\[7pt]

- A_0 \rto \frac{1}{a_{33}}\big(\gamma a_0 - f^7[\nabla\varphi^{\e},\nabla\varphi,\partial_j v^{\e,i},\partial_j v^i]\big)
+ i\tau \frac{a_0}{a_{33}} - i\e \sum\limits_{j=1,2}\frac{\partial_z a_{j3}}{a_{33}}\xi^j .
\end{array}
\end{equation*}
When $\e>0$ is sufficiently small, the values of $A_1$ and $A_0$ are around their limits.

The solution of the ODE $(\ref{Sect4_BoundarLayer_Boundary_BC_2})$ is that
\begin{equation}\label{Sect4_BoundarLayer_Boundary_BC_4}
\begin{array}{ll}
\tilde{W}^{bdy} = \exp\{\frac{- A_1 + \sqrt{A_1^2 -4 A_0}}{2} \frac{z}{\sqrt{\e}}\}\tilde{W}^{bdy}|_{z=0}.
\end{array}
\end{equation}

The complex root $\sqrt{A_1^2 -4 A_0}$ has two branches, but one of which always has positive real part,
since $\Re (A_1^2 -4 A_0)>0$ when $\e$ is sufficiently small, where $\Re$ represents the real part. Then we choose this branch.
Since $\Re (- 4 A_0) >0$ when $\e$ is sufficiently small, then $|\Re \sqrt{A_1^2 -4 A_0}| >| - A_1|$,
and then $\Re\frac{- A_1 + \sqrt{A_1^2 -4 A_0}}{2} >0$ and $\big\|\frac{- A_1 + \sqrt{A_1^2 -4 A_0}}{2}\big\|_{L^{\infty}}<+\infty$.

Define $\mathcal{T}[W^{bdy}] = \mathcal{F}^{-1}[\tilde{W}^{bdy}]$.
Note that $(\ref{Sect4_BoundarLayer_Boundary_BC_3})$ has the same form with $(\ref{Sect4_HeatEq_Diffusion_2})_3$,
apply the same argument in Proposition $\ref{Sect4_HeatEq_Diffusion_Proposition}$ to $(\ref{Sect4_BoundarLayer_Boundary_BC_3})$,
we can prove that if $z =O(\e^{\frac{1}{2} +\delta_z})$ where $\delta_z > 0$,
$\|\tilde{W}^{bdy}\|_{L^1} \nrightarrow 0$ as $\e\rto 0$,
then $\|\mathcal{T}[W^{bdy}]\|_{L^{\infty}} \neq 0$.
If $z =O(\e^{\frac{1}{2} -\delta_z})$ where $\delta_z \geq 0$,
$\|\tilde{W}^{bdy}\|_{L^1} \rto 0$ as $\e\rto 0$,
then $\|\mathcal{T}[W^{bdy}]\|_{L^{\infty}} \rto 0$.

The difference between $\mathcal{T}[W^{bdy}]$ and $W^{bdy}$ is bounded by $W^{bdy}$ (refer to the results of paradifferential calculus in 
\cite{Masmoudi_Rousset_2012_FreeBC,Metivier_Zumbrun_2005}).
If we assume $\lim\limits_{\e\rto 0}\|W^{bdy}\|_{L^{\infty}} =0$, then $\lim\limits_{\e\rto 0}\|\mathcal{T}[W^{bdy}]\|_{L^{\infty}} =0$.
This is a contradiction. So $\lim\limits_{\e\rto 0}\|W^{bdy}\|_{L^{\infty}} \neq 0$ in some set located in the interior.
Thus, Theorem $\ref{Sect4_BoundarLayer_Boundary_Thm}$ is proved.
\end{proof}

%%% find 5
\section{Convergence Rates of Inviscid Limit for $\sigma=0$}

In this section, we estimate convergence rates of the inviscid limit when $\sigma=0$. We denote $\hat{v} =v^{\e} -v,\hat{q} =q^{\e} -q, \hat{h} =h^{\e} -h$,
we denote the $i-$th components of $v^{\e}$ and $v$ by $v^{\e,i}$ and $v^i$ respectively.

\subsection{Estimates for the Pressure Gradient}

\begin{lemma}\label{Sect5_Pressure_Lemma}
Assume $0\leq s\leq k-1,\, k \leq m-2$, the difference of the pressure $\hat{q}$ has the following gradient estimate:
\begin{equation}\label{Sect5_Pressure_Lemma_Eq}
\begin{array}{ll}
\|\nabla \hat{q}\|_{X^s} \lem \|\hat{v}\|_{X^{s,1}} + \|\partial_z\hat{v}\|_{X^s} + |\hat{h}|_{X^{s,\frac{1}{2}}}
+ O(\e).
\end{array}
\end{equation}
\end{lemma}

\begin{proof}
\cite{Masmoudi_Rousset_2012_FreeBC} introduced the following matrices $\textsf{E}$ and $\textsf{P}$ satisfying $\textsf{E} = \frac{1}{\partial_z\varphi} \textsf{P} \textsf{P}^{\top}$:
\begin{equation*}
\begin{array}{ll}
\textsf{E} = \left(\begin{array}{ccc}
\partial_z \varphi & 0 & -\partial_1\varphi \\[1pt]
0 & \partial_z \varphi & -\partial_2\varphi \\[2pt]
-\partial_1 \varphi & -\partial_2 \varphi & \frac{1+(\partial_1\varphi)^2 + (\partial_2\varphi)^2}{\partial_z\varphi}
\end{array}\right)

\, , \quad

\textsf{P} = \left(\begin{array}{ccc}
\partial_z \varphi & 0 & 0 \\[2pt]
0 & \partial_z \varphi & 0 \\[2pt]
-\partial_1 \varphi & -\partial_2 \varphi & 1
\end{array}\right)\, .
\end{array}
\end{equation*}

Apply the divergence operator $\nabla^{\varphi^{\e}}\cdot$ to $(\ref{Sect1_NS_Eq})_1$ and apply the divergence operator $\nabla^{\varphi}\cdot$ to $(\ref{Sect1_Euler_Eq})_1$, then we get
\begin{equation}\label{Sect5_Pressure_Estimates_1}
\left\{\begin{array}{ll}
\nabla\cdot(\textsf{E}^{\e}\nabla q^{\e}) = \partial_z \varphi^{\e}\triangle^{\varphi^{\e}} q^{\e}
= -\partial_z \varphi^{\e}\nabla^{\varphi^{\e}}\cdot (v^{\e} \cdot\nabla^{\varphi^{\e}} v^{\e}), \\[6pt]

\nabla\cdot(\textsf{E}\nabla q) = \partial_z \varphi\triangle^{\varphi} q
= -\partial_z \varphi\nabla^{\varphi}\cdot (v \cdot\nabla^{\varphi} v).
\end{array}\right.
\end{equation}

It follows from $(\ref{Sect5_Pressure_Estimates_1})$ and $(\ref{SectA_Difference_Transform_2})$ that
\begin{equation}\label{Sect5_Pressure_Estimates_2}
\begin{array}{ll}
\nabla\cdot(\textsf{E}^{\e}\nabla \hat{q}) + \nabla\cdot((\textsf{E}^{\e} -\textsf{E}) \nabla q)
= \nabla\cdot(\textsf{E}^{\e}\nabla q^{\e}) - \nabla\cdot(\textsf{E}\nabla q) \\[7pt]

= - \partial_z \varphi^{\e}\nabla^{\varphi^{\e}}\cdot (v^{\e} \cdot\nabla^{\varphi^{\e}} v^{\e})
+ \partial_z \varphi \nabla^{\varphi}\cdot (v \cdot\nabla^{\varphi} v) \\[7pt]

= -\nabla\cdot \big[\textsf{P}^{\e} (v^{\e} \cdot\nabla^{\varphi^{\e}} v^{\e})\big]
+ \nabla\cdot \big[\textsf{P} (v \cdot\nabla^{\varphi} v)\big] \\[7pt]

= -\nabla\cdot \big[\textsf{P}^{\e} (v^{\e} \cdot\nabla^{\varphi^{\e}} v^{\e} - v \cdot\nabla^{\varphi} v)\big]
- \nabla\cdot \big[(\textsf{P}^{\e} - \textsf{P}) (v \cdot\nabla^{\varphi} v)\big] \\[7pt]

= -\nabla\cdot \big[\textsf{P}^{\e} (v^{\e} \cdot\nabla^{\varphi^{\e}} \hat{v} - v^{\e}\cdot\nabla^{\varphi^{\e}}\hat{\varphi}\partial_z^{\varphi} v
+ \hat{v}\cdot\nabla^{\varphi} v)\big]
- \nabla\cdot \big[(\textsf{P}^{\e} - \textsf{P}) (v \cdot\nabla^{\varphi} v)\big].
\end{array}
\end{equation}

Namely, $\hat{q}$ satisfies the following elliptic equation:
\begin{equation}\label{Sect5_Pressure_Estimates_3}
\left\{\begin{array}{ll}
\nabla\cdot(\textsf{E}^{\e}\nabla \hat{q})
= -\nabla\cdot((\textsf{E}^{\e} -\textsf{E}) \nabla q) - \nabla\cdot [(\textsf{P}^{\e} - \textsf{P}) (v \cdot\nabla^{\varphi} v)] \\[5pt]\hspace{2.1cm}

-\nabla\cdot [\textsf{P}^{\e} (v^{\e} \cdot\nabla^{\varphi^{\e}} \hat{v} - v^{\e}\cdot\nabla^{\varphi^{\e}}\hat{\varphi}\partial_z^{\varphi} v
+ \hat{v}\cdot\nabla^{\varphi} v)], \\[7pt]

q|_{z=0} = g\hat{h} + 2\e\mathcal{S}^{\varphi^{\e}} v^{\e} \nn^{\e}\cdot\nn^{\e}.
\end{array}\right.
\end{equation}

The matrix $\textsf{E}^{\e}$ is definitely positive, then
it is standard to prove that $\hat{q}$ satisfies the following gradient estimate:
\begin{equation}\label{Sect5_Pressure_Estimates_6}
\begin{array}{ll}
\|\nabla \hat{q}\|_{X^s} \lem \|(\textsf{E}^{\e} -\textsf{E}) \nabla q \|_{X^s}
+ \|(\textsf{P}^{\e} - \textsf{P}) (v \cdot\nabla^{\varphi} v)\|_{X^s} \\[6pt]\hspace{1.74cm}

+ \|\textsf{P}^{\e} (v^{\e} \cdot\nabla^{\varphi^{\e}} \hat{v} - v^{\e}\cdot\nabla^{\varphi^{\e}}\hat{\varphi}\partial_z^{\varphi} v
+ \hat{v}\cdot\nabla^{\varphi} v)\|_{X^s} \\[5pt]\hspace{1.74cm}

+ |g\hat{h} + 2\e\mathcal{S}^{\varphi^{\e}} v^{\e} \nn^{\e}\cdot\nn^{\e}|_{X^{s,\frac{1}{2}}} \\[9pt]\hspace{1.3cm}

\lem \|\textsf{E}^{\e} -\textsf{E}\|_{X^s} + \|\textsf{P}^{\e} - \textsf{P}\|_{X^s}
+ \|\hat{v}\|_{X^s} + \|\nabla\hat{v}\|_{X^s} + \|\nabla\hat{\varphi}\|_{X^s}\\[7pt]\hspace{1.74cm}

+ g|\hat{h}|_{X^{s,\frac{1}{2}}} + 2\e \big|\mathcal{S}^{\varphi^{\e}} v^{\e}|_{z=0}\big|_{X^{s,\frac{1}{2}}} \\[9pt]\hspace{1.3cm}

\lem \|\hat{v}\|_{X^{s,1}} + \|\partial_z\hat{v}\|_{X^s} + \|\nabla\hat{\eta}\|_{X^s}
+ g|\hat{h}|_{X^{s,\frac{1}{2}}} + 2\e \big|\mathcal{S}^{\varphi^{\e}} v^{\e}|_{z=0}\big|_{X^{s,\frac{1}{2}}} \\[9pt]\hspace{1.3cm}

\lem \|\hat{v}\|_{X^{s,1}} + \|\partial_z\hat{v}\|_{X^s} + |\hat{h}|_{X^{s,\frac{1}{2}}}
+ O(\e),
\end{array}
\end{equation}
where $\big|\mathcal{S}^{\varphi^{\e}} v^{\e}|_{z=0}\big|_{X^{s,\frac{1}{2}}}
\lem \|\partial_z\partial_j v^{\e}\|_{X^s}^{\frac{1}{2}}\|\partial_j v^{\e}\|_{X^{s+1}}^{\frac{1}{2}} <+\infty$.

Thus, Lemma $\ref{Sect5_Pressure_Lemma}$ is proved.
\end{proof}

\subsection{Estimates for Tangential Derivatives}

In order to close the estimates of tangential derivatives of $\hat{v}$, that is to bound $\|\partial_t^{\ell} \hat{v}\|_{L^2}$ and
$\sqrt{\e}\|\nabla \partial_t^{\ell}\mathcal{Z}^{\alpha}\hat{v}\|_{L^2}$, we must prove two preliminary lemmas of 
$\hat{h}$ by using the kinetical boundary condition $(\ref{Sect1_T_Derivatives_Difference_Eq})_3$. 

The first preliminary lemma concerns $|\partial_t^{\ell} \hat{h}|_{L^2}$ where $0\leq\ell\leq k-1$.
Note that the estimates of mix derivatives $\partial_t^{\ell}\mathcal{Z}^{\alpha} \hat{h}$
will be obtained when we estimate mix derivatives $\partial_t^{\ell}\mathcal{Z}^{\alpha} \hat{v}$, where $|\alpha|>0$.
\begin{lemma}\label{Sect5_Height_Estimates_Lemma}
Assume $0\leq k\leq m-2$, $0\leq\ell\leq k-1$,
$|\partial_t^{\ell}\hat{h}|_{L^2}$ has the estimates:
\begin{equation}\label{Sect5_Height_Estimates_Lemma_Eq}
\begin{array}{ll}
|\partial_t^{\ell}\hat{h}|_{L^2}^2 
\lem |\hat{h}_0|_{X^{k-1}}^2
+ \int\limits_0^t |\hat{h}|_{X^{k-1,1}}^2 + \|\hat{v}\|_{X^{k-1,1}}^2 \,\mathrm{d}t
+ \|\partial_z\hat{v}\|_{L^4([0,T],X^{k-1})}^2.
\end{array}
\end{equation}
\end{lemma}

\begin{proof}
Apply $\partial_t^{\ell}$ to the kinetical boundary condition $(\ref{Sect1_T_Derivatives_Difference_Eq})_3$, we get
\begin{equation}\label{Sect5_Height_Estimates_Lemma_Eq_1}
\begin{array}{ll}
\partial_t \partial_t^{\ell}\hat{h} + v_y\cdot \nabla_y \partial_t^{\ell}\hat{h} = \partial_t^{\ell}\hat{v}\cdot\NN^{\e}
+ [\partial_t^{\ell}, \NN^{\e}\cdot]\hat{v} - [\partial_t^{\ell}, v_y\cdot \nabla_y]\hat{h}.
\end{array}
\end{equation}

Multiply $(\ref{Sect5_Height_Estimates_Lemma_Eq_1})$ with $\partial_t^{\ell}\hat{h}$, integrate in $\mathbb{R}^2$, we have
\begin{equation}\label{Sect5_Height_Estimates_Lemma_Eq_2}
\begin{array}{ll}
\frac{\mathrm{d}}{\mathrm{d}t}\int\limits_{\mathbb{R}^2} |\partial_t^{\ell}\hat{h}|^2 \,\mathrm{d}y

= 2\int\limits_{\mathbb{R}^2}\big( \partial_t^{\ell}\hat{v}\cdot\NN^{\e}
+ [\partial_t^{\ell}, \NN^{\e}\cdot]\hat{v} - [\partial_t^{\ell}, v_y\cdot \nabla_y]\hat{h} \big) \partial_t^{\ell}\hat{h} \,\mathrm{d}y \\[7pt]\quad
+ \int\limits_{\mathbb{R}^2} |\partial_t^{\ell}\hat{h}|^2 \nabla_y\cdot v_y  \,\mathrm{d}y 

\lem |\partial_t^{\ell}\hat{h}|_{L^2}^2 + |\hat{h}|_{X^{\ell,1}}^2 + \big|\hat{v}|_{z=0}\big|_{X^{\ell}}^2 \\[7pt]
\lem |\partial_t^{\ell}\hat{h}|_{L^2}^2 + |\hat{h}|_{X^{k-1,1}}^2 + \|\hat{v}\|_{X^{k-1,1}}^2 + \|\partial_z\hat{v}\|_{X^{k-1}}^2.
\end{array}
\end{equation}

Sum $\ell$, integrate $(\ref{Sect5_Height_Estimates_Lemma_Eq_2})$ in time and apply the integral form of Gronwall's inequality, we have
\begin{equation}\label{Sect5_Height_Estimates_Lemma_Eq_3}
\begin{array}{ll}
\int\limits_{\mathbb{R}^2} |\partial_t^{\ell}\hat{h}|^2 \,\mathrm{d}y
\lem |\hat{h}_0|_{X^{k-1}}^2
+ \int\limits_0^t |\hat{h}|_{X^{k-1,1}}^2 + \|\hat{v}\|_{X^{k-1,1}}^2 + \|\partial_z\hat{v}\|_{X^{k-1}}^2
\,\mathrm{d}t \\[8pt]

\lem |\hat{h}_0|_{X^{k-1}}^2
+ \int\limits_0^t |\hat{h}|_{X^{k-1,1}}^2 + \|\hat{v}\|_{X^{k-1,1}}^2 \,\mathrm{d}t
+ \|\partial_z\hat{v}\|_{L^4([0,T],X^{k-1})}^2.
\end{array}
\end{equation}
Thus, Lemma $\ref{Sect5_Height_Estimates_Lemma}$ is proved.
\end{proof}

The second preliminary lemma concerns $\sqrt{\e}|\partial_t^{\ell}\mathcal{Z}^{\alpha} \hat{h}|_{\frac{1}{2}}$,
by which we bound $\sqrt{\e}\|\mathcal{S}^{\varphi}\partial_t^{\ell}\mathcal{Z}^{\alpha}\hat{\eta}\|_{L^2}$ and then
we can bound $\sqrt{\e}\|\mathcal{S}^{\varphi}\partial_t^{\ell}\mathcal{Z}^{\alpha} \hat{v}\|_{L^2}$.
\begin{lemma}\label{Sect5_Height_Viscous_Estimates_Lemma}
Assume $0\leq k\leq m-2$, $0\leq\ell\leq k-1$, $\ell+|\alpha| \leq k$,
$\sqrt{\e}|\partial_t^{\ell}\mathcal{Z}^{\alpha}\hat{h}|_{\frac{1}{2}}$ has the estimates:
\begin{equation}\label{Sect5_Height_Viscous_Estimates_Lemma_Eq}
\begin{array}{ll}
\e|\hat{h}|_{X^{k-1,\frac{3}{2}}}^2 \leq \e|\hat{h}_0|_{X^{k-1,\frac{3}{2}}}^2 + \int\limits_0^t|\hat{h}|_{X^{k-1,1}}^2
+ \e\|\nabla \hat{v}\|_{X^{k-1,1}}^2 \,\mathrm{d}t.
\end{array}
\end{equation}
\end{lemma}

\begin{proof}
The differential operator $\Lambda$ is defined in the proof of Lemma $\ref{Sect2_Height_Viscous_Estimates_Lemma}$.
Apply $\partial_t^{\ell}\mathcal{Z}^{\alpha}\Lambda^{\frac{1}{2}}$ to the kinetical boundary condition $(\ref{Sect1_T_Derivatives_Difference_Eq})_3$, we get
\begin{equation}\label{Sect5_Height_Viscous_Estimates_Lemma_Eq_1}
\begin{array}{ll}
\partial_t \partial_t^{\ell}\mathcal{Z}^{\alpha}\Lambda^{\frac{1}{2}}\hat{h} 
+ v_y\cdot \nabla_y \partial_t^{\ell}\mathcal{Z}^{\alpha}\Lambda^{\frac{1}{2}}\hat{h} \\[6pt]
= \partial_t^{\ell}\mathcal{Z}^{\alpha}\Lambda^{\frac{1}{2}}\hat{v}\cdot\NN^{\e} 
+ [\partial_t^{\ell}\mathcal{Z}^{\alpha}\Lambda^{\frac{1}{2}}, \NN^{\e}\cdot]\hat{v} 
- [\partial_t^{\ell}\mathcal{Z}^{\alpha}\Lambda^{\frac{1}{2}}, v_y\cdot \nabla_y]\hat{h}.
\end{array}
\end{equation}

Multiply $(\ref{Sect5_Height_Viscous_Estimates_Lemma_Eq_1})$ with $\e\partial_t^{\ell}\mathcal{Z}^{\alpha}\Lambda^{\frac{1}{2}} \hat{h}$,
integrate in $\mathbb{R}^2$, we have
\begin{equation}\label{Sect5_Height_Viscous_Estimates_Lemma_Eq_2}
\begin{array}{ll}
\e\frac{\mathrm{d}}{\mathrm{d}t}\int\limits_{\mathbb{R}^2} |\partial_t^{\ell}\mathcal{Z}^{\alpha}\Lambda^{\frac{1}{2}} \hat{h}|^2 \,\mathrm{d}y

= \e\int\limits_{\mathbb{R}^2} |\partial_t^{\ell}\mathcal{Z}^{\alpha}\Lambda^{\frac{1}{2}}\hat{h}|^2 \nabla_y\cdot v_y  \,\mathrm{d}y \\[7pt]\quad
+ 2\e\int\limits_{\mathbb{R}^2}\big(\partial_t^{\ell}\mathcal{Z}^{\alpha}\Lambda^{\frac{1}{2}}\hat{v}\cdot\NN^{\e}
+ [\partial_t^{\ell}\mathcal{Z}^{\alpha}\Lambda^{\frac{1}{2}}, \NN^{\e}\cdot]\hat{v}
- [\partial_t^{\ell}\mathcal{Z}^{\alpha}\Lambda^{\frac{1}{2}}, v_y\cdot \nabla_y]\hat{h}\big)
\partial_t^{\ell}\mathcal{Z}^{\alpha}\Lambda^{\frac{1}{2}} \hat{h} \,\mathrm{d}y 
\end{array}
\end{equation}

\begin{equation*}
\begin{array}{ll}
\lem \e|\hat{h}|_{X^{k-1,\frac{3}{2}}}^2 + \e\big|\hat{v}|_{z=0}\big|_{X^{k-1,\frac{3}{2}}}^2
+ \e|\partial_t^{\ell}\mathcal{Z}^{\alpha}\Lambda^{\frac{1}{2}} \hat{h}|_{L^2}^2\\[7pt]
\lem \e|\hat{h}|_{X^{k-1,\frac{3}{2}}}^2
+ \e\|\hat{v}\|_{X_{tan}^{k-1,2}}^2 + \e\|\partial_z \hat{v}\|_{X_{tan}^{k-1,1}}^2
+ \e|\partial_t^{\ell}\mathcal{Z}^{\alpha}\Lambda^{\frac{1}{2}} \hat{h}|_{L^2}^2\\[8pt]

\lem \e|\hat{h}|_{X^{k-1,\frac{3}{2}}}^2 +|\hat{h}|_{X^{k-1,1}}^2 + \e\|\nabla \hat{v}\|_{X^{k-1,1}}^2.
\hspace{5.cm}
\end{array}
\end{equation*}

Sum $\ell,\alpha$, integrate $(\ref{Sect5_Height_Viscous_Estimates_Lemma_Eq_2})$ in time and apply the integral form of Gronwall's inequality, 
we have $(\ref{Sect5_Height_Viscous_Estimates_Lemma_Eq})$.
Thus, Lemma $\ref{Sect5_Height_Viscous_Estimates_Lemma}$ is proved.
\end{proof}

We state that $\partial_z \hat{v}^3$ can be estimated by $\partial_z \hat{v}_h$,
that is $\|\partial_z \hat{v}^3\|_{X^s} \lem \|\hat{v}_h\|_{X^{s,1}} + \|\partial_z \hat{v}_h\|_{X^s}
+ |\hat{h}|_{X^{s,\frac{1}{2}}}$. The proof is based on the following equality that follows from
the divergence free condition $(\ref{Sect2_DivFreeCondition})$.
\begin{equation}\label{Sect5_NDerivatives_Lemma_Eq_1}
\begin{array}{ll}
\partial_z \hat{v}^3 = -\partial_z\varphi^{\e}(\partial_1 \hat{v}^1 + \partial_2 \hat{v}^2)
- \partial_z\hat{\varphi}(\partial_1 v^1 + \partial_2 v^2) \\[7pt]\hspace{1.15cm}

+\partial_1\varphi^{\e}\partial_z \hat{v}^1 +\partial_1\hat{\varphi}\partial_z v^1
+ \partial_2\varphi^{\e}\partial_z \hat{v}^2 +\partial_2\hat{\varphi}\partial_z v^2.
\end{array}
\end{equation}

Now we develop the estimates for tangential derivatives.
\begin{lemma}\label{Sect5_Tangential_Estimates_Lemma}
Assume $0\leq k\leq m-2$,
$\partial_t^{\ell}\mathcal{Z}^{\alpha}\hat{v}$ and $\partial_t^{\ell}\mathcal{Z}^{\alpha}\hat{h}$ have the estimates:
\begin{equation}\label{Sect5_Tangential_Estimates_Lemma_Eq}
\begin{array}{ll}
\|\hat{v}\|_{X^{k-1,1}}^2 + |\hat{h}|_{X^{k-1,1}}^2 + \e|\hat{h}|_{X^{k-1,\frac{3}{2}}}^2
+ \e\int\limits_0^t\|\nabla\hat{v}\|_{X^{k-1,1}}^2 \,\mathrm{d}t \\[5pt]

\lem \|\hat{v}_0\|_{X^{k-1,1}}^2 + |\hat{h}_0|_{X^{k-1,1}}^2 + \e|\hat{h}_0|_{X^{k-1,\frac{3}{2}}}^2
+\|\partial_z \hat{v}\|_{L^4([0,T],X^{k-1})}^2 \\[6pt]\quad
+ |\partial_t^k\hat{h}|_{L^4([0,T],L^2)}^2
+ \|\nabla\hat{q}\|_{L^4([0,T],X^{k-1})}^2 + O(\e).
\end{array}
\end{equation}
\end{lemma}

\begin{proof}
$(\hat{V}^{\ell,\alpha}, \hat{Q}^{\ell,\alpha})$ satisfy the following equations:
\begin{equation}\label{Sect5_TangentialEstimates_Diff_Eq}
\left\{\begin{array}{ll}
\partial_t^{\varphi^{\e}} \hat{V}^{\ell,\alpha}
+ v^{\e} \cdot\nabla^{\varphi^{\e}} \hat{V}^{\ell,\alpha}
+ \nabla^{\varphi^{\e}} \hat{Q}^{\ell,\alpha}
- 2\e\nabla^{\varphi^{\e}}\cdot\mathcal{S}^{\varphi^{\e}} \partial_t^{\ell}\mathcal{Z}^{\alpha}\hat{v} \\[8pt]\quad

= 2\e[\partial_t^{\ell}\mathcal{Z}^{\alpha},\nabla^{\varphi^{\e}}\cdot]\mathcal{S}^{\varphi^{\e}} \hat{v}
+ 2\e\nabla^{\varphi^{\e}}\cdot[\partial_t^{\ell}\mathcal{Z}^{\alpha},\mathcal{S}^{\varphi^{\e}}] \hat{v}
+ \e\partial_t^{\ell}\mathcal{Z}^{\alpha}\triangle^{\varphi^\e} v
\\[8pt]\quad

- \partial_t^{\ell}\mathcal{Z}^{\alpha}\hat{\varphi}\partial_t^{\varphi^{\e}}\partial_z^{\varphi} v
-  \partial_t^{\ell}\mathcal{Z}^{\alpha}\hat{\varphi}\, v^{\e}\cdot \nabla^{\varphi^{\e}}\partial_z^{\varphi} v
- \partial_t^{\ell}\mathcal{Z}^{\alpha}\hat{v}\cdot\nabla^{\varphi} v
- \partial_t^{\ell}\mathcal{Z}^{\alpha}\hat{\varphi} \nabla^{\varphi^{\e}}\partial_z^{\varphi} q
\\[8pt]\quad

- [\partial_t^{\ell}\mathcal{Z}^{\alpha},\partial_t^{\varphi^{\e}}]\hat{v}
+ [\partial_t^{\ell}\mathcal{Z}^{\alpha}, \partial_z^{\varphi} v \partial_t^{\varphi^{\e}}]\hat{\varphi}

- [\partial_t^{\ell}\mathcal{Z}^{\alpha}, v^{\e} \cdot\nabla^{\varphi^{\e}}] \hat{v}
- [\partial_t^{\ell}\mathcal{Z}^{\alpha}, \nabla^{\varphi} v\cdot]\hat{v}\\[8pt]\quad

+ [\partial_t^{\ell}\mathcal{Z}^{\alpha}, \partial_z^{\varphi} v \, v^{\e}\cdot \nabla^{\varphi^{\e}}]\hat{\varphi}

- [\partial_t^{\ell}\mathcal{Z}^{\alpha},\nabla^{\varphi^{\e}}] \hat{q}
+ [\partial_t^{\ell}\mathcal{Z}^{\alpha},\partial_z^{\varphi} q\nabla^{\varphi^{\e}}]\hat{\varphi} 

:= \mathcal{I}_4 , \\[12pt]

\nabla^{\varphi^{\e}}\cdot \hat{V}^{\ell,\alpha}
= -[\partial_t^{\ell}\mathcal{Z}^{\alpha},\nabla^{\varphi^{\e}}\cdot] \hat{v}
+ [\partial_t^{\ell}\mathcal{Z}^{\alpha},\partial_z^{\varphi}v \cdot\nabla^{\varphi^{\e}}]\hat{\eta}
- \partial_t^{\ell}\mathcal{Z}^{\alpha}\hat{\eta} \nabla^{\varphi^{\e}}\cdot \partial_z^{\varphi}v,
\\[12pt]

\partial_t \partial_t^{\ell}\mathcal{Z}^{\alpha}\hat{h} + v_y^{\e}\cdot \nabla_y \partial_t^{\ell}\mathcal{Z}^{\alpha}\hat{h}
- \NN^{\e}\cdot \hat{V}^{\ell,\alpha}
= \NN^{\e}\cdot \partial_z^{\varphi} v \partial_t^{\ell}\mathcal{Z}^{\alpha}\hat{\eta} \\[7pt]\quad
 - \hat{v}_y \cdot \nabla_y \partial_t^{\ell}\mathcal{Z}^{\alpha} h
- \partial_y \hat{h}\cdot \partial_t^{\ell}\mathcal{Z}^{\alpha}v_y
 + [\partial_t^{\ell}\mathcal{Z}^{\alpha}, \hat{v},\NN^{\e}]
 - [\partial_t^{\ell}\mathcal{Z}^{\alpha}, v_y, \partial_y \hat{h}], \\[12pt]

\hat{Q}^{\ell,\alpha}\NN^{\e}

-2\e \mathcal{S}^{\varphi^{\e}} \partial_t^{\ell}\mathcal{Z}^{\alpha}\hat{v}\,\NN^{\e}
- (g-\partial_z^{\varphi}q)\partial_t^{\ell}\mathcal{Z}^{\alpha}\hat{h} \,\NN^{\e}
\\[8pt]\quad
= 2\e [\partial_t^{\ell}\mathcal{Z}^{\alpha},\mathcal{S}^{\varphi^{\e}}] v^{\e}\,\NN^{\e}

+ (2\e \mathcal{S}^{\varphi^{\e}}v^{\e} - 2\e \mathcal{S}^{\varphi^{\e}}v^{\e}\nn^{\e}\cdot\nn^{\e})
\,\partial_t^{\ell}\mathcal{Z}^{\alpha}\NN^{\e} \\[8pt]\quad
- [\partial_t^{\ell}\mathcal{Z}^{\alpha},2\e \mathcal{S}^{\varphi^{\e}}v^{\e}\nn^{\e}\cdot\nn^{\e},\NN^{\e}]
+2\e [\partial_t^{\ell}\mathcal{Z}^{\alpha},\mathcal{S}^{\varphi^{\e}}v^{\e}, \NN^{\e}]
+ 2\e \mathcal{S}^{\varphi^{\e}} \partial_t^{\ell}\mathcal{Z}^{\alpha}v\,\NN^{\e}, \\[12pt]

(\partial_t^{\ell}\mathcal{Z}^{\alpha}\hat{v},\partial_t^{\ell}\mathcal{Z}^{\alpha}\hat{h})|_{t=0}
= (\partial_t^{\ell}\mathcal{Z}^{\alpha}v_0^\e -\partial_t^{\ell}\mathcal{Z}^{\alpha}v_0,
\partial_t^{\ell}\mathcal{Z}^{\alpha}h_0^\e -\partial_t^{\ell}\mathcal{Z}^{\alpha}h_0),
\end{array}\right.
\end{equation}

When $|\alpha|\geq 1, \ell\leq k-1, 1\leq \ell+|\alpha|\leq k$, we develop the $L^2$ estimate of $\hat{V}^{\ell,\alpha}$, we have
\begin{equation}\label{Sect5_Tangential_Estimates_1}
\begin{array}{ll}
\frac{1}{2}\frac{\mathrm{d}}{\mathrm{d}t} \int\limits_{\mathbb{R}^3_{-}} |\hat{V}^{\ell,\alpha}|^2 \,\mathrm{d}\mathcal{V}_t
- \int\limits_{\mathbb{R}^3_{-}} \hat{Q}^{\ell,\alpha} \nabla^{\varphi^{\e}}\cdot \hat{V}^{\ell,\alpha} \,\mathrm{d}\mathcal{V}_t
+ 2\e \int\limits_{\mathbb{R}^3_{-}} |\mathcal{S}^{\varphi^{\e}}\partial_t^{\ell}\mathcal{Z}^{\alpha} \hat{v}|^2
\,\mathrm{d}\mathcal{V}_t
\\[10pt]

= -\int\limits_{\{z=0\}} \big(\hat{Q}^{\ell,\alpha} \NN^{\e} - 2\e \mathcal{S}^{\varphi^{\e}}\partial_t^{\ell}\mathcal{Z}^{\alpha} \hat{v} \NN^{\e} \big)
\cdot \hat{V}^{\ell,\alpha} \,\mathrm{d}y

+ \int\limits_{\mathbb{R}^3_{-}} \mathcal{I}_4 \cdot V^{\ell,\alpha}
\,\mathrm{d}\mathcal{V}_t \\[10pt]\quad

+ 2\e \int\limits_{\mathbb{R}^3_{-}} \mathcal{S}^{\varphi^{\e}}\partial_t^{\ell}\mathcal{Z}^{\alpha} \hat{v}
\cdot \mathcal{S}^{\varphi^{\e}} (\partial_z^{\varphi}v \partial_t^{\ell}\mathcal{Z}^{\alpha}\hat{\eta})\,\mathrm{d}\mathcal{V}_t \\[10pt]

\lem -\int\limits_{\{z=0\}} \big(\hat{Q}^{\ell,\alpha} \NN^{\e} - 2\e \mathcal{S}^{\varphi^{\e}}\partial_t^{\ell}\mathcal{Z}^{\alpha} \hat{v} \NN^{\e} \big)
\cdot \hat{V}^{\ell,\alpha} \,\mathrm{d}y + \|\hat{V}^{\ell,\alpha}\|_{L^2}^2

+ \|\partial_z \hat{v}\|_{X^{k-1}}^2 \\[10pt]\quad
 + \|\hat{v}\|_{X^{k-1,1}}^2
+ \|\hat{\eta}\|_{X^{k-1,1}}^2 + \e|\hat{h}|_{X^{k-1,\frac{3}{2}}}^2 + \|\partial_t^k\hat{\eta}\|_{L^2}^2
+ \|\nabla\hat{q}\|_{X^{k-1}}^2 + O(\e).
\end{array}
\end{equation}

We develop the boundary estimates in $(\ref{Sect5_Tangential_Estimates_1})$,
\begin{equation}\label{Sect5_Tangential_Estimates_2}
\begin{array}{ll}
-\int\limits_{\{z=0\}} \big(\hat{Q}^{\ell,\alpha} \NN^{\e} - 2\e \mathcal{S}^{\varphi^{\e}}\partial_t^{\ell}\mathcal{Z}^{\alpha} \hat{v} \NN^{\e} \big)
\cdot \hat{V}^{\ell,\alpha} \,\mathrm{d}y \\[10pt]

= \int\limits_{\{z=0\}}
-(g -\partial_z^{\varphi}q) \partial_t^{\ell}\mathcal{Z}^{\alpha}\hat{h}\NN^{\e}\cdot \hat{V}^{\ell,\alpha}
- \big(2\e [\partial_t^{\ell}\mathcal{Z}^{\alpha},\mathcal{S}^{\varphi^{\e}}] v^{\e}\,\NN^{\e} \\[9pt]\qquad

+ (2\e \mathcal{S}^{\varphi^{\e}}v^{\e} - 2\e \mathcal{S}^{\varphi^{\e}}v^{\e}\nn^{\e}\cdot\nn^{\e})\,\partial_t^{\ell}\mathcal{Z}^{\alpha}\NN^{\e}
- [\partial_t^{\ell}\mathcal{Z}^{\alpha},2\e \mathcal{S}^{\varphi^{\e}}v^{\e}\nn^{\e}\cdot\nn^{\e},\NN^{\e}] \\[7pt]\qquad
+ 2\e [\partial_t^{\ell}\mathcal{Z}^{\alpha},\mathcal{S}^{\varphi^{\e}}v^{\e}, \NN^{\e}]
+ 2\e \mathcal{S}^{\varphi^{\e}} \partial_t^{\ell}\mathcal{Z}^{\alpha}v\,\NN^{\e}\big)\cdot \hat{V}^{\ell,\alpha}
\,\mathrm{d}y \\[8pt]

\lem \int\limits_{\{z=0\}}
-(g -\partial_z^{\varphi}q) \partial_t^{\ell}\mathcal{Z}^{\alpha}\hat{h}\NN^{\e}\cdot \hat{V}^{\ell,\alpha} \,\mathrm{d}y + O(\e),
\end{array}
\end{equation}
note that normal derivatives $\partial_z v^{\e}$ on the boundary can be expressed in terms of tangential derivatives of $v^{\e}$, thus we get $O(\e)$ rather than $O(\sqrt{\e})$. Namely,
\begin{equation}\label{Sect5_Tangential_Estimates_3}
\begin{array}{ll}
-\int\limits_{\{z=0\}} \big(\hat{Q}^{\ell,\alpha} \NN^{\e} - 2\e \partial_t^{\ell}\mathcal{Z}^{\alpha}\mathcal{S}^{\varphi^{\e}} \hat{v} \NN^{\e} \big)
\cdot \hat{V}^{\ell,\alpha} \,\mathrm{d}y \\[13pt]

\lem \int\limits_{\{z=0\}}
-(g -\partial_z^{\varphi}q) \partial_t^{\ell}\mathcal{Z}^{\alpha}\hat{h}
\big(\partial_t \partial_t^{\ell}\mathcal{Z}^{\alpha}\hat{h} + v_y^{\e}\cdot \nabla_y \partial_t^{\ell}\mathcal{Z}^{\alpha}\hat{h}
- \NN^{\e}\cdot \partial_z^{\varphi} v \partial_t^{\ell}\mathcal{Z}^{\alpha}\hat{\eta} \\[11pt]\quad
+ \hat{v}_y \cdot \nabla_y \partial_t^{\ell}\mathcal{Z}^{\alpha} h
+ \partial_y \hat{h}\cdot \partial_t^{\ell}\mathcal{Z}^{\alpha}v_y
- [\partial_t^{\ell}\mathcal{Z}^{\alpha}, \hat{v},\NN^{\e}]
+ [\partial_t^{\ell}\mathcal{Z}^{\alpha}, v_y, \partial_y \hat{h}]\big)\,\mathrm{d}y
+ O(\e) \\[11pt]

\lem - \frac{1}{2}\frac{\mathrm{d}}{\mathrm{d}t}\int\limits_{\{z=0\}}
(g -\partial_z^{\varphi}q) |\partial_t^{\ell}\mathcal{Z}^{\alpha}\hat{h}|^2
\,\mathrm{d}y

+ \frac{1}{2}\int\limits_{\{z=0\}} \nabla\cdot (g v_y -v_y \partial_z^{\varphi}q) |\partial_t^{\ell}\mathcal{Z}^{\alpha}\hat{h}|^2
\,\mathrm{d}y \\[11pt]\quad

- \frac{1}{2}\int\limits_{\{z=0\}}
\partial_t\partial_z^{\varphi}v |\partial_t^{\ell}\mathcal{Z}^{\alpha}\hat{h}|^2
\,\mathrm{d}y

- \int\limits_{\{z=0\}}
(g -\partial_z^{\varphi}q) \partial_t^{\ell}\mathcal{Z}^{\alpha}\hat{h}
\big(- \NN^{\e}\cdot \partial_z^{\varphi} v \partial_t^{\ell}\mathcal{Z}^{\alpha}\hat{\eta} \\[11pt]\quad
+ \hat{v}_y \cdot \nabla_y \partial_t^{\ell}\mathcal{Z}^{\alpha} h
+ \partial_y \hat{h}\cdot \partial_t^{\ell}\mathcal{Z}^{\alpha}v_y
- [\partial_t^{\ell}\mathcal{Z}^{\alpha}, \hat{v},\NN^{\e}]
+ [\partial_t^{\ell}\mathcal{Z}^{\alpha}, v_y, \partial_y \hat{h}]\big)\,\mathrm{d}y + O(\e)
 \\[11pt]

\lem - \frac{1}{2}\frac{\mathrm{d}}{\mathrm{d}t}\int\limits_{\{z=0\}}
(g -\partial_z^{\varphi}q) |\partial_t^{\ell}\mathcal{Z}^{\alpha}\hat{h}|^2
\,\mathrm{d}y

+ \|\hat{v}\|_{X^{k-1}} + \|\partial_z\hat{v}\|_{X^{k-1}} + |\hat{h}|_{X^{k-1,1}}
+ O(\e).
\end{array}
\end{equation}

By $(\ref{Sect5_Tangential_Estimates_1})$ and $(\ref{Sect5_Tangential_Estimates_3})$, we have
\begin{equation}\label{Sect5_Tangential_Estimates_4}
\begin{array}{ll}
\frac{\mathrm{d}}{\mathrm{d}t} \int\limits_{\mathbb{R}^3_{-}} |\hat{V}^{\ell,\alpha}|^2 \,\mathrm{d}\mathcal{V}_t
+ \frac{\mathrm{d}}{\mathrm{d}t}\int\limits_{\{z=0\}}
(g -\partial_z^{\varphi}q) |\partial_t^{\ell}\mathcal{Z}^{\alpha}\hat{h}|^2
\,\mathrm{d}y

+ \e \int\limits_{\mathbb{R}^3_{-}} |\partial_t^{\ell}\mathcal{Z}^{\alpha}\mathcal{S}^{\varphi^{\e}} \hat{v}|^2 \,\mathrm{d}\mathcal{V}_t
\\[13pt]

\lem \|\hat{V}^{\ell,\alpha}\|_{L^2}^2

+ \|\partial_z \hat{v}\|_{X^{k-1}}^2  + \|\hat{v}\|_{X^{k-1,1}}^2
 + |\hat{h}|_{X^{k-1,1}}^2 + |\partial_t^k\hat{h}|_{L^2}^2
 + \e|\hat{h}|_{X^{k-1,\frac{3}{2}}}^2 \\[5pt]\quad
+ \|\nabla\hat{q}\|_{X^{k-1}}^2 
+ O(\e).
\end{array}
\end{equation}

Since $g -\partial_z^{\varphi}q \geq c_0>0$, integrate $(\ref{Sect5_Tangential_Estimates_4})$ in time, apply the integral form of Gronwall's inequality, we get
\begin{equation}\label{Sect5_Tangential_Estimates_5}
\begin{array}{ll}
\|\partial_t^{\ell}\mathcal{Z}^{\alpha}\hat{v}\|_{L^2}^2
+ |\partial_t^{\ell}\mathcal{Z}^{\alpha}\hat{h}|^2 
+\e|\hat{h}|_{X^{k-1,\frac{3}{2}}}^2
+ \e \int\limits_0^t\|\nabla\partial_t^{\ell}\mathcal{Z}^{\alpha}\hat{v}\|_{L^2}^2 \,\mathrm{d}t
\\[8pt]
\lem \|\hat{v}_0\|_{X^{k-1,1}}^2 + |\hat{h}_0|_{X^{k-1,1}}^2
+ \int\limits_0^t \|\hat{v}\|_{X^{k-1,1}}^2 + |\hat{h}|_{X^{k-1,1}}^2 +\e|\hat{h}|_{X^{k-1,\frac{3}{2}}}^2 \,\mathrm{d}t
\\[8pt]\quad

+ \|\partial_z \hat{v}\|_{L^4([0,T],X^{k-1})}^2 + |\partial_t^k\hat{h}|_{L^4([0,T],L^2)}^2
+ \|\nabla\hat{q}\|_{L^4([0,T],X^{k-1})}^2 + O(\e).
\end{array}
\end{equation}

When $|\alpha|=0, \, 0\leq\ell\leq k-1$, we have no bounds of $\hat{q}$ and $\partial_t^{\ell} \hat{q}$,
so we neither use the variable $\hat{Q}^{\ell,\alpha}$ and nor apply the integration by parts to the pressure terms.
Also, the dynamical boundary condition will not be used. Since the main equation of $\hat{V}^{\ell,0}$ and its kinetical boundary condition
satisfy
\begin{equation}\label{Sect5_Tangential_Estimates_Time}
\left\{\begin{array}{ll}
\partial_t^{\varphi^{\e}} \hat{V}^{\ell,0}
+ v^{\e} \cdot\nabla^{\varphi^{\e}}\hat{V}^{\ell,0}
- 2\e\nabla^{\varphi^{\e}}\cdot\mathcal{S}^{\varphi^{\e}} \partial_t^{\ell}\hat{v} \\[8pt]\quad

= \e\partial_t^{\ell}\triangle^{\varphi^\e} v^\e
+ 2\e [\partial_t^{\ell},\nabla^{\varphi^{\e}}\cdot]\mathcal{S}^{\varphi^{\e}} \hat{v}
+ 2\e \nabla^{\varphi^{\e}}\cdot[\partial_t^{\ell},\mathcal{S}^{\varphi^{\e}}] \hat{v} 

- \partial_t^{\ell}\nabla^{\varphi^{\e}}\hat{q} \\[8pt]\quad
+ \partial_z^{\varphi} q\nabla^{\varphi^{\e}} \partial_t^{\ell}\hat{\varphi}

- \partial_t^{\ell}\hat{\varphi}\partial_t^{\varphi^{\e}}\partial_z^{\varphi} v
- \partial_t^{\ell}\hat{v}\cdot\nabla^{\varphi} v
- \partial_t^{\ell}\hat{\varphi}\, v^{\e}\cdot \nabla^{\varphi^{\e}}\partial_z^{\varphi} v \\[8pt]\quad

- [\partial_t^{\ell},\partial_t^{\varphi^{\e}}]\hat{v}
+ [\partial_t^{\ell}, \partial_z^{\varphi} v \partial_t^{\varphi^{\e}}]\hat{\varphi} 

- [\partial_t^{\ell}, v^{\e} \cdot\nabla^{\varphi^{\e}}] \hat{v}
+ [\partial_t^{\ell}, \partial_z^{\varphi} v \, v^{\e}\cdot \nabla^{\varphi^{\e}}]\hat{\varphi} \\[8pt]\quad
- [\partial_t^{\ell}, \nabla^{\varphi} v\cdot]\hat{v}

+ [\partial_t^{\ell},\partial_z^{\varphi} q\nabla^{\varphi^{\e}}]\hat{\varphi}
:=\mathcal{I}_5,  \\[13pt]

\partial_t \partial_t^{\ell}\hat{h} + v_y^{\e}\cdot \nabla_y \partial_t^{\ell}\hat{h}
- \NN^{\e}\cdot \hat{V}^{\ell,0}
= \NN^{\e}\cdot \partial_z^{\varphi} v \partial_t^{\ell}\hat{\eta} \\[7pt]\quad
 - \hat{v}_y \cdot \nabla_y \partial_t^{\ell} h
- \partial_y \hat{h}\cdot \partial_t^{\ell}v_y
 + [\partial_t^{\ell}, \hat{v},\NN^{\e}]
 - [\partial_t^{\ell}, v_y, \partial_y \hat{h}], \\[13pt]

(\partial_t^{\ell}\hat{v},\partial_t^{\ell}\hat{h})|_{t=0}
= (\partial_t^{\ell}v_0^\e -\partial_t^{\ell}v_0,
\partial_t^{\ell}h_0^\e -\partial_t^{\ell}h_0),
\end{array}\right.
\end{equation}
then we have $L^2$ estimate of $\hat{V}^{\ell,0}$:
\begin{equation}\label{Sect5_Tangential_Estimates_6}
\begin{array}{ll}
\frac{1}{2}\frac{\mathrm{d}}{\mathrm{d}t} \int\limits_{\mathbb{R}^3_{-}} |\hat{V}^{\ell,0}|^2 \,\mathrm{d}\mathcal{V}_t
+ 2\e \int\limits_{\mathbb{R}^3_{-}} |\mathcal{S}^{\varphi^{\e}}\partial_t^{\ell} \hat{v}|^2
\,\mathrm{d}\mathcal{V}_t 

= 2\e \int\limits_{\{z=0\}} \mathcal{S}^{\varphi^{\e}}\partial_t^{\ell} \hat{v} \NN^{\e} \cdot \hat{V}^{\ell,0} \,\mathrm{d}y \\[12pt]\quad
 + 2\e \int\limits_{\mathbb{R}^3_{-}} \mathcal{S}^{\varphi^{\e}}\partial_t^{\ell} \hat{v}
 \cdot \mathcal{S}^{\varphi^{\e}} (\partial_z^{\varphi}v\partial_t^{\ell} \hat{\eta}) \,\mathrm{d}\mathcal{V}_t
+ \int\limits_{\mathbb{R}^3_{-}}\mathcal{I}_5 \cdot \hat{V}^{\ell,0} \,\mathrm{d}\mathcal{V}_t.
\end{array}
\end{equation}

\vspace{-0.3cm}
Now we estimate the right hand side of $(\ref{Sect5_Tangential_Estimates_6})$:
\begin{equation}\label{Sect5_Tangential_Estimates_7}
\begin{array}{ll}
2\e\int\limits_{\{z=0\}} \mathcal{S}^{\varphi^\e} \partial_t^{\ell} v^\e \NN^{\e} \cdot \hat{V}^{\ell,0} \,\mathrm{d}y
= 2\e\int\limits_{\{z=0\}} \mathcal{S}^{\varphi^\e} \partial_t^{\ell} v^\e \NN^{\e} \cdot
(\partial_t^{\ell}\hat{v} - \partial_z^{\varphi} v \partial_t^{\ell}\hat{\eta}) \,\mathrm{d}y \\[12pt]

\lem \big|\partial_t^{\ell}\hat{v}|_{z=0}\big|_{L^2}^2 + |\partial_t^{\ell}\hat{h}|_{L^2}^2 + O(\e)
\lem \|\partial_t^{\ell}\hat{v}\|_{L^2}^2 + \|\partial_t^{\ell}\partial_z \hat{v}\|_{L^2}^2 + |\partial_t^{\ell}\hat{h}|_{L^2}^2 + O(\e).
\end{array}
\end{equation}

It is easy to check that
\begin{equation}\label{Sect5_Tangential_Estimates_8}
\begin{array}{ll}
\int\limits_{\mathbb{R}^3_{-}}\mathcal{I}_5 \cdot \hat{V}^{\ell,0} \,\mathrm{d}\mathcal{V}_t
\lem \|\hat{V}^{\ell,0}\|_{L^2}^2 + \|\partial_t^{\ell-1}\partial_z \hat{v}\|_{L^2}^2 + \|\partial_t^{\ell-1}\partial_y \hat{v}\|_{L^2}^2
+ \|\partial_t^{\ell-1} \hat{v}\|_{L^2}^2 \\[8pt]\quad

+ \|\partial_t^{\ell-1}\nabla\hat{\eta}\|_{L^2}^2 + \|\partial_t^{\ell}\hat{\eta}\|_{L^2}^2 + \|\partial_t^{\ell}\nabla\hat{\eta}\|_{L^2}^2
+ \|\partial_t^{\ell}\nabla\hat{q}\|_{L^2}^2 +O(\e).
\end{array}
\end{equation}

By $(\ref{Sect5_Tangential_Estimates_6})$, $(\ref{Sect5_Tangential_Estimates_7})$ and $(\ref{Sect5_Tangential_Estimates_8})$, we have
\begin{equation}\label{Sect5_Tangential_Estimates_9}
\begin{array}{ll}
\frac{\mathrm{d}}{\mathrm{d}t} \int\limits_{\mathbb{R}^3_{-}} |\hat{V}^{\ell,0}|^2 \,\mathrm{d}\mathcal{V}_t
+ \e \int\limits_{\mathbb{R}^3_{-}} |\mathcal{S}^{\varphi^{\e}}\partial_t^{\ell} \hat{v}|^2
\,\mathrm{d}\mathcal{V}_t 
\lem \|\hat{V}^{\ell,0}\|_{L^2}^2 + \|\partial_z \hat{v}\|_{X^{k-1}}^2 + \|\hat{v}\|_{X^{k-1,1}}^2 \\[8pt]\quad
+ |\hat{h}|_{X^{k-1,1}}^2 + \e|\hat{h}|_{X^{k-1,\frac{3}{2}}}^2 
+ \|\nabla\hat{q}\|_{X^{k-1}}^2 +O(\e).
\end{array}
\end{equation}

Integrate $(\ref{Sect5_Tangential_Estimates_9})$ in time, apply the integral form of Gronwall's inequality, we have
\begin{equation}\label{Sect5_Tangential_Estimates_10}
\begin{array}{ll}
\|\hat{V}^{\ell,0}\|_{L^2}^2 + \e\|\nabla\partial_t^{\ell}\hat{v}\|_{L^2}^2 \\[5pt]

\lem \|\hat{v}_0\|_{X^{k-1}}^2 +|\hat{h}_0|_{X^{k-1}}^2 + \int\limits_0^t \|\partial_z \hat{v}\|_{X^{k-1}}^2 + \|\hat{v}\|_{X^{k-1,1}}^2
+ |\hat{h}|_{X^{k-1,1}}^2  \\[7pt]\quad
+ \|\nabla\hat{q}\|_{X^{k-1}}^2 + \e|\hat{h}|_{X^{k-1,\frac{3}{2}}}^2\,\mathrm{d}t +O(\e) \\[5pt]

\lem \|\hat{v}_0\|_{X^{k-1}}^2 +|\hat{h}_0|_{X^{k-1}}^2 + \int\limits_0^t \|\hat{v}\|_{X^{k-1,1}}^2 + |\hat{h}|_{X^{k-1,1}}^2 
+ \e|\hat{h}|_{X^{k-1,\frac{3}{2}}}^2\,\mathrm{d}t
 \\[9pt]\quad

 + \|\partial_z \hat{v}\|_{L^4([0,T],X^{k-1})}^2 + \|\nabla\hat{q}\|_{L^4([0,T],X^{k-1})}^2 +O(\e).
\end{array}
\end{equation}

Combining $(\ref{Sect5_Tangential_Estimates_10})$ and Lemma $\ref{Sect5_Height_Estimates_Lemma}$, we have
\begin{equation}\label{Sect5_Tangential_Estimates_11}
\begin{array}{ll}
\|\partial_t^{\ell}\hat{v}\|_{L^2}^2 + |\partial_t^{\ell}\hat{h}|_{L^2}^2 + \e|\hat{h}|_{X^{k-1,\frac{3}{2}}}^2
+ \e \int\limits_0^t\|\nabla\partial_t^{\ell}\hat{v}\|_{L^2}^2 \,\mathrm{d}t
\\[5pt] 

\lem \|\hat{v}_0\|_{X^{k-1}}^2 +|\hat{h}_0|_{X^{k-1}}^2 + \int\limits_0^t \|\hat{v}\|_{X^{k-1,1}}^2 + |\hat{h}\|_{X^{k-1,1}}^2 
+ \e|\hat{h}|_{X^{k-1,\frac{3}{2}}}^2\,\mathrm{d}t
 \\[9pt]\quad

+ \|\partial_z \hat{v}\|_{L^4([0,T],X^{k-1})}^2 + \|\nabla\hat{q}\|_{L^4([0,T],X^{k-1})}^2 +O(\e).
\end{array}
\end{equation}

Apply the integral form of Gronwall's inequality to $(\ref{Sect5_Tangential_Estimates_5})$ and $(\ref{Sect5_Tangential_Estimates_11})$,
we get $(\ref{Sect5_Tangential_Estimates_Lemma_Eq})$.
Thus, Lemma $\ref{Sect5_Tangential_Estimates_Lemma}$ is proved.
\end{proof}

\subsection{Estimates for Normal Derivatives when $\Pi\mathcal{S}^{\varphi} v\nn|_{z=0} \neq 0$}

In this subsection, we develop the estimates for normal derivatives $\partial_z \hat{v}$.
In the following lemma, we estimate $\|\partial_z\hat{v}\|_{L^4([0,T],X^{k-1})}^2$ by studying the equations of $\hat{\omega}_h$.
\begin{lemma}\label{Sect5_Vorticity_Lemma}
Assume $k\leq m-2$, if $\Pi\mathcal{S}^{\varphi} v \nn|_{z=0} \neq 0$,
then the vorticity has the following estimate:
\begin{equation}\label{Sect5_Vorticity_Lemma_Eq}
\begin{array}{ll}
\|\partial_z\hat{v}_h\|_{L^4([0,T],X^{k-1})}^2 + \|\hat{\omega}_h\|_{L^4([0,T],X^{k-1})}^2
\\[6pt]
\lem \big\|\hat{\omega}_0\big\|_{X^{k-1}}^2
+ \int\limits_0^T\|\hat{v}\|_{X^{k-1,1}}^2 \,\mathrm{d}t
+ \int\limits_0^T|\hat{h}|_{X^{k-1,1}}^2 \,\mathrm{d}t
+ \|\partial_t^k\hat{h}\|_{L^4([0,T],L^2)}^2
+ O(\sqrt{\e}).
\end{array}
\end{equation}
\end{lemma}

\begin{proof}
Assume $\ell+|\alpha|\leq k-1$, we study the equations $(\ref{Sect1_N_Derivatives_Difference_Eq})$ and decompose $\hat{\omega}_h = \hat{\omega}_h^{nhom} + \hat{\omega}_h^{hom}$, such that
$\hat{\omega}_h^{nhom}$ satisfies the following nonhomogeneous equations:
\begin{equation}\label{Sect5_N_Derivatives_Difference_Eq_Nonhom}
\left\{\begin{array}{ll}
\partial_t^{\varphi^{\e}} \hat{\omega}_h^{nhom}
+ v^{\e} \cdot\nabla^{\varphi^{\e}} \hat{\omega}_h^{nhom}
- \e\triangle^{\varphi^{\e}}\hat{\omega}_h^{nhom} \\[6pt]\quad
= \vec{\textsf{F}}^0[\nabla\varphi^{\e}](\omega_h^{\e},\partial_j v^{\e,i})
- \vec{\textsf{F}}^0[\nabla\varphi](\omega_h,\partial_j v^i)
+ \e\triangle^{\varphi^{\e}}\omega_h \\[6pt]\qquad
+ \partial_z^{\varphi}\omega_h \partial_t^{\varphi^{\e}} \hat{\eta}
+ \partial_z^{\varphi} \omega_h\, v^{\e}\cdot \nabla^{\varphi^{\e}}\hat{\eta}
- \hat{v}\cdot\nabla^{\varphi} \omega_h , \\[11pt]

\hat{\omega}_h^{nhom}|_{z=0} =0, \\[7pt]

\hat{\omega}_h^{nhom}|_{t=0} = (\hat{\omega}_0^1, \hat{\omega}_0^2)^{\top},
\end{array}\right.
\end{equation}
and $\hat{\omega}_h^{hom}$ satisfies the following homogeneous equations:
\begin{equation}\label{Sect5_N_Derivatives_Difference_Eq_Hom}
\left\{\begin{array}{ll}
\partial_t^{\varphi^{\e}} \hat{\omega}_h^{hom}
+ v^{\e} \cdot\nabla^{\varphi^{\e}} \hat{\omega}_h^{hom}
- \e\triangle^{\varphi^{\e}}\hat{\omega}_h^{hom} =0, \\[9pt]

\hat{\omega}_h^{hom}|_{z=0} =\textsf{F}^{1,2} [\nabla\varphi^{\e}](\partial_j v^{\e,i}) - \omega^b, \\[7pt]

\hat{\omega}_h^{hom}|_{t=0} = 0,
\end{array}\right.
\end{equation}

By using $(\ref{Sect3_Preliminaries_Vorticity_Eq_2})$ and $\partial_t^{\varphi^{\e}} + v^{\e} \cdot\nabla^{\varphi^{\e}}
= \partial_t + v_y^{\e}\cdot\nabla_y + V_z^{\e}\partial_z$,
$(\ref{Sect5_N_Derivatives_Difference_Eq_Nonhom})$ is equivalent to the following equations:
\begin{equation}\label{Sect5_N_Derivatives_Difference_Eq_Nonhom_1}
\left\{\begin{array}{ll}
\partial_t \hat{\omega}_h^{nhom} + v_y^{\e}\cdot\nabla_y\hat{\omega}_h^{nhom}
+ V_z^{\e}\partial_z \hat{\omega}_h^{nhom}
- \e\triangle^{\varphi^{\e}}\hat{\omega}_h^{nhom} \\[8pt]\quad
= f^7[\nabla\varphi^{\e},\nabla\varphi,\partial_j v^{\e,i},\partial_j v^i]\hat{\omega}_h
+ f^8[\nabla\varphi^{\e},\nabla\varphi,\partial_j v^{\e,i},\partial_j v^i,\omega_h^{\e},\omega_h]\partial_j\hat{v}^i \\[7pt]\qquad
+ f^9[\nabla\varphi^{\e},\nabla\varphi,\partial_j v^{\e,i},\partial_j v^i,\omega_h^{\e},\omega_h]\nabla\hat{\varphi}
+ \e\triangle^{\varphi^{\e}}\omega_h \\[7pt]\qquad
+ \partial_z^{\varphi}\omega_h \partial_t^{\varphi^{\e}} \hat{\eta}
+ \partial_z^{\varphi} \omega_h\, v^{\e}\cdot \nabla^{\varphi^{\e}}\hat{\eta}
- \hat{v}\cdot\nabla^{\varphi} \omega_h := \mathcal{I}_6, \\[9pt]

\hat{\omega}_h^{nhom}|_{z=0} =0, \\[7pt]

\hat{\omega}_h^{nhom}|_{t=0} = (\hat{\omega}_0^1, \hat{\omega}_0^2)^{\top},
\end{array}\right.
\end{equation}

Apply $\partial_t^{\ell}\mathcal{Z}^{\alpha}$ to $(\ref{Sect5_N_Derivatives_Difference_Eq_Nonhom_1})$, we get
\begin{equation}\label{Sect5_N_Derivatives_Difference_Eq_Nonhom_2}
\left\{\begin{array}{ll}
\partial_t \partial_t^{\ell}\mathcal{Z}^{\alpha}\hat{\omega}_h^{nhom} + v_y^{\e}\cdot\nabla_y \partial_t^{\ell}\mathcal{Z}^{\alpha}\hat{\omega}_h^{nhom}
+ V_z^{\e}\partial_z \partial_t^{\ell}\mathcal{Z}^{\alpha}\hat{\omega}_h^{nhom}
- \e\triangle^{\varphi^{\e}}\partial_t^{\ell}\mathcal{Z}^{\alpha}\hat{\omega}_h^{nhom} \\[8pt]\quad
= \partial_t^{\ell}\mathcal{Z}^{\alpha} \mathcal{I}_6
-[\partial_t^{\ell}\mathcal{Z}^{\alpha}, v_y^{\e}\cdot\nabla_y]\hat{\omega}_h^{nhom}
-[\partial_t^{\ell}\mathcal{Z}^{\alpha}, V_z\partial_z]\hat{\omega}_h^{nhom} \\[8pt]\qquad

+ \e\nabla^{\varphi^{\e}}\cdot [\partial_t^{\ell}\mathcal{Z}^{\alpha}, \nabla^{\varphi}]\hat{\omega}_h^{nhom}
+ \e[\partial_t^{\ell}\mathcal{Z}^{\alpha}, \nabla^{\varphi}\cdot]\nabla^{\varphi^{\e}}\hat{\omega}_h^{nhom}, \\[9pt]

\partial_t^{\ell}\mathcal{Z}^{\alpha}\hat{\omega}_h^{nhom}|_{z=0} =0, \\[7pt]

\partial_t^{\ell}\mathcal{Z}^{\alpha}\hat{\omega}_h^{nhom}|_{t=0}
= (\partial_t^{\ell}\mathcal{Z}^{\alpha}\hat{\omega}_0^1, \partial_t^{\ell}\mathcal{Z}^{\alpha}\hat{\omega}_0^2)^{\top},
\end{array}\right.
\end{equation}

Develop the $L^2$ estimate of $\partial_t^{\ell}\mathcal{Z}^{\alpha}\hat{\omega}_h^{nhom}$, we get
\begin{equation}\label{Sect5_N_Derivatives_Difference_Eq_Nonhom_3}
\begin{array}{ll}
\frac{\mathrm{d}}{\mathrm{d}t} \|\partial_t^{\ell}\mathcal{Z}^{\alpha}\hat{\omega}_h^{nhom}\|_{L^2}^2
+ 2\e \|\nabla^{\varphi^{\e}}\partial_t^{\ell}\mathcal{Z}^{\alpha}\hat{\omega}_h^{nhom}\|_{L^2}^2 \\[10pt]
\lem \|\partial_t^{\ell}\mathcal{Z}^{\alpha}\hat{\omega}_h^{nhom}\|_{L^2}^2
+ \|\partial_t^{\ell}\mathcal{Z}^{\alpha} \mathcal{I}_6\|_{L^2}^2
+ \|[\partial_t^{\ell}\mathcal{Z}^{\alpha}, V_z\partial_z]\hat{\omega}_h^{nhom}\|_{L^2}^2 \\[8pt]\quad

+ \e \int\limits_{\mathbb{R}^3_{-}}\nabla^{\varphi^{\e}}\cdot [\partial_t^{\ell}\mathcal{Z}^{\alpha}, \nabla^{\varphi}]\hat{\omega}_h^{nhom}
\partial_t^{\ell}\mathcal{Z}^{\alpha}\hat{\omega}_h^{nhom} \,\mathrm{d}\mathcal{V}_t \\[8pt]\quad
+ \e \int\limits_{\mathbb{R}^3_{-}}[\partial_t^{\ell}\mathcal{Z}^{\alpha}, \nabla^{\varphi}\cdot]\nabla^{\varphi^{\e}}\hat{\omega}_h^{nhom}
\partial_t^{\ell}\mathcal{Z}^{\alpha}\hat{\omega}_h^{nhom} \,\mathrm{d}\mathcal{V}_t 
\hspace{3.5cm}
\end{array}
\end{equation}

\begin{equation*}
\begin{array}{ll}
\lem \|\partial_t^{\ell}\mathcal{Z}^{\alpha}\hat{\omega}_h^{nhom}\|_{L^2}^2
+ \|\hat{\omega}_h\|_{X^{k-1}}^2 + \|\hat{v}\|_{X^{k-1,1}}^2 + \|\nabla\hat{\eta}\|_{X^{k-1}}^2
+ \|\partial_t^k\hat{\eta}\|_{L^2}^2 \\[8pt]\
+ \sum\limits_{\ell_1+|\alpha_1|>0}
\|\frac{1-z}{z}\partial_t^{\ell}\mathcal{Z}^{\alpha} V_z \cdot \partial_t^{\ell}\mathcal{Z}^{\alpha} \frac{z}{1-z}\partial_z\hat{\omega}_h^{nhom}\|_{L^2}^2

\\[8pt]\
- \e \int\limits_{\mathbb{R}^3_{-}}[\partial_t^{\ell}\mathcal{Z}^{\alpha}, \NN\partial_z^{\varphi}]\hat{\omega}_h^{nhom}
\cdot \nabla^{\varphi^{\e}}\partial_t^{\ell}\mathcal{Z}^{\alpha}\hat{\omega}_h^{nhom} \,\mathrm{d}\mathcal{V}_t \\[8pt]\
+ \e \int\limits_{\mathbb{R}^3_{-}} \sum\limits_{\ell_1+|\alpha_1|>0}\big[(\partial_z^{\varphi})^{-1}
\partial_t^{\ell_1}\mathcal{Z}^{\alpha_1}(\frac{\NN}{\partial_z\varphi})
\partial_t^{\ell_2}\mathcal{Z}^{\alpha_2}\partial_z\big]\cdot
\nabla^{\varphi^{\e}}\hat{\omega}_h^{nhom} \,
\partial_z^{\varphi}\partial_t^{\ell}\mathcal{Z}^{\alpha}\hat{\omega}_h^{nhom} \,\mathrm{d}\mathcal{V}_t.
\end{array}
\end{equation*}
where the notation $(\partial_z^{\varphi})^{-1}$ means a cancellation such that $(\partial_z^{\varphi})^{-1}(\partial_z^{\varphi}) =1$.

Integrate $(\ref{Sect5_N_Derivatives_Difference_Eq_Nonhom_3})$ in time, apply the integral form of Gronwall's inequality, it is easy to have
\begin{equation}\label{Sect5_N_Derivatives_Difference_Eq_Nonhom_4}
\begin{array}{ll}
\|\hat{\omega}_h^{nhom}\|_{X^{k-1}}^2
+ 2\e \int\limits_0^t\|\nabla\hat{\omega}_h^{nhom}\|_{X^{k-1}}^2 \,\mathrm{d}t \\[7pt]

\leq \|\hat{\omega}_{0,h}\|_{X^{k-1}}^2 + \int\limits_0^t\|\hat{\omega}_h\|_{X^{k-1}}^2\,\mathrm{d}t
+ \|\hat{h}\|_{X^{k-1,1}}^2\,\mathrm{d}t + \|\partial_t^k\hat{h}\|_{L^2}^2\,\mathrm{d}t + O(\e).
\end{array}
\end{equation}

Similar to $(\ref{Sect2_Vorticity_Estimate_10})$, we have
\begin{equation}\label{Sect5_N_Derivatives_Difference_Eq_Nonhom_5}
\begin{array}{ll}
\|\hat{\omega}_h^{nhom}\|_{L^4([0,T],X^{k-1})}^2
\lem \sqrt{T}\big\|\hat{\omega}_{0,h}\big\|_{X^{k-1}}^2 + T \|\hat{\omega}_h\|_{L^4([0,T],X^{k-1})}^2 \\[9pt]\quad

+ \sqrt{T}\int\limits_0^T\|\hat{v}\|_{X^{k-1,1}}^2 \,\mathrm{d}t + \sqrt{T}\int\limits_0^T|\hat{h}|_{X^{k-1,1}}^2 \,\mathrm{d}t
+ \sqrt{T}|\partial_t^k \hat{h}|_{L^4([0,T],L^2)}^2 + O(\e).
\end{array}
\end{equation}

For the homogeneous equations $(\ref{Sect5_N_Derivatives_Difference_Eq_Hom})$,
similar to the estimates of the equations $(\ref{Sect2_Vorticity_Estimate_3})$
or \cite{Masmoudi_Rousset_2012_FreeBC}, we have
\begin{equation}\label{Sect5_N_Derivatives_Difference_Eq_Hom_1}
\begin{array}{ll}
\|\partial_t^{\ell}\mathcal{Z}^{\alpha}\hat{\omega}_h^{hom}\|_{L^4([0,T],L^2(\mathbb{R}^3_{-}))}^2
\lem \|\partial_t^{\ell}\mathcal{Z}^{\alpha}\hat{\omega}_h^{hom}\|_{H^{\frac{1}{4}}([0,T],L^2(\mathbb{R}^3_{-}))}^2 \\[10pt]
\lem \sqrt{\e}\int\limits_0^T\big|\hat{\omega}_h^{hom}|_{z=0}\big|_{X^{k-1}(\mathbb{R}^2)}^2 \,\mathrm{d}t \\[10pt]

\lem \sqrt{\e}\int\limits_0^T\big|\textsf{F}^{1,2}[\nabla\varphi^{\e}](\partial_j v^{\e,i}) -\textsf{F}^{1,2}[\nabla\varphi](\partial_j v^i)
\big|_{X^{k-1}(\mathbb{R}^2)}^2 \,\mathrm{d}t \\[8pt]\quad
+ \sqrt{\e}\int\limits_0^T\big|\varsigma_1\Theta^1 + \varsigma_2\Theta^2 + \varsigma_3\Theta^3\big|_{X^{k-1}(\mathbb{R}^2)}^2 \,\mathrm{d}t \\[8pt]\quad
+ \sqrt{\e}\int\limits_0^T\big|\varsigma_4\Theta^4 + \varsigma_5\Theta^5 + \varsigma_6\Theta^6\big|_{X^{k-1}(\mathbb{R}^2)}^2 \,\mathrm{d}t

\lem O(\sqrt{\e}),
\end{array}
\end{equation}
where $\varsigma_i$ and $\Theta^i$ are defined in the proof of Lemma $\ref{Sect4_Vorticity_Discrepancy_Lemma}$.

By $(\ref{Sect5_N_Derivatives_Difference_Eq_Nonhom_5})$ and $(\ref{Sect5_N_Derivatives_Difference_Eq_Hom_1})$,
we have
\begin{equation}\label{Sect5_NormalEstimates_1}
\begin{array}{ll}
\|\hat{\omega}_h\|_{L^4([0,T],X^{k-1})}^2
\lem \|\hat{\omega}_h^{nhom}\|_{L^4([0,T],X^{k-1})}^2
+ \|\hat{\omega}_h^{hom}\|_{L^4([0,T],X^{k-1})}^2 \\[9pt]

\lem \big\|\hat{\omega}_{0,h}\big\|_{X^{k-1}}^2
+ |\partial_t^k\hat{h}|_{L^4([0,T],L^2)}^2
+ \int\limits_0^T\|\hat{v}\|_{X^{k-1,1}}^2 \,\mathrm{d}t
+ \int\limits_0^T|\hat{h}|_{X^{k-1,1}}^2 \,\mathrm{d}t
+ O(\sqrt{\e}).
\end{array}
\end{equation}
Thus, Lemma $\ref{Sect5_Vorticity_Lemma}$ is proved.
\end{proof}

\begin{remark}\label{Sect5_Euler_BoundaryData_Remark}
If $\Pi\mathcal{S}^{\varphi} v\nn|_{z=0} = 0$, then $\Theta^i =0$ where $i=1,\cdots,6$,
and then the estimate $(\ref{Sect5_N_Derivatives_Difference_Eq_Hom_1})$ is reduced into the following estimate:
\begin{equation}\label{Sect5_Euler_BoundaryData_Remark_Estimate}
\begin{array}{ll}
\|\partial_t^{\ell}\mathcal{Z}^{\alpha}\hat{\omega}_h^{hom}\|_{L^4([0,T],L^2(\mathbb{R}^3_{-}))}^2 \\[7pt]

\lem \sqrt{\e}\int\limits_0^T\big|\textsf{F}^{1,2}[\nabla\varphi^{\e}](\partial_j v^{\e,i}) -\textsf{F}^{1,2}[\nabla\varphi](\partial_j v^i)
\big|_{X^{k-1}(\mathbb{R}^2)}^2 \,\mathrm{d}t

\lem O(\sqrt{\e}),
\end{array}
\end{equation}
Since we do not have the convergence rates of $|\partial_j v^{\e,i} - \partial_j v^i|_{X^{k-1}(\mathbb{R}^2)}$,
thus we can not improve the convergence rates of $\|\omega\|_{L^4([0,T],X^{k-1})}^2$.
However, we can improve the convergence rates of $\|\omega\|_{L^4([0,T],X^{k-2})}^2$,
see subsection 5.4.
\end{remark}

If $\Pi\mathcal{S}^{\varphi}v\nn|_{z=0} \neq 0$, we estimate $\|\partial_z \hat{v}\|_{L^{\infty}([0,T],X^{m-4})}$ and $\|\hat{\omega}\|_{L^{\infty}([0,T],X^{m-4})}$.
Note that when $\Pi\mathcal{S}^{\varphi}v\nn|_{z=0} \neq 0$,
not only $\big|\nabla^{\varphi^{\e}} \times \partial_t^{\ell}\mathcal{Z}^{\alpha}(v^{\e} -v)|_{z=0}\big|_{L^2} \neq 0$ but also
$\big|\NN^{\e}\times(\nabla^{\varphi^{\e}} \times \partial_t^{\ell}\mathcal{Z}^{\alpha}(v^{\e} -v))|_{z=0}\big|_{L^2} \neq 0$.
\begin{lemma}\label{Sect5_NormalDer_Lemma}
Assume $0\leq k\leq m-2$,
 $\hat{\omega}_h =\omega_h^{\e} -\omega_h$, $\partial_z\hat{v} =\partial_z v^{\e} -\partial_z v$, then
$\hat{\omega}_h$ and $\partial_z\hat{v}$ satisfy the following estimate:
\begin{equation}\label{Sect5_NormalDer_Estimate}
\begin{array}{ll}
\|\hat{\omega}\|_{X^{k-2}}^2 + \|\partial_z\hat{v}\|_{X^{k-2}}^2

\lem \|\hat{\omega}_0\|_{X^{k-2}}^2

+ \int\limits_0^t\|\hat{v}\|_{X^{k-2}} + \|\partial_z \hat{v}\|_{X^{k-2}} \\[6pt]\quad
+ \|\nabla\hat{q} \|_{X^{k-2}} + \|\hat{h}\|_{X^{k-1}}\,\mathrm{d}t + O(\e).
\end{array}
\end{equation}
\end{lemma}

\begin{proof}
By using $(\ref{Sect2_NormalDer_Estimate_Laplacian})$, we rewrite $(\ref{Sect1_T_Derivatives_Difference_Eq})_1$ as
\begin{equation}\label{Sect5_NormalDer_Estimate_L2_1}
\begin{array}{ll}
\partial_t^{\varphi^{\e}}\hat{v}-\partial_z^{\varphi} v \partial_t^{\varphi^{\e}}\hat{\eta}
+ v^{\e} \cdot\nabla^{\varphi^{\e}} \hat{v} - v^{\e}\cdot \nabla^{\varphi^{\e}}\hat{\eta}\, \partial_z^{\varphi} v + \hat{v}\cdot\nabla^{\varphi} v \\[7pt]\quad
+ \nabla^{\varphi^{\e}} \hat{q} - \partial_z^{\varphi} q\nabla^{\varphi^{\e}}\hat{\eta}
= -\e\nabla^{\varphi^{\e}}\times \hat{\omega} -\e\nabla^{\varphi^{\e}}\times \omega
\end{array}
\end{equation}

Firstly, we develop $L^2$ estimate of $\hat{\omega}$. Multiple $(\ref{Sect5_NormalDer_Estimate_L2_1})$
with $\nabla^{\varphi^{\e}}\times(\nabla^{\varphi^{\e}}\times\hat{v})$, integrate in $\mathbb{R}^3_{-}$, use the integration by parts formula $(\ref{Sect1_Formulas_CanNotUse})_3$,
we get
\begin{equation}\label{Sect5_NormalDer_Estimate_L2_2}
\begin{array}{ll}
\int\limits_{\mathbb{R}^3_{-}}
\big(\partial_t^{\varphi^{\e}}\hat{v}-\partial_z^{\varphi} v \partial_t^{\varphi^{\e}}\hat{\eta}
+ v^{\e} \cdot\nabla^{\varphi^{\e}} \hat{v} - v^{\e}\cdot \nabla^{\varphi^{\e}}\hat{\eta}\, \partial_z^{\varphi} v + \hat{v}\cdot\nabla^{\varphi} v
\\[7pt]\
+ \nabla^{\varphi^{\e}} \hat{q} - \partial_z^{\varphi} q\nabla^{\varphi^{\e}}\hat{\eta}
+ \e\nabla^{\varphi^{\e}}\times (\nabla^{\varphi^{\e}}\times\hat{v}) +\e\triangle^{\varphi^{\e}}v \big)
\cdot \nabla^{\varphi^{\e}}\times(\nabla^{\varphi^{\e}}\times\hat{v}) \,\mathrm{d}\mathcal{V}_t =0, \\[13pt]

\int\limits_{\mathbb{R}^3_{-}}
\nabla^{\varphi^{\e}}\times\big(\partial_t^{\varphi^{\e}}\hat{v}
+ v^{\e} \cdot\nabla^{\varphi^{\e}} \hat{v} + \nabla^{\varphi^{\e}} \hat{q}\big) \,\mathrm{d}\mathcal{V}_t

+ \e\int\limits_{\mathbb{R}^3_{-}} |\nabla^{\varphi^{\e}}\times(\nabla^{\varphi^{\e}}\times\hat{v})|^2 \,\mathrm{d}\mathcal{V}_t\\[8pt]\quad

= \int\limits_{\mathbb{R}^3_{-}}
\nabla^{\varphi^{\e}}\times\big(\partial_z^{\varphi} v \partial_t^{\varphi^{\e}}\hat{\eta}
+ v^{\e}\cdot \nabla^{\varphi^{\e}}\hat{\eta}\, \partial_z^{\varphi} v - \hat{v}\cdot\nabla^{\varphi} v
+ \partial_z^{\varphi} q\nabla^{\varphi^{\e}}\hat{\eta}\big)\cdot \hat{\omega} \,\mathrm{d}\mathcal{V}_t \\[8pt]\qquad

- \int\limits_{z=0}
\big(\partial_t\hat{v} + v^{\e}_y \cdot\nabla_y \hat{v} + \hat{v}\cdot\nabla^{\varphi} v
-\partial_z^{\varphi} v \partial_t\hat{\eta}
- v^{\e}_y\cdot \nabla_y \hat{\eta}\, \partial_z^{\varphi} v

+ \nabla^{\varphi^{\e}} \hat{q} \\[11pt]\qquad
- \partial_z^{\varphi} q\nabla^{\varphi^{\e}}\hat{\eta}
\big)\cdot \NN^{\e}\times (\nabla^{\varphi^{\e}}\times\hat{v}) \,\mathrm{d}y

- \e\int\limits_{\mathbb{R}^3_{-}}
\nabla^{\varphi^{\e}}\times \omega \cdot \nabla^{\varphi^{\e}}\times (\nabla^{\varphi^{\e}}\times\hat{v}) \,\mathrm{d}\mathcal{V}_t \\[8pt]\quad

\lem \|\hat{\omega}\|_{L^2}^2 + |\hat{h}|_{X^{1,\frac{1}{2}}}^2 + \|\hat{v}\|_{X^1}^2 + \|\partial_z \hat{v}\|_{L^2}^2
+ \frac{\e}{2}\int\limits_{\mathbb{R}^3_{-}} |\nabla^{\varphi^{\e}}\times\hat{\omega}|^2 \,\mathrm{d}\mathcal{V}_t + O(\e)\\[8pt]\qquad
+ \big|\NN^{\e}\times (\nabla^{\varphi^{\e}}\times\hat{v})|_{z=0}\big|_{L^2} \big(\big|\hat{v}|_{z=0}\big|_{X_{tan}^1}
+ \big|\hat{h}|_{z=0}\big|_{X^1} \big) \\[8pt]\qquad
+ \big|\NN^{\e}\times (\nabla^{\varphi^{\e}}\times\hat{v})|_{z=0}\big|_{\frac{1}{2}} \big|\nabla \hat{q}|_{z=0}\big|_{-\frac{1}{2}}.
\end{array}
\end{equation}

Since $\nabla^{\varphi^{\e}}\times \nabla^{\varphi^{\e}} \hat{q} =0$, $\nabla^{\varphi^{\e}}\times \omega$ is bounded,
$\big|\NN^{\e}\times (\nabla^{\varphi^{\e}}\times\hat{v})|_{z=0}\big|_{L^2}\neq 0$,
\begin{equation}\label{Sect5_NormalDer_Estimate_L2_3}
\begin{array}{ll}
\|\hat{\omega}|_{L^2}
+ \e\int\limits_0^t \|\nabla\hat{\omega}\|^2 \,\mathrm{d}t
\lem \|\hat{\omega}_0\|_{L^2}^2 + \int\limits_0^t |\hat{h}|_{X^{1,1}}^2 + \|\hat{v}\|_{X^1}^2 + \|\partial_z \hat{v}\|_{L^2}^2
\\[8pt]\quad
+ \big|\hat{v}|_{z=0}\big|_{X_{tan}^1} + \big|\hat{h}|_{z=0}\big|_{X^1}
+ \big|\nabla \hat{q}|_{z=0}\big|_{-\frac{1}{2}} \,\mathrm{d}t  + O(\e) \\[9pt]

\lem \|\hat{\omega}_0\|_{L^2}^2 + \int\limits_0^t |\hat{h}|_{X^{1,1}}^2 + \|\hat{v}\|_{X^1}^2 + \|\partial_z \hat{v}\|_{L^2}^2
+ \|\hat{v}\|_{X_{tan}^1} + \|\partial_z\hat{v}\|_{X_{tan}^1} \\[8pt]\quad
+ \|\hat{h}\|_{X^1} + \|\nabla \hat{q}\|_{L^2} \,\mathrm{d}t  + O(\e).
\end{array}
\end{equation}

When $\ell+|\alpha|\leq k-2$, we study the quantity $\nabla^{\varphi^{\e}}\times \partial_t^{\ell}\mathcal{Z}^{\alpha}\hat{v}$.

The equations $(\ref{SectA_Difference_Eq2_1})_2$ is rewritten as
\begin{equation}\label{Sect5_NormalDer_Estimate_1}
\begin{array}{ll}
\partial_t^{\varphi^{\e}} \partial_t^{\ell}\mathcal{Z}^{\alpha}\hat{v}
+ v^{\e} \cdot\nabla^{\varphi^{\e}} \partial_t^{\ell}\mathcal{Z}^{\alpha}\hat{v}
+ \nabla^{\varphi^{\e}} \partial_t^{\ell}\mathcal{Z}^{\alpha}\hat{q}
+ \e\nabla^{\varphi^\e}\times \nabla^{\varphi^\e} \partial_t^{\ell}\mathcal{Z}^{\alpha} \hat{v}

= \e \, \mathcal{I}_{7,1} + \mathcal{I}_{7,2},
\end{array}
\end{equation}
where
\begin{equation}\label{Sect5_NormalDer_Estimate_2}
\begin{array}{ll}
\mathcal{I}_{7,1} = -[\partial_t^{\ell}\mathcal{Z}^{\alpha}, \nabla^{\varphi^\e}\times]\nabla^{\varphi^\e}\times\hat{v}
- \nabla^{\varphi^\e}\times[\partial_t^{\ell}\mathcal{Z}^{\alpha}, \nabla^{\varphi^\e}\times] \hat{v}
+ \partial_t^{\ell}\mathcal{Z}^{\alpha}\triangle^{\varphi^{\e}}v,
\\[9pt]

\mathcal{I}_{7,2} := \partial_z^{\varphi} v (\partial_t + v^{\e}_y\cdot \nabla_y + V_z^{\e}\partial_z)
\partial_t^{\ell}\mathcal{Z}^{\alpha}\hat{\varphi}
- \partial_t^{\ell}\mathcal{Z}^{\alpha}\hat{v}\cdot\nabla^{\varphi} v
+ \partial_z^{\varphi} q\nabla^{\varphi^{\e}}\partial_t^{\ell}\mathcal{Z}^{\alpha}\hat{\varphi} \\[8pt]\quad

- [\partial_t^{\ell}\mathcal{Z}^{\alpha},\partial_t + v^{\e}\partial_y + V_z^{\e}\partial_z]\hat{v}
+ [\partial_t^{\ell}\mathcal{Z}^{\alpha}, \partial_z^{\varphi} v (\partial_t + v^{\e}_y\cdot \nabla_y + V_z^{\e}\partial_z]\hat{\varphi} \\[8pt]\quad

- [\partial_t^{\ell}\mathcal{Z}^{\alpha}, \nabla^{\varphi} v\cdot]\hat{v}
- [\partial_t^{\ell}\mathcal{Z}^{\alpha},\nabla^{\varphi^{\e}}] \hat{q}
+ [\partial_t^{\ell}\mathcal{Z}^{\alpha},\partial_z^{\varphi} q\nabla^{\varphi^{\e}}]\hat{\varphi}.
\end{array}
\end{equation}

Multiply $(\ref{Sect5_NormalDer_Estimate_1})$ with $\nabla^{\varphi^{\e}}\times (\nabla^{\varphi^{\e}}\times
\partial_t^{\ell}\mathcal{Z}^{\alpha}\hat{v})$, integrate in $\mathbb{R}^3_{-}$, we get
\begin{equation}\label{Sect5_NormalDer_Estimate_3}
\begin{array}{ll}
\int\limits_{\mathbb{R}^3_{-}}\partial_t^{\varphi^{\e}} \partial_t^{\ell}\mathcal{Z}^{\alpha}\hat{v} \cdot
\nabla^{\varphi^{\e}}\times (\nabla^{\varphi^{\e}}\times
\partial_t^{\ell}\mathcal{Z}^{\alpha}\hat{v}) \,\mathrm{d}\mathcal{V}_t^{\e} \\[10pt]\quad

+ \int\limits_{\mathbb{R}^3_{-}}v^{\e} \cdot\nabla^{\varphi^{\e}} \partial_t^{\ell}\mathcal{Z}^{\alpha}\hat{v} \cdot
\nabla^{\varphi^{\e}}\times (\nabla^{\varphi^{\e}}\times
\partial_t^{\ell}\mathcal{Z}^{\alpha}\hat{v}) \,\mathrm{d}\mathcal{V}_t^{\e} \\[10pt]\quad

+ \int\limits_{\mathbb{R}^3_{-}}\nabla^{\varphi^{\e}} \partial_t^{\ell}\mathcal{Z}^{\alpha}\hat{q} \cdot
\nabla^{\varphi^{\e}}\times (\nabla^{\varphi^{\e}}\times
\partial_t^{\ell}\mathcal{Z}^{\alpha}\hat{v}) \,\mathrm{d}\mathcal{V}_t^{\e} \\[11pt]\quad

+ \e\int\limits_{\mathbb{R}^3_{-}}|\nabla^{\varphi^\e}\times \nabla^{\varphi^\e} \partial_t^{\ell}\mathcal{Z}^{\alpha} \hat{v}|^2
\,\mathrm{d}\mathcal{V}_t^{\e} \\[11pt]

= \int\limits_{\mathbb{R}^3_{-}}(\e \, \mathcal{I}_{7,1} + \mathcal{I}_{7,2}) \cdot \nabla^{\varphi^{\e}}\times (\nabla^{\varphi^{\e}}\times
\partial_t^{\ell}\mathcal{Z}^{\alpha}\hat{v}) \,\mathrm{d}\mathcal{V}_t^{\e}.
\end{array}
\end{equation}

Use the integration by parts formula $(\ref{Sect1_Formulas_CanNotUse})_3$ and note that
$[\partial_t^{\varphi^{\e}}, \nabla^{\varphi^{\e}}] =0$, we have
\begin{equation}\label{Sect5_NormalDer_Estimate_4}
\begin{array}{ll}
\int\limits_{z=0}\partial_t^{\varphi^{\e}} \partial_t^{\ell}\mathcal{Z}^{\alpha}\hat{v} \cdot
\NN^{\e}\times (\nabla^{\varphi^{\e}}\times
\partial_t^{\ell}\mathcal{Z}^{\alpha}\hat{v}) \,\mathrm{d}y \\[10pt]\quad
+ \int\limits_{\mathbb{R}^3_{-}}\partial_t^{\varphi^{\e}} (\nabla^{\varphi^{\e}}\times
\partial_t^{\ell}\mathcal{Z}^{\alpha}\hat{v}) \cdot (\nabla^{\varphi^{\e}}\times
\partial_t^{\ell}\mathcal{Z}^{\alpha}\hat{v}) \,\mathrm{d}\mathcal{V}_t^{\e} \\[10pt]\quad

+ \int\limits_{z=0}v^{\e} \cdot\nabla^{\varphi^{\e}} \partial_t^{\ell}\mathcal{Z}^{\alpha}\hat{v} \cdot
\NN^{\e}\times (\nabla^{\varphi^{\e}}\times
\partial_t^{\ell}\mathcal{Z}^{\alpha}\hat{v}) \,\mathrm{d}y \\[10pt]\quad

+ \int\limits_{\mathbb{R}^3_{-}}
v^{\e} \cdot\nabla^{\varphi^{\e}} (\nabla^{\varphi^{\e}}\times \partial_t^{\ell}\mathcal{Z}^{\alpha}\hat{v})
\cdot (\nabla^{\varphi^{\e}}\times
\partial_t^{\ell}\mathcal{Z}^{\alpha}\hat{v}) \,\mathrm{d}\mathcal{V}_t^{\e} \hspace{2.5cm}
\end{array}
\end{equation}

\begin{equation*}
\begin{array}{ll}
\quad 
+ \int\limits_{\mathbb{R}^3_{-}} [(\sum\limits_{i=1}^3 \nabla^{\varphi^{\e}}
v^{\e,i} \cdot\partial_i^{\varphi^{\e}}) \times \partial_t^{\ell}\mathcal{Z}^{\alpha}\hat{v}]
\cdot (\nabla^{\varphi^{\e}}\times
\partial_t^{\ell}\mathcal{Z}^{\alpha}\hat{v}) \,\mathrm{d}\mathcal{V}_t^{\e} \\[10pt]\quad

+ \int\limits_{z=0}\nabla^{\varphi^{\e}} \partial_t^{\ell}\mathcal{Z}^{\alpha}\hat{q} \cdot
\NN^{\e}\times (\nabla^{\varphi^{\e}}\times
\partial_t^{\ell}\mathcal{Z}^{\alpha}\hat{v}) \,\mathrm{d}y \\[10pt]\quad
+ \int\limits_{\mathbb{R}^3_{-}}\nabla^{\varphi^{\e}}\times \nabla^{\varphi^{\e}} \partial_t^{\ell}\mathcal{Z}^{\alpha}\hat{q}
\cdot (\nabla^{\varphi^{\e}}\times
\partial_t^{\ell}\mathcal{Z}^{\alpha}\hat{v}) \,\mathrm{d}\mathcal{V}_t^{\e} \\[11pt]\quad

+ \e\int\limits_{\mathbb{R}^3_{-}}|\nabla^{\varphi^\e}\times \nabla^{\varphi^\e} \partial_t^{\ell}\mathcal{Z}^{\alpha} \hat{v}|^2
\,\mathrm{d}\mathcal{V}_t^{\e}

= \e\int\limits_{\mathbb{R}^3_{-}}\mathcal{I}_{7,1} \cdot \nabla^{\varphi^{\e}}\times (\nabla^{\varphi^{\e}}\times
\partial_t^{\ell}\mathcal{Z}^{\alpha}\hat{v}) \,\mathrm{d}\mathcal{V}_t^{\e} \\[11pt]\quad

+ \int\limits_{z=0}\mathcal{I}_{7,2} \cdot \NN^{\e}\times (\nabla^{\varphi^{\e}}\times
\partial_t^{\ell}\mathcal{Z}^{\alpha}\hat{v}) \,\mathrm{d}y
+ \int\limits_{\mathbb{R}^3_{-}} \nabla^{\varphi^{\e}}\times \mathcal{I}_{7,2} \cdot (\nabla^{\varphi^{\e}}\times
\partial_t^{\ell}\mathcal{Z}^{\alpha}\hat{v}) \,\mathrm{d}\mathcal{V}_t^{\e}.
\end{array}
\end{equation*}

Note that $\nabla^{\varphi^{\e}}\times \nabla^{\varphi^{\e}} \partial_t^{\ell}\mathcal{Z}^{\alpha}\hat{q} =0$
and $(\partial_t^{\varphi^{\e}} + v^{\e} \cdot\nabla^{\varphi^{\e}})|_{z=0} = (\partial_t + v_y^{\e}\cdot\nabla_y)$, we have
\begin{equation}\label{Sect5_NormalDer_Estimate_5}
\begin{array}{ll}
\frac{1}{2}\frac{\mathrm{d}}{\mathrm{d}t} \int\limits_{\mathbb{R}^3_{-}}
|\nabla^{\varphi^{\e}}\times\partial_t^{\ell}\mathcal{Z}^{\alpha}\hat{v}|^2 \,\mathrm{d}\mathcal{V}_t^{\e}

+ \e\int\limits_{\mathbb{R}^3_{-}}|\nabla^{\varphi^\e}\times \nabla^{\varphi^\e} \partial_t^{\ell}\mathcal{Z}^{\alpha} \hat{v}|^2
\,\mathrm{d}\mathcal{V}_t^{\e}
 \\[10pt]

= - \int\limits_{z=0}(\partial_t + v_y^{\e}\cdot\nabla_y) \partial_t^{\ell}\mathcal{Z}^{\alpha}\hat{v} \cdot
\NN^{\e}\times (\nabla^{\varphi^{\e}}\times
\partial_t^{\ell}\mathcal{Z}^{\alpha}\hat{v}) \,\mathrm{d}y \\[10pt]\quad

- \int\limits_{\mathbb{R}^3_{-}} [(\sum\limits_{i=1}^3 \nabla^{\varphi^{\e}}
v^{\e,i} \cdot\partial_i^{\varphi^{\e}}) \times \partial_t^{\ell}\mathcal{Z}^{\alpha}\hat{v}]
\cdot (\nabla^{\varphi^{\e}}\times
\partial_t^{\ell}\mathcal{Z}^{\alpha}\hat{v}) \,\mathrm{d}\mathcal{V}_t^{\e} \\[10pt]\quad

- \int\limits_{z=0}\nabla^{\varphi^{\e}} \partial_t^{\ell}\mathcal{Z}^{\alpha}\hat{q} \cdot
\NN^{\e}\times (\nabla^{\varphi^{\e}}\times
\partial_t^{\ell}\mathcal{Z}^{\alpha}\hat{v}) \,\mathrm{d}y \\[10pt]\quad

+ \e\int\limits_{\mathbb{R}^3_{-}}\mathcal{I}_{7,1} \cdot \nabla^{\varphi^{\e}}\times (\nabla^{\varphi^{\e}}\times
\partial_t^{\ell}\mathcal{Z}^{\alpha}\hat{v}) \,\mathrm{d}\mathcal{V}_t^{\e}

+ \int\limits_{z=0}\mathcal{I}_{7,2} \cdot \NN^{\e}\times (\nabla^{\varphi^{\e}}\times
\partial_t^{\ell}\mathcal{Z}^{\alpha}\hat{v}) \,\mathrm{d}y \\[11pt]\quad
+ \int\limits_{\mathbb{R}^3_{-}} \nabla^{\varphi^{\e}}\times \mathcal{I}_{7,2} \cdot (\nabla^{\varphi^{\e}}\times
\partial_t^{\ell}\mathcal{Z}^{\alpha}\hat{v}) \,\mathrm{d}\mathcal{V}_t^{\e} \\[12pt]

\lem \|\nabla^{\varphi^{\e}}\times \partial_t^{\ell}\mathcal{Z}^{\alpha}\hat{v}\|_{L^2}^2

+ |\NN^{\e}\times (\nabla^{\varphi^{\e}}\times \partial_t^{\ell}\mathcal{Z}^{\alpha}\hat{v})|_{\frac{1}{2}}
\big|\nabla^{\varphi^{\e}} \partial_t^{\ell}\mathcal{Z}^{\alpha}\hat{q}|_{z=0} \big|_{-\frac{1}{2}} \\[10pt]\quad

+ |\NN^{\e}\times (\nabla^{\varphi^{\e}}\times \partial_t^{\ell}\mathcal{Z}^{\alpha}\hat{v})|_{L^2}
\big(\big|\mathcal{I}_{7,2}|_{z=0}\big|_{L^2}
+ \big|\partial_t^{\ell}\mathcal{Z}^{\alpha}\hat{v}|_{z=0} \big|_{X_{tan}^1} \big)

 \\[10pt]\quad

+ \|\partial_t^{\ell}\mathcal{Z}^{\alpha}\hat{v}\|_{X^1}^2
+ \|\partial_z \partial_t^{\ell}\mathcal{Z}^{\alpha}\hat{v}\|_{L^2}^2

+ \e\|\mathcal{I}_{7,1}\|_{L^2}^2
+ \|\nabla^{\varphi^{\e}}\times \mathcal{I}_{7,2}\|_{L^2}^2.
\end{array}
\end{equation}

It is easy to prove that
\begin{equation}\label{Sect5_NormalDer_Estimate_6}
\begin{array}{ll}
\big|\mathcal{I}_{7,2}|_{z=0}\big|_{L^2} \lem |\hat{h}|_{X^{k-1}} + \big|\hat{v}|_{z=0}\big|_{X^{k-2}}
+ \big|\nabla\hat{q}|_{z=0}\big|_{X^{k-3}} \\[6pt]\hspace{1.85cm}

\lem |\hat{h}|_{X^{k-1}} + \|\hat{v}\|_{X^{k-2}} + \|\partial_z\hat{v}\|_{X^{k-2}}
+ \|\nabla\hat{q}\|_{X^{k-2}}, \\[9pt]

\e\|\mathcal{I}_{7,1}\|_{L^2}^2 \lem \e\sum\limits_{\ell+|\alpha|\leq k-2} \|\nabla^{\varphi^\e}\times \nabla^{\varphi^\e} \partial_t^{\ell}\mathcal{Z}^{\alpha} \hat{v}\|_{L^2}^2 + O(\e), \\[13pt]

\|\nabla^{\varphi^{\e}}\times \mathcal{I}_{7,2}\|_{L^2}^2 \lem \|\hat{\eta}\|_{X^{k-1,1}}^2
+ \|\nabla^{\varphi^{\e}}\times \hat{v}\|_{X_{tan}^{k-2}}^2
+ \|\nabla^{\varphi^{\e}}\times [\partial_t^{\ell}\mathcal{Z}^{\alpha},V_z^{\e}\partial_z]\hat{v}\|_{L^2}^2
\\[6pt]\hspace{2.8cm}
+ \|\nabla^{\varphi^{\e}}\times[\partial_t^{\ell}\mathcal{Z}^{\alpha},\NN^{\e}\partial_z^{\varphi^{\e}}] \hat{q}\|_{L^2}^2,
\end{array}
\end{equation}
where the estimates for the last two terms are similar to $(\ref{Sect2_NormalDer_Estimate_5}), (\ref{Sect2_NormalDer_Estimate_6})$.

Integrate $(\ref{Sect5_NormalDer_Estimate_5})$ in time, apply the integral form of Gronwall's inequality, we have
\begin{equation}\label{Sect5_NormalDer_Estimate_7}
\begin{array}{ll}
\|\nabla^{\varphi^{\e}}\times \partial_t^{\ell}\mathcal{Z}^{\alpha}\hat{v}\|_{L^2}^2
+ \e\int\limits_{\mathbb{R}^3_{-}}|\nabla^{\varphi^\e}\times \nabla^{\varphi^\e} \partial_t^{\ell}\mathcal{Z}^{\alpha} \hat{v}|^2
\,\mathrm{d}\mathcal{V}_t^{\e} \\[5pt]
\lem \big\|\nabla^{\varphi^{\e}}\times \partial_t^{\ell}\mathcal{Z}^{\alpha}\hat{v}|_{t=0}\big\|_{L^2}^2
+ \int\limits_0^t \|\hat{v}\|_{X^{k-2}} + \|\partial_z \hat{v}\|_{X^{k-2}}

+ \|\nabla\hat{q} \|_{X^{k-2}} + \|\hat{h}\|_{X^{k-1}} \,\mathrm{d}t \\[7pt]\quad

+ \int\limits_0^t \|\hat{v}\|_{X^{k-1}}^2 + \|\partial_z \hat{v}\|_{X^{k-1}}^2 + \|\hat{h}\|_{X^{k-1,1}}^2 \,\mathrm{d}t
+ O(\e) \\[8pt]

\lem \big\|\partial_t^{\ell}\mathcal{Z}^{\alpha}\hat{\omega}|_{t=0}\big\|_{L^2}^2

+ \int\limits_0^t\|\hat{v}\|_{X^{k-2}} + \|\partial_z \hat{v}\|_{X^{k-2}}
+ \|\nabla\hat{q} \|_{X^{k-2}} + \|\hat{h}\|_{X^{k-1}}\,\mathrm{d}t + O(\e).
\end{array}
\end{equation}

Since $\hat{\omega} = \nabla^{\varphi^{\e}} \times \hat{v} - \nabla^{\varphi^{\e}}\hat{\eta} \times \partial_z^{\varphi} v$, we have
\begin{equation}\label{Sect5_NormalDer_Estimate_8}
\begin{array}{ll}
\|\hat{\omega}\|_{X^{k-2}}^2
+ \|\partial_z\hat{v}\|_{X^{k-2}}^2

\lem \|\hat{\omega}|_{t=0}\|_{X^{k-2}}^2

+ \int\limits_0^t\|\hat{v}\|_{X^{k-2}} + \|\partial_z \hat{v}\|_{X^{k-2}} \\[6pt]\quad
+ \|\nabla\hat{q} \|_{X^{k-2}} + \|\hat{h}\|_{X^{k-1}}\,\mathrm{d}t + O(\e).
\end{array}
\end{equation}
Thus, Lemma $\ref{Sect5_NormalDer_Lemma}$ is proved.
\end{proof}

\begin{remark}\label{Sect5_NormalDer_Remark}
We can not use the following variables to estimate $\hat{\omega}_h$:
\begin{equation}\label{Sect5_NormalDer_Remark_1}
\left\{\begin{array}{ll}
\hat{\zeta}^1 = \hat{\omega}^1 -\textsf{F}^1 [\nabla\varphi^{\e}](\partial_j v^{\e,i}) + \omega^{b,1}, \ j=1,2,\ i=1,2,3,
\\[6pt]

\hat{\zeta}^2 = \hat{\omega}^2 -\textsf{F}^2 [\nabla\varphi^{\e}](\partial_j v^{\e,i}) + \omega^{b,2}, \ j=1,2,\ i=1,2,3,
\end{array}\right.
\end{equation}
because $\textsf{F}^{1,2} [\nabla\varphi^{\e}](\partial_j v^{\e,i})$ may not converge to the extension of $\omega^b$.
That is, $\hat{\zeta}$ may not be a small quantity.
\end{remark}

\subsection{Estimates for Normal Derivatives when $\Pi\mathcal{S}^{\varphi} v\nn|_{z=0} = 0$}

When $\Pi\mathcal{S}^{\varphi} v\nn|_{z=0} = 0$, the boundary value of Navier-Stokes vorticity converges to that of Euler vorticity,
the convergence rates of the inviscid limits can be improved.
In the following lemma, we estimate normal derivatives for the special Euler boundary data.
\begin{lemma}\label{Sect5_Special_EularData_Lemma}
Assume $k\leq m-2$, if $\Pi\mathcal{S}^{\varphi} v \nn|_{z=0} = 0$,
then the vorticity has the following estimate:
\begin{equation}\label{Sect5_Special_EularData_Lemma_Eq}
\begin{array}{ll}
\|\partial_z\hat{v}_h\|_{L^4([0,T],X^{k-2})}^2 + \|\hat{\omega}_h\|_{L^4([0,T],X^{k-2})}^2
\\[6pt]
\lem \big\|\hat{\omega}_0\big\|_{X^{k-2}}^2
+ \int\limits_0^T\|\hat{v}\|_{X^{k-1,1}}^2 \,\mathrm{d}t
+ \int\limits_0^T|\hat{h}|_{X^{k-2,1}}^2 \,\mathrm{d}t
+ \|\partial_t^{k-1}\hat{h}\|_{L^4([0,T],L^2)}^2 \\[8pt]\quad
+ \sqrt{\e}\|\partial_z\hat{v}\|_{L^4([0,T],X^{k-1})}^2
+ O(\e).
\end{array}
\end{equation}
\end{lemma}

\begin{proof}
If $\Pi\mathcal{S}^{\varphi} v\nn|_{z=0} = 0$, $\Theta^i =0$ where $i=1,\cdots,6$, See Remark $\ref{Sect5_Euler_BoundaryData_Remark}$.

When $\ell+|\alpha|\leq k-2$, we study the equations $(\ref{Sect1_N_Derivatives_Difference_Eq})$ and decompose $\hat{\omega}_h = \hat{\omega}_h^{nhom} + \hat{\omega}_h^{hom}$, such that
$\hat{\omega}_h^{nhom}$ satisfies the nonhomogeneous equations $(\ref{Sect5_N_Derivatives_Difference_Eq_Nonhom})$
and $\hat{\omega}_h^{hom}$ satisfies the homogeneous equations $(\ref{Sect5_N_Derivatives_Difference_Eq_Hom})$.

While $\hat{\omega}_h^{nhom}$ satisfies the following estimate:
\begin{equation}\label{Sect5_Special_EularData_1}
\begin{array}{ll}
\|\hat{\omega}_h^{nhom}\|_{L^4([0,T],X^{k-2})}^2
\lem \sqrt{T}\big\|\hat{\omega}_{0,h}\big\|_{X^{k-2}}^2 + T \|\hat{\omega}_h\|_{L^4([0,T],X^{k-2})}^2 \\[9pt]\quad

+ \sqrt{T}\int\limits_0^T\|\hat{v}\|_{X^{k-2,1}}^2 \,\mathrm{d}t + \sqrt{T}\int\limits_0^T|\hat{h}|_{X^{k-2,1}}^2 \,\mathrm{d}t
+ \sqrt{T}|\partial_t^{k-1} \hat{h}|_{L^4([0,T],L^2)}^2 + O(\e).
\end{array}
\end{equation}

When $\ell+|\alpha|\leq k-2$, the estimate $(\ref{Sect5_N_Derivatives_Difference_Eq_Hom_1})$ is reduced as follows:
\begin{equation}\label{Sect5_Special_EularData_2}
\begin{array}{ll}
\|\partial_t^{\ell}\mathcal{Z}^{\alpha}\hat{\omega}_h^{hom}\|_{L^4([0,T],L^2(\mathbb{R}^3_{-}))}^2 \\[7pt]

\lem \sqrt{\e}\int\limits_0^T\big|\textsf{F}^{1,2}[\nabla\varphi^{\e}](\partial_j v^{\e,i}) -\textsf{F}^{1,2}[\nabla\varphi](\partial_j v^i)
\big|_{X^{k-2}(\mathbb{R}^2)}^2 \,\mathrm{d}t \\[7pt]

\lem \sqrt{\e}\int\limits_0^T\big|\hat{v}|_{z=0}\big|_{X^{k-1}(\mathbb{R}^2)}^2 \,\mathrm{d}t
+ \sqrt{\e}\int\limits_0^T|\hat{h}|_{X^{k-2,1}(\mathbb{R}^2)}^2 \,\mathrm{d}t \\[7pt]

\lem \sqrt{\e}\int\limits_0^T\|\hat{v}\|_{X^{k-1,1}(\mathbb{R}^2)}^2 \,\mathrm{d}t
+ \sqrt{\e}\sqrt{T}\|\partial_z\hat{v}\|_{L^4([0,T],X^{k-1})}^2
+ \sqrt{\e}\int\limits_0^T|\hat{h}|_{X^{k-2,1}(\mathbb{R}^2)}^2 \,\mathrm{d}t.
\end{array}
\end{equation}

By $(\ref{Sect5_Special_EularData_1})$ and $(\ref{Sect5_Special_EularData_2})$, we have
\begin{equation}\label{Sect5_Special_EularData_3}
\begin{array}{ll}
\|\partial_z\hat{v}_h\|_{L^4([0,T],X^{k-2})}^2 + \|\hat{\omega}_h\|_{L^4([0,T],X^{k-2})}^2
\\[10pt]

\lem \|\hat{\omega}_h^{nhom}\|_{L^4([0,T],X^{k-2})}^2
+ \|\partial_t^{\ell}\mathcal{Z}^{\alpha}\hat{\omega}_h^{hom}\|_{L^4([0,T],L^2(\mathbb{R}^3_{-}))}^2 \\[7pt]

\lem \big\|\hat{\omega}_0\big\|_{X^{k-2}}^2
+ \int\limits_0^T\|\hat{v}\|_{X^{k-1,1}}^2 \,\mathrm{d}t
+ \int\limits_0^T|\hat{h}|_{X^{k-2,1}}^2 \,\mathrm{d}t
+ \|\partial_t^{k-1}\hat{h}\|_{L^4([0,T],L^2)}^2 \\[8pt]\quad
+ \sqrt{\e}\|\partial_z\hat{v}\|_{L^4([0,T],X^{k-1})}^2
+ O(\e).
\end{array}
\end{equation}

Thus, Lemma $\ref{Sect5_Special_EularData_Lemma}$ is proved.
\end{proof}

\subsection{Convergence Rates of the Inviscid Limit}
In this subsection, we calculate convergence rates of the inviscid limit.

\begin{theorem}\label{Sect5_ConvergenceRates_Thm}
Assume $T>0$ is finite, fixed and independent of $\e$, $(v^{\e},h^{\e})$ is the solution in $[0,T]$ of Navier-Stokes equations $(\ref{Sect1_NS_Eq})$ with initial data $(v^{\e}_0,h^{\e}_0)$ satisfying $(\ref{Sect1_Proposition_TimeRegularity_1})$, $\omega^{\e}$ is its vorticity.
$(v,h)$ is the solution in $[0,T]$ of Euler equations $(\ref{Sect1_Euler_Eq})$ with initial data $(v_0,h_0)\in X^{m-1,1}(\mathbb{R}^3_{-}) \times X^{m-1,1}(\mathbb{R}^2)$, $\omega$ is its vorticity. Assume there exists an integer $k$ where $1\leq k\leq m-2$, such that $\|v^{\e}_0 -v_0\|_{X^{k-1,1}(\mathbb{R}^3_{-})} =O(\e^{\lambda^v})$, $|h^{\e}_0 -h_0|_{X^{k-1,1}(\mathbb{R}^2)} =O(\e^{\lambda^h})$, $\|\omega^{\e}_0 - \omega_0\|_{X^{k-1}(\mathbb{R}^3_{-})} =O(\e^{\lambda^{\omega}_1})$, where
$\lambda^v>0, \lambda^h>0, \lambda^{\omega}_1>0$.

If the Euler boundary data satisfies $\Pi\mathcal{S}^{\varphi} v\nn|_{z=0}\neq 0$ in $[0,T]$, then the convergence rates of the inviscid limit
satisfy $(\ref{Sect1_Thm4_ConvergenceRates_1})$.

If the Euler boundary data satisfies $\Pi\mathcal{S}^{\varphi} v\nn|_{z=0} = 0$ in $[0,T]$,
assume $\|\omega^{\e}_0 - \omega_0\|_{X^{k-2}(\mathbb{R}^3_{-})} =O(\e^{\lambda^{\omega}_2})$ where $\lambda^{\omega}_2>0$,
then the convergence rates of the inviscid limit satisfy $(\ref{Sect1_Thm4_ConvergenceRates_2})$.
\end{theorem}

\begin{proof}
If $\Pi\mathcal{S}^{\varphi} v\nn|_{z=0}\neq 0$, we prove the converge rates of the inviscid limit:

By Lemmas $\ref{Sect5_Pressure_Lemma}, \ref{Sect5_Tangential_Estimates_Lemma},
\ref{Sect5_Vorticity_Lemma}$, we have
\begin{equation}\label{Sect5_ConvergenceRates_Thm_Eq_1}
\begin{array}{ll}
\|\hat{v}\|_{X^{k-1,1}}^2 + |\hat{h}|_{X^{k-1,1}}^2

\lem \|\hat{v}_0\|_{X^{k-1,1}}^2 + |\hat{h}_0|_{X^{k-1,1}}^2 + |\hat{\omega}_0|_{X^{k-1}}^2 \\[5pt]\quad
+ \int\limits_0^t\|\hat{v}\|_{X^{k-1,1}}^2
+ \int\limits_0^t\|\hat{h}\|_{X^{k-1,1}}^2 \,\mathrm{d}t + O(\sqrt{\e}).
\end{array}
\end{equation}

Apply the integral form of Gronwall's inequality to $(\ref{Sect5_ConvergenceRates_Thm_Eq_1})$, we get
\begin{equation}\label{Sect5_ConvergenceRates_Thm_Eq_2}
\begin{array}{ll}
\|\hat{v}\|_{X^{k-1,1}}^2 + |\hat{h}|_{X^{k-1,1}}^2

\lem \|\hat{v}_0\|_{X^{k-1,1}}^2 + |\hat{h}_0|_{X^{k-1,1}}^2
+ \big\|\hat{\omega}_0\big\|_{X^{k-1}}^2 + O(\sqrt{\e}) \\[7pt]\hspace{3.3cm}

\lem O(\e^{\min\{\frac{1}{2}, 2\lambda^v, 2\lambda^h, 2\lambda^{\omega}_1\}}).
\end{array}
\end{equation}

By Lemma $\ref{Sect5_Vorticity_Lemma}$, we have
\begin{equation}\label{Sect5_ConvergenceRates_Thm_Eq_3}
\begin{array}{ll}
\|\partial_z\hat{v}_h\|_{L^4([0,T],X^{k-1})}^2 + \|\hat{\omega}_h\|_{L^4([0,T],X^{k-1})}^2

\lem O(\e^{\min\{\frac{1}{2}, 2\lambda^v, 2\lambda^h, 2\lambda^{\omega}_1\}}).
\end{array}
\end{equation}

By Lemmas $\ref{Sect5_Pressure_Lemma}, \ref{Sect5_Vorticity_Lemma}, \ref{Sect5_NormalDer_Lemma}$, we have
\begin{equation}\label{Sect5_ConvergenceRates_Thm_Eq_4}
\begin{array}{ll}
\|\hat{\omega}\|_{X^{k-2}}^2 + \|\partial_z\hat{v}\|_{X^{k-2}}^2

\lem \|\hat{\omega}_0\|_{X^{k-2}}^2

+ \int\limits_0^t\|\hat{v}\|_{X^{k-2}} + \|\partial_z \hat{v}\|_{X^{k-2}} \\[6pt]\quad
+ \|\nabla\hat{q} \|_{X^{k-2}} + \|\hat{h}\|_{X^{k-1}}\,\mathrm{d}t + O(\e)

\lem O(\e^{\min\{\frac{1}{4}, \lambda^v, \lambda^h, \lambda^{\omega}_1\}}).
\end{array}
\end{equation}

If $\Pi\mathcal{S}^{\varphi} v\nn|_{z=0}= 0$, we prove the converge rates of the inviscid limit:

By Lemma $\ref{Sect5_Special_EularData_Lemma}$, we have
\begin{equation}\label{Sect5_ConvergenceRates_Thm_Eq_5}
\begin{array}{ll}
\|\partial_z\hat{v}_h\|_{L^4([0,T],X^{k-2})}^2 + \|\hat{\omega}_h\|_{L^4([0,T],X^{k-2})}^2
\\[6pt]
\lem \big\|\hat{\omega}_0\big\|_{X^{k-2}}^2
+ \int\limits_0^T\|\hat{v}\|_{X^{k-1,1}}^2 \,\mathrm{d}t
+ \int\limits_0^T|\hat{h}|_{X^{k-2,1}}^2 \,\mathrm{d}t
+ \sqrt{\e}O(\e^{\min\{\frac{1}{2}, 2\lambda^v, 2\lambda^h, 2\lambda^{\omega}_1\}}).
\end{array}
\end{equation}

Couple $(\ref{Sect5_ConvergenceRates_Thm_Eq_5})$ with the following tangential estimates,
\begin{equation}\label{Sect5_ConvergenceRates_Thm_Eq_6}
\begin{array}{ll}
\|\hat{v}\|_{X^{k-2,1}}^2 + |\hat{h}|_{X^{k-2,1}}^2

\lem \|\hat{v}_0\|_{X^{k-2,1}}^2 + |\hat{h}_0|_{X^{k-2,1}}^2
+\|\partial_z \hat{v}\|_{L^4([0,T],X^{k-2})}^2 \\[8pt]\quad
+ \int\limits_0^t\|\hat{v}\|_{X^{k-2,1}}^2
+ \int\limits_0^t\|\hat{h}\|_{X^{k-2,1}}^2 \,\mathrm{d}t + O(\e),
\end{array}
\end{equation}
apply the integral form of Gronwall's inequality, then we get
\begin{equation}\label{Sect5_ConvergenceRates_Thm_Eq_7}
\begin{array}{ll}
\|\hat{v}\|_{X^{k-2,1}}^2 + |\hat{h}|_{X^{k-2,1}}^2 \\[6pt]
\lem \big\|\hat{\omega}_0\big\|_{X^{k-2}}^2
+ \|\hat{v}_0\|_{X^{k-2,1}}^2 + |\hat{h}_0|_{X^{k-2,1}}^2 + \sqrt{\e}O(\e^{\min\{\frac{1}{2}, 2\lambda^v, 2\lambda^h, 2\lambda^{\omega}_1\}})
\\[8pt]
\lem O(\e^{\min\{1, 2\lambda^v, 2\lambda^h, 2\lambda^{\omega}_2, 2\lambda^{\omega}_1+1\}})
= O(\e^{\min\{1, 2\lambda^v, 2\lambda^h, 2\lambda^{\omega}_2\}}).
\end{array}
\end{equation}

Similar to $(\ref{Sect5_ConvergenceRates_Thm_Eq_4})$, we have
\begin{equation}\label{Sect5_ConvergenceRates_Thm_Eq_8}
\begin{array}{ll}
\|\hat{\omega}\|_{X^{k-3}}^2 + \|\partial_z\hat{v}\|_{X^{k-3}}^2
\lem O(\e^{\min\{\frac{1}{2}, \lambda^v, \lambda^h, \lambda^{\omega}_2\}}).
\end{array}
\end{equation}
Thus, Theorem $\ref{Sect5_ConvergenceRates_Thm}$ is proved.
\end{proof}

To estimate $\NN^{\e}\cdot \partial_z^{\varphi^{\e}} v^{\e} - \NN \cdot \partial_z^{\varphi} v$, we use the equality
$\NN\cdot\partial_z^{\varphi} v = - (\partial_1 v^1 + \partial_2 v^2)$.
To estimate $\NN^{\e}\cdot {\omega}^{\e} -\NN\cdot \omega$, we use the following equality:
\begin{equation}\label{Sect5_Vorticity_Multiple_N}
\begin{array}{ll}
\NN\cdot\omega
= -\partial_1\varphi(\partial_2 v^3 - \frac{\partial_2\varphi}{\partial_z\varphi}\partial_z v^3
- \frac{1}{\partial_z\varphi}\partial_z v^2)
-\partial_2\varphi(- \partial_1 v^3 + \frac{\partial_1\varphi}{\partial_z\varphi}\partial_z v^3
+ \frac{1}{\partial_z\varphi}\partial_z v^1)
\\[8pt]\hspace{1.3cm}
+ \partial_1 v^2 - \frac{\partial_1\varphi}{\partial_z\varphi}\partial_z v^2 - \partial_2 v^1 + \frac{\partial_2\varphi}{\partial_z\varphi}\partial_z v^1
\\[8pt]\hspace{0.9cm}

= -\partial_1\varphi\partial_2 v^3 +\partial_2\varphi\partial_1 v^3 + \partial_1 v^2 - \partial_2 v^1.
\end{array}
\end{equation}

%%% find 6
\section{Regularity Structure of Navier-Stokes Solutions for Fixed $\sigma>0$}
In this section, $\sigma>0$, we prove Proposition $\ref{Sect1_Proposition_Regularity_Tension}$
on the regularities of Navier-Stokes solutions and Euler solutions.
For simplicity, we omit the superscript ${}^{\e}$ in this section, which represents Navier-Stokes solutions.

Since the estimates of normal derivatives are the same as the $\sigma=0$ case, we only focus on the estimates of the pressure gradient and
tangential derivatives when $\sigma>0$.

The following lemma concerns the estimate of the pressure gradient.
\begin{lemma}\label{Sect6_Pressure_Estimates}
Assume the pressure $q$ satisfies the elliptic equation with Neumann boundary condition $(\ref{Sect1_Pressure_Neumann})$, then
$q$ has the following gradient estimate:
\begin{equation}\label{Sect6_Pressure_Estimate_Eq}
\begin{array}{ll}
\|\nabla q\|_{X^{m-1}}
\lem \|\partial_t^m v\|_{X^m} + \|v\|_{X^{m-1,1}} + \|\partial_z v\|_{X^{m-1}} +  |h|_{X^{m,1}} \\[6pt]\quad
+ \e \|\nabla_y v\|_{X^{m-1,1}} + \e \|\partial_z v\|_{X^{m-1}} + \e |h|_{X^{m-1,\frac{3}{2}}}.
\end{array}
\end{equation}
\end{lemma}

\begin{proof}
The $L^2$ estimate of the elliptic equation with its Neumann boundary condition $(\ref{Sect1_Pressure_Neumann})$ is standard, that is
\begin{equation}\label{Sect6_Pressure_Estimate_L^2}
\begin{array}{ll}
\|\nabla q\|_{L^2} \lem \|v\cdot \nabla^{\varphi} v\|_{L^2}
+ \big|\nabla^{\varphi} q\cdot\NN|_{z=0} \big|_{-\frac{1}{2}} \\[8pt]

\lem \|v\|_{X^{0,1}}+ \|\partial_z v\|_{L^2} + |h|_{X^{0,1}}
+ \big|\partial_t^{\varphi} v\cdot\NN|_{z=0}\big|_{-\frac{1}{2}} \\[6pt]\quad
+ \big|v\cdot\nabla^{\varphi} v\cdot\NN|_{z=0} \big|_{-\frac{1}{2}}
+ \e\big|\triangle^{\varphi} v\cdot\NN|_{z=0}\big|_{-\frac{1}{2}} \\[8pt]

\lem \|v\|_{X^{0,1}}+ \|\partial_z v\|_{L^2} + |h|_{X^{0,1}}
+ \|\partial_t^{\varphi} v\|_{L^2} + \|\nabla^{\varphi}\cdot\partial_t^{\varphi} v\|_{L^2} \\[6pt]\quad
+ \|v\cdot\nabla^{\varphi} v\|_{L^2} + \|\nabla\cdot(v\cdot\nabla^{\varphi} v)\|_{L^2}
+ \e\big|v|_{z=0}\big|_{\frac{3}{2}} + \e|h|_{\frac{3}{2}}.
\end{array}
\end{equation}
where we used the inequality $|v\cdot\NN|_{-\frac{1}{2}} \lem \|v\| + \|\nabla^{\varphi}\cdot v\|$ (see \cite{Wang_Xin_2015}).

Similar to \cite{Wang_Xin_2015,Masmoudi_Rousset_2012_NavierBC}, we have higher order estimates for $(\ref{Sect1_Pressure_Neumann})$:
\begin{equation}\label{Sect6_Pressure_Estimate_1}
\begin{array}{ll}
\|\nabla q\|_{X^{m-1}} \lem \|v\cdot \nabla^{\varphi} v\|_{X^{m-1}}
+ \big|\nabla^{\varphi} q\cdot\NN|_{z=0} \big|_{X^{m-1, -\frac{1}{2}}} \\[9pt]

\lem \|v\|_{X^{m-1,1}}+ \|\partial_z v\|_{X^{m-1}} + |h|_{X^{m-1,1}}
+ \big|\nabla^{\varphi} q\cdot\NN|_{z=0} \big|_{X^{m-1, -\frac{1}{2}}}.
\end{array}
\end{equation}

Next, we estimate $\big|\nabla^{\varphi} q\cdot\NN|_{z=0} \big|_{X^{m-1, -\frac{1}{2}}}$.

Firstly, it is easy to estimate $\e \big|\triangle^{\varphi} v\cdot\NN|_{z=0} \big|_{X^{m-1,-\frac{1}{2}}}$.
It follows from the divergence free condition $\nabla^{\varphi}\cdot v=0$ that
\begin{equation}\label{Sect6_Pressure_Estimate_2}
\begin{array}{ll}
\partial_z v \cdot\NN = -\partial_z\varphi \nabla_y\cdot v_y, \quad z\leq 0, \\[11pt]

\partial_{zz} v\cdot\NN|_{z=0} = \partial_z(\partial_z v \cdot\NN) -\partial_z v \cdot \partial_z\NN
= - \partial_z(\partial_z\varphi \nabla_y\cdot v_y) + \partial_z v_y \cdot \partial_z\partial_y\varphi \\[5pt]
= - \partial_z^2\varphi \nabla_y\cdot v_y - \partial_z\varphi \nabla_y\cdot\partial_z v_y
+ \partial_z v_y \cdot \partial_z\partial_y\varphi  \\[5pt]
= - \partial_z^2\varphi \nabla_y\cdot v_y - \partial_z\varphi \nabla_y\cdot[f^{5,6}[\nabla\varphi](\partial_j v^i)]
+ [f^{5,6}[\nabla\varphi](\partial_j v^i)] \cdot \partial_z\partial_y\varphi ,
\end{array}
\end{equation}
where $\partial_z v_h = f^{5,6}[\nabla\varphi](\partial_j v^i)$ is proved in $(\ref{Sect2_Vorticity_H_BC_10})$.
Thus,
\begin{equation}\label{Sect6_Pressure_Estimate_3}
\begin{array}{ll}
\e \big|\triangle^{\varphi} v\cdot\NN|_{z=0} \big|_{X^{m-1,-\frac{1}{2}}}
\lem \e \big|\nabla_y v|_{z=0}\big|_{X^{m-1,\frac{1}{2}}} + \e |h|_{X^{m-1,\frac{3}{2}}}
 \\[7pt]

\lem \e \|\nabla_y v\|_{X^{m-1,1}}
 + \e \|\partial_z v\|_{X^{m-1}} + \e |h|_{X^{m-1,\frac{3}{2}}}.
\end{array}
\end{equation}

Secondly, we estimate $\big|(\partial_t^{\varphi} v + v\cdot\nabla^{\varphi} v)\cdot\NN|_{z=0} \big|_{X^{m-1, -\frac{1}{2}}}$.
\begin{equation}\label{Sect6_Pressure_Estimate_4}
\begin{array}{ll}
\big|(\partial_t^{\varphi} v + v\cdot\nabla^{\varphi} v)\cdot\NN|_{z=0} \big|_{X^{m-1, -\frac{1}{2}}}
= \big|(\partial_t v + v_y\cdot\nabla_y v)\cdot\NN|_{z=0} \big|_{X^{m-1, -\frac{1}{2}}} \\[8pt]

\lem \sum\limits_{\ell+|\alpha|\leq m-1}
\big(\big|\partial_t^{\ell}\mathcal{Z}^{\alpha}(\partial_t v + v_y\cdot\nabla_y v)\cdot\NN|_{z=0} \big|_{-\frac{1}{2}}
+ \big|\partial_t^{\ell}\mathcal{Z}^{\alpha}\NN|_{z=0} \big|_{-\frac{1}{2}}\big) \\[12pt]

\lem \sum\limits_{\ell+|\alpha|\leq m-1}
\big(\|\partial_t^{\ell}\mathcal{Z}^{\alpha} \partial_t v\|_{L^2}
+ \|\nabla^{\varphi}\cdot\partial_t^{\ell}\mathcal{Z}^{\alpha} \partial_t v\|_{L^2}\big) \\[10pt]\quad
+ \sum\limits_{\ell+|\alpha|\leq m-1}
\big(\|\partial_t^{\ell}\mathcal{Z}^{\alpha} \nabla_y v\|_{L^2}
+ \|\nabla^{\varphi}\cdot\partial_t^{\ell}\mathcal{Z}^{\alpha} \nabla_y v\|_{L^2}\big)
+ |h|_{X^{m-1, \frac{1}{2}}} \\[10pt]

\lem \|v\|_{X^m} + \|\partial_z v\|_{X^{m-1}} +  |h|_{X^{m,1}}.
\end{array}
\end{equation}
Thus, Lemma $\ref{Sect6_Pressure_Estimates}$ is proved.
\end{proof}

Before estimating tangential derivatives of $v$, we have the estimate of $\partial_t^{\ell} h$
by using the kinetical boundary condition $(\ref{Sect1_NS_Eq_ST})_3$, which is the same with $(\ref{Sect1_NS_Eq})_3$,
we give the following lemma without proof, which is the same with Lemma $\ref{Sect2_Height_Estimates_Lemma}$.
\begin{lemma}\label{Sect6_Height_Estimates_Lemma}
Assume $0\leq\ell\leq m-1$,
$\partial_t^{\ell}h$ have the estimates:
\begin{equation}\label{Sect6_Height_Estimates_Lemma_Eq}
\begin{array}{ll}
\int\limits_{\mathbb{R}^2} |\partial_t^{\ell}h|^2 \,\mathrm{d}y

\lem |h_0|_{X^{m-1}}^2
+ \int\limits_0^t |h|_{X^{m-1,1}}^2 + \|v\|_{X^{m-1,1}}^2 \,\mathrm{d}t
+ \|\partial_z v\|_{L^4([0,T],X^{m-1})}^2.
\end{array}
\end{equation}
\end{lemma}

Now, we develop a priori estimates for tangential derivatives including time derivatives.
Our equations and variables are different from \cite{Wang_Xin_2015} which used Alinhac's good unknown.
\begin{lemma}\label{Sect6_Tangential_Estimate_Lemma}
Assume the conditions are the same with those of Proposition $\ref{Sect1_Proposition_Regularity_Tension}$,
then $v$ and $h$ satisfy the a priori estimate:
\begin{equation}\label{Sect6_Tangential_Estimate}
\begin{array}{ll}
\|v\|_{X^{m-1,1}}^2 + |h|_{X^{m-1,1}}^2 + \e |h|_{X^{m-1,\frac{3}{2}}}^2 + \sigma |h|_{X^{m-1,2}}^2
+ \e\int\limits_0^t \|\nabla v\|_{X^{m-1,1}}^2 \,\mathrm{d}t \\[9pt]

\lem \|v_0\|_{X^{m-1,1}}^2 + |h_0|_{X^{m-1,1}}^2 + \e |h_0|_{X^{m-1,\frac{3}{2}}}^2 + \sigma |h_0|_{X^{m-1,2}}^2 \\[6pt]\quad
+ \|\partial_z v\|_{L^4([0,T],X^{m-1})}^2 + \|\partial_t^m v\|_{L^4([0,T],L^2)}^2 + \|\partial_t^m h\|_{L^4([0,T],X^{0,1})}^2.
\end{array}
\end{equation}
\end{lemma}

\begin{proof}
For the fixed $\sigma>0$, we do not need Alinhac's good unknown $(\ref{Sect1_Good_Unknown_1})$. 
Apply $\partial_t^{\ell}\mathcal{Z}^{\alpha}$ to $(\ref{Sect1_NS_Eq_ST})$, 
then $\partial_t^{\ell}\mathcal{Z}^{\alpha}v$ and $\partial_t^{\ell}\mathcal{Z}^{\alpha}q$ satisfy the following equations:
\begin{equation}\label{Sect6_Tangential_Estimate_2}
\left\{\begin{array}{ll}
\partial_t^{\varphi} \partial_t^{\ell}\mathcal{Z}^{\alpha}v
+ v\cdot\nabla^{\varphi} \partial_t^{\ell}\mathcal{Z}^{\alpha} v
+ \nabla^{\varphi} \partial_t^{\ell}\mathcal{Z}^{\alpha} q
- 2\e\nabla^{\varphi}\cdot \mathcal{S}^{\varphi} \partial_t^{\ell}\mathcal{Z}^{\alpha}v
\\[8pt]\quad
= \partial_t^{\ell+1}\mathcal{Z}^{\alpha}\eta \partial_z^{\varphi}v
+ \partial_t^{\ell}\mathcal{Z}^{\alpha}\nabla \eta \cdot v \partial_z^{\varphi} v
+ \partial_t^{\ell}\mathcal{Z}^{\alpha}\nabla \eta \cdot \partial_z^{\varphi} q \\[8pt]\quad

+ 2\e\nabla^{\varphi}\cdot [\partial_t^{\ell}\mathcal{Z}^{\alpha}, \mathcal{S}^{\varphi}]v
+ 2\e[\partial_t^{\ell}\mathcal{Z}^{\alpha}, \nabla^{\varphi}\cdot] \mathcal{S}^{\varphi}v + b.t.,  \\[10pt]

\nabla^{\varphi}\cdot \partial_t^{\ell}\mathcal{Z}^{\alpha}v = \partial_t^{\ell}\mathcal{Z}^{\alpha}\nabla\eta \cdot \partial_z^{\varphi}v + b.t. , \\[10pt]

\partial_t \partial_t^{\ell}\mathcal{Z}^{\alpha} h
+ v_y \cdot\nabla_y \partial_t^{\ell}\mathcal{Z}^{\alpha} h
= \partial_t^{\ell}\mathcal{Z}^{\alpha}v \cdot \NN
+ [\partial_t^{\ell}\mathcal{Z}^{\alpha}, v,\NN],  \\[10pt]

\partial_t^{\ell}\mathcal{Z}^{\alpha}q \NN
-2\e \mathcal{S}^{\varphi} \partial_t^{\ell}\mathcal{Z}^{\alpha}v\,\NN \\[8pt]\quad

= g \partial_t^{\ell}\mathcal{Z}^{\alpha}h\NN
- \sigma \nabla_y\cdot\frac{1}{\sqrt{1+|\nabla_y h|^2}}
\big(\nabla_y\partial_t^{\ell}\mathcal{Z}^{\alpha} h
- \frac{\nabla_y h(\nabla_y h\cdot \nabla_y\partial_t^{\ell}\mathcal{Z}^{\alpha}h)}{1+|\nabla_y h|^2}\big) \NN \\[10pt]\quad

+ 2\e [\partial_t^{\ell}\mathcal{Z}^{\alpha},\mathcal{S}^{\varphi}] v\,\NN

+ (2\e \mathcal{S}^{\varphi}v - (q-g h))\,\partial_t^{\ell}\mathcal{Z}^{\alpha}\NN \\[8pt]\quad
- [\partial_t^{\ell}\mathcal{Z}^{\alpha},q-g h,\NN]
+2\e [\partial_t^{\ell}\mathcal{Z}^{\alpha},\mathcal{S}^{\varphi}v, \NN]

-\sigma [\partial_t^{\ell}\mathcal{Z}^{\alpha}, \NN]H \\[8pt]\quad

- \sigma \nabla_y\cdot [\partial_t^{\ell}\mathcal{Z}^{\alpha}, \nabla_y h, \frac{1}{\sqrt{1+|\nabla_y h|^2}}] \NN, \\[12pt]

(\partial_t^{\ell}\mathcal{Z}^{\alpha}v, \partial_t^{\ell}\mathcal{Z}^{\alpha}h)|_{t=0}
= (\partial_t^{\ell}\mathcal{Z}^{\alpha}v_0, \partial_t^{\ell}\mathcal{Z}^{\alpha}h_0).
\end{array}\right.
\end{equation}

When $|\alpha|\geq 1, \, 1\leq \ell+|\alpha|\leq m, 0\leq\ell\leq m-1$, we develop the $L^2$ estimate 
$\partial_t^{\ell}\mathcal{Z}^{\alpha} v$ and $\partial_t^{\ell}\mathcal{Z}^{\alpha} h$, we get
\begin{equation}\label{Sect6_Tangential_Estimate_3}
\begin{array}{ll}
\frac{1}{2}\frac{\mathrm{d}}{\mathrm{d}t}\int\limits_{\mathbb{R}^3_{-}} |\partial_t^{\ell}\mathcal{Z}^{\alpha}v|^2 \,\mathrm{d}\mathcal{V}_t
- \int\limits_{\mathbb{R}^3_{-}} \partial_t^{\ell}\mathcal{Z}^{\alpha}q \, \nabla^{\varphi}\cdot \partial_t^{\ell}\mathcal{Z}^{\alpha}v \,\mathrm{d}\mathcal{V}_t
+ 2\e \int\limits_{\mathbb{R}^3_{-}} |\mathcal{S}^{\varphi}\partial_t^{\ell}\mathcal{Z}^{\alpha}v|^2 \,\mathrm{d}\mathcal{V}_t \\[14pt]

\leq \int\limits_{\{z=0\}} (2\e \mathcal{S}^{\varphi}\partial_t^{\ell}\mathcal{Z}^{\alpha}v\NN
- \partial_t^{\ell}\mathcal{Z}^{\alpha}q\NN)\cdot \partial_t^{\ell}\mathcal{Z}^{\alpha}v \mathrm{d}y
+ \|\partial_z v\|_{X^{m-1}}^2 + \|\nabla q\|_{X^{m-1}}^2 \\[11pt]\quad
+ |h|_{X^{m-1,2}}^2 + |\partial_t^m h|_{L^2}^2
+ \text{b.t.} \\[7pt]

\leq -\int\limits_{\{z=0\}} \big[
g \partial_t^{\ell}\mathcal{Z}^{\alpha}h
- \sigma \nabla_y\cdot\frac{1}{\sqrt{1+|\nabla_y h|^2}}
\big(\nabla_y\partial_t^{\ell}\mathcal{Z}^{\alpha} h
- \frac{\nabla_y h(\nabla_y h\cdot \nabla_y\partial_t^{\ell}\mathcal{Z}^{\alpha}h)}{1+|\nabla_y h|^2}\big)
\big]\NN \\[13pt]\quad
\cdot \partial_t^{\ell}\mathcal{Z}^{\alpha}v \mathrm{d}y
+ \|\partial_z v\|_{X^{m-1}}^2 + \|\nabla q\|_{X^{m-1}}^2
+ |h|_{X^{m-1,2}}^2 + |\partial_t^m h|_{L^2}^2 + \text{b.t.} \\[8pt]

\leq \sigma\int\limits_{\{z=0\}} \nabla_y\cdot\frac{1}{\sqrt{1+|\nabla_y h|^2}}
\big(\nabla_y\partial_t^{\ell}\mathcal{Z}^{\alpha} h
- \frac{\nabla_y h(\nabla_y h\cdot \nabla_y\partial_t^{\ell}\mathcal{Z}^{\alpha}h)}{1+|\nabla_y h|^2}\big)

\cdot(\partial_t \partial_t^{\ell}\mathcal{Z}^{\alpha} h \\[13pt]\quad
+ v_y \cdot\nabla_y \partial_t^{\ell}\mathcal{Z}^{\alpha} h) \mathrm{d}y

-\int\limits_{\{z=0\}}
g \partial_t^{\ell}\mathcal{Z}^{\alpha}h
\cdot(\partial_t \partial_t^{\ell}\mathcal{Z}^{\alpha} h
+ v_y \cdot\nabla_y \partial_t^{\ell}\mathcal{Z}^{\alpha} h) \mathrm{d}y \\[10pt]\quad
+ \|\partial_z v\|_{X^{m-1}}^2 + \|\nabla q\|_{X^{m-1}}^2
+ |h|_{X^{m-1,2}}^2 + |\partial_t^m h|_{L^2}^2 + \text{b.t.} \\[7pt]

\leq - \sigma \int\limits_{\{z=0\}} \frac{1}{\sqrt{1+|\nabla_y h|^2}}
\big(\nabla_y\partial_t^{\ell}\mathcal{Z}^{\alpha} h
- \frac{\nabla_y h(\nabla_y h\cdot \nabla_y\partial_t^{\ell}\mathcal{Z}^{\alpha}h)}{1+|\nabla h|^2}\big)
\cdot(\partial_t \nabla_y\partial_t^{\ell}\mathcal{Z}^{\alpha} h \\[13pt]\quad
+ v_y \cdot\nabla_y \nabla_y\partial_t^{\ell}\mathcal{Z}^{\alpha} h) \mathrm{d}y

- \frac{g}{2}\frac{\mathrm{d}}{\mathrm{d}t}
\int\limits_{\{z=0\}} |\partial_t^{\ell}\mathcal{Z}^{\alpha}h|^2 \mathrm{d}y
+ \|\partial_z v\|_{X^{m-1}}^2 \\[10pt]\quad
+ \|\nabla q\|_{X^{m-1}}^2
+ |h|_{X^{m-1,2}}^2 + |\partial_t^m h|_{L^2}^2 + \text{b.t.} 
\end{array}
\end{equation}

\begin{equation*}
\begin{array}{ll}
\leq - \frac{g}{2}\frac{\mathrm{d}}{\mathrm{d}t}
\int\limits_{\{z=0\}} |\partial_t^{\ell}\mathcal{Z}^{\alpha}h|^2 \mathrm{d}y

- \frac{\sigma}{2}\frac{\mathrm{d}}{\mathrm{d}t} \int\limits_{\{z=0\}} \frac{1}{\sqrt{1+|\nabla_y h|^2}}
\big(|\nabla_y\partial_t^{\ell}\mathcal{Z}^{\alpha} h|^2 \\[12pt]\quad
- \frac{|\nabla_y h\cdot \nabla_y\partial_t^{\ell}\mathcal{Z}^{\alpha}h|^2}{1+|\nabla_y h|^2}\big) \mathrm{d}y

+ \|\partial_z v\|_{X^{m-1}}^2 + \|\nabla q\|_{X^{m-1}}^2
+ |h|_{X^{m-1,2}}^2 + |\partial_t^m h|_{L^2}^2 + \text{b.t.}
\end{array}
\end{equation*}
then
\begin{equation}\label{Sect6_Tangential_Estimate_4}
\begin{array}{ll}
\frac{1}{2}\frac{\mathrm{d}}{\mathrm{d}t}\int\limits_{\mathbb{R}^3_{-}} |\partial_t^{\ell}\mathcal{Z}^{\alpha}v|^2 \,\mathrm{d}\mathcal{V}_t
+ \frac{g}{2}\frac{\mathrm{d}}{\mathrm{d}t}
\int\limits_{\{z=0\}} |\partial_t^{\ell}\mathcal{Z}^{\alpha}h|^2 \mathrm{d}y
+ 2\e \int\limits_{\mathbb{R}^3_{-}} |\mathcal{S}^{\varphi}\partial_t^{\ell}\mathcal{Z}^{\alpha}v|^2 \,\mathrm{d}\mathcal{V}_t \\[10pt]\quad

+ \frac{\sigma}{2}\frac{\mathrm{d}}{\mathrm{d}t} \int\limits_{\{z=0\}} \frac{1}{\sqrt{1+|\nabla_y h|^2}}
\big(|\nabla_y\partial_t^{\ell}\mathcal{Z}^{\alpha} h|^2
- \frac{|\nabla_y h\cdot \nabla_y\partial_t^{\ell}\mathcal{Z}^{\alpha}h|^2}{1+|\nabla_y h|^2}\big) \mathrm{d}y \\[12pt]

\leq \|\partial_z v\|_{X^{m-1}}^2 + \|\nabla q\|_{X^{m-1}}^2
+ |h|_{X^{m-1,2}}^2 + |\partial_t^m h|_{L^2}^2 + \text{b.t.}
\end{array}
\end{equation}

Integrate $(\ref{Sect6_Tangential_Estimate_4})$ in time, apply the integral form of Gronwall's inequality, we have
\begin{equation}\label{Sect6_Tangential_Estimate_5}
\begin{array}{ll}
\|\partial_t^{\ell}\mathcal{Z}^{\alpha}v\|^2 + |\partial_t^{\ell}\mathcal{Z}^{\alpha} h|^2 + \e |\partial_t^{\ell}\mathcal{Z}^{\alpha} h|_{\frac{1}{2}}^2
+ \frac{\sigma}{4} |\partial_t^{\ell}\mathcal{Z}^{\alpha} h|_{1}^2
+ \e\int\limits_0^t \|\nabla \partial_t^{\ell}\mathcal{Z}^{\alpha}v\|^2 \,\mathrm{d}t \\[9pt]

\lem \|v_0\|_{X^{m-1,1}}^2 + |h_0|_{X^{m-1,1}}^2 + \e |h_0|_{X^{m-1,\frac{3}{2}}}^2 + \sigma |h_0|_{X^{m-1,2}}^2 \\[7pt]\quad
+ \int\limits_0^T \|\partial_z v\|_{X^{m-2,1}}^2 + \|\nabla q\|_{X^{m-2,1}}^2 + |\partial_t^m h|_{L^2}^2\,\mathrm{d}t.
\end{array}
\end{equation}
Note that we use the following inequality to control the surface tension term:
\begin{equation}\label{Sect6_Tangential_Estimate_6}
\begin{array}{ll}
\frac{\sigma}{2}\int\limits_{\{z=0\}} \frac{1}{\sqrt{1+|\nabla_y h|^2}}
\big(|\nabla_y\partial_t^{\ell}\mathcal{Z}^{\alpha} h|^2
- \frac{|\nabla_y h\cdot \nabla_y\partial_t^{\ell}\mathcal{Z}^{\alpha}h|^2}{1+|\nabla_y h|^2}\big) \mathrm{d}y \\[15pt]
\geq \frac{\sigma}{2}\int\limits_{\{z=0\}}\frac{1}{2(1+|\nabla_y h|^2)^{\frac{3}{2}}} |\nabla_y\partial_t^{\ell}\mathcal{Z}^{\alpha} h|^2 \,\mathrm{d}y

\geq \frac{\sigma}{4}\int\limits_{\{z=0\}}|\nabla_y\partial_t^{\ell}\mathcal{Z}^{\alpha} h|^2 \,\mathrm{d}y.
\end{array}
\end{equation}

When $|\alpha|=0$ and $0\leq\ell\leq m-1$, we have no bounds of $q$ and $\partial_t^{\ell} q$,
so we can not apply the integration by parts to the pressure terms.
The divergence free condition and the dynamical boundary condition will not be used here. Then
\begin{equation}\label{Sect6_Tangential_Estimate_7}
\begin{array}{ll}
\frac{1}{2}\frac{\mathrm{d}}{\mathrm{d}t}\int\limits_{\mathbb{R}^3_{-}} |\partial_t^{\ell} v|^2 \,\mathrm{d}\mathcal{V}_t
+ 2\e \int\limits_{\mathbb{R}^3_{-}} |\mathcal{S}^{\varphi}\partial_t^{\ell} v|^2 \,\mathrm{d}\mathcal{V}_t \\[12pt]

\leq - \int\limits_{\mathbb{R}^3_{-}} \partial_t^{\ell}\nabla^{\varphi} q \cdot \partial_t^{\ell} v \,\mathrm{d}\mathcal{V}_t
+ \int\limits_{\{z=0\}} 2\e \mathcal{S}^{\varphi}\partial_t^{\ell} v\NN \cdot \partial_t^{\ell} v \mathrm{d}y + \|\partial_z v\|_{X^{m-2}}^2 \\[6pt]\quad
+ \|\partial_z q\|_{X^{m-2}}^2
+ \sum\limits_{\ell=0}^{m-1}|\partial_t^{\ell+1} h|_{L^2}^2 + \text{b.t.} \\[11pt]

\lem \|\nabla q\|_{X^{m-1}}^2
+ \|\partial_t^{\ell} v\|_{L^2}^2 + \|\partial_z v\|_{X^{m-2}}^2 + \e\|\nabla_y\partial_t^{\ell} v\|_{X^{0,1}}^2 
+ \e\|\partial_z\nabla_y\partial_t^{\ell} v\|_{L^2}^2 \\[4pt]\quad
+ \sum\limits_{\ell=0}^m|\partial_t^{\ell} h|_{L^2}^2 + \text{b.t.}
\end{array}
\end{equation}

Combining $(\ref{Sect6_Tangential_Estimate_7})$ and $(\ref{Sect6_Height_Estimates_Lemma_Eq})$, we have
\begin{equation}\label{Sect6_Tangential_Estimate_8}
\begin{array}{ll}
\|\partial_t^{\ell} v\|^2 + \|\partial_t^{\ell} h\|^2 + \e\int\limits_0^t \|\nabla \partial_t^{\ell} v\|^2 \,\mathrm{d}t 

\lem \|v_0\|_{X^{m-1}}^2 + \|\nabla q\|_{L^4([0,T],X^{m-1})}^2 \\[6pt]\quad
+ \|\partial_z v\|_{L^4([0,T],X^{m-1})}^2 + |\partial_t^m h|_{L^4([0,T],L^2)}^2 + b.t.
\end{array}
\end{equation}

Sum $\ell$ and $\alpha$. By $(\ref{Sect6_Tangential_Estimate_5}),(\ref{Sect6_Tangential_Estimate_8})$ and Lemma $\ref{Sect6_Height_Estimates_Lemma}$, we get the estimate $(\ref{Sect6_Tangential_Estimate})$.
Thus, Lemma $\ref{Sect6_Tangential_Estimate_Lemma}$ is proved.
\end{proof}

In order to close our estimates of tangential derivatives, we need to bound $\|\partial_t^m v\|_{L^4([0,T],L^2)}^2$ and
$\|\partial_t^m h\|_{L^4([0,T],X^{0,1})}^2$, which appear in Lemma $\ref{Sect6_Tangential_Estimate_Lemma}$.
Thus, we estimate $\partial_t^m v$ and $\partial_t^m h$.
\begin{lemma}\label{Sect6_TimeDer_Estimate_Lemma}
$\partial_t^m v, \partial_t^m h, \partial_t^{m+1}h$ satisfies the following estimate:
\begin{equation}\label{Sect6_TimeDer_Estimate}
\begin{array}{ll}
\|\partial_t^m v\|_{L^4([0,T],L^2)}^2 + |\partial_t^m h|_{L^4([0,T],X^{0,1})}^2 + |\partial_t^{m+1}\nabla h|_{L^4([0,T],L^2)}^2
\\[6pt]

\lem \|\partial_t^m v_0\|_{L^2}^2 + g|\partial_t^m h_0|_{L^2}^2 + \sigma|\partial_t^m\nabla h_0|_{L^2}^2
+ \|\partial_z\partial_t^{m-1} v_0\|_{L^2}^2
\\[6pt]\quad

+ \|\partial_z v\|_{L^4([0,T],X^{m-1})}^2
+ \text{b.t.}
\end{array}
\end{equation}
\end{lemma}

\begin{proof}
In $(\ref{Sect6_Tangential_Estimate_2})$, let $\alpha=0$ and $\ell=m$. Then multiply with $\partial_t^m v$, integrate in $\mathbb{R}^3_{-}$, then we get
\begin{equation}\label{Sect6_TimeDer_Estimate_1}
\begin{array}{ll}
\frac{1}{2}\frac{\mathrm{d}}{\mathrm{d}t}\int\limits_{\mathbb{R}^3_{-}} |\partial_t^m v|^2 \,\mathrm{d}\mathcal{V}_t
- \int\limits_{\mathbb{R}^3_{-}} \partial_t^m q \, \nabla^{\varphi}\cdot \partial_t^m v \,\mathrm{d}\mathcal{V}_t
+ 2\e \int\limits_{\mathbb{R}^3_{-}} |\mathcal{S}^{\varphi}\partial_t^m v|^2 \,\mathrm{d}\mathcal{V}_t \\[14pt]

\leq \int\limits_{\{z=0\}} (2\e \mathcal{S}^{\varphi}\partial_t^m v\NN
- \partial_t^m q\NN)\cdot \partial_t^m v \mathrm{d}y
+ \|\partial_z v\|_{X^{m-1}}^2 + \|\nabla q\|_{X^{m-1}}^2 \\[11pt]\quad
+ |h|_{X^{m-1,2}}^2 + |\partial_t^m h|_{X^{0,1}}^2  + |\partial_t^{m+1} h|_{L^2}^2
+ \text{b.t.} \\[7pt]

\leq -\int\limits_{\{z=0\}} \big[
g \partial_t^m h
- \sigma \nabla_y\cdot\frac{1}{\sqrt{1+|\nabla_y h|^2}}
\big(\nabla_y\partial_t^m  h
- \frac{\nabla_y h(\nabla_y h\cdot \nabla_y\partial_t^m h)}{1+|\nabla_y h|^2}\big)
\big]\NN \\[13pt]\quad
\cdot \partial_t^m v \mathrm{d}y
+ \|\partial_z v\|_{X^{m-1}}^2 + \|\nabla q\|_{X^{m-1}}^2
+ |h|_{X^{m-1,2}}^2 + |\partial_t^m h|_{X^{0,1}}^2 \\[6pt]\quad
+ |\partial_t^{m+1} h|_{L^2}^2 + \text{b.t.} \\[6pt]

\leq \sigma\int\limits_{\{z=0\}} \nabla_y\cdot\frac{1}{\sqrt{1+|\nabla_y h|^2}}
\big(\nabla_y\partial_t^m  h
- \frac{\nabla_y h(\nabla_y h\cdot \nabla_y\partial_t^m h)}{1+|\nabla_y h|^2}\big)

\cdot(\partial_t \partial_t^m  h \\[13pt]\quad
+ v_y \cdot\nabla_y \partial_t^m  h) \mathrm{d}y

-\int\limits_{\{z=0\}}
g \partial_t^m h
\cdot(\partial_t \partial_t^m  h
+ v_y \cdot\nabla_y \partial_t^m  h) \mathrm{d}y \\[10pt]\quad
+ \|\partial_z v\|_{X^{m-1}}^2 + \|\nabla q\|_{X^{m-1}}^2
+ |h|_{X^{m-1,2}}^2 + |\partial_t^m h|_{X^{0,1}}^2  + |\partial_t^{m+1} h|_{L^2}^2 + \text{b.t.} \\[7pt]

\leq - \sigma \int\limits_{\{z=0\}} \frac{1}{\sqrt{1+|\nabla_y h|^2}}
\big(\nabla_y\partial_t^m  h
- \frac{\nabla_y h(\nabla_y h\cdot \nabla_y\partial_t^m h)}{1+|\nabla_y h|^2}\big)
\cdot(\partial_t \nabla_y\partial_t^m  h \\[13pt]\quad
+ v_y \cdot\nabla_y \nabla_y\partial_t^m  h) \mathrm{d}y

- \frac{g}{2}\frac{\mathrm{d}}{\mathrm{d}t}
\int\limits_{\{z=0\}} |\partial_t^m h|^2 \mathrm{d}y
+ \|\partial_z v\|_{X^{m-1}}^2 \\[10pt]\quad
+ \|\nabla q\|_{X^{m-1}}^2
+ |h|_{X^{m-1,2}}^2 + |\partial_t^m h|_{X^{0,1}}^2  + |\partial_t^{m+1} h|_{L^2}^2 + \text{b.t.} \\[8pt]

\leq - \frac{g}{2}\frac{\mathrm{d}}{\mathrm{d}t}
\int\limits_{\{z=0\}} |\partial_t^m h|^2 \mathrm{d}y

- \frac{\sigma}{2}\frac{\mathrm{d}}{\mathrm{d}t} \int\limits_{\{z=0\}} \frac{1}{\sqrt{1+|\nabla_y h|^2}}
\big(|\nabla_y\partial_t^m  h|^2
- \frac{|\nabla_y h\cdot \nabla_y\partial_t^m h|^2}{1+|\nabla_y h|^2}\big) \mathrm{d}y
\\[12pt]\quad
+ \|\partial_z v\|_{X^{m-1}}^2 + \|\nabla q\|_{X^{m-1}}^2
+ |h|_{X^{m-1,2}}^2 + |\partial_t^m h|_{X^{0,1}}^2  + |\partial_t^{m+1} h|_{L^2}^2 + \text{b.t.}
\end{array}
\end{equation}

The same as \cite{Wang_Xin_2015}, we will integrate in time twice, we get the $L^4([0,T],L^2)$ type estimate.
After the first integration in time, we have
\begin{equation}\label{Sect6_TimeDer_Estimate_3}
\begin{array}{ll}
\|\partial_t^m v\|_{L^2}^2 + g|\partial_t^m h|_{L^2}^2 + \frac{\sigma}{2}|\partial_t^m\nabla h|_{L^2}^2
+ 4\e \int\limits_{\mathbb{R}^3_{-}} |\mathcal{S}^{\varphi}\partial_t^m v|^2 \,\mathrm{d}\mathcal{V}_t
\\[6pt]

\lem \|\partial_t^m v_0\|_{L^2}^2 + g|\partial_t^m h_0|_{L^2}^2 + \sigma|\partial_t^m\nabla h_0|_{L^2}^2
+ \int\limits_0^t\int\limits_{\mathbb{R}^3_{-}} \partial_t^m q \, \nabla^{\varphi}\cdot \partial_t^m v \,\mathrm{d}\mathcal{V}_t\mathrm{d}t 
\end{array}
\end{equation}

\begin{equation*}
\begin{array}{ll}
\quad
+ \int\limits_0^t\|\partial_z v\|_{X^{m-1}}^2 + \|\nabla q\|_{X^{m-1}}^2
+ |h|_{X^{m-1,2}}^2 + |\partial_t^m h|_{X^{0,1}}^2  + |\partial_t^{m+1} h|_{L^2}^2\mathrm{d}t + \text{b.t.} \\[6pt]

\lem \|\partial_t^m v_0\|_{L^2}^2 + g|\partial_t^m h_0|_{L^2}^2 + \sigma|\partial_t^m\nabla h_0|_{L^2}^2
+ \int\limits_0^t\int\limits_{\mathbb{R}^3_{-}} \partial_t^m q \, \nabla^{\varphi}\cdot \partial_t^m v \,\mathrm{d}\mathcal{V}_t\mathrm{d}t \\[6pt]\quad

+ \|\partial_z v\|_{L^4([0,T],X^{m-1})}^2 + |\partial_t^m h|_{L^4([0,T],X^{0,1})}^2
+ |\partial_t^m v|_{L^4([0,T],L^2)}^2 + \text{b.t.}.
\end{array}
\end{equation*}

Now we deal with the pressure term:
\begin{equation}\label{Sect6_TimeDer_Estimate_4}
\begin{array}{ll}
\int\limits_0^t\int\limits_{\mathbb{R}^3_{-}} \partial_t^m q \, \nabla^{\varphi}\cdot \partial_t^m v \,\mathrm{d}\mathcal{V}_t\mathrm{d}t
= -\int\limits_0^t\int\limits_{\mathbb{R}^3_{-}} \partial_t^m q \, [\partial_t^m, \nabla^{\varphi}\cdot] v \,\mathrm{d}\mathcal{V}_t\mathrm{d}t \\[7pt]

= \sum\limits_{\ell_1>0}\int\limits_0^t
\int\limits_{\mathbb{R}^3_{-}} \partial_t^m q \, \big(\partial_z \varphi\partial_t^{\ell_1}(\frac{\NN}{\partial_z \varphi}) \big)
\cdot \partial_t^{\ell_2}\partial_z v \,\mathrm{d}x\mathrm{d}t \\[11pt]

= \sum\limits_{\ell_1>0}\int\limits_0^t\frac{\mathrm{d}}{\mathrm{d}t}\int\limits_{\mathbb{R}^3_{-}} \partial_t^{m-1} q \,
\big(\partial_z \varphi\partial_t^{\ell_1}(\frac{\NN}{\partial_z \varphi}) \big)
\cdot \partial_t^{\ell_2}\partial_z v \,\mathrm{d}x\mathrm{d}t \\[11pt]\quad

- \sum\limits_{\ell_1>0}\int\limits_0^t\int\limits_{\mathbb{R}^3_{-}} \partial_t^{m-1} q \,
\partial_t\big(\partial_z \varphi\partial_t^{\ell_1}(\frac{\NN}{\partial_z \varphi}) \big)
\cdot \partial_t^{\ell_2}\partial_z v \,\mathrm{d}x\mathrm{d}t \\[11pt]\quad

- \sum\limits_{\ell_1>0}\int\limits_0^t\int\limits_{\mathbb{R}^3_{-}} \partial_t^{m-1} q \,
\big(\partial_z \varphi\partial_t^{\ell_1}(\frac{\NN}{\partial_z \varphi}) \big)
\cdot \partial_t^{\ell_2+1}\partial_z v \,\mathrm{d}x\mathrm{d}t \\[11pt]

= \sum\limits_{\ell_1>0}\int\limits_{\{z=0\}} \partial_t^{m-1} q \,
\big(\partial_z \varphi\partial_t^{\ell_1}(\frac{\NN}{\partial_z \varphi}) \big)
\cdot \partial_t^{\ell_2}v \,\mathrm{d}y \\[11pt]\quad

-\sum\limits_{\ell_1>0}\int\limits_{\mathbb{R}^3_{-}} \partial_z\big[\partial_t^{m-1} q \,
\big(\partial_z \varphi\partial_t^{\ell_1}(\frac{\NN}{\partial_z \varphi}) \big)\big]
\cdot \partial_t^{\ell_2} v \,\mathrm{d}x \\[11pt]\quad

- \sum\limits_{\ell_1>0} \int\limits_{\{z=0\}} \partial_t^{m-1} q|_{t=0} \,
\big(\partial_z \varphi\partial_t^{\ell_1}(\frac{\NN}{\partial_z \varphi})|_{t=0} \big)
\cdot \partial_t^{\ell_2} v|_{t=0} \,\mathrm{d}y \\[11pt]\quad

+ \sum\limits_{\ell_1>0} \int\limits_{\mathbb{R}^3_{-}} \partial_z\big[\partial_t^{m-1} q|_{t=0} \,
\big(\partial_z \varphi\partial_t^{\ell_1}(\frac{\NN}{\partial_z \varphi}) \big)|_{t=0}\big]
\cdot \partial_t^{\ell_2} v|_{t=0} \,\mathrm{d}x \\[11pt]\quad

- \sum\limits_{\ell_1>0}\int\limits_0^t\int\limits_{\{z=0\}} \partial_t^{m-1} q \,
\partial_t\big(\partial_z \varphi\partial_t^{\ell_1}(\frac{\NN}{\partial_z \varphi}) \big)
\cdot \partial_t^{\ell_2} v \,\mathrm{d}y\mathrm{d}t \\[11pt]\quad

+ \sum\limits_{\ell_1>0}\int\limits_0^t\int\limits_{\mathbb{R}^3_{-}} \partial_z\big[\partial_t^{m-1} q \,
\partial_t\big(\partial_z \varphi\partial_t^{\ell_1}(\frac{\NN}{\partial_z \varphi}) \big)\big]
\cdot \partial_t^{\ell_2}v \,\mathrm{d}x\mathrm{d}t \\[11pt]\quad

- \sum\limits_{\ell_1>0} \int\limits_0^t\int\limits_{\{z=0\}} \partial_t^{m-1} q \,
\big(\partial_z \varphi\partial_t^{\ell_1}(\frac{\NN}{\partial_z \varphi}) \big)
\cdot \partial_t^{\ell_2+1} v \,\mathrm{d}y\mathrm{d}t \\[11pt]\quad

+ \sum\limits_{\ell_1>0} \int\limits_0^t\int\limits_{\mathbb{R}^3_{-}} \partial_z\big[\partial_t^{m-1} q \,
\big(\partial_z \varphi\partial_t^{\ell_1}(\frac{\NN}{\partial_z \varphi}) \big)\big]
\cdot \partial_t^{\ell_2+1} v \,\mathrm{d}x\mathrm{d}t .
\end{array}
\end{equation}

$(\ref{Sect6_TimeDer_Estimate_4})$ contains $\partial_z\big[\partial_t^{m-1} q f \big]$, where $f$ represents the terms
$\big(\partial_z \varphi\partial_t^{\ell_1}(\frac{\NN}{\partial_z \varphi}) \big)$,
$\big(\partial_z \varphi\partial_t^{\ell_1}(\frac{\NN}{\partial_z \varphi}) \big)|_{t=0}$,
$\partial_t\big(\partial_z \varphi\partial_t^{\ell_1}(\frac{\NN}{\partial_z \varphi}) \big)$.
By using Hardy's inequality $(\ref{Sect1_HardyIneq})$, we get
\begin{equation}\label{Sect6_TimeDer_Estimate_5}
\begin{array}{ll}
\int\limits_{\mathbb{R}^3_{-}}\partial_z\big[\partial_t^{m-1} q f \big]\cdot \partial_t^{\ell_2}v \,\mathrm{d}x \\[7pt]
= \int\limits_{\mathbb{R}^3_{-}}(\partial_z\partial_t^{m-1} q) f \cdot \partial_t^{\ell_2}v \,\mathrm{d}x
+ \int\limits_{\mathbb{R}^3_{-}}\frac{1}{1-z}\partial_t^{m-1} q[(1-z)\partial_z f]\cdot \partial_t^{\ell_2}v \,\mathrm{d}x 
\end{array}
\end{equation}

\begin{equation*}
\begin{array}{ll}
\lem \|\partial_z \partial_t^{m-1} q\|_{L^2}^2 + \|\frac{1}{1-z}\partial_t^{m-1} q\|_{L^2}^2
+ \|(1-z)\partial_z f\|_{L^2}^2 + \|\partial_t^{\ell_2}v\|_{L^2}^2 \\[7pt]

\lem \|\partial_z \partial_t^{m-1} q\|_{L^2}^2 + \big|\partial_t^{m-1} q|_{z=0}\big|_{L^2}^2
+ |\partial_t^m h|_{X^{0,1}}^2 + \|\partial_t^{m-1}v\|_{L^2}^2, \hspace{1.3cm}
\end{array}
\end{equation*}
where $(1-z)\partial_z f \sim (1-z)\partial_z \psi\ast \partial_t^{\ell} h$ and $(1-z)\partial_z\psi \in L^1(\mathrm{d}z)$.

Denote
\begin{equation}\label{Sect6_TimeDer_Estimate_6}
\begin{array}{ll}
\mathcal{I}_8 := \|\partial_z \partial_t^{m-1} q\|_{L^2}^2 + \big|\partial_t^{m-1} q|_{z=0}\big|_{L^2}^2
+ |\partial_t^m h|_{X^{0,1}}^2 \\[5pt]\hspace{0.94cm}
+ \|\partial_t^{m-1}v\|_{L^2}^2 + \big|\partial_t^{m-1}v|_{z=0}\big|_{L^2}^2 \\[8pt]\hspace{0.53cm}

\lem \|\partial_t^m v\|_{L^2}^2 + \|\partial_z\partial_t^{m-1} v\|_{L^2}^2
+ |\partial_t^m h|_{X^{0,1}}^2 + b.t.
\end{array}
\end{equation}

Plug $(\ref{Sect6_TimeDer_Estimate_5})$ into $(\ref{Sect6_TimeDer_Estimate_4})$, we get
\begin{equation}\label{Sect6_TimeDer_Estimate_7}
\begin{array}{ll}
\int\limits_0^t\int\limits_{\mathbb{R}^3_{-}} \partial_t^m q \, \nabla^{\varphi}\cdot \partial_t^m v
\,\mathrm{d}\mathcal{V}_t\mathrm{d}t

\lem \mathcal{I}_8|_{t=0} + \mathcal{I}_8 + \int\limits_0^T \mathcal{I}_8 \,\mathrm{d}s \\[10pt]

\lem \|\partial_t^m v_0\|_{L^2}^2 + \|\partial_z\partial_t^{m-1} v_0\|_{L^2}^2
+ |\partial_t^m h_0|_{X^{0,1}}^2

+ \|\partial_t^m v\|_{L^2}^2 + \|\partial_z\partial_t^{m-1} v\|_{L^2}^2 \\[6pt]\quad
+ |\partial_t^m h|_{X^{0,1}}^2

+ \|\partial_t^m v\|_{L^4([0,T],L^2)}^2
+ \|\partial_z\partial_t^{m-1} v\|_{L^4([0,T],L^2)}^2
+ |\partial_t^m h|_{L^4([0,T],L^2)}^2.
\end{array}
\end{equation}

By $(\ref{Sect6_TimeDer_Estimate_3})$ and $(\ref{Sect6_TimeDer_Estimate_7})$, we get
\begin{equation}\label{Sect6_TimeDer_Estimate_8}
\begin{array}{ll}
\|\partial_t^m v\|_{L^2}^2 + g|\partial_t^m h|_{L^2}^2 + \frac{\sigma}{2}|\partial_t^m\nabla_y h|_{L^2}^2
\\[6pt]

\lem \|\partial_t^m v_0\|_{L^2}^2 + g|\partial_t^m h_0|_{L^2}^2 + \sigma|\partial_t^m\nabla_y h_0|_{L^2}^2
 + \|\partial_z\partial_t^{m-1} v_0\|_{L^2}^2

+ \|\partial_t^m v\|_{L^2}^2 \\[6pt]\quad
+ \|\partial_z\partial_t^{m-1} v\|_{L^2}^2
+ |\partial_t^m h|_{X^{0,1}}^2

+ \|\partial_t^m v\|_{L^4([0,T],L^2)}^2
+ \|\partial_z v\|_{L^4([0,T],X^{m-1})}^2 \\[6pt]\quad
+ |\partial_t^m h|_{L^4([0,T],X^{0,1})}^2 + \text{b.t.}
\end{array}
\end{equation}

Square $(\ref{Sect6_TimeDer_Estimate_8})$ and integrate in time again, apply the integral form of Gronwall's inequality, we have
\begin{equation}\label{Sect6_TimeDer_Estimate_9}
\begin{array}{ll}
\|\partial_t^m v\|_{L^4([0,T],L^2)}^2 + g|\partial_t^m h|_{L^4([0,T],L^2)}^2 + \frac{\sigma}{2}|\partial_t^m\nabla_y h|_{L^4([0,T],L^2)}^2
\\[6pt]

\lem \|\partial_t^m v_0\|_{L^2}^2 + g|\partial_t^m h_0|_{L^2}^2 + \sigma|\partial_t^m\nabla_y h_0|_{L^2}^2
+ \|\partial_z\partial_t^{m-1} v_0\|_{L^2}^2
\\[6pt]\quad

+ \|\partial_z v\|_{L^4([0,T],X^{m-1})}^2
+ \text{b.t.}
\end{array}
\end{equation}
Thus, Lemma $\ref{Sect6_TimeDer_Estimate_Lemma}$ is proved.
\end{proof}

The same as the $\sigma=0$ case, we have the estimates of normal derivatives:
\begin{equation}\label{Sect6_NormalDer_Estimates}
\begin{array}{ll}
\partial_z v,\, \omega \in L^4([0,T],X^{m-1}) \cap L^{\infty}([0,T],X^{m-2}).
\end{array}
\end{equation}
Couple Lemmas $\ref{Sect6_Pressure_Estimate_Eq},
\ref{Sect6_Tangential_Estimate_Lemma}, \ref{Sect6_TimeDer_Estimate_Lemma}$ with the normal estimates $(\ref{Sect6_NormalDer_Estimates})$,
it is standard to prove Proposition $\ref{Sect1_Proposition_Regularity_Tension}$.

%%% find 7
\section{Convergence Rates of Inviscid Limit for Fixed $\sigma>0$}

In this section, we estimate convergence rates of the inviscid limit for the $\sigma>0$ case.

We denote $\hat{v} =v^{\e} -v,\hat{q} =q^{\e} -q, \hat{h} =h^{\e} -h$, we denote the $i-$th components of $v^{\e}$ and $v$ by $v^{\e,i}$ and $v^i$ respectively.
$\hat{v},\hat{h},\hat{q}$ satisfy the following equations
\begin{equation}\label{Sect7_DifferenceEq_1}
\left\{\begin{array}{ll}
\partial_t^{\varphi^{\e}}\hat{v}
+ v^{\e} \cdot\nabla^{\varphi^{\e}} \hat{v}
+ \nabla^{\varphi^{\e}} \hat{q}
- 2\e \nabla^{\varphi^{\e}} \cdot\mathcal{S}^{\varphi^{\e}} \hat{v} \\[6pt]\quad
= \partial_z^{\varphi} v \partial_t^{\varphi^{\e}}\hat{\eta}
+ v^{\e}\cdot \nabla^{\varphi^{\e}}\hat{\eta}\, \partial_z^{\varphi} v
- \hat{v}\cdot\nabla^{\varphi} v
+ \partial_z^{\varphi} q\nabla^{\varphi^{\e}}\hat{\eta}
+ \e \triangle^{\varphi^{\e}} v, \quad  x\in\mathbb{R}^3_{-}, \\[9pt]

\nabla^{\varphi^{\e}}\cdot \hat{v} = \partial_z^{\varphi}v \cdot\nabla^{\varphi^{\e}}\hat{\eta}, \hspace{6.42cm} x\in\mathbb{R}^3_{-},\\[7pt]

\partial_t \hat{h} + v_y\cdot \nabla \hat{h} = \hat{v}\cdot\NN^{\e},  \hspace{6.2cm} \{z=0\},
\\[8pt]

\hat{q}\NN^{\e} -2\e \mathcal{S}^{\varphi^\e} \hat{v} \,\NN^\e
= g \hat{h} \NN^{\e} -\sigma \nabla_y \cdot \big( \mathfrak{H}_1\nabla_y\hat{h} \\[5pt]\quad
+ \mathfrak{H}_2 \nabla_y\hat{h}\cdot\nabla_y(h^{\e}+h) \nabla_y (h^{\e} +h)\big)\NN^{\e}
+ 2\e \mathcal{S}^{\varphi^\e}v\,\NN^\e, \hspace{1.47cm} \{z=0\},
\\[9pt]

(\hat{v},\hat{h})|_{t=0} = (v_0^\e -v_0,h_0^\e -h_0),
\end{array}\right.
\end{equation}
where the quantities $\mathfrak{H}_1$ and $\mathfrak{H}_2$ are defined as
\begin{equation}\label{Sect7_DifferenceEq_2}
\begin{array}{ll}
\mathfrak{H}_1 = \frac{1}{2\sqrt{1+|\nabla_y h^{\e}|^2}}+ \frac{1}{2\sqrt{1+|\nabla_y h|^2}}, \\[12pt]

\mathfrak{H}_2 = \frac{-1}
{2\sqrt{1+|\nabla_y h^{\e}|^2}\sqrt{1+|\nabla_y h|^2}(\sqrt{1+|\nabla_y h^{\e}|^2} + \sqrt{1+|\nabla_y h|^2})}.
\end{array}
\end{equation}

Since the estimates for normal derivatives are the same as the $\sigma>0$ case, we focus on the estimates of the pressure and tangential derivatives.

The following lemma concerns the estimate of $\nabla\hat{q} = \nabla q^{\e} -\nabla q$:
\begin{lemma}\label{Sect7_Pressure_Lemma}
Assume $0\leq s\leq k-1,\, k \leq m-1$, the difference of the pressure $\hat{q}$ has the following gradient estimate:
\begin{equation}\label{Sect7_Pressure_Lemma_Eq}
\begin{array}{ll}
\|\nabla \hat{q}\|_{X^s}
\lem \|\hat{v}\|_{X^{s,1}} + \|\partial_z\hat{v}\|_{X^s} + \|\partial_t^{s+1}\hat{v}\|_{L^2} + |\partial_t^s\hat{h}\|_{X^{0,\frac{1}{2}}} + |\hat{h}|_{X^{s,\frac{3}{2}}} +O(\e).
\end{array}
\end{equation}
\end{lemma}

\begin{proof}
The Navier-Stokes pressure $q^{\e}$ satisfies the elliptic equations $(\ref{Sect1_Pressure_Neumann})$,
while the Euler pressure satisfies the following equations:
\begin{equation}\label{Sect1_Pressure_Neumann_EulerEq}
\left\{\begin{array}{ll}
\triangle^{\varphi} q = -\partial_j^{\varphi} v^i \partial_i^{\varphi} v^j, \\[6pt]

\nabla^{\varphi} q\cdot\NN|_{z=0} = - \partial_t^{\varphi} v\cdot\NN - v\cdot\nabla^{\varphi} v\cdot\NN,
\end{array}\right.
\end{equation}

Then the difference between boundary values is
\begin{equation}\label{Sect7_Pressure_Estimates_1}
\begin{array}{ll}
\nabla^{\varphi^{\e}} q^{\e}\cdot\NN^{\e} |_{z=0} - \nabla^{\varphi} q\cdot\NN |_{z=0} \\[5pt]
= \e\triangle^{\varphi^{\e}} v^{\e}\cdot\NN^{\e} - (\partial_t^{\varphi^{\e}} v^{\e}\cdot\NN^{\e}-\partial_t^{\varphi} v\cdot\NN)
- (v^{\e}\cdot\nabla^{\varphi^{\e}} v^{\e}\cdot\NN^{\e} - v\cdot\nabla^{\varphi} v\cdot\NN), 
\end{array}
\end{equation}

\begin{equation*}
\begin{array}{ll}
(\nabla^{\varphi^{\e}} \hat{q} -\partial_z^{\varphi}q \nabla^{\varphi^{\e}}\hat{\eta}) \cdot\NN^{\e} |_{z=0}
+ \nabla^{\varphi} q\cdot \hat{\NN}|_{z=0} \\[6pt]

= \e\triangle^{\varphi^{\e}} v^{\e}\cdot\NN^{\e}
- (\partial_t^{\varphi^{\e}} \hat{v} - \partial_z^{\varphi} v\partial_t^{\varphi^{\e}}\hat{\eta})\cdot\NN^{\e}
- \partial_t^{\varphi} v \cdot\hat{\NN} \\[6pt]\quad
- (v^{\e}\cdot\nabla^{\varphi^{\e}} \hat{v} - v^{\e}\cdot\nabla^{\varphi^{\e}} \hat{\eta}\, \partial_z^{\varphi}v
+ \hat{v}\cdot\nabla^{\varphi} v)\cdot\NN^{\e}
- v\cdot\nabla^{\varphi} v\cdot\hat{\NN}, \\[10pt]

\nabla^{\varphi^{\e}} \hat{q} \cdot\NN^{\e} |_{z=0} =\partial_z^{\varphi}q \nabla^{\varphi^{\e}}\hat{\eta} \cdot\NN^{\e} |_{z=0}
- \nabla^{\varphi} q\cdot \hat{\NN}|_{z=0} \\[7pt]\quad

+ \e\triangle^{\varphi^{\e}} v^{\e}\cdot\NN^{\e}
- (\partial_t \hat{v} + v^{\e}_y\cdot\nabla_y \hat{v})\cdot\NN^{\e}
- (\partial_t v + v_y\cdot\nabla_y v) \cdot\hat{\NN} \\[7pt]\quad
+ [(\partial_t\hat{\eta} + v_y^{\e}\cdot\nabla_y \hat{\eta})\, \partial_z^{\varphi}v
- \hat{v}\cdot\nabla^{\varphi} v]\cdot\NN^{\e} :=\mathcal{I}_9.
\end{array}
\end{equation*}

Similar to $(\ref{Sect5_Pressure_Estimates_3})$, $\hat{q}$ satisfies the following elliptic equation:
\begin{equation}\label{Sect7_Pressure_Estimates_2}
\left\{\begin{array}{ll}
\nabla\cdot(\textsf{E}^{\e}\nabla \hat{q})
= -\nabla\cdot((\textsf{E}^{\e} -\textsf{E}) \nabla q) - \nabla\cdot [(\textsf{P}^{\e} - \textsf{P}) (v \cdot\nabla^{\varphi} v)] \\[5pt]\hspace{2.07cm}

-\nabla\cdot [\textsf{P}^{\e} (v^{\e} \cdot\nabla^{\varphi^{\e}} \hat{v} - v^{\e}\cdot\nabla^{\varphi^{\e}}\hat{\varphi}\partial_z^{\varphi} v
+ \hat{v}\cdot\nabla^{\varphi} v)], \\[7pt]

\nabla^{\varphi^{\e}} \hat{q} \cdot\NN^{\e} |_{z=0} = \mathcal{I}_9.
\end{array}\right.
\end{equation}

The matrix $\textsf{E}^{\e}$ is definitely positive, then it is standard to prove that $\hat{q}$ satisfies the following gradient estimate:
\begin{equation}\label{Sect7_Pressure_Estimates_3}
\begin{array}{ll}
\|\nabla \hat{q}\|_{X^s} \lem \|(\textsf{E}^{\e} -\textsf{E}) \nabla q \|_{X^s}
+ \|(\textsf{P}^{\e} - \textsf{P}) (v \cdot\nabla^{\varphi} v)\|_{X^s} \\[5pt]\quad

+ \|\textsf{P}^{\e} (v^{\e} \cdot\nabla^{\varphi^{\e}} \hat{v} - v^{\e}\cdot\nabla^{\varphi^{\e}}\hat{\varphi}\partial_z^{\varphi} v
+ \hat{v}\cdot\nabla^{\varphi} v)\|_{X^s}

+ |\mathcal{I}_9|_{X^{s,-\frac{1}{2}}} \\[9pt]

\lem \|\textsf{E}^{\e} -\textsf{E}\|_{X^s} + \|\textsf{P}^{\e} - \textsf{P}\|_{X^s}
+ \|\hat{v}\|_{X^s} + \|\nabla\hat{v}\|_{X^s}
+ \|\nabla\hat{\varphi}\|_{X^s}

+ |\mathcal{I}_9|_{X^{s,-\frac{1}{2}}} \\[9pt]

\lem \|\hat{v}\|_{X^{s,1}} + \|\partial_z\hat{v}\|_{X^s} + |\hat{h}|_{X^{s,\frac{1}{2}}}
+ |\mathcal{I}_9|_{X^{s,-\frac{1}{2}}}.
\end{array}
\end{equation}

Now we estimate the boundary terms.
\begin{equation}\label{Sect7_Pressure_Estimates_4}
\begin{array}{ll}
|\mathcal{I}_9|_{X^{s,-\frac{1}{2}}}

\lem \big|\partial_t\hat{v}\cdot\NN^{\e}|_{z=0}\big|_{X^{s,-\frac{1}{2}}}
+ \big|v_y^{\e}\cdot\nabla_y \hat{v}\cdot\NN^{\e}|_{z=0}\big|_{X^{s,-\frac{1}{2}}}
+ \big|\hat{v}\cdot\NN^{\e}|_{z=0}\big|_{X^{s,-\frac{1}{2}}} \\[9pt]\quad
+ \big|\partial_t\hat{\eta}\cdot\NN^{\e}|_{z=0}\big|_{X^{s,-\frac{1}{2}}}
+ \big|\nabla_y \hat{\eta}\cdot\NN^{\e}|_{z=0}\big|_{X^{s,-\frac{1}{2}}}
+|\hat{h}|_{X^{s,\frac{1}{2}}} +O(\e) \\[10pt]

\lem |\hat{h}|_{X^{s,\frac{1}{2}}}
+ \|\partial_t\hat{v}\|_{X^s} + \|\hat{v}\|_{X^{s,1}}
+ \|\partial_t\hat{\eta}\|_{X^s} + \|\partial_t\hat{\eta}\|_{X^{s,1}} + \|\nabla\hat{\eta}\|_{X^{s,1}} +O(\e) \\[10pt]

\lem \|\partial_t^{s+1}\hat{v}\|_{L^2} + \|\hat{v}\|_{X^{s,1}}
+ |\partial_t\hat{h}\|_{X^{s,\frac{1}{2}}} + |\hat{h}|_{X^{s,\frac{3}{2}}} +O(\e),
\end{array}
\end{equation}
refer to $(\ref{Sect6_Pressure_Estimate_2})$ for the estimate of $\e\triangle^{\varphi^{\e}}\cdot\NN^{\e}$.

By $(\ref{Sect7_Pressure_Estimates_3})$ and $(\ref{Sect7_Pressure_Estimates_4})$, we obtain $(\ref{Sect7_Pressure_Lemma_Eq})$.
Thus, Lemma $\ref{Sect7_Pressure_Lemma}$ is proved.
\end{proof}

Before estimating tangential derivatives of $\hat{v}$, we have the estimate of $\partial_t^{\ell} \hat{h}$
by using the kinetical boundary condition $(\ref{Sect7_DifferenceEq_1})_3$, which is the same with $(\ref{Sect1_T_Derivatives_Difference_Eq})_3$,
we give the following lemma without proof, which is the same with Lemma $\ref{Sect5_Height_Estimates_Lemma}$.
\begin{lemma}\label{Sect7_Height_Estimates_Lemma}
Assume $0\leq k\leq m-2$, $0\leq\ell\leq k-1$,
$\partial_t^{\ell}\hat{h}$ have the estimates:
\begin{equation}\label{Sect7_Height_Estimates_Lemma_Eq}
\begin{array}{ll}
\int\limits_{\mathbb{R}^2} |\partial_t^{\ell}\hat{h}|^2 \,\mathrm{d}y

\lem |\hat{h}_0|_{X^{k-1}}^2
+ \int\limits_0^t |\hat{h}|_{X^{k-1,1}}^2 + \|\hat{v}\|_{X^{k-1,1}}^2 \,\mathrm{d}t
+ \|\partial_z\hat{v}\|_{L^4([0,T],X^{k-1})}^2.
\end{array}
\end{equation}
\end{lemma}

We develop the estimates for tangential derivatives.
\begin{lemma}\label{Sect7_Tangential_Estimates_Lemma}
Assume $0\leq k\leq m-2$,
$\partial_t^{\ell}\mathcal{Z}^{\alpha}\hat{v}$ and $\partial_t^{\ell}\mathcal{Z}^{\alpha}\hat{h}$ have the estimates:
\begin{equation}\label{Sect7_Tangential_Estimates_Lemma_Eq}
\begin{array}{ll}
\|\hat{v}\|_{X^{k-1,1}}^2 + |\hat{h}|_{X^{k-1,1}}^2

\lem \|\hat{v}_0\|_{X^{k-1,1}}^2 + |\hat{h}_0|_{X^{k-1,1}}^2
+ \int\limits_0^t \|\partial_z \hat{v}\|_{X^{k-1}}^2 \\[8pt]\quad
+ \|\hat{v}\|_{X^{k-1,1}}^2
+ \|\hat{h}\|_{X^{k-1,1}}^2 \,\mathrm{d}t + O(\e).
\end{array}
\end{equation}
\end{lemma}

\begin{proof}
$(\partial_t^{\ell}\mathcal{Z}^{\alpha}\hat{v}, \partial_t^{\ell}\mathcal{Z}^{\alpha}\hat{h}, \partial_t^{\ell}\mathcal{Z}^{\alpha}\hat{q})$ satisfy the following equations:
\begin{equation}\label{Sect7_TangentialEstimates_Diff_Eq}
\left\{\begin{array}{ll}
\partial_t^{\varphi^{\e}} \partial_t^{\ell}\mathcal{Z}^{\alpha}\hat{v}
+ v^{\e} \cdot\nabla^{\varphi^{\e}} \partial_t^{\ell}\mathcal{Z}^{\alpha}\hat{v}
+ \nabla^{\varphi^{\e}} \partial_t^{\ell}\mathcal{Z}^{\alpha}\hat{q}
- 2\e \nabla^{\varphi^{\e}}\cdot\mathcal{S}^{\varphi^{\e}} \partial_t^{\ell}\mathcal{Z}^{\alpha}\hat{v} \\[8pt]\quad

= \e\partial_t^{\ell}\mathcal{Z}^{\alpha}\triangle^{\varphi^\e} v
+ 2\e [\partial_t^{\ell}\mathcal{Z}^{\alpha},\nabla^{\varphi^{\e}}\cdot]\mathcal{S}^{\varphi^{\e}} \hat{v}
+ 2\e \nabla^{\varphi^{\e}}\cdot[\partial_t^{\ell}\mathcal{Z}^{\alpha}, \mathcal{S}^{\varphi^{\e}}] \hat{v} \\[8pt]\quad

+ \partial_z^{\varphi} v \partial_t^{\varphi^{\e}} \partial_t^{\ell}\mathcal{Z}^{\alpha}\hat{\varphi}
+ \partial_z^{\varphi} v \, v^{\e}\cdot \nabla^{\varphi^{\e}}\partial_t^{\ell}\mathcal{Z}^{\alpha}\hat{\varphi}
- \partial_t^{\ell}\mathcal{Z}^{\alpha}\hat{v}\cdot\nabla^{\varphi} v
+ \partial_z^{\varphi} q\nabla^{\varphi^{\e}}\partial_t^{\ell}\mathcal{Z}^{\alpha}\hat{\varphi}
\\[8pt]\quad

- [\partial_t^{\ell}\mathcal{Z}^{\alpha},\partial_t^{\varphi^{\e}}]\hat{v}
+ [\partial_t^{\ell}\mathcal{Z}^{\alpha}, \partial_z^{\varphi} v \partial_t^{\varphi^{\e}}]\hat{\varphi}

- [\partial_t^{\ell}\mathcal{Z}^{\alpha}, v^{\e} \cdot\nabla^{\varphi^{\e}}] \hat{v}
- [\partial_t^{\ell}\mathcal{Z}^{\alpha}, \nabla^{\varphi} v\cdot]\hat{v}\\[8pt]\quad
+ [\partial_t^{\ell}\mathcal{Z}^{\alpha}, \partial_z^{\varphi} v \, v^{\e}\cdot \nabla^{\varphi^{\e}}]\hat{\varphi}

- [\partial_t^{\ell}\mathcal{Z}^{\alpha},\nabla^{\varphi^{\e}}] \hat{q}
+ [\partial_t^{\ell}\mathcal{Z}^{\alpha},\partial_z^{\varphi} q\nabla^{\varphi^{\e}}]\hat{\varphi} :=\mathcal{I}_{10}, \\[11pt]

\nabla^{\varphi^{\e}}\cdot \partial_t^{\ell}\mathcal{Z}^{\alpha}\hat{v}
= \partial_z^{\varphi}v \cdot\nabla^{\varphi^{\e}}\partial_t^{\ell}\mathcal{Z}^{\alpha}\hat{\eta}
-[\partial_t^{\ell}\mathcal{Z}^{\alpha},\nabla^{\varphi^{\e}}\cdot] \hat{v}
+ [\partial_t^{\ell}\mathcal{Z}^{\alpha},\partial_z^{\varphi}v \cdot\nabla^{\varphi^{\e}}]\hat{\eta}, \\[11pt]

\partial_t \partial_t^{\ell}\mathcal{Z}^{\alpha}\hat{h} + v_y^{\e}\cdot \nabla_y \partial_t^{\ell}\mathcal{Z}^{\alpha}\hat{h}
- \NN^{\e}\cdot \partial_t^{\ell}\mathcal{Z}^{\alpha}\hat{v}
 = - \hat{v}_y \cdot \nabla_y \partial_t^{\ell}\mathcal{Z}^{\alpha} h
- \partial_y \hat{h}\cdot \partial_t^{\ell}\mathcal{Z}^{\alpha}v_y \\[7pt]\quad
 + [\partial_t^{\ell}\mathcal{Z}^{\alpha}, \hat{v},\NN^{\e}]
 - [\partial_t^{\ell}\mathcal{Z}^{\alpha}, v_y, \partial_y \hat{h}], \\[11pt]

\partial_t^{\ell}\mathcal{Z}^{\alpha} \hat{q}\NN^{\e}
-2\e \mathcal{S}^{\varphi^{\e}} \partial_t^{\ell}\mathcal{Z}^{\alpha} \hat{v}\,\NN^{\e}

- g\partial_t^{\ell}\mathcal{Z}^{\alpha}\hat{h}\NN^{\e}
+ \sigma \nabla_y \cdot \big( \mathfrak{H}_1\nabla_y \partial_t^{\ell}\mathcal{Z}^{\alpha} \hat{h} \big) \NN^{\e} \\[8pt]\quad
+ \sigma\nabla_y \cdot \big( \mathfrak{H}_2 \nabla_y \partial_t^{\ell}\mathcal{Z}^{\alpha} \hat{h}\cdot\nabla_y(h^{\e}+h) \nabla_y (h^{\e} +h)\big)\NN^{\e}
 
= \mathcal{I}_{11,1} + \mathcal{I}_{11,2},
\\[11pt]

(\partial_t^{\ell}\mathcal{Z}^{\alpha}\hat{v},\partial_t^{\ell}\mathcal{Z}^{\alpha}\hat{h})|_{t=0}
= (\partial_t^{\ell}\mathcal{Z}^{\alpha}v_0^\e -\partial_t^{\ell}\mathcal{Z}^{\alpha}v_0,
\partial_t^{\ell}\mathcal{Z}^{\alpha}h_0^\e -\partial_t^{\ell}\mathcal{Z}^{\alpha}h_0),
\end{array}\right.
\end{equation}
where 
\begin{equation}\label{Sect7_TangentialEstimates_Diff_Eq_Appendix}
\begin{array}{ll}
\mathcal{I}_{11,1}
:= 2\e \mathcal{S}^{\varphi^{\e}} \partial_t^{\ell}\mathcal{Z}^{\alpha}v\,\NN^{\e}
+ 2\e(\mathcal{S}^{\varphi^{\e}}v^{\e} - \mathcal{S}^{\varphi^{\e}}v^{\e}\nn^{\e}\cdot\nn^{\e})\,\partial_t^{\ell}\mathcal{Z}^{\alpha}\NN^{\e} 
\\[4pt]\hspace{1.25cm}
+ 2\e[\partial_t^{\ell}\mathcal{Z}^{\alpha}, \mathcal{S}^{\varphi^{\e}}v^{\e} - \mathcal{S}^{\varphi^{\e}}v^{\e}\nn^{\e}\cdot\nn^{\e}, \NN^{\e}] 
- 2\e [\partial_t^{\ell}\mathcal{Z}^{\alpha},\mathcal{S}^{\varphi^{\e}}] v^{\e}\,\NN^{\e}, \\[10pt]

\mathcal{I}_{11,2}
:= - \sigma\nabla_y \cdot \big([\partial_t^{\ell}\mathcal{Z}^{\alpha}, \mathfrak{H}_1\nabla_y] \hat{h}\big) \NN^{\e} \\[4pt]\hspace{1.25cm}
- \sigma\nabla_y \cdot \big([\partial_t^{\ell}\mathcal{Z}^{\alpha}, 
\mathfrak{H}_2 \nabla_y (h^{\e} +h) \nabla_y(h^{\e}+h)\cdot \nabla_y] \hat{h} \big)\NN^{\e}.
\end{array}
\end{equation}

When $|\alpha|\geq 1$ and $1\leq \ell+|\alpha|\leq k$, we develop the $L^2$ estimate of $\partial_t^{\ell}\mathcal{Z}^{\alpha}\hat{v}$,
we have
\begin{equation}\label{Sect7_Tangential_Estimates_1}
\begin{array}{ll}
\frac{1}{2}\frac{\mathrm{d}}{\mathrm{d}t} \int\limits_{\mathbb{R}^3_{-}} |\partial_t^{\ell}\mathcal{Z}^{\alpha}\hat{v}|^2 \,\mathrm{d}\mathcal{V}_t^{\e}
- \int\limits_{\mathbb{R}^3_{-}} \partial_t^{\ell}\mathcal{Z}^{\alpha}\hat{q} \nabla^{\varphi^{\e}}\cdot \partial_t^{\ell}\mathcal{Z}^{\alpha}\hat{v} \,\mathrm{d}\mathcal{V}_t^{\e}
+ 2\e\int\limits_{\mathbb{R}^3_{-}} |\mathcal{S}^{\varphi^{\e}} \partial_t^{\ell}\mathcal{Z}^{\alpha}\hat{v}|^2 \,\mathrm{d}\mathcal{V}_t^{\e}
\\[10pt]

\lem -\int\limits_{\{z=0\}} \big(\partial_t^{\ell}\mathcal{Z}^{\alpha}\hat{q} \NN^{\e}
- 2\e \mathcal{S}^{\varphi^{\e}}\partial_t^{\ell}\mathcal{Z}^{\alpha} \hat{v} \NN^{\e} \big)
\cdot \partial_t^{\ell}\mathcal{Z}^{\alpha}\hat{v} \,\mathrm{d}y
+ \int\limits_{\mathbb{R}^3_{-}} \mathcal{I}_{10} \cdot \partial_t^{\ell}\mathcal{Z}^{\alpha}\hat{v} \,\mathrm{d}\mathcal{V}_t^{\e} \\[11pt]

\lem \int\limits_{\{z=0\}} \big[- g\partial_t^{\ell}\mathcal{Z}^{\alpha}\hat{h}
+ \sigma \nabla_y \cdot \big( \mathfrak{H}_1\nabla_y \partial_t^{\ell}\mathcal{Z}^{\alpha} \hat{h} \big)  \\[9pt]\quad
+ \sigma\nabla_y \cdot \big( \mathfrak{H}_2 \nabla_y \partial_t^{\ell}\mathcal{Z}^{\alpha} \hat{h}\cdot\nabla_y(h^{\e}+h) \nabla_y (h^{\e} +h)\big)
  \big] \NN^{\e}\cdot \partial_t^{\ell}\mathcal{Z}^{\alpha}\hat{v} \,\mathrm{d}y \\[9pt]\quad

- \int\limits_{\{z=0\}} (\mathcal{I}_{11,1} + \mathcal{I}_{11,2}) \cdot \partial_t^{\ell}\mathcal{Z}^{\alpha}\hat{v} \,\mathrm{d}y
+ \int\limits_{\mathbb{R}^3_{-}} \mathcal{I}_{10}\cdot \partial_t^{\ell}\mathcal{Z}^{\alpha}\hat{v} \,\mathrm{d}\mathcal{V}_t^{\e} +O(\e) 
\end{array}
\end{equation}

\begin{equation*}
\begin{array}{ll}
\lem \int\limits_{\{z=0\}} \big[- g\partial_t^{\ell}\mathcal{Z}^{\alpha}\hat{h}
+ \sigma \nabla_y \cdot \big( \mathfrak{H}_1\nabla_y \partial_t^{\ell}\mathcal{Z}^{\alpha} \hat{h} \big)  \\[8pt]\quad
+ \sigma\nabla_y \cdot \big( \mathfrak{H}_2 \nabla_y \partial_t^{\ell}\mathcal{Z}^{\alpha} \hat{h}\cdot\nabla_y(h^{\e}+h) \nabla_y (h^{\e} +h)\big)\big] \\[8pt]\quad
\cdot\Big(\partial_t \partial_t^{\ell}\mathcal{Z}^{\alpha}\hat{h} + v_y^{\e}\cdot \nabla_y \partial_t^{\ell}\mathcal{Z}^{\alpha}\hat{h}
+ \hat{v}_y \cdot \nabla_y \partial_t^{\ell}\mathcal{Z}^{\alpha} h
+ \partial_y \hat{h}\cdot \partial_t^{\ell}\mathcal{Z}^{\alpha}v_y \\[7pt]\quad
- [\partial_t^{\ell}\mathcal{Z}^{\alpha}, \hat{v},\NN^{\e}]
+ [\partial_t^{\ell}\mathcal{Z}^{\alpha}, v_y, \partial_y \hat{h}]
\Big) \,\mathrm{d}y

+ \int\limits_{\mathbb{R}^3_{-}} \mathcal{I}_{10}\cdot \partial_t^{\ell}\mathcal{Z}^{\alpha}\hat{v} \,\mathrm{d}\mathcal{V}_t^{\e} \\[7pt]\quad
- \int\limits_{\{z=0\}} \mathcal{I}_{11,1} \cdot \partial_t^{\ell}\mathcal{Z}^{\alpha}\hat{v} \,\mathrm{d}y
- \int\limits_{\{z=0\}} \mathcal{I}_{11,2} \cdot \partial_t^{\ell}\mathcal{Z}^{\alpha}\hat{v} \,\mathrm{d}y +O(\e).
\hspace{2.3cm}
\end{array}
\end{equation*}

We develop the following boundary estimates in $(\ref{Sect7_Tangential_Estimates_1})$,
\begin{equation}\label{Sect7_Tangential_Estimates_2_1}
\begin{array}{ll}
\int\limits_{\{z=0\}} \big(\partial_t \partial_t^{\ell}\mathcal{Z}^{\alpha}\hat{h} + v_y^{\e}\cdot \nabla_y \partial_t^{\ell}\mathcal{Z}^{\alpha}\hat{h}\big)
\big(- g\partial_t^{\ell}\mathcal{Z}^{\alpha}\hat{h}
+ \sigma \nabla_y \cdot \big( \mathfrak{H}_1\nabla_y \partial_t^{\ell}\mathcal{Z}^{\alpha} \hat{h} \big)  \\[8pt]\quad
+ \sigma\nabla_y \cdot \big( \mathfrak{H}_2 \nabla_y \partial_t^{\ell}\mathcal{Z}^{\alpha} \hat{h}\cdot\nabla_y(h^{\e}+h) \nabla_y (h^{\e} +h)\big)\big)
 \,\mathrm{d}y
\\[9pt]

= -\frac{g}{2} \frac{\mathrm{d}}{\mathrm{d}t}
\int\limits_{\{z=0\}} |\partial_t^{\ell}\mathcal{Z}^{\alpha}\hat{h}|^2 \,\mathrm{d}y
+ \frac{g}{2} \int\limits_{\{z=0\}} |\partial_t^{\ell}\mathcal{Z}^{\alpha}\hat{h}|^2 \nabla_y\cdot v^{\e}_y \,\mathrm{d}y \\[9pt]\quad

- \sigma\int\limits_{\{z=0\}} \big(\partial_t \nabla_y\partial_t^{\ell}\mathcal{Z}^{\alpha}\hat{h} + v_y^{\e}\cdot \nabla_y \nabla_y\partial_t^{\ell}\mathcal{Z}^{\alpha}\hat{h} + \nabla_y v_y^{\e,j}\cdot \partial_j \partial_t^{\ell}\mathcal{Z}^{\alpha}\hat{h}\big)  \\[8pt]\quad
\cdot \big[\mathfrak{H}_1\nabla_y \partial_t^{\ell}\mathcal{Z}^{\alpha} \hat{h}
+ \mathfrak{H}_2 \nabla_y \partial_t^{\ell}\mathcal{Z}^{\alpha} \hat{h}\cdot\nabla_y(h^{\e}+h) \nabla_y (h^{\e} +h)\big]
 \,\mathrm{d}y \\[10pt]

\lem -\frac{g}{2} \frac{\mathrm{d}}{\mathrm{d}t}
\int\limits_{\{z=0\}} |\partial_t^{\ell}\mathcal{Z}^{\alpha}\hat{h}|^2 \,\mathrm{d}y

- \frac{\sigma}{2}\frac{\mathrm{d}}{\mathrm{d}t}
\int\limits_{\{z=0\}} \mathfrak{H}_1|\nabla_y \partial_t^{\ell}\mathcal{Z}^{\alpha} \hat{h}|^2 \\[11pt]\quad

+ \mathfrak{H}_2 |\nabla_y \partial_t^{\ell}\mathcal{Z}^{\alpha} \hat{h}\cdot\nabla_y(h^{\e}+h)|^2
 \,\mathrm{d}y 
+ \|\nabla_y \partial_t^{\ell}\mathcal{Z}^{\alpha} \hat{h}\|_{L^2}^2 + \|\partial_t^{\ell}\mathcal{Z}^{\alpha} \hat{h}\|_{L^2}^2.
\end{array}
\end{equation}

It is easy to check that
\begin{equation}\label{Sect7_Tangential_Estimates_2_2}
\begin{array}{ll}
- \int\limits_{\{z=0\}} \mathcal{I}_{11,1} \cdot \partial_t^{\ell}\mathcal{Z}^{\alpha}\hat{v} \,\mathrm{d}y =O(\e).
\end{array}
\end{equation}
Another boundary estimate is that
\begin{equation}\label{Sect7_Tangential_Estimates_2_3}
\begin{array}{ll}
- \int\limits_{\{z=0\}} \mathcal{I}_{11,2} \cdot \partial_t^{\ell}\mathcal{Z}^{\alpha}\hat{v} \,\mathrm{d}y

= \sigma\int\limits_{\{z=0\}} \nabla_y \cdot \Big( [\partial_t^{\ell}\mathcal{Z}^{\alpha}, \mathfrak{H}_1\nabla_y] \hat{h} \\[12pt]\quad
+ [\partial_t^{\ell}\mathcal{Z}^{\alpha},
\mathfrak{H}_2 \nabla_y (h^{\e} +h) \nabla_y(h^{\e}+h)\cdot \nabla_y] \hat{h} \Big)
\NN^{\e}\cdot \partial_t^{\ell}\mathcal{Z}^{\alpha}\hat{v}\,\mathrm{d}y \\[10pt]

= \sigma\int\limits_{\{z=0\}} \nabla_y \cdot \Big( [\partial_t^{\ell}\mathcal{Z}^{\alpha}, \mathfrak{H}_1\nabla_y] \hat{h} 
+ [\partial_t^{\ell}\mathcal{Z}^{\alpha},
\mathfrak{H}_2 \nabla_y (h^{\e} +h) \nabla_y(h^{\e}+h)\cdot \nabla_y] \hat{h} \Big)\\[12pt]\quad

\cdot\Big(\partial_t \partial_t^{\ell}\mathcal{Z}^{\alpha}\hat{h} + v_y^{\e}\cdot \nabla_y \partial_t^{\ell}\mathcal{Z}^{\alpha}\hat{h}
+ \hat{v}_y \cdot \nabla_y \partial_t^{\ell}\mathcal{Z}^{\alpha} h
+ \partial_y \hat{h}\cdot \partial_t^{\ell}\mathcal{Z}^{\alpha}v_y \\[7pt]\quad
- [\partial_t^{\ell}\mathcal{Z}^{\alpha}, \hat{v},\NN^{\e}]
+ [\partial_t^{\ell}\mathcal{Z}^{\alpha}, v_y, \partial_y \hat{h}]
\Big)\,\mathrm{d}y \\[10pt]

\lem \sigma |\hat{h}|_{X^{k-1,2}}^2 + \sigma |\partial_t^k \hat{h}|_{X^{0,1}}^2 + \big|\hat{v}|_{z=0}\big|_{X^{k-1}}^2 \\[8pt]
\lem \sigma |\hat{h}|_{X^{k-1,2}}^2 + \sigma |\partial_t^k \hat{h}|_{X^{0,1}}^2 
+ \|\hat{v}\big|_{X^{k-1,1}}^2 + \|\partial_z\hat{v}\big|_{X^{k-1}}^2 .
\end{array}
\end{equation}

Plug $(\ref{Sect7_Tangential_Estimates_2_1}),(\ref{Sect7_Tangential_Estimates_2_2}),(\ref{Sect7_Tangential_Estimates_2_3})$ into $(\ref{Sect7_Tangential_Estimates_1})$, we have
\begin{equation}\label{Sect7_Tangential_Estimates_3}
\begin{array}{ll}
\frac{1}{2}\frac{\mathrm{d}}{\mathrm{d}t} \int\limits_{\mathbb{R}^3_{-}} |\partial_t^{\ell}\mathcal{Z}^{\alpha}\hat{v}|^2 \,\mathrm{d}\mathcal{V}_t^{\e}
+ 2\e\int\limits_{\mathbb{R}^3_{-}} |\mathcal{S}^{\varphi^{\e}} \partial_t^{\ell}\mathcal{Z}^{\alpha}\hat{v}|^2 \,\mathrm{d}\mathcal{V}_t^{\e}
+ \frac{g}{2} \frac{\mathrm{d}}{\mathrm{d}t}
\int\limits_{\{z=0\}} |\partial_t^{\ell}\mathcal{Z}^{\alpha}\hat{h}|^2 \,\mathrm{d}y \\[10pt]\quad

+ \frac{\sigma}{2}\frac{\mathrm{d}}{\mathrm{d}t}
\int\limits_{\{z=0\}} \mathfrak{H}_1|\nabla_y \partial_t^{\ell}\mathcal{Z}^{\alpha} \hat{h}|^2
+ \mathfrak{H}_2 |\nabla_y \partial_t^{\ell}\mathcal{Z}^{\alpha} \hat{h}\cdot\nabla_y(h^{\e}+h)|^2\big]
 \,\mathrm{d}y \\[10pt]

\lem \|\hat{v}\|_{X^{k-1,1}}^2 + \|\partial_t^k\hat{v}\|_{L^2}^2 + |\hat{h}|_{X^{k-1,2}}^2 + |\partial_t^k \hat{h}|_{X^{0,1}}^2
+ \|\hat{q}\|_{X^{k-1}}^2 + O(\e).
\end{array}
\end{equation}

Since
\begin{equation}\label{Sect7_Tangential_Estimates_4}
\begin{array}{ll}
\int\limits_{\{z=0\}} \mathfrak{H}_1|\nabla_y \partial_t^{\ell}\mathcal{Z}^{\alpha} \hat{h}|^2
+ \mathfrak{H}_2 |\nabla_y \partial_t^{\ell}\mathcal{Z}^{\alpha} \hat{h}\cdot\nabla_y(h^{\e}+h)|^2\big]
 \,\mathrm{d}y \\[10pt]

\geq \int\limits_{\{z=0\}} |\nabla_y \partial_t^{\ell}\mathcal{Z}^{\alpha} \hat{h}|^2
(\mathfrak{H}_1 - |\mathfrak{H}_2| |\nabla_y(h^{\e}+h)|^2) \,\mathrm{d}y

\geq \int\limits_{\{z=0\}} 4|\mathfrak{H}_2| |\nabla_y \partial_t^{\ell}\mathcal{Z}^{\alpha} \hat{h}|^2 \,\mathrm{d}y.
\end{array}
\end{equation}
where $4|\mathfrak{H}_2| \geq \delta_{\sigma}>0$.

Integrate $(\ref{Sect7_Tangential_Estimates_3})$ in time, apply the integral form of Gronwall's inequality, note that $(\ref{Sect7_Tangential_Estimates_4})$,
we get
\begin{equation}\label{Sect7_Tangential_Estimates_5}
\begin{array}{ll}
\int\limits_{\mathbb{R}^3_{-}} |\partial_t^{\ell}\mathcal{Z}^{\alpha}\hat{v}|^2 \,\mathrm{d}\mathcal{V}_t^{\e}
+ g \int\limits_{\{z=0\}} |\partial_t^{\ell}\mathcal{Z}^{\alpha}\hat{h}|^2 \,\mathrm{d}y

+ \sigma\int\limits_{\{z=0\}} |\nabla_y \partial_t^{\ell}\mathcal{Z}^{\alpha} \hat{h}|^2 \,\mathrm{d}y \\[10pt]

\lem \|\hat{v}_0\|_{X^{k-1,1}}^2 + |\hat{h}_0|_{X^{k-1,1}}^2
+ \int\limits_0^t \|\partial_t^k\hat{v}\|_{L^2}^2 + |\partial_t^k \hat{h}|_{X^{0,1}}^2
+ \|\hat{q}\|_{X^{k-1}}^2 \,\mathrm{d}t + O(\e).
\end{array}
\end{equation}

When $|\alpha|=0, \, 0\leq\ell\leq k-1$, we have no bounds of $\hat{q}$ and $\partial_t^{\ell} \hat{q}$,
so we can not apply the integration by parts to the pressure terms.
Also, the dynamical boundary condition will not be used. Since the main equation of $\partial_t^{\ell}\hat{v}$ and its kinetical boundary condition
satisfy
\begin{equation}\label{Sect7_Tangential_Estimates_Time}
\left\{\begin{array}{ll}
\partial_t^{\varphi^{\e}} \partial_t^{\ell}\hat{v}
+ v^{\e} \cdot\nabla^{\varphi^{\e}} \partial_t^{\ell}\hat{v}
- 2\e\nabla^{\varphi^{\e}}\cdot\mathcal{S}^{\varphi^{\e}} \partial_t^{\ell}\hat{v} \\[8pt]\quad

= - \nabla^{\varphi^{\e}} \partial_t^{\ell}\hat{q}
+ 2\e [\partial_t^{\ell},\nabla^{\varphi^{\e}}\cdot]\mathcal{S}^{\varphi^{\e}} \hat{v}
+ 2\e\nabla^{\varphi^{\e}}\cdot[\partial_t^{\ell}, \mathcal{S}^{\varphi^{\e}}] \hat{v}
+ \e\partial_t^{\ell}\triangle^{\varphi^\e} v \\[8pt]\quad

+ \partial_z^{\varphi} v \partial_t^{\varphi^{\e}} \partial_t^{\ell}\hat{\varphi}
+ \partial_z^{\varphi} v \, v^{\e}\cdot \nabla^{\varphi^{\e}}\partial_t^{\ell}\hat{\varphi}
- \partial_t^{\ell}\hat{v}\cdot\nabla^{\varphi} v
+ \partial_z^{\varphi} q\nabla^{\varphi^{\e}}\partial_t^{\ell}\hat{\varphi}
\\[8pt]\quad

- [\partial_t^{\ell},\partial_t^{\varphi^{\e}}]\hat{v}
+ [\partial_t^{\ell}, \partial_z^{\varphi} v \partial_t^{\varphi^{\e}}]\hat{\varphi}

- [\partial_t^{\ell}, v^{\e} \cdot\nabla^{\varphi^{\e}}] \hat{v}
- [\partial_t^{\ell}, \nabla^{\varphi} v\cdot]\hat{v}\\[8pt]\quad
+ [\partial_t^{\ell}, \partial_z^{\varphi} v \, v^{\e}\cdot \nabla^{\varphi^{\e}}]\hat{\varphi}

- [\partial_t^{\ell},\nabla^{\varphi^{\e}}] \hat{q}
+ [\partial_t^{\ell},\partial_z^{\varphi} q\nabla^{\varphi^{\e}}]\hat{\varphi} :=\mathcal{I}_{12}, \\[11pt]

\partial_t \partial_t^{\ell}\hat{h} + v_y^{\e}\cdot \nabla_y \partial_t^{\ell}\hat{h}
- \NN^{\e}\cdot \partial_t^{\ell}\hat{v}
 = - \hat{v}_y \cdot \nabla_y \partial_t^{\ell} h
- \partial_y \hat{h}\cdot \partial_t^{\ell}v_y \\[6pt]\quad
 + [\partial_t^{\ell}, \hat{v},\NN^{\e}]
 - [\partial_t^{\ell}, v_y, \partial_y \hat{h}], \\[11pt]

(\partial_t^{\ell}\hat{v},\partial_t^{\ell}\hat{h})|_{t=0}
= (\partial_t^{\ell}v_0^\e -\partial_t^{\ell}v_0,
\partial_t^{\ell}h_0^\e -\partial_t^{\ell}h_0),
\end{array}\right.
\end{equation}

then we have $L^2$ estimate of $\partial_t^{\ell}\hat{v}$:
\begin{equation}\label{Sect7_Tangential_Estimates_6}
\begin{array}{ll}
\frac{1}{2}\frac{\mathrm{d}}{\mathrm{d}t} \int\limits_{\mathbb{R}^3_{-}} |\partial_t^{\ell}\hat{v}|^2 \,\mathrm{d}\mathcal{V}_t^{\e}
+ 2\e\int\limits_{\mathbb{R}^3_{-}}|\mathcal{S}^{\varphi^{\e}} \partial_t^{\ell}\hat{v}|^2 \,\mathrm{d}\mathcal{V}_t^{\e}
\\[6pt]
\lem 2\e\int\limits_{\{z=0\}} \mathcal{S}^{\varphi^\e}\partial_t^{\ell} \hat{v} \NN^{\e} \cdot \partial_t^{\ell}\hat{v} \,\mathrm{d}y 
+ \int\limits_{\mathbb{R}^3_{-}}\mathcal{I}_{12}\cdot \partial_t^{\ell} \hat{v} \,\mathrm{d}\mathcal{V}_t^{\e} \\[6pt]

\lem \|\mathcal{I}_{12}\|_{L^2}^2 + \|\partial_t^{\ell} \hat{v}\|_{L^2}^2 + O(\e).
\end{array}
\end{equation}

It is easy to check the last term of $(\ref{Sect7_Tangential_Estimates_6})$ satisfies
\begin{equation}\label{Sect7_Tangential_Estimates_7}
\begin{array}{ll}
\|\mathcal{I}_{12}\|_{L^2} \lem \|\partial_z \hat{v}\|_{X^{k-1}} + \|\hat{v}\|_{X^{k-1,1}} + \|\partial_t^k\hat{v}\|_{L^2} \\[5pt]\hspace{1.67cm}
+ |\partial_t^k \hat{h}|_{L^2}  + |\hat{h}|_{X^{k-1,\frac{1}{2}}}
+ \|\nabla\hat{q}\|_{X^{k-1}} + O(\e).
\end{array}
\end{equation}

Plug $(\ref{Sect7_Tangential_Estimates_7})$ into $(\ref{Sect7_Tangential_Estimates_6})$, integrate in time and apply
the integral form of Gronwall's inequality, we have
\begin{equation}\label{Sect7_Tangential_Estimates_8}
\begin{array}{ll}
\|\partial_t^{\ell}\hat{v}\|_{L^2}^2 + \e\int\limits_0^t \|\nabla \partial_t^{\ell}\hat{v}\|_{L^2}^2 \\[7pt]

\lem \|\partial_t^{\ell}\hat{v}_0\|_{L^2}^2 + \int\limits_0^t\|\partial_z \hat{v}\|_{X^{k-1}}^2 
+ \|\hat{v}\|_{X^{k-1,1}}^2 + \|\partial_t^k\hat{v}\|_{L^2}^2
+ |\partial_t^k \hat{h}|_{L^2}^2  \\[6pt]\quad
+ |\hat{h}|_{X^{k-1,1}}^2 + \|\nabla\hat{q}\|_{X^{k-1}}^2\,\mathrm{d}t + O(\e) \\[6pt]

\lem \|\hat{v}_0\|_{X^{k-1}}^2
+ \|\partial_z \hat{v}\|_{L^4([0,T],X^{k-2})}^2 + |\partial_t^k \hat{h}|_{L^4([0,T],L^2)}^2 \\[5pt]\quad
+ \|\nabla\hat{q}\|_{L^4([0,T],X^{k-1})}^2

+ \int\limits_0^t\|\hat{v}\|_{X^{k-1}}^2 + |\hat{h}|_{X^{k-1,1}}^2 \,\mathrm{d}t + O(\e).
\end{array}
\end{equation}

Sum $\ell$ and $\alpha$.
By $(\ref{Sect7_Tangential_Estimates_5})$, $(\ref{Sect7_Tangential_Estimates_8})$ and Lemma $\ref{Sect7_Height_Estimates_Lemma}$,
we have $(\ref{Sect7_Tangential_Estimates_Lemma_Eq})$.
Thus, Lemma $\ref{Sect7_Tangential_Estimates_Lemma}$ is proved.
\end{proof}

In order to close our estimates of tangential derivatives, we need to bound $\|\partial_t^k \hat{v}\|_{L^4([0,T],L^2)}^2$ and
$\|\partial_t^k \hat{h}\|_{L^4([0,T],X^{0,1})}^2$, which appear in Lemma $\ref{Sect7_Tangential_Estimates_Lemma}$.
Thus, we estimate $\partial_t^k \hat{v}$ and $\partial_t^k \hat{h}$.
\begin{lemma}\label{Sect7_TimeDer_Estimate_Lemma}
$\partial_t^k \hat{v}, \partial_t^k \hat{h}, \partial_t^{k+1}\hat{h}$ satisfies the following estimate:
\begin{equation}\label{Sect7_TimeDer_Estimate}
\begin{array}{ll}
\|\partial_t^k \hat{v}\|_{L^4([0,T],L^2)}^2 + |\partial_t^k \hat{h}|_{L^4([0,T],X^{0,1})}^2 + |\partial_t^{k+1}\nabla \hat{h}|_{L^4([0,T],L^2)}^2
\\[6pt]

\lem \|\partial_t^k \hat{v}_0\|_{L^2}^2 + g|\partial_t^k \hat{h}_0|_{L^2}^2 + \sigma|\partial_t^k\nabla \hat{h}_0|_{L^2}^2
+ \|\partial_z\partial_t^{k-1} \hat{v}_0\|_{L^2}^2
\\[6pt]\quad

+ \|\partial_z \hat{v}\|_{L^4([0,T],X^{k-1})}^2 + O(\e).
\end{array}
\end{equation}
\end{lemma}

\begin{proof}
$(\partial_t^k \hat{v},\partial_t^k \hat{h},\partial_t^k \hat{q})$ satisfy the following equations:
\begin{equation}\label{Sect7_TimeDer_Eq}
\left\{\begin{array}{ll}
\partial_t^{\varphi^{\e}} \partial_t^k\hat{v}
+ v^{\e} \cdot\nabla^{\varphi^{\e}} \partial_t^k\hat{v}
+ \nabla^{\varphi^{\e}} \partial_t^k\hat{q}
- 2\e \nabla^{\varphi^{\e}}\cdot\mathcal{S}^{\varphi^{\e}} \partial_t^k\hat{v} 

= \mathcal{I}_{10}|_{\ell=k,|\alpha|=0}, \\[11pt]

\nabla^{\varphi^{\e}}\cdot \partial_t^k\hat{v}
= \partial_z^{\varphi}v \cdot\nabla^{\varphi^{\e}}\partial_t^k\hat{\eta}
-[\partial_t^k,\nabla^{\varphi^{\e}}\cdot] \hat{v}
+ [\partial_t^k,\partial_z^{\varphi}v \cdot\nabla^{\varphi^{\e}}]\hat{\eta}, \\[11pt]

\partial_t \partial_t^k\hat{h} + v_y^{\e}\cdot \nabla_y \partial_t^k\hat{h}
- \NN^{\e}\cdot \partial_t^k\hat{v}
 = - \hat{v}_y \cdot \nabla_y \partial_t^k h
- \partial_y \hat{h}\cdot \partial_t^k v_y \\[5pt]\quad
 + [\partial_t^k, \hat{v},\NN^{\e}]
 - [\partial_t^k, v_y, \partial_y \hat{h}], \\[11pt]

\partial_t^k \hat{q}\NN^{\e}
-2\e \mathcal{S}^{\varphi^{\e}} \partial_t^k \hat{v}\,\NN^{\e}

- g\partial_t^k\hat{h}\NN^{\e}
+ \sigma \nabla_y \cdot \big( \mathfrak{H}_1\nabla_y \partial_t^k \hat{h} \big) \NN^{\e} \\[5pt]\quad
+ \sigma\nabla_y \cdot \big( \mathfrak{H}_2 \nabla_y \partial_t^k \hat{h}\cdot\nabla_y(h^{\e}+h) \nabla_y (h^{\e} +h)\big)\NN^{\e} \\[5pt]\quad

= \mathcal{I}_{11,1}|_{\ell=k,|\alpha|=0} + \mathcal{I}_{11,2}|_{\ell=k,|\alpha|=0},
\\[11pt]

(\partial_t^k\hat{v},\partial_t^k\hat{h})|_{t=0}
= (\partial_t^k v_0^\e -\partial_t^k v_0,
\partial_t^k h_0^\e -\partial_t^k h_0), 
\end{array}\right.
\end{equation}

Then multiply $(\ref{Sect7_TimeDer_Eq})$ with $\partial_t^k \hat{v}$,
integrate in $\mathbb{R}^3_{-}$, then we get
\begin{equation}\label{Sect7_TimeDer_Estimate_1}
\begin{array}{ll}
\frac{1}{2}\frac{\mathrm{d}}{\mathrm{d}t}\int\limits_{\mathbb{R}^3_{-}} |\partial_t^k \hat{v}|^2 \,\mathrm{d}\mathcal{V}_t^{\e}
- \int\limits_{\mathbb{R}^3_{-}} \partial_t^k \hat{q} \, \nabla^{\varphi^{\e}}\cdot \partial_t^k \hat{v} \,\mathrm{d}\mathcal{V}_t^{\e}
+ 2\e\int\limits_{\mathbb{R}^3_{-}} |\mathcal{S}^{\varphi^{\e}} \partial_t^k\hat{v}|^2 \,\mathrm{d}\mathcal{V}_t^{\e}
 \\[14pt]

\leq \int\limits_{\{z=0\}} (2\e \mathcal{S}^{\varphi}\partial_t^k \hat{v} \NN^{\e}
- \partial_t^k \hat{q}\NN^{\e})\cdot \partial_t^k \hat{v} \mathrm{d}y
+ \|\partial_z \hat{v}\|_{X^{k-1}}^2 + \|\nabla \hat{q}\|_{X^{k-1}}^2 \\[11pt]\quad
+ |\hat{h}|_{X^{k-1,2}}^2 + |\partial_t^k \hat{h}|_{X^{0,1}}^2  + |\partial_t^{k+1} \hat{h}|_{L^2}^2
+ O(\e) \\[7pt]

\leq -\int\limits_{\{z=0\}} \Big[
g\partial_t^{\ell}\hat{h}
- \sigma \nabla_y \cdot \big( \mathfrak{H}_1\nabla_y \partial_t^{\ell} \hat{h} \big) \\[6pt]\quad
- \sigma\nabla_y \cdot \big( \mathfrak{H}_2 \nabla_y \partial_t^{\ell} \hat{h} 
\cdot\nabla_y(h^{\e}+h) \nabla_y (h^{\e} +h)\big)
\Big]\NN^{\e}
\cdot \partial_t^k \hat{v} \mathrm{d}y \\[6pt]\quad
+ \|\partial_z \hat{v}\|_{X^{k-1}}^2 + \|\nabla \hat{q}\|_{X^{k-1}}^2
+ |\hat{h}|_{X^{k-1,2}}^2 + |\partial_t^k \hat{h}|_{X^{0,1}}^2
+ |\partial_t^{k+1} \hat{h}|_{L^2}^2 + O(\e) \\[6pt]

\leq \sigma\int\limits_{\{z=0\}} \nabla_y\cdot
\big(\mathfrak{H}_1\nabla_y \partial_t^{\ell} \hat{h}
+ \mathfrak{H}_2 \nabla_y \partial_t^{\ell} \hat{h}\cdot\nabla_y(h^{\e}+h) \nabla_y (h^{\e} +h) \big)

\cdot(\partial_t \partial_t^k  \hat{h} \\[13pt]\quad
+ v_y \cdot\nabla_y \partial_t^k  \hat{h}) \mathrm{d}y

-\int\limits_{\{z=0\}}
g \partial_t^k \hat{h}
\cdot(\partial_t \partial_t^k  \hat{h}
+ v_y \cdot\nabla_y \partial_t^k  \hat{h}) \mathrm{d}y \\[10pt]\quad
+ \|\partial_z \hat{v}\|_{X^{k-1}}^2 + \|\nabla \hat{q}\|_{X^{k-1}}^2
+ |\hat{h}|_{X^{k-1,2}}^2 + |\partial_t^k \hat{h}|_{X^{0,1}}^2
+ |\partial_t^{k+1} \hat{h}|_{L^2}^2 + O(\e) \\[7pt]

\leq - \sigma \int\limits_{\{z=0\}}
\big(\mathfrak{H}_1\nabla_y \partial_t^{\ell} \hat{h}
+ \mathfrak{H}_2 \nabla_y \partial_t^{\ell} \hat{h}\cdot\nabla_y(h^{\e}+h) \nabla_y (h^{\e} +h) \big)
\cdot(\partial_t \nabla_y\partial_t^k  \hat{h} \\[13pt]\quad
+ v_y \cdot\nabla_y \nabla_y\partial_t^k  \hat{h}) \mathrm{d}y

- \frac{g}{2}\frac{\mathrm{d}}{\mathrm{d}t}
\int\limits_{\{z=0\}} |\partial_t^k \hat{h}|^2 \mathrm{d}y
+ \|\partial_z \hat{v}\|_{X^{k-1}}^2 + \|\nabla \hat{q}\|_{X^{k-1}}^2 \\[7pt]\quad
+ |\hat{h}|_{X^{k-1,2}}^2 + |\partial_t^k \hat{h}|_{X^{0,1}}^2
+ |\partial_t^{k+1} \hat{h}|_{L^2}^2 + O(\e) \\[8pt]

\leq - \frac{g}{2}\frac{\mathrm{d}}{\mathrm{d}t}
\int\limits_{\{z=0\}} |\partial_t^k \hat{h}|^2 \mathrm{d}y

- \frac{\sigma}{2}\frac{\mathrm{d}}{\mathrm{d}t} \int\limits_{\{z=0\}}
\big(\mathfrak{H}_1|\nabla_y\partial_t^k \hat{h}|^2
+ \mathfrak{H}_2 |\nabla_y \partial_t^{\ell} \hat{h}\cdot\nabla_y(h^{\e}+h)|^2 \big) \mathrm{d}y \\[12pt]\quad
+ \|\partial_z \hat{v}\|_{X^{k-1}}^2 + \|\nabla \hat{q}\|_{X^{k-1}}^2
+ |\hat{h}|_{X^{k-1,2}}^2 + |\partial_t^k \hat{h}|_{X^{0,1}}^2
+ |\partial_t^{k+1} \hat{h}|_{L^2}^2 + O(\e).
\end{array}
\end{equation}

The same as $(\ref{Sect7_Tangential_Estimates_4})$, we have
\begin{equation}\label{Sect7_TimeDer_Estimate_1_Inequality}
\begin{array}{ll}
\int\limits_{\{z=0\}} \mathfrak{H}_1|\nabla_y \partial_t^k \hat{h}|^2
+ \mathfrak{H}_2 |\nabla_y \partial_t^k \hat{h}\cdot\nabla_y(h^{\e}+h)|^2\big]
 \,\mathrm{d}y \\[10pt]

\geq \int\limits_{\{z=0\}} |\nabla_y \partial_t^k \hat{h}|^2
(\mathfrak{H}_1 - |\mathfrak{H}_2| |\nabla_y(h^{\e}+h)|^2) \,\mathrm{d}y

\geq \int\limits_{\{z=0\}} 4|\mathfrak{H}_2| |\nabla_y \partial_t^k \hat{h}|^2 \,\mathrm{d}y.
\end{array}
\end{equation}
where $4|\mathfrak{H}_2| \geq \delta_{\sigma}>0$, since $|\nabla_y h^{\e}|_{\infty}$ and $|\nabla_y h|_{\infty}$ have their upper bounds.

The same as \cite{Wang_Xin_2015}, we will integrate in time twice, we get the $L^4([0,T],L^2)$ type estimate.
After the first integration in time, we have
\begin{equation}\label{Sect7_TimeDer_Estimate_3}
\begin{array}{ll}
\|\partial_t^k \hat{v}\|_{L^2}^2 + g|\partial_t^k \hat{h}|_{L^2}^2 + \sigma|\partial_t^k\nabla_y \hat{h}|_{L^2}^2
+\e\int\limits_0^t \|\nabla\partial_t^k \hat{v}\|_{L^2}^2 \,\mathrm{d}t
\\[6pt]

\lem \|\partial_t^k \hat{v}_0\|_{L^2}^2 + g|\partial_t^k \hat{h}_0|_{L^2}^2 + \sigma|\partial_t^k\nabla_y \hat{h}_0|_{L^2}^2
+ \int\limits_0^t\int\limits_{\mathbb{R}^3_{-}} \partial_t^k \hat{q} \, \nabla^{\varphi^{\e}}\cdot \partial_t^k \hat{v} \,\mathrm{d}\mathcal{V}_t^{\e}\mathrm{d}t \\[6pt]\quad

+ \int\limits_0^t\|\partial_z \hat{v}\|_{X^{k-1}}^2 + \|\nabla \hat{q}\|_{X^{k-1}}^2
+ |\hat{h}|_{X^{k-1,2}}^2 + |\partial_t^k \hat{h}|_{X^{0,1}}^2  + |\partial_t^{k+1} \hat{h}|_{L^2}^2\mathrm{d}t + O(\e) 
\end{array}
\end{equation}

\begin{equation*}
\begin{array}{ll}
\lem \|\partial_t^k \hat{v}_0\|_{L^2}^2 + g|\partial_t^k \hat{h}_0|_{L^2}^2 + \sigma|\partial_t^k\nabla \hat{h}_0|_{L^2}^2
+ \int\limits_0^t\int\limits_{\mathbb{R}^3_{-}} \partial_t^k \hat{q} \, \nabla^{\varphi^{\e}}\cdot \partial_t^k \hat{v} \,\mathrm{d}\mathcal{V}_t^{\e}\mathrm{d}t \\[6pt]\quad

+ \|\partial_z \hat{v}\|_{L^4([0,T],X^{k-1})}^2 + |\partial_t^k \hat{h}|_{L^4([0,T],X^{0,1})}^2
+ |\partial_t^k \hat{v}|_{L^4([0,T],L^2)}^2 + O(\e). \hspace{1cm}
\end{array}
\end{equation*}

Similar to the procedures
$(\ref{Sect6_TimeDer_Estimate_4})$, $(\ref{Sect6_TimeDer_Estimate_5})$, $(\ref{Sect6_TimeDer_Estimate_6})$, $(\ref{Sect6_TimeDer_Estimate_7})$,
we deal with the pressure term $\int\limits_0^t\int\limits_{\mathbb{R}^3_{-}} \partial_t^k \hat{q} \, \nabla^{\varphi^{\e}}\cdot \partial_t^k \hat{v} \,\mathrm{d}\mathcal{V}_t^{\e}\mathrm{d}t$ by using Hardy's inequality.  Denote
\begin{equation}\label{Sect7_TimeDer_Estimate_6}
\begin{array}{ll}
\mathcal{I}_{13} := \|\partial_z \partial_t^{k-1} \hat{q}\|_{L^2}^2 + \big|\partial_t^{k-1} \hat{q}|_{z=0}\big|_{L^2}^2
+ |\partial_t^k \hat{h}|_{X^{0,1}}^2 \\[4pt]\hspace{1.15cm}
+ \|\partial_t^{k-1}\hat{v}\|_{L^2}^2 + \big|\partial_t^{k-1}\hat{v}|_{z=0}\big|_{L^2}^2 \\[8pt]\hspace{0.7cm}

\lem \|\partial_t^k \hat{v}\|_{L^2}^2 + \|\partial_z\partial_t^{k-1} \hat{v}\|_{L^2}^2
+ |\partial_t^k \hat{h}|_{X^{0,1}}^2 + O(\e),
\end{array}
\end{equation}
then we have
\begin{equation}\label{Sect7_TimeDer_Estimate_7}
\begin{array}{ll}
\int\limits_0^t\int\limits_{\mathbb{R}^3_{-}} \partial_t^k \hat{q} \, \nabla^{\varphi}\cdot \partial_t^k \hat{v}
\,\mathrm{d}\mathcal{V}_t^{\e}\mathrm{d}t

\lem \mathcal{I}_{13}|_{t=0} + \mathcal{I}_{13} + \int\limits_0^T \mathcal{I}_{13} \,\mathrm{d}s \\[10pt]

\lem \|\partial_t^k \hat{v}_0\|_{L^2}^2 + \|\partial_z\partial_t^{k-1} \hat{v}_0\|_{L^2}^2
+ |\partial_t^k \hat{h}_0|_{X^{0,1}}^2

+ \|\partial_t^k \hat{v}\|_{L^2}^2 + \|\partial_z\partial_t^{k-1} \hat{v}\|_{L^2}^2 \\[6pt]\quad
+ |\partial_t^k \hat{h}|_{X^{0,1}}^2

+ \|\partial_t^k \hat{v}\|_{L^4([0,T],L^2)}^2
+ \|\partial_z\partial_t^{k-1} \hat{v}\|_{L^4([0,T],L^2)}^2
+ |\partial_t^k \hat{h}|_{L^4([0,T],L^2)}^2.
\end{array}
\end{equation}

By $(\ref{Sect7_TimeDer_Estimate_3})$ and $(\ref{Sect7_TimeDer_Estimate_7})$, we get
\begin{equation}\label{Sect7_TimeDer_Estimate_8}
\begin{array}{ll}
\|\partial_t^k \hat{v}\|_{L^2}^2 + g|\partial_t^k \hat{h}|_{L^2}^2 + \sigma|\partial_t^k\nabla_y \hat{h}|_{L^2}^2
\\[6pt]

\lem \|\partial_t^k \hat{v}_0\|_{L^2}^2 + g|\partial_t^k \hat{h}_0|_{L^2}^2 + \sigma|\partial_t^k\nabla_y \hat{h}_0|_{L^2}^2
 + \|\partial_z\partial_t^{k-1} \hat{v}_0\|_{L^2}^2

+ \|\partial_t^k \hat{v}\|_{L^2}^2 \\[6pt]\quad
+ \|\partial_z\partial_t^{k-1} \hat{v}\|_{L^2}^2
+ |\partial_t^k \hat{h}|_{X^{0,1}}^2

+ \|\partial_t^k \hat{v}\|_{L^4([0,T],L^2)}^2
+ \|\partial_z \hat{v}\|_{L^4([0,T],X^{k-1})}^2 \\[6pt]\quad
+ |\partial_t^k \hat{h}|_{L^4([0,T],X^{0,1})}^2 + O(\e).
\end{array}
\end{equation}

Square $(\ref{Sect7_TimeDer_Estimate_8})$ and integrate in time again (see \cite{Wang_Xin_2015}), apply the integral form of Gronwall's inequality, we have
\begin{equation}\label{Sect7_TimeDer_Estimate_9}
\begin{array}{ll}
\|\partial_t^k \hat{v}\|_{L^4([0,T],L^2)}^2 + g|\partial_t^k \hat{h}|_{L^4([0,T],L^2)}^2 + \sigma|\partial_t^k\nabla_y \hat{h}|_{L^4([0,T],L^2)}^2
\\[6pt]

\lem \|\partial_t^k \hat{v}_0\|_{L^2}^2 + g|\partial_t^k \hat{h}_0|_{L^2}^2 + \sigma|\partial_t^k\nabla_y \hat{h}_0|_{L^2}^2
+ \|\partial_z\partial_t^{k-1} \hat{v}_0\|_{L^2}^2
\\[6pt]\quad

+ \|\partial_z \hat{v}\|_{L^4([0,T],X^{k-1})}^2
+ O(\e).
\end{array}
\end{equation}
Thus, Lemma $\ref{Sect7_TimeDer_Estimate_Lemma}$ is proved.
\end{proof}

Based on Lemmas $\ref{Sect7_Height_Estimates_Lemma}$, $\ref{Sect7_Tangential_Estimates_Lemma}$ and $\ref{Sect7_TimeDer_Estimate_Lemma}$,
the estimates of tangential derivatives can be closed. The estimates of normal derivatives are the same as the $\sigma=0$ case.
Finally, it is standard to estimates $(\ref{Sect1_Thm7_ConvergenceRates_1})$ and $(\ref{Sect1_Thm7_ConvergenceRates_2})$ in
Theorem $\ref{Sect1_Thm_StrongLayer_ST}$.

\appendix

%%% find 8
\section{Derivation of the Equations and Boundary Conditions}

In this appendix, we derive the equations and their boundary conditions for the $\sigma=0$ case.

Since $\partial_i^{\varphi^{\e}}\varphi^{\e} = \partial_i \varphi^{\e}
- \frac{\partial_i\varphi^{\e}}{\partial_z \varphi^{\e}}\partial_z \varphi^{\e} =0$ and
$\partial_z^{\varphi^{\e}}\varphi^{\e} = \frac{1}{\partial_z\varphi^{\e}}\partial_z \varphi^{\e} =1$,
\begin{equation}\label{SectA_Difference_Transform_1}
\begin{array}{ll}
\partial_i^{\varphi^{\e}}v^{\e} - \partial_i^{\varphi} v
= \partial_i^{\varphi^{\e}}\hat{v}
- (\frac{\partial_i \varphi^{\e}}{\partial_z \varphi^{\e}} - \frac{\partial_i \varphi}{\partial_z \varphi})\partial_z v

= \partial_i^{\varphi^{\e}}\hat{v}
+ (\partial_i \varphi - \frac{\partial_i \varphi^{\e}}{\partial_z \varphi^{\e}}\partial_z \varphi)\frac{1}{\partial_z \varphi}\partial_z v
\\[8pt]\hspace{1.95cm}

= \partial_i^{\varphi^{\e}}\hat{v}
+ \partial_z^{\varphi} v \partial_i^{\varphi^{\e}}\varphi
= \partial_i^{\varphi^{\e}}\hat{v}
+ \partial_z^{\varphi} v \partial_i^{\varphi^{\e}}\varphi -\partial_z^{\varphi} v \partial_i^{\varphi^{\e}}\varphi^{\e} \\[8pt]\hspace{1.95cm}

= \partial_i^{\varphi^{\e}}\hat{v}
-\partial_z^{\varphi} v \partial_i^{\varphi^{\e}}\hat{\varphi}

= \partial_i^{\varphi^{\e}}\hat{v}
- \partial_i^{\varphi^{\e}}\hat{\eta} \, \partial_z^{\varphi} v,
\hspace{0.5cm} i=t,1,2,\\[11pt]

\partial_z^{\varphi^{\e}}v^{\e} - \partial_z^{\varphi} v
= \partial_z^{\varphi^{\e}} \hat{v}
+ (\frac{1}{\partial_z \varphi^{\e}} - \frac{1}{\partial_z \varphi}) \partial_z v \\[8pt]\hspace{1.95cm}

= \partial_z^{\varphi^{\e}} \hat{v}
+ (\frac{1}{\partial_z \varphi^{\e}}\partial_z \varphi - 1) \frac{1}{\partial_z \varphi}\partial_z v
= \partial_z^{\varphi^{\e}} \hat{v}
+ (\partial_z^{\varphi^{\e}} \varphi - 1) \frac{1}{\partial_z \varphi}\partial_z v \\[8pt]\hspace{1.95cm}

= \partial_z^{\varphi^{\e}} \hat{v}
+ (\partial_z^{\varphi^{\e}} \varphi - \partial_z^{\varphi^{\e}} \varphi^{\e}) \partial_z^{\varphi} v \\[8pt]\hspace{1.95cm}

= \partial_z^{\varphi^{\e}} \hat{v}
- \partial_z^{\varphi} v\partial_z^{\varphi^{\e}}\hat{\varphi}

= \partial_z^{\varphi^{\e}} \hat{v}
- \partial_z^{\varphi^{\e}}\hat{\eta} \, \partial_z^{\varphi} v.
\end{array}
\end{equation}

Similarly, we have
\begin{equation}\label{SectA_Difference_Transform_2}
\begin{array}{ll}
\partial_i^{\varphi^{\e}}q^{\e} - \partial_i^{\varphi} q
= \partial_i^{\varphi^{\e}}\hat{q}
- \partial_i^{\varphi^{\e}}\hat{\eta} \, \partial_z^{\varphi} q,
\hspace{0.5cm} i=t,1,2, 3 \\[8pt]

v^{\e} \cdot\nabla^{\varphi^{\e}} v^{\e} - v\cdot\nabla^{\varphi} v
= v^{\e} \cdot\nabla^{\varphi^{\e}} \hat{v} - v^{\e}\cdot \nabla^{\varphi^{\e}}\hat{\eta}\, \partial_z^{\varphi} v + \hat{v}\cdot\nabla^{\varphi} v, \\[8pt]

\omega^{\e} -\omega = \nabla^{\varphi^{\e}}\times v^{\e} - \nabla^{\varphi}\times v
= \nabla^{\varphi^{\e}} \times \hat{v} - \nabla^{\varphi^{\e}}\hat{\eta} \times \partial_z^{\varphi} v.
\end{array}
\end{equation}

\begin{lemma}\label{SectA_DifferenceEq1_Lemma}
$(\hat{v} = v^{\e} -v,\ \hat{h} =h^{\e} -h,\ \hat{q} = q^{\e} -q)$ satisfy the equations
$(\ref{Sect1_T_Derivatives_Difference_Eq})$.
\end{lemma}

\begin{proof}

Plug $(\ref{SectA_Difference_Transform_1}),(\ref{SectA_Difference_Transform_2})$ into
\begin{equation}\label{SectA_Difference_Eq1_4}
\begin{array}{ll}
\partial_t^{\varphi^\e} v^\e -\partial_t^{\varphi} v + v^\e\cdot\nabla^{\varphi^\e} v^\e - v\cdot\nabla^{\varphi} v
+ \nabla^{\varphi^\e} q^\e - \nabla^{\varphi} q \\[5pt]
= 2\e \nabla^{\varphi^{\e}} \cdot\mathcal{S}^{\varphi^{\e}} (v^{\e}-v)
+ \e \triangle^{\varphi^{\e}} v,
\end{array}
\end{equation}
then we get the equation $(\ref{Sect1_T_Derivatives_Difference_Eq})_1$.

It follows from the divergence free condition that
\begin{equation}\label{SectA_Difference_Eq1_5}
\begin{array}{ll}
0 = \nabla^{\varphi^\e}\cdot v^\e - \nabla^{\varphi}\cdot v
= \sum\limits_{i=1}^3(\partial_i^{\varphi^{\e}}\hat{v}^i
-\partial_z^{\varphi} v^i \partial_i^{\varphi^{\e}}\hat{\eta})

= \nabla^{\varphi^{\e}}\cdot \hat{v} - \partial_z^{\varphi}v \cdot\nabla^{\varphi^{\e}}\hat{\eta}.
\end{array}
\end{equation}

It follows from the kinetical boundary condition that
\begin{equation}\label{SectA_Difference_Eq1_6}
\begin{array}{ll}
\partial_t\hat{h} = \partial_t h^{\e} -\partial_t h = v^\e(t,y,0)\cdot \NN^{\e} -v(t,y,0)\cdot \NN, \\[6pt]

v^\e\cdot \NN^{\e} -v\cdot \NN = \hat{v}\cdot \NN^{\e} +v\cdot \hat{\NN}
= v\cdot (-\nabla_y \hat{h},0) + \hat{v}\cdot \NN^{\e}, \\[6pt]

\partial_t \hat{h} + v_y\cdot \nabla_y \hat{h} = \hat{v}\cdot\NN^{\e}.
\end{array}
\end{equation}

The dynamical boundary condition for the Euler equation with $\sigma=0$
is a scalar equation, that is $q= gh$. For any vector such as $\NN^{\e}$,
$q \NN^{\e}= gh \NN^{\e}$ makes sense.
It follows from the dynamical boundary condition that
\begin{equation}\label{SectA_Difference_Eq1_7}
\begin{array}{ll}
q^\e \NN^{\e} -q\NN^{\e} -2\e \mathcal{S}^{\varphi^\e}(v^\e -v)\,\NN^\e =gh^{\e} \NN^{\e} -gh \NN^{\e}
+ 2\e \mathcal{S}^{\varphi^\e}v\,\NN^\e, \\[6pt]

\hat{q}\NN^{\e} -2\e \mathcal{S}^{\varphi^\e}\hat{v}\,\NN^\e =g \hat{h} \NN^{\e} + 2\e \mathcal{S}^{\varphi^\e}v\,\NN^\e, \\[6pt]

(\hat{q} -g \hat{h})\NN^{\e} -2\e \mathcal{S}^{\varphi^\e}\hat{v}\,\NN^\e = 2\e \mathcal{S}^{\varphi^\e}v\,\NN^\e.
\end{array}
\end{equation}
Thus, Lemma $\ref{SectA_DifferenceEq1_Lemma}$ is proved.
\end{proof}

\begin{lemma}\label{SectA_DifferenceEq2_Lemma}
Assume $0\leq \ell +|\alpha|\leq k, \ 0\leq \ell\leq k-1, \ |\alpha|\geq 1$,
let $\hat{V}^{\ell,\alpha} = \partial_t^{\ell}\mathcal{Z}^{\alpha}\hat{v} - \partial_z^{\varphi}v \partial_t^{\ell}\mathcal{Z}^{\alpha}\hat{\varphi}$,
$\hat{Q}^{\ell,\alpha} = \partial_t^{\ell}\mathcal{Z}^{\alpha}\hat{q} - \partial_z^{\varphi}q \partial_t^{\ell}\mathcal{Z}^{\alpha}\hat{\varphi}$,
then $\hat{V}^{\ell,\alpha}, \hat{Q}^{\ell,\alpha}$ satisfy the equations $(\ref{Sect5_TangentialEstimates_Diff_Eq})$.
\end{lemma}

\begin{proof}
Apply $\partial_t^{\ell}\mathcal{Z}^{\alpha}$ to the equations $(\ref{Sect1_T_Derivatives_Difference_Eq})$, we prove
$(\ref{Sect5_TangentialEstimates_Diff_Eq})$. The derivation of the main equation $(\ref{Sect5_TangentialEstimates_Diff_Eq})_1$ is as follows:
\begin{equation}\label{SectA_Difference_Eq2_1}
\begin{array}{ll}
\partial_t^{\varphi^{\e}} \partial_t^{\ell}\mathcal{Z}^{\alpha}\hat{v} + [\partial_t^{\ell}\mathcal{Z}^{\alpha},\partial_t^{\varphi^{\e}}]\hat{v}
-\partial_z^{\varphi} v \partial_t^{\varphi^{\e}} \partial_t^{\ell}\mathcal{Z}^{\alpha}\hat{\varphi}
- [\partial_t^{\ell}\mathcal{Z}^{\alpha}, \partial_z^{\varphi} v \partial_t^{\varphi^{\e}}]\hat{\varphi} \\[7pt]\quad

+ v^{\e} \cdot\nabla^{\varphi^{\e}} \partial_t^{\ell}\mathcal{Z}^{\alpha}\hat{v} + [\partial_t^{\ell}\mathcal{Z}^{\alpha}, v^{\e} \cdot\nabla^{\varphi^{\e}}] \hat{v}
- \partial_z^{\varphi} v \, v^{\e}\cdot \nabla^{\varphi^{\e}}\partial_t^{\ell}\mathcal{Z}^{\alpha}\hat{\varphi} \\[7pt]\quad
- [\partial_t^{\ell}\mathcal{Z}^{\alpha}, \partial_z^{\varphi} v \, v^{\e}\cdot \nabla^{\varphi^{\e}}]\hat{\varphi}

+ \partial_t^{\ell}\mathcal{Z}^{\alpha}\hat{v}\cdot\nabla^{\varphi} v + [\partial_t^{\ell}\mathcal{Z}^{\alpha}, \nabla^{\varphi} v\cdot]\hat{v}

+ \nabla^{\varphi^{\e}} \partial_t^{\ell}\mathcal{Z}^{\alpha}\hat{q} \\[7pt]\quad
+ [\partial_t^{\ell}\mathcal{Z}^{\alpha},\nabla^{\varphi^{\e}}] \hat{q}
- \partial_z^{\varphi} q\nabla^{\varphi^{\e}}\partial_t^{\ell}\mathcal{Z}^{\alpha}\hat{\varphi}
- [\partial_t^{\ell}\mathcal{Z}^{\alpha},\partial_z^{\varphi} q\nabla^{\varphi^{\e}}]\hat{\varphi} \\[7pt]\quad
= \e\partial_t^{\ell}\mathcal{Z}^{\alpha}\triangle^{\varphi^\e} \hat{v} + \e\partial_t^{\ell}\mathcal{Z}^{\alpha}\triangle^{\varphi^\e} v, \\[13pt]

\partial_t^{\varphi^{\e}} \partial_t^{\ell}\mathcal{Z}^{\alpha}\hat{v}
-\partial_z^{\varphi} v \partial_t^{\varphi^{\e}} \partial_t^{\ell}\mathcal{Z}^{\alpha}\hat{\varphi}

+ v^{\e} \cdot\nabla^{\varphi^{\e}} \partial_t^{\ell}\mathcal{Z}^{\alpha}\hat{v}
- \partial_z^{\varphi} v \, v^{\e}\cdot \nabla^{\varphi^{\e}}\partial_t^{\ell}\mathcal{Z}^{\alpha}\hat{\varphi} \\[8pt]\quad

+ \partial_t^{\ell}\mathcal{Z}^{\alpha}\hat{v}\cdot\nabla^{\varphi} v

+ \nabla^{\varphi^{\e}} \partial_t^{\ell}\mathcal{Z}^{\alpha}\hat{q}
- \partial_z^{\varphi} q\nabla^{\varphi^{\e}}\partial_t^{\ell}\mathcal{Z}^{\alpha}\hat{\varphi}

- 2\e\partial_t^{\ell}\mathcal{Z}^{\alpha}\nabla^{\varphi^{\e}}\cdot\mathcal{S}^{\varphi^{\e}} \hat{v} \\[8pt]\quad

= \e\partial_t^{\ell}\mathcal{Z}^{\alpha}\triangle^{\varphi^\e} v
- [\partial_t^{\ell}\mathcal{Z}^{\alpha},\partial_t^{\varphi^{\e}}]\hat{v}
+ [\partial_t^{\ell}\mathcal{Z}^{\alpha}, \partial_z^{\varphi} v \partial_t^{\varphi^{\e}}]\hat{\varphi}

- [\partial_t^{\ell}\mathcal{Z}^{\alpha}, v^{\e} \cdot\nabla^{\varphi^{\e}}] \hat{v} \\[8pt]\quad
+ [\partial_t^{\ell}\mathcal{Z}^{\alpha}, \partial_z^{\varphi} v \, v^{\e}\cdot \nabla^{\varphi^{\e}}]\hat{\varphi}
- [\partial_t^{\ell}\mathcal{Z}^{\alpha}, \nabla^{\varphi} v\cdot]\hat{v}

- [\partial_t^{\ell}\mathcal{Z}^{\alpha},\nabla^{\varphi^{\e}}] \hat{q}
+ [\partial_t^{\ell}\mathcal{Z}^{\alpha},\partial_z^{\varphi} q\nabla^{\varphi^{\e}}]\hat{\varphi},
\\[13pt]

\partial_t^{\varphi^{\e}} (\partial_t^{\ell}\mathcal{Z}^{\alpha}\hat{v}
-\partial_z^{\varphi} v \partial_t^{\ell}\mathcal{Z}^{\alpha}\hat{\varphi})

+ v^{\e} \cdot\nabla^{\varphi^{\e}}(\partial_t^{\ell}\mathcal{Z}^{\alpha}\hat{v}
- \partial_z^{\varphi} v \partial_t^{\ell}\mathcal{Z}^{\alpha}\hat{\varphi}) \\[8pt]\quad

+ \nabla^{\varphi^{\e}} (\partial_t^{\ell}\mathcal{Z}^{\alpha}\hat{q}
- \partial_z^{\varphi} q\partial_t^{\ell}\mathcal{Z}^{\alpha}\hat{\varphi}) 

- 2\e\nabla^{\varphi^{\e}}\cdot\mathcal{S}^{\varphi^{\e}} \partial_t^{\ell}\mathcal{Z}^{\alpha}\hat{v}

= \mathcal{I}_4, \\[13pt]

\partial_t^{\varphi^{\e}} \hat{V}^{\ell,\alpha}
+ v^{\e} \cdot\nabla^{\varphi^{\e}} \hat{V}^{\ell,\alpha}
+ \nabla^{\varphi^{\e}} \hat{Q}^{\ell,\alpha} 
- 2\e\nabla^{\varphi^{\e}}\cdot\mathcal{S}^{\varphi^{\e}} \partial_t^{\ell}\mathcal{Z}^{\alpha}\hat{v}
= \mathcal{I}_4.
\end{array}
\end{equation}

The derivation of the divergence free condition $(\ref{Sect5_TangentialEstimates_Diff_Eq})_2$ is as follows:
\begin{equation}\label{SectA_Difference_Eq2_2}
\begin{array}{ll}
\nabla^{\varphi^{\e}}\cdot \partial_t^{\ell}\mathcal{Z}^{\alpha}\hat{v} + [\partial_t^{\ell}\mathcal{Z}^{\alpha},\nabla^{\varphi^{\e}}\cdot] \hat{v}
- \partial_z^{\varphi}v \cdot\nabla^{\varphi^{\e}}\partial_t^{\ell}\mathcal{Z}^{\alpha}\hat{\eta}
- [\partial_t^{\ell}\mathcal{Z}^{\alpha},\partial_z^{\varphi}v \cdot\nabla^{\varphi^{\e}}]\hat{\eta} =0,
\\[12pt]

\nabla^{\varphi^{\e}}\cdot \partial_t^{\ell}\mathcal{Z}^{\alpha}\hat{v}
- \partial_z^{\varphi}v \cdot\nabla^{\varphi^{\e}}\partial_t^{\ell}\mathcal{Z}^{\alpha}\hat{\eta} \\[8pt]\quad
= -[\partial_t^{\ell}\mathcal{Z}^{\alpha},\nabla^{\varphi^{\e}}\cdot] \hat{v}
+ [\partial_t^{\ell}\mathcal{Z}^{\alpha},\partial_z^{\varphi}v \cdot\nabla^{\varphi^{\e}}]\hat{\eta},
\\[12pt]

\nabla^{\varphi^{\e}}\cdot (\partial_t^{\ell}\mathcal{Z}^{\alpha}\hat{v}
- \partial_z^{\varphi}v \partial_t^{\ell}\mathcal{Z}^{\alpha}\hat{\eta})
+ \partial_t^{\ell}\mathcal{Z}^{\alpha}\hat{\eta} \nabla^{\varphi^{\e}}\cdot \partial_z^{\varphi}v \\[8pt]\quad
= -[\partial_t^{\ell}\mathcal{Z}^{\alpha},\nabla^{\varphi^{\e}}\cdot] \hat{v}
+ [\partial_t^{\ell}\mathcal{Z}^{\alpha},\partial_z^{\varphi}v \cdot\nabla^{\varphi^{\e}}]\hat{\eta},
\end{array}
\end{equation}

\begin{equation*}
\begin{array}{ll}
\nabla^{\varphi^{\e}}\cdot \hat{V}^{\ell,\alpha}
= -[\partial_t^{\ell}\mathcal{Z}^{\alpha},\nabla^{\varphi^{\e}}\cdot] \hat{v}
+ [\partial_t^{\ell}\mathcal{Z}^{\alpha},\partial_z^{\varphi}v \cdot\nabla^{\varphi^{\e}}]\hat{\eta}
- \partial_t^{\ell}\mathcal{Z}^{\alpha}\hat{\eta} \nabla^{\varphi^{\e}}\cdot \partial_z^{\varphi}v,
\end{array}
\end{equation*}

Next, we derive the kinetical boundary condition $(\ref{Sect5_TangentialEstimates_Diff_Eq})_3$.
Apply $\partial_t^{\ell}\mathcal{Z}^{\alpha}$ to Navier-Stokes and Euler kinetical boundary conditions, we get
\begin{equation}\label{SectA_Difference_Eq2_3}
\begin{array}{ll}
\partial_t \partial_t^{\ell}\mathcal{Z}^{\alpha} h^{\e}
+ v_y^{\e} \cdot\nabla_y \partial_t^{\ell}\mathcal{Z}^{\alpha} h^{\e}
= \partial_t^{\ell}\mathcal{Z}^{\alpha}v^{\e}\cdot \NN^{\e} + [\partial_t^{\ell}\mathcal{Z}^{\alpha}, v^{\e},\NN^{\e}], \\[8pt]

\partial_t \partial_t^{\ell}\mathcal{Z}^{\alpha} h
+ v_y \cdot\nabla_y \partial_t^{\ell}\mathcal{Z}^{\alpha} h
= \partial_t^{\ell}\mathcal{Z}^{\alpha}v\cdot \NN + [\partial_t^{\ell}\mathcal{Z}^{\alpha}, v,\NN],
\end{array}
\end{equation}
then the kinetical boundary condition $(\ref{Sect5_TangentialEstimates_Diff_Eq})_3$ is derived as follows:
\begin{equation}\label{SectA_Difference_Eq2_4}
\begin{array}{ll}
\partial_t \partial_t^{\ell}\mathcal{Z}^{\alpha}\hat{h} + v_y^{\e}\cdot \nabla_y \partial_t^{\ell}\mathcal{Z}^{\alpha}\hat{h}
+ \hat{v}_y \cdot \nabla_y \partial_t^{\ell}\mathcal{Z}^{\alpha} h
 = \NN^{\e}\cdot \partial_t^{\ell}\mathcal{Z}^{\alpha}\hat{v}
- \partial_y \hat{h}\cdot \partial_t^{\ell}\mathcal{Z}^{\alpha}v_y \\[7pt]\quad
 + [\partial_t^{\ell}\mathcal{Z}^{\alpha}, \hat{v},\NN^{\e}]
 - [\partial_t^{\ell}\mathcal{Z}^{\alpha}, v_y, \partial_y \hat{h}], \\[12pt]

\partial_t \partial_t^{\ell}\mathcal{Z}^{\alpha}\hat{h} + v_y^{\e}\cdot \nabla_y \partial_t^{\ell}\mathcal{Z}^{\alpha}\hat{h}
- \NN^{\e}\cdot \partial_t^{\ell}\mathcal{Z}^{\alpha}\hat{v}
 = - \hat{v}_y \cdot \nabla_y \partial_t^{\ell}\mathcal{Z}^{\alpha} h
- \partial_y \hat{h}\cdot \partial_t^{\ell}\mathcal{Z}^{\alpha}v_y \\[7pt]\quad
 + [\partial_t^{\ell}\mathcal{Z}^{\alpha}, \hat{v},\NN^{\e}]
 - [\partial_t^{\ell}\mathcal{Z}^{\alpha}, v_y, \partial_y \hat{h}], \\[12pt]

\partial_t \partial_t^{\ell}\mathcal{Z}^{\alpha}\hat{h} + v_y^{\e}\cdot \nabla_y \partial_t^{\ell}\mathcal{Z}^{\alpha}\hat{h}
- \NN^{\e}\cdot \hat{V}^{\ell,\alpha}
= \NN^{\e}\cdot \partial_z^{\varphi} v \partial_t^{\ell}\mathcal{Z}^{\alpha}\hat{\eta} \\[8pt]\quad
 - \hat{v}_y \cdot \nabla_y \partial_t^{\ell}\mathcal{Z}^{\alpha} h
- \partial_y \hat{h}\cdot \partial_t^{\ell}\mathcal{Z}^{\alpha}v_y
 + [\partial_t^{\ell}\mathcal{Z}^{\alpha}, \hat{v},\NN^{\e}]
 - [\partial_t^{\ell}\mathcal{Z}^{\alpha}, v_y, \partial_y \hat{h}],
\end{array}
\end{equation}

Finally, we derive the dynamical boundary condition $(\ref{Sect5_TangentialEstimates_Diff_Eq})_4$.
Apply $\partial_t^{\ell}\mathcal{Z}^{\alpha}$ to Navier-Stokes and Euler dynamical boundary conditions, we get
\begin{equation}\label{SectA_Difference_Eq2_5}
\begin{array}{ll}
(\partial_t^{\ell}\mathcal{Z}^{\alpha}q^{\e} - g\partial_t^{\ell}\mathcal{Z}^{\alpha}h^{\e})\NN^{\e}
-2\e \mathcal{S}^{\varphi^{\e}} \partial_t^{\ell}\mathcal{Z}^{\alpha}v^{\e}\,\NN^{\e} \\[8pt]\quad
= 2\e [\partial_t^{\ell}\mathcal{Z}^{\alpha},\mathcal{S}^{\varphi^{\e}}] v^{\e}\,\NN^{\e}

+ (2\e \mathcal{S}^{\varphi^{\e}}v^{\e} - (q^{\e}-g h^{\e}))\,\partial_t^{\ell}\mathcal{Z}^{\alpha}\NN^{\e} \\[8pt]\qquad
- [\partial_t^{\ell}\mathcal{Z}^{\alpha},q^{\e}-g h^{\e},\NN^{\e}]
+2\e [\partial_t^{\ell}\mathcal{Z}^{\alpha},\mathcal{S}^{\varphi^{\e}}v^{\e}, \NN^{\e}], \\[11pt]

\partial_t^{\ell}\mathcal{Z}^{\alpha}q = g\partial_t^{\ell}\mathcal{Z}^{\alpha}h,
\end{array}
\end{equation}
then the dynamical boundary condition $(\ref{Sect5_TangentialEstimates_Diff_Eq})_4$ is derived as follows:
\begin{equation}\label{SectA_Difference_Eq2_6}
\begin{array}{ll}
(\partial_t^{\ell}\mathcal{Z}^{\alpha}\hat{q} - g\partial_t^{\ell}\mathcal{Z}^{\alpha}\hat{h})\NN^{\e}
-2\e \mathcal{S}^{\varphi^{\e}} \partial_t^{\ell}\mathcal{Z}^{\alpha}\hat{v}\,\NN^{\e} \\[8pt]\quad
= 2\e [\partial_t^{\ell}\mathcal{Z}^{\alpha},\mathcal{S}^{\varphi^{\e}}] v^{\e}\,\NN^{\e}

+ (2\e \mathcal{S}^{\varphi^{\e}}v^{\e} - 2\e \mathcal{S}^{\varphi^{\e}}v^{\e}\nn^{\e}\cdot\nn^{\e})\,\partial_t^{\ell}\mathcal{Z}^{\alpha}\NN^{\e} \\[8pt]\qquad
- [\partial_t^{\ell}\mathcal{Z}^{\alpha},2\e \mathcal{S}^{\varphi^{\e}}v^{\e}\nn^{\e}\cdot\nn^{\e},\NN^{\e}]
+2\e [\partial_t^{\ell}\mathcal{Z}^{\alpha},\mathcal{S}^{\varphi^{\e}}v^{\e}, \NN^{\e}]
+ 2\e \mathcal{S}^{\varphi^{\e}} \partial_t^{\ell}\mathcal{Z}^{\alpha}v\,\NN^{\e}, \\[12pt]

\hat{Q}^{\ell,\alpha}\NN^{\e}

-2\e \mathcal{S}^{\varphi^{\e}} \partial_t^{\ell}\mathcal{Z}^{\alpha}\hat{v}\,\NN^{\e}
- (g-\partial_z^{\varphi}q)\partial_t^{\ell}\mathcal{Z}^{\alpha}\hat{h} \,\NN^{\e}
\\[8pt]\quad
= 2\e [\partial_t^{\ell}\mathcal{Z}^{\alpha},\mathcal{S}^{\varphi^{\e}}] v^{\e}\,\NN^{\e}

+ (2\e \mathcal{S}^{\varphi^{\e}}v^{\e} - 2\e \mathcal{S}^{\varphi^{\e}}v^{\e}\nn^{\e}\cdot\nn^{\e})\,\partial_t^{\ell}\mathcal{Z}^{\alpha}\NN^{\e} \\[8pt]\qquad
- [\partial_t^{\ell}\mathcal{Z}^{\alpha},2\e \mathcal{S}^{\varphi^{\e}}v^{\e}\nn^{\e}\cdot\nn^{\e},\NN^{\e}]
+2\e [\partial_t^{\ell}\mathcal{Z}^{\alpha},\mathcal{S}^{\varphi^{\e}}v^{\e}, \NN^{\e}]
+ 2\e \mathcal{S}^{\varphi^{\e}} \partial_t^{\ell}\mathcal{Z}^{\alpha}v\,\NN^{\e},
\end{array}
\end{equation}
Thus, Lemma $\ref{SectA_DifferenceEq2_Lemma}$ is proved.
\end{proof}

\begin{lemma}\label{SectA_DifferenceEq2_Lemma_Time}
Assume $0\leq \ell \leq k-1, |\alpha|=0$,
let $\hat{V}^{\ell,0} = \partial_t^{\ell}\hat{v} - \partial_z^{\varphi}v \partial_t^{\ell}\hat{\varphi}$,
then the main equation of $\hat{V}^{\ell,0}$ and its kinetical boundary condition satisfy $(\ref{Sect5_Tangential_Estimates_Time})$.
\end{lemma}

\begin{proof}
The derivation of the main equation of $\hat{V}^{\ell,0}$ is as follows:
\begin{equation}\label{SectA_DifferenceEq2_Lemma_Time_1}
\begin{array}{ll}
\partial_t^{\varphi^{\e}} (\partial_t^{\ell}\hat{v}
-\partial_z^{\varphi} v \partial_t^{\ell}\hat{\varphi})
+ \partial_t^{\ell}\hat{\varphi}\partial_t^{\varphi^{\e}}\partial_z^{\varphi} v

+ v^{\e} \cdot\nabla^{\varphi^{\e}}(\partial_t^{\ell}\hat{v}
- \partial_z^{\varphi} v \partial_t^{\ell}\hat{\varphi}) 
+ \partial_t^{\ell}\hat{v}\cdot\nabla^{\varphi} v\\[8pt]\quad
+  \partial_t^{\ell}\hat{\varphi}\, v^{\e}\cdot \nabla^{\varphi^{\e}}\partial_z^{\varphi} v

+ \partial_t^{\ell}\nabla^{\varphi^{\e}}\hat{q}
- \partial_z^{\varphi} q\nabla^{\varphi^{\e}} \partial_t^{\ell}\hat{\varphi}
- 2\e\partial_t^{\ell}\nabla^{\varphi^{\e}}\cdot\mathcal{S}^{\varphi^{\e}} \hat{v} \\[8pt]\quad

= \e\partial_t^{\ell}\triangle^{\varphi^\e} v^\e
- [\partial_t^{\ell},\partial_t^{\varphi^{\e}}]\hat{v}
+ [\partial_t^{\ell}, \partial_z^{\varphi} v \partial_t^{\varphi^{\e}}]\hat{\varphi}

- [\partial_t^{\ell}, v^{\e} \cdot\nabla^{\varphi^{\e}}] \hat{v} \\[8pt]\quad
+ [\partial_t^{\ell}, \partial_z^{\varphi} v \, v^{\e}\cdot \nabla^{\varphi^{\e}}]\hat{\varphi}
- [\partial_t^{\ell}, \nabla^{\varphi} v\cdot]\hat{v}

+ [\partial_t^{\ell},\partial_z^{\varphi} q\nabla^{\varphi^{\e}}]\hat{\varphi}, \\[12pt]

\partial_t^{\varphi^{\e}} \hat{V}^{\ell,0}
+ v^{\e} \cdot\nabla^{\varphi^{\e}}\hat{V}^{\ell,0}
- 2\e\nabla^{\varphi^{\e}}\cdot\mathcal{S}^{\varphi^{\e}} \partial_t^{\ell}\hat{v} 
= \mathcal{I}_5.
\end{array}
\end{equation}

The derivation of the kinetical boundary condition is the same as Lemma $\ref{SectA_DifferenceEq2_Lemma}$, 
but let $\alpha =0$ in $(\ref{Sect5_TangentialEstimates_Diff_Eq})_3$.
Thus, Lemma $\ref{SectA_DifferenceEq2_Lemma_Time}$ is proved.
\end{proof}

\begin{lemma}\label{SectA_Vorticity_Eq_Lemma}
$\hat{\omega}_h =\omega_h^{\e} -\omega_h$ satisfies the equations $(\ref{Sect1_N_Derivatives_Difference_Eq})$.
\end{lemma}

\begin{proof}

By Lemma $\ref{Sect2_Vorticity_H_Eq_BC_Lemma}$ and $(\ref{Sect2_Vorticity_Estimate_4})$, the tangential components of Navier-Stokes vorticity $\omega_h^{\e}$ satisfies
\begin{equation}\label{SectA_Vorticity_Lemma_Eq_1}
\left\{\begin{array}{ll}
\partial_t^{\varphi^{\e}} \omega_h^{\e} + v^{\e}\cdot\nabla^{\varphi^{\e}}\omega_h^{\e}
- \e\triangle^{\varphi^{\e}}\omega_h^{\e} = \vec{\textsf{F}}^0[\nabla\varphi^{\e}](\omega_h^{\e},\partial_j v^{\e,i}),
\\[7pt]

\omega^{\e,1}|_{z=0} =\textsf{F}^1 [\nabla\varphi^{\e}](\partial_j v^{\e,i}),
\\[7pt]

\omega^{\e,2}|_{z=0} =\textsf{F}^2 [\nabla\varphi^{\e}](\partial_j v^{\e,i}),
\end{array}\right.
\end{equation}

Similar to the arguments in $(\ref{Sect2_Vorticity_H_Eq_1}), (\ref{Sect2_Vorticity_H_Eq_2})$,
the tangential components of Euler vorticity $\omega_h^{\e}$ satisfies
\begin{equation}\label{SectA_Vorticity_Lemma_Eq_2}
\left\{\begin{array}{ll}
\partial_t^{\varphi} \omega_h + v\cdot\nabla^{\varphi}\omega_h
= \vec{\textsf{F}}^0[\nabla\varphi](\omega_h,\partial_j v^i),
\\[7pt]

\omega^1|_{z=0} = \partial_2^{\varphi} v^3 - \partial_z^{\varphi} v^2
= \partial_2 v^3 -\frac{\partial_2\varphi}{\partial_z\varphi}\partial_z v^3 - \frac{1}{\partial_z\varphi}\partial_z v^2 := \omega^{b,1},
\\[7pt]

\omega^2|_{z=0} = \partial_z^{\varphi} v^1 - \partial_1^{\varphi} v^3
= \frac{1}{\partial_z\varphi}\partial_z v^1 - \partial_1 v^3 +\frac{\partial_1\varphi}{\partial_z\varphi}\partial_z v^3 := \omega^{b,2},
\end{array}\right.
\end{equation}

By $(\ref{SectA_Vorticity_Lemma_Eq_1})-(\ref{SectA_Vorticity_Lemma_Eq_2})$, we get
\begin{equation}\label{SectA_Vorticity_Lemma_Eq_3}
\left\{\begin{array}{ll}
\partial_t^{\varphi^{\e}} \hat{\omega}_h - \partial_z^{\varphi} \omega_h \partial_t^{\varphi^{\e}} \hat{\eta}

+ v^{\e}\cdot\nabla^{\varphi^{\e}}\hat{\omega}_h - \partial_z^{\varphi} \omega_h v^{\e}\cdot\nabla^{\varphi^{\e}}\hat{\eta}
+ \hat{v}\cdot\nabla^{\varphi}\omega_h

- \e\triangle^{\varphi^{\e}}\hat{\omega}_h \\[6pt]\quad
= \vec{\textsf{F}}^0[\nabla\varphi^{\e}](\omega_h^{\e},\partial_j v^{\e,i})
- \vec{\textsf{F}}^0[\nabla\varphi](\omega_h,\partial_j v^i) + \e\triangle^{\varphi^{\e}}\omega_h, \\[7pt]

\hat{\omega}^1|_{z=0} =\textsf{F}^1 [\nabla\varphi^{\e}](\partial_j v^{\e,i}) - \omega^{b,1}, \\[6pt]

\hat{\omega}^2|_{z=0} =\textsf{F}^2 [\nabla\varphi^{\e}](\partial_j v^{\e,i}) - \omega^{b,2}.
\end{array}\right.
\end{equation}
Thus, Lemma $\ref{SectA_Vorticity_Eq_Lemma}$ is proved.
\end{proof}

%%% find 9
\section{Derivation of the Equations for the Surface Tension}

In this appendix, we derive the equations and their boundary conditions for the $\sigma>0$ case.

\begin{lemma}\label{SectB_DifferenceEq1_Corollary}
$(\hat{v} = v^{\e} -v,\ \hat{h} =h^{\e} -h,\ \hat{q} = q^{\e} -q)$ satisfy the equations $(\ref{Sect7_DifferenceEq_1})$.
\end{lemma}

\begin{proof}
The surface tension term appear in the dynamical boundary condition, thus we only need to derive the difference equation of the
dynamical boundary condition, other equations and the kinetical boundary condition are the same with the $\sigma=0$ case,
see $(\ref{Sect1_T_Derivatives_Difference_Eq})$.

We derive the difference equation of the dynamical boundary condition.
The dynamical boundary condition for the Euler equation with $\sigma>0$
is a scalar equation, that is $q= gh -\sigma H$. For any vector such as $\NN^{\e}$,
$q \NN^{\e}= gh \NN^{\e} -\sigma H \NN^{\e}$ makes sense.

Denote $\hat{H} = H^{\e} -H$, it follows from the dynamical boundary condition that
\begin{equation}\label{SectB_DifferenceEq1_Corollary_1}
\begin{array}{ll}
q^\e \NN^{\e} -q\NN^{\e} -2\e \mathcal{S}^{\varphi^\e}v^{\e}\,\NN^\e
=gh^{\e} \NN^{\e} -gh \NN^{\e} -\sigma H^{\e} \NN^{\e} + \sigma H \NN^{\e}, \\[7pt]

\hat{q}\NN^{\e} -2\e \mathcal{S}^{\varphi^\e}\hat{v}\,\NN^\e
=(g \hat{h} -\sigma \hat{H})\NN^{\e}
+ 2\e \mathcal{S}^{\varphi^\e}v\,\NN^\e.
\end{array}
\end{equation}

$\hat{v},\hat{h},\hat{q}$ satisfy the following equations
\begin{equation}\label{SectB_DifferenceEq1_Corollary_2}
\left\{\begin{array}{ll}
\partial_t^{\varphi^{\e}}\hat{v}-\partial_z^{\varphi} v \partial_t^{\varphi^{\e}}\hat{\eta}
+ v^{\e} \cdot\nabla^{\varphi^{\e}} \hat{v} - v^{\e}\cdot \nabla^{\varphi^{\e}}\hat{\eta}\, \partial_z^{\varphi} v + \hat{v}\cdot\nabla^{\varphi} v \\[6pt]\quad
+ \nabla^{\varphi^{\e}} \hat{q} - \partial_z^{\varphi} q\nabla^{\varphi^{\e}}\hat{\eta}
= 2\e \nabla^{\varphi^{\e}} \cdot\mathcal{S}^{\varphi^{\e}} \hat{v} 
+ \e\triangle^{\varphi^{\e}} v, \hspace{2cm}  x\in\mathbb{R}^3_{-}, \\[7pt]

\nabla^{\varphi^{\e}}\cdot \hat{v} - \partial_z^{\varphi}v \cdot\nabla^{\varphi^{\e}}\hat{\eta} =0, \hspace{5.27cm} x\in\mathbb{R}^3_{-},\\[7pt]

\partial_t \hat{h} + v_y\cdot \nabla \hat{h} = \hat{v}\cdot\NN^{\e},  \hspace{5.66cm} \{z=0\},
\\[6pt]

\hat{q}\NN^{\e} -2\e \mathcal{S}^{\varphi^\e}\hat{v}\,\NN^\e
=(g \hat{h} -\sigma \hat{H})\NN^{\e} + 2\e \mathcal{S}^{\varphi^\e}v\,\NN^\e, \hspace{1.88cm} \{z=0\},
\\[6pt]

(\hat{v},\hat{h})|_{t=0} = (v_0^\e -v_0,h_0^\e -h_0),
\end{array}\right.
\end{equation}
where
\begin{equation}\label{SectB_DifferenceEq1_Corollary_3}
\begin{array}{ll}
\hat{H} = \nabla_y \cdot \Big(\frac{\nabla_y h^{\e}}{\sqrt{1+|\nabla_y h^{\e}|^2}}- \frac{\nabla_y h}{\sqrt{1+|\nabla_y h|^2}}\Big) \\[9pt]\quad

= \nabla_y \cdot \Big(\nabla_y \hat{h} \big(\frac{1}{2\sqrt{1+|\nabla_y h^{\e}|^2}}+ \frac{1}{2\sqrt{1+|\nabla_y h|^2}}\big)\Big) \\[11pt]\qquad
+ \nabla_y \cdot \Big(\frac{\sqrt{1+|\nabla_y h|^2} - \sqrt{1+|\nabla_y h^{\e}|^2}}{\sqrt{1+|\nabla_y h^{\e}|^2}\sqrt{1+|\nabla_y h|^2}}\cdot
\frac{\nabla_y h^{\e} +\nabla_y h}{2}\Big) \\[10pt]\quad

= \nabla_y \cdot \Big(\nabla_y\hat{h} \big(\frac{1}{2\sqrt{1+|\nabla_y h^{\e}|^2}}+ \frac{1}{2\sqrt{1+|\nabla_y h|^2}}\big)\Big) \\[10pt]\qquad
- \nabla_y \cdot \Big(\big(\frac{\nabla_y\hat{h}\cdot\nabla_y(h^{\e}+h)}
{2\sqrt{1+|\nabla_y h^{\e}|^2}\sqrt{1+|\nabla_y h|^2}(\sqrt{1+|\nabla_y h^{\e}|^2} + \sqrt{1+|\nabla_y h|^2})}\big)
\nabla_y (h^{\e} + h)\Big).
\end{array}
\end{equation}

Plug $(\ref{SectB_DifferenceEq1_Corollary_3})$ into $(\ref{SectB_DifferenceEq1_Corollary_2})$, we obtain $(\ref{Sect7_DifferenceEq_1})$
and $(\ref{Sect7_DifferenceEq_2})$.
Thus, Lemma $\ref{SectB_DifferenceEq1_Corollary}$ is proved.
\end{proof}

\begin{lemma}\label{SectB_DifferenceEq2_Corollary}
Assume $0\leq \ell +|\alpha|\leq k, \ 0\leq \ell\leq k-1, \ |\alpha|\geq 1$,
then $(\partial_t^{\ell}\mathcal{Z}^{\alpha}\hat{v},\ \partial_t^{\ell}\mathcal{Z}^{\alpha}\hat{h},\ \partial_t^{\ell}\mathcal{Z}^{\alpha}\hat{q})$
satisfy the equations $(\ref{Sect7_TangentialEstimates_Diff_Eq})$.
\end{lemma}

\begin{proof}
Apply $\partial_t^{\ell}\mathcal{Z}^{\alpha}$ to the equations $(\ref{Sect7_DifferenceEq_1})$, we prove $(\ref{Sect7_TangentialEstimates_Diff_Eq})$.
The main equation $(\ref{Sect7_TangentialEstimates_Diff_Eq})_1$ follows from $(\ref{SectA_Difference_Eq2_1})_2$,
\begin{equation}\label{SectB_DifferenceEq2_Corollary_1}
\begin{array}{ll}
\partial_t^{\varphi^{\e}} \partial_t^{\ell}\mathcal{Z}^{\alpha}\hat{v}
+ v^{\e} \cdot\nabla^{\varphi^{\e}} \partial_t^{\ell}\mathcal{Z}^{\alpha}\hat{v}
+ \nabla^{\varphi^{\e}} \partial_t^{\ell}\mathcal{Z}^{\alpha}\hat{q} 
- 2\e\partial_t^{\ell}\mathcal{Z}^{\alpha}\nabla^{\varphi^{\e}}\cdot\mathcal{S}^{\varphi^{\e}} \hat{v} \\[8pt]\quad

= \partial_z^{\varphi} v \partial_t^{\varphi^{\e}} \partial_t^{\ell}\mathcal{Z}^{\alpha}\hat{\varphi}
+ \partial_z^{\varphi} v \, v^{\e}\cdot \nabla^{\varphi^{\e}}\partial_t^{\ell}\mathcal{Z}^{\alpha}\hat{\varphi}
- \partial_t^{\ell}\mathcal{Z}^{\alpha}\hat{v}\cdot\nabla^{\varphi} v
+ \partial_z^{\varphi} q\nabla^{\varphi^{\e}}\partial_t^{\ell}\mathcal{Z}^{\alpha}\hat{\varphi}
\\[8pt]\quad

+ \e\partial_t^{\ell}\mathcal{Z}^{\alpha}\triangle^{\varphi^\e} v
- [\partial_t^{\ell}\mathcal{Z}^{\alpha},\partial_t^{\varphi^{\e}}]\hat{v}
+ [\partial_t^{\ell}\mathcal{Z}^{\alpha}, \partial_z^{\varphi} v \partial_t^{\varphi^{\e}}]\hat{\varphi}

- [\partial_t^{\ell}\mathcal{Z}^{\alpha}, v^{\e} \cdot\nabla^{\varphi^{\e}}] \hat{v} \\[8pt]\quad
+ [\partial_t^{\ell}\mathcal{Z}^{\alpha}, \partial_z^{\varphi} v \, v^{\e}\cdot \nabla^{\varphi^{\e}}]\hat{\varphi}
- [\partial_t^{\ell}\mathcal{Z}^{\alpha}, \nabla^{\varphi} v\cdot]\hat{v}

- [\partial_t^{\ell}\mathcal{Z}^{\alpha},\nabla^{\varphi^{\e}}] \hat{q}
+ [\partial_t^{\ell}\mathcal{Z}^{\alpha},\partial_z^{\varphi} q\nabla^{\varphi^{\e}}]\hat{\varphi}, \\[13pt]

\partial_t^{\varphi^{\e}} \partial_t^{\ell}\mathcal{Z}^{\alpha}\hat{v}
+ v^{\e} \cdot\nabla^{\varphi^{\e}} \partial_t^{\ell}\mathcal{Z}^{\alpha}\hat{v}
+ \nabla^{\varphi^{\e}} \partial_t^{\ell}\mathcal{Z}^{\alpha}\hat{q}
- 2\e \nabla^{\varphi^{\e}}\cdot\mathcal{S}^{\varphi^{\e}} \partial_t^{\ell}\mathcal{Z}^{\alpha}\hat{v} 
= \mathcal{I}_{10}.
\end{array}
\end{equation}

The divergence free condition $(\ref{Sect7_TangentialEstimates_Diff_Eq})_2$ follows from $(\ref{SectA_Difference_Eq2_2})$,
\begin{equation}\label{SectB_DifferenceEq2_Corollary_2}
\begin{array}{ll}
\nabla^{\varphi^{\e}}\cdot \partial_t^{\ell}\mathcal{Z}^{\alpha}\hat{v}
= \partial_z^{\varphi}v \cdot\nabla^{\varphi^{\e}}\partial_t^{\ell}\mathcal{Z}^{\alpha}\hat{\eta}
-[\partial_t^{\ell}\mathcal{Z}^{\alpha},\nabla^{\varphi^{\e}}\cdot] \hat{v}
+ [\partial_t^{\ell}\mathcal{Z}^{\alpha},\partial_z^{\varphi}v \cdot\nabla^{\varphi^{\e}}]\hat{\eta}.
\end{array}
\end{equation}

Next, the kinetical boundary condition $(\ref{Sect7_TangentialEstimates_Diff_Eq})_3$ is exactly $(\ref{SectA_Difference_Eq2_4})_2$.

Finally, we derive the dynamical boundary condition $(\ref{Sect7_TangentialEstimates_Diff_Eq})_4$.
Apply $\partial_t^{\ell}\mathcal{Z}^{\alpha}$ to Navier-Stokes and Euler dynamical boundary conditions, we get
\begin{equation}\label{SectB_DifferenceEq2_Corollary_3}
\begin{array}{ll}
(\partial_t^{\ell}\mathcal{Z}^{\alpha}q^{\e} - g\partial_t^{\ell}\mathcal{Z}^{\alpha}h^{\e}
+ \sigma \partial_t^{\ell}\mathcal{Z}^{\alpha}H^{\e})\NN^{\e}
-2\e \mathcal{S}^{\varphi^{\e}} \partial_t^{\ell}\mathcal{Z}^{\alpha}v^{\e}\,\NN^{\e} \\[8pt]\quad
= 2\e [\partial_t^{\ell}\mathcal{Z}^{\alpha},\mathcal{S}^{\varphi^{\e}}] v^{\e}\,\NN^{\e}

+ (2\e \mathcal{S}^{\varphi^{\e}}v^{\e} - (q^{\e}-g h^{\e} 
+ \sigma H^{\e}))\,\partial_t^{\ell}\mathcal{Z}^{\alpha}\NN^{\e} \\[8pt]\qquad
- [\partial_t^{\ell}\mathcal{Z}^{\alpha},q^{\e}-g h^{\e} + \sigma H^{\e},\NN^{\e}]
+2\e [\partial_t^{\ell}\mathcal{Z}^{\alpha},\mathcal{S}^{\varphi^{\e}}v^{\e}, \NN^{\e}], \\[9pt]

(\partial_t^{\ell}\mathcal{Z}^{\alpha}q - g\partial_t^{\ell}\mathcal{Z}^{\alpha}h + \sigma \partial_t^{\ell}\mathcal{Z}^{\alpha}H)\NN^{\e} =0.
\end{array}
\end{equation}

By $(\ref{SectB_DifferenceEq2_Corollary_3})_1 -(\ref{SectB_DifferenceEq2_Corollary_3})_2$, we get
\begin{equation}\label{SectB_DifferenceEq2_Corollary_4}
\begin{array}{ll}
(\partial_t^{\ell}\mathcal{Z}^{\alpha}\hat{q} - g\partial_t^{\ell}\mathcal{Z}^{\alpha}\hat{h}
+ \sigma \partial_t^{\ell}\mathcal{Z}^{\alpha}\hat{H})\NN^{\e}
- 2\e \mathcal{S}^{\varphi^{\e}} \partial_t^{\ell}\mathcal{Z}^{\alpha}\hat{v} \,\NN^{\e} \\[8pt]\quad

= 2\e [\partial_t^{\ell}\mathcal{Z}^{\alpha},\mathcal{S}^{\varphi^{\e}}] v^{\e}\,\NN^{\e}

+ (2\e \mathcal{S}^{\varphi^{\e}}v^{\e} - (q^{\e}-g h^{\e}
+ \sigma H^{\e}))\,\partial_t^{\ell}\mathcal{Z}^{\alpha}\NN^{\e} \\[8pt]\qquad
- [\partial_t^{\ell}\mathcal{Z}^{\alpha},q^{\e}-g h^{\e} + \sigma H^{\e},\NN^{\e}]
+2\e [\partial_t^{\ell}\mathcal{Z}^{\alpha},\mathcal{S}^{\varphi^{\e}}v^{\e}, \NN^{\e}]
+ 2\e \mathcal{S}^{\varphi^{\e}} \partial_t^{\ell}\mathcal{Z}^{\alpha}v\,\NN^{\e},
\end{array}
\end{equation}
and then we calculate $\partial_t^{\ell}\mathcal{Z}^{\alpha} \hat{H}$,
\begin{equation}\label{SectB_DifferenceEq2_Corollary_5}
\begin{array}{ll}
\partial_t^{\ell}\mathcal{Z}^{\alpha} \hat{H} = \partial_t^{\ell}\mathcal{Z}^{\alpha} \nabla_y \cdot \big[ \mathfrak{H}_1\nabla_y\hat{h}
+ \mathfrak{H}_2 \nabla_y\hat{h}\cdot\nabla_y(h^{\e}+h) \nabla_y (h^{\e} +h)\big] \\[6pt]

= \nabla_y \cdot \big[ \mathfrak{H}_1\nabla_y \partial_t^{\ell}\mathcal{Z}^{\alpha} \hat{h}
+ \mathfrak{H}_2 \nabla_y \partial_t^{\ell}\mathcal{Z}^{\alpha} \hat{h}\cdot\nabla_y(h^{\e}+h) \nabla_y (h^{\e} +h)\big] \\[6pt]\quad

+ \nabla_y \cdot \big[ [\partial_t^{\ell}\mathcal{Z}^{\alpha}, \mathfrak{H}_1\nabla_y] \hat{h}
+ [\partial_t^{\ell}\mathcal{Z}^{\alpha}, \mathfrak{H}_2 \nabla_y (h^{\e} +h) \nabla_y(h^{\e}+h)\cdot \nabla_y] \hat{h} \big].
\end{array}
\end{equation}

Plug $(\ref{SectB_DifferenceEq2_Corollary_5})$ into $(\ref{SectB_DifferenceEq2_Corollary_4})$, we get
\begin{equation}\label{SectB_DifferenceEq2_Corollary_6}
\begin{array}{ll}
\partial_t^{\ell}\mathcal{Z}^{\alpha} \hat{q}\NN^{\e}
-2\e \mathcal{S}^{\varphi^{\e}} \partial_t^{\ell}\mathcal{Z}^{\alpha} \hat{v}\,\NN^{\e}

= g\partial_t^{\ell}\mathcal{Z}^{\alpha}\hat{h}\NN^{\e}
- \sigma \nabla_y \cdot \big( \mathfrak{H}_1\nabla_y \partial_t^{\ell}\mathcal{Z}^{\alpha} \hat{h} \big) \NN^{\e} \\[8pt]\quad
- \sigma\nabla_y \cdot \big( \mathfrak{H}_2 \nabla_y \partial_t^{\ell}\mathcal{Z}^{\alpha} \hat{h}\cdot\nabla_y(h^{\e}+h) \nabla_y (h^{\e} +h)\big)\NN^{\e}
+ 2\e [\partial_t^{\ell}\mathcal{Z}^{\alpha},\mathcal{S}^{\varphi^{\e}}] v^{\e}\,\NN^{\e} \\[6pt]\quad

+ 2\e(\mathcal{S}^{\varphi^{\e}}v^{\e} - \mathcal{S}^{\varphi^{\e}}v^{\e}\nn^{\e}\cdot\nn^{\e})\,\partial_t^{\ell}\mathcal{Z}^{\alpha}\NN^{\e}
+ 2\e[\partial_t^{\ell}\mathcal{Z}^{\alpha}, \mathcal{S}^{\varphi^{\e}}v^{\e} - \mathcal{S}^{\varphi^{\e}}v^{\e}\nn^{\e}\cdot\nn^{\e}, \NN^{\e}]
\end{array}
\end{equation}

\begin{equation*}
\begin{array}{ll}
\quad 
+ 2\e \mathcal{S}^{\varphi^{\e}} \partial_t^{\ell}\mathcal{Z}^{\alpha}v\,\NN^{\e}
-  \sigma\nabla_y \cdot \big[ [\partial_t^{\ell}\mathcal{Z}^{\alpha}, \mathfrak{H}_1\nabla_y]\big] \hat{h}\NN^{\e}
\\[6pt]\quad
- \sigma\nabla_y \cdot \big[[\partial_t^{\ell}\mathcal{Z}^{\alpha}, \mathfrak{H}_2 \nabla_y (h^{\e} +h) \nabla_y(h^{\e}+h)\cdot \nabla_y] \hat{h} \big]\NN^{\e}.
\hspace{3cm}
\end{array}
\end{equation*}
Thus, Lemma $\ref{SectB_DifferenceEq2_Corollary}$ is proved.
\end{proof}

\begin{lemma}\label{SectB_DifferenceEq3_Corollary}
Assume $0\leq \ell\leq k-1, \ |\alpha|=0$,
then the main equation of $\partial_t^{\ell}\hat{v}$ and its kinetical boundary condition satisfy the equations $(\ref{Sect7_Tangential_Estimates_Time})$.
\end{lemma}

\begin{proof}
The derivation of the main equation of $\partial_t^{\ell} \hat{v}$ is as follows:
\begin{equation}\label{SectB_DifferenceEq3_Corollary_1}
\begin{array}{ll}
\partial_t^{\varphi^{\e}} \partial_t^{\ell}\hat{v}
+ v^{\e} \cdot\nabla^{\varphi^{\e}} \partial_t^{\ell}\hat{v}
+ \nabla^{\varphi^{\e}} \partial_t^{\ell}\hat{q}
- 2\e\partial_t^{\ell}\nabla^{\varphi^{\e}}\cdot\mathcal{S}^{\varphi^{\e}} \hat{v} \\[8pt]\quad

= \partial_z^{\varphi} v \partial_t^{\varphi^{\e}} \partial_t^{\ell}\hat{\varphi}
+ \partial_z^{\varphi} v \, v^{\e}\cdot \nabla^{\varphi^{\e}}\partial_t^{\ell}\hat{\varphi}
- \partial_t^{\ell}\hat{v}\cdot\nabla^{\varphi} v
+ \partial_z^{\varphi} q\nabla^{\varphi^{\e}}\partial_t^{\ell}\hat{\varphi}
\\[8pt]\quad

+ \e\partial_t^{\ell}\triangle^{\varphi^\e} v
- [\partial_t^{\ell},\partial_t^{\varphi^{\e}}]\hat{v}
+ [\partial_t^{\ell}, \partial_z^{\varphi} v \partial_t^{\varphi^{\e}}]\hat{\varphi}

- [\partial_t^{\ell}, v^{\e} \cdot\nabla^{\varphi^{\e}}] \hat{v} \\[8pt]\quad
+ [\partial_t^{\ell}, \partial_z^{\varphi} v \, v^{\e}\cdot \nabla^{\varphi^{\e}}]\hat{\varphi}
- [\partial_t^{\ell}, \nabla^{\varphi} v\cdot]\hat{v}

- [\partial_t^{\ell},\nabla^{\varphi^{\e}}] \hat{q}
+ [\partial_t^{\ell},\partial_z^{\varphi} q\nabla^{\varphi^{\e}}]\hat{\varphi}, \\[11pt]

\partial_t^{\varphi^{\e}} \partial_t^{\ell}\hat{v}
+ v^{\e} \cdot\nabla^{\varphi^{\e}} \partial_t^{\ell}\hat{v}
- 2\e\partial_t^{\ell}\nabla^{\varphi^{\e}}\cdot\mathcal{S}^{\varphi^{\e}} \hat{v} = \mathcal{I}_{12}.
\end{array}
\end{equation}

The derivation of the kinetical boundary condition is the same as Lemma $\ref{SectA_DifferenceEq2_Lemma}$,
but let $\alpha =0$ in $(\ref{Sect7_TangentialEstimates_Diff_Eq})_3$.
Thus, Lemma $\ref{SectB_DifferenceEq3_Corollary}$ is proved.
\end{proof}

%%% find 10
\section*{Acknowledgements}
%The author is grateful to anonymous referees for their many helpful suggestions.
%The author is thankful for the support of the Center of Mathematical Sciences and Applications, Harvard University.
%This paper is supported by the scholarship of Chinese Scholarship Council (No. 201500090074).

The author is extremely grateful to Prof. S. T. Yau for all his help.
This research is conducted under the supervision of Prof. S. T. Yau and supported by the Center of Mathematical Sciences and Applications, Harvard University.
This research is also supported by the scholarship of Chinese Scholarship Council (No. 201500090074).

%\newpage
\bibliographystyle{siam}
\addcontentsline{toc}{section}{References}
\bibliography{FuzhouWu_NavierStokesEq}

\end{document}